\documentclass[10pt,a4paper]{amsart}

\usepackage[left=2.8cm, right=2.8cm, top=2.2cm, bottom=2.2cm]{geometry}

\usepackage{amsmath,amssymb,amsthm,mathrsfs,stmaryrd,latexsym}
\usepackage{amsfonts}
\usepackage{accents}

\usepackage{tcolorbox}
\tcbuselibrary{skins, breakable, theorems}

\usepackage{enumerate} 
\usepackage{tikz}
\usepackage{etoolbox}

\def\Luoma#1{\uppercase\expandafter{\romannumeral#1}}
\def\luoma#1{\romannumeral#1}

\usepackage{url}

\newtheorem{mythm}{Theorem}[section]
\newtheorem{mylem}[mythm]{Lemma}
\newtheorem{myprop}[mythm]{Proposition}
\newtheorem{mycor}[mythm]{Corollary}
\newtheorem{myques}[mythm]{Question}
\newtheorem{myconj}[mythm]{Conjecture}

\theoremstyle{definition}
\newtheorem{mydefn}[mythm]{Definition}
\newtheorem{myexample}[mythm]{Example}

\theoremstyle{remark}
\newtheorem{myrem}[mythm]{Remark}
\newtheorem{mypara}[mythm]{}

\usepackage[all]{xy}
\usepackage{graphicx}
\usepackage{subfigure}

\newcommand{\bb}{\mathbb}
\newcommand{\ca}{\mathcal}
\newcommand{\ak}{\mathfrak}
\newcommand{\scr}{\mathscr}

\newcommand{\mbf}{\mathbf}
\newcommand{\mrm}{\mathrm}

\newcommand{\trm}{\textrm}

\def\op#1{\mathop{\mathrm{#1}}}

\newcommand{\mo}{\mrm{Mor}}
\newcommand{\ho}{\mrm{Hom}}

\newcommand{\ke}{\mrm{Ker}}
\newcommand{\cok}{\mrm{Coker}}

\newcommand{\df}{\mrm{d}}

\newcommand{\spec}{\op{Spec}}
\newcommand{\colim}{\op{colim}}

\newcommand{\ob}{\mrm{Ob}}
\newcommand{\rr}{\mrm{R}}

\newcommand{\iso}{\stackrel{\sim}{\longrightarrow}}

\newcommand{\surj}{\twoheadrightarrow}

\newcommand{\spa}{\op{Spa}}

\newcommand{\oppo}{\mrm{op}}
\newcommand{\al}{\mrm{al}}

\newcommand{\triv}{\mrm{tr}}

\newcommand{\gal}{\mrm{Gal}}

\newcommand{\dd}{\mbf{D}}

\newcommand{\module}{\trm{-}\mbf{Mod}}

\newcommand{\pro}{\mbf{Pro}}

\newcommand{\et}{\mrm{\acute{e}t}}
\newcommand{\fet}{\mrm{f\acute{e}t}}

\newcommand{\proet}{\mrm{pro\acute{e}t}}
\newcommand{\profet}{\mrm{prof\acute{e}t}}

\newcommand{\fal}{\mbf{E}}

\newcommand{\falb}{\overline{\scr{B}}}

\newcommand{\rz}{\mrm{RZ}}
\newcommand{\modf}{\scr{MD}}
\newcommand{\val}{\mrm{Val}}
\newcommand{\supprz}{\mrm{Supp}^{\rz}}

\def\ff#1{\scr{F}^{\mrm{fini}}_{\overline{#1}/#1}}

\title[Purity for Perfectoidness]{Purity for Perfectoidness}%
\author{Tongmu He}
\date{\today}
\address{Tongmu He, Institute for Advanced Study, 1 Einstein Drive, 08540 New Jersey, the United States}
\email{hetm@ias.edu}

\usepackage{hyperref}
\hypersetup{colorlinks,
	citecolor=black,
	filecolor=black,
	linkcolor=black,
	urlcolor=black,
	pdftitle={Purity for Perfectoidness},
	pdfauthor={Tongmu He},
	pdftex}

\numberwithin{equation}{mythm}

\begin{document}
	\maketitle
	
\begin{abstract}
	There has been a long-standing question about whether being perfectoid for an algebra is local in the analytic topology. We provide affirmative answers for the algebras (e.g., over $\overline{\mathbb{Z}_p}$) whose spectra are inverse limits of semi-stable affine schemes. In fact, we established a valuative criterion for such an algebra being perfectoid, saying that it suffices to check the perfectoidness of the stalks of the associated Riemann-Zariski space. Combining with Gabber-Ramero's computation of differentials of valuation rings, we obtain a differential criterion for perfectoidness. We also establish a purity result for perfectoidness when the limit preserves generic points of the special fibres.
	
	As an application to limits of smooth $p$-adic varieties (on the generic point), assuming either the poly-stable modification conjecture or working only with curves, we prove that stalk-wise perfectoidness implies vanishing of the higher completed \'etale cohomology groups of the smooth varieties, which is inspired by Scholze's vanishing for Shimura varieties. Moreover, we give an explicit description of the completed \'etale cohomology group in the top degree in terms of the colimit of Zariski cohomology groups of the structural sheaves.
\end{abstract}
\footnotetext{\emph{2020 Mathematics Subject Classification} 14F30 (primary), 11F77.\\Keywords: perfectoid, purity, criterion, Riemann-Zariski space, Faltings topos, completed cohomology}
	
\tableofcontents

\section{Introduction}

\begin{mypara}
	To compute the cohomology of a scheme, we typically start by finding a \v Cech covering (or hypercovering in general) consisting of affine open subschemes and use the associated \v Cech complex for calculation. This is because rings satisfy cohomological descent in Zariski topology (or even faithfully flat topology). As shown by Scholze, \emph{perfectoid algebras} play a similar role in the cohomology computation in $p$-adic analytic topology \cite{scholze2012perfectoid} (or even v-topology \cite{scholze2017diamond}). By constructing perfectoid coverings, countless results in $p$-adic geometry have been proven. There is a fundamental question that arose with the emergence of perfectoid algebras:
\end{mypara}

\begin{myques}[{\cite[2.16]{scholze2013perfsurv}}]
	Is an algebra perfectoid if it is locally perfectoid in the analytic topology?
\end{myques}

\begin{mypara}
	The difficulty lies in the lack of cohomological descent properties for general algebras in the analytic topology, and there seems to be limited evidence or examples in the existing literature. In this article, we provide affirmative answers to this question for a broad class of algebras encountered in $p$-adic geometry. More precisely, we prove that for an algebra (e.g., over $\overline{\bb{Z}_p}$) whose spectrum is an inverse limit of semi-stable affine schemes, (its $p$-adic completion) being perfectoid is local in analytic topology. 
\end{mypara}

\begin{mypara}
	In fact, we established a \emph{valuative criterion} for such an algebra being perfectoid saying that it suffices to check the perfectoidness of the stalks of the associated Riemann-Zariski space (see \ref{intro-thm:criterion-1}). Our proof comprises two main steps: firstly, an elaborate computation of Galois cohomology (measuring the distance from being perfectoid, see \ref{intro-thm:acyclic}) of semi-stable affine schemes after Faltings \cite{faltings1988p,faltings2002almost}, Abbes-Gros \cite{abbes2016p}, Tsuji \cite{tsuji2018localsimpson}, Bhatt-Morrow-Scholze \cite{bhattmorrowscholze2018integral}, and secondly interpreting the Galois cohomology as the cohomology of Faltings ringed topos, which possesses a limit-preserving property, facilitating the passage to the limit (see \ref{intro-thm:essential-adequate-coh}). Along the way, we also established a \emph{purity} result for perfectoidness when the limit preserves generic points of the special fibres, i.e., it suffices to check the perfectoidness of the localizations at those generic points (see \ref{intro-cor:purity}). Combining with Gabber-Ramero's computation of differentials of valuation rings \cite{gabber2003almost}, we obtain a \emph{differential criterion} for perfectoidness (see \ref{intro-thm:criterion-2}).
\end{mypara}

\begin{mypara}
	We insist to work in the generality with semi-stable schemes for (potential) applications of resolution of singularities. In fact, if we assume that the poly-stable modification conjecture is true or only work with curves by Temkin's stable modification theorem \cite{temkin2010stablecurve}, then we are able to give some applications to countable limits of smooth $p$-adic varieties (on the generic point): inspired by Scholze's work on Calegari-Emerton's conjecture on the vanishing of higher completed cohomology groups of Shimura varieties \cite{scholze2015torsion}, we prove in a similar manner that the completed \'etale cohomology groups of the smooth $p$-adic varieties vanish in higher degrees if the stalks of the associated Riemann-Zariski space are perfectoid (see \ref{intro-thm:completed}), and moreover our framework gives an explicit relation between the completed \'etale cohomology group and the countable colimit of Zariski cohomology groups of the structural sheaves in the top degree (see \ref{intro-cor:completed}).
\end{mypara}

\begin{mypara}
	Before going into the details of this article, we firstly recall the definition of perfectoid algebras. Let $K$ be a complete discrete valuation field of characteristic $0$ with perfect residue field of characteristic $p>0$, $\overline{K}$ an algebraic closure of $K$. For a flat $\ca{O}_{\overline{K}}$-algebra $R$, we say that it is \emph{pre-perfectoid} if the Frobenius induces an isomorphism $R/p^{1/p}R\iso R/pR,\ x\mapsto x^p$ (see \ref{para:notation-perfd}).
\end{mypara}

\begin{mypara}
	The first ingredient for determining whether $R$ is pre-perfectoid is the cohomology of Faltings ringed topos (we also call it \emph{Faltings cohomology} for short), which measures the distance of $R$ from being pre-perfectoid. We put
	\begin{align}
		\eta=\spec(\overline{K}),\quad S=\spec(\ca{O}_{\overline{K}}),\quad s=\spec(\ca{O}_{\overline{K}}/\ak{m}_{\overline{K}}).
	\end{align}
	Let $X$ be a coherent scheme over $S$ and let $X_{\eta}$ (resp. $X_s$) denote its generic fibre (resp. special fibre) (here ``coherent" stands for ``quasi-compact and quasi-separated", see \ref{para:notation-intclos}). Recall that the Faltings ringed site $(\fal^\et_{X_{\eta}\to X},\falb)$ associated to $X$ is defined as follows: $\fal^\et_{X_{\eta}\to X}$ is the category of morphisms of coherent schemes $V\to U$ over $X_\eta\to X$, i.e. commutative diagrams
	\begin{align}
		\xymatrix{
			V\ar[r]\ar[d]& U\ar[d]\\
			X_{\eta}\ar[r]& X
		}
	\end{align}
	such that $U$ is \'etale over $X$ and that $V$ is finite \'etale over $U_{\eta}=X_{\eta}\times_X U$, endowed with the topology generated by the following two types of families of morphisms
	\begin{enumerate}
		\item(vertical coverings) $\{(V_m \to U) \to (V \to U)\}_{m \in M}$, where $M$ is a finite set and $\coprod_{m\in M} V_m\to V$ is surjective;
		\item(Cartesian coverings) $\{(V\times_U{U_n} \to U_n) \to (V \to U)\}_{n \in N}$, where $N$ is a finite set and $\coprod_{n\in N} U_n\to U$ is surjective.
	\end{enumerate}
	The sheaf $\falb$ on $\fal^{\et}_{X_\eta\to X}$ is given by
	\begin{align}
			\falb(V \to U) = \Gamma(U^V , \ca{O}_{U^V}),
	\end{align}
	where $U^V$ is the integral closure of $U$ in $V$. One of the main results of cohomological descent for Faltings ringed topos is the following perfectoidness criterion.
\end{mypara}

\begin{mythm}[{\cite[8.24]{he2024coh}, see \ref{thm:acyclic}}]\label{intro-thm:acyclic}
	Assume that the $\ca{O}_{\overline{K}}$-algebra $R$ is integrally closed in $R[1/p]$. Then, the following statements are equivalent:
	\begin{enumerate}
		\renewcommand{\labelenumi}{{\rm(\theenumi)}}
		\item(Perfectoidness) The $\ca{O}_{\overline{K}}$-algebra $R$ is pre-perfectoid.
		\item(Faltings acyclicity) The canonical morphism $R/pR\to\rr\Gamma(\fal^\et_{X_{\eta}\to X},\falb/p\falb)$ is an almost isomorphism, where $X=\spec(R)$.
	\end{enumerate} 
\end{mythm}

\begin{mypara}
	The next step is to compute Faltings cohomology for semi-stable $S$-schemes $X=\spec(R)$. Namely, \'etale locally $X$ admits a standard semi-stable chart (see \ref{exam:semi-stable}), i.e., an \'etale homomorphism over $\ca{O}_{\overline{K}}$,
	\begin{align}\label{intro-eq:exam:semi-stable-1}
		\ca{O}_{\overline{K}}[T_0,T_1,\dots,T_b,T_{b+1}^{\pm 1},\dots,T_c^{\pm 1}]/(T_0\cdots T_b-a)\longrightarrow R,
	\end{align}
	for some integers $0\leq b\leq c$ and $a\in\ak{m}_{\overline{K}}\setminus \{0\}$. It turns out that the Faltings cohomology of $X$ can be controlled by the Faltings cohomology of the localizations of $X$ at generic points of its special fibre $X_s$.
\end{mypara}

\begin{mythm}[{see \ref{thm:essential-adequate-coh}}]\label{intro-thm:essential-adequate-coh}
	For any integer $n$, the natural map of $R$-modules 
	\begin{align}\label{intro-eq:thm:essential-adequate-coh-1}
		H^n(\fal^\et_{X_\eta\to X},\falb/\pi\falb)\longrightarrow \prod_{x\in\ak{G}(X_s)}H^n(\fal^\et_{X_{(x),\eta}\to X_{(x)}},\falb/\pi\falb)
	\end{align}
	is almost injective, where $\ak{G}(X_s)$ is the finite set of generic points of $X_s$ and $X_{(x)}=\spec(\ca{O}_{X,x})$. Moreover, the canonical exact sequence $0\to \falb/p\falb\stackrel{\cdot p}{\longrightarrow}\falb/p^2\falb\to \falb/p\falb\to 0$ induces an almost exact sequence of $R$-modules
	\begin{align}\label{intro-eq:thm:essential-adequate-coh-2}
		\xymatrix{
			0\ar[r]&R/pR\ar[r]&H^0(\fal^\et_{X_\eta\to X},\falb/p\falb)\ar[r]^-{\delta^0}&H^1(\fal^\et_{X_\eta\to X},\falb/p\falb).
		}
	\end{align}
\end{mythm}

\begin{mypara}
	As Faltings cohomology commutes with \'etale base change (see \ref{cor:acyclic}), we reduce to the local case \eqref{intro-eq:exam:semi-stable-1} so that Faltings cohomology can be computed by Galois cohomology (see \ref{cor:profet-coh} and \ref{rem:profet-coh}). Then, we can use a perfectoid tower associated to the standard semi-stable chart constructed by Tsuji \cite[\textsection4]{tsuji2018localsimpson} to reduce to the cohomology computation of Koszul complexes. Indeed, such a perfectoid tower admits a Galois equivariant decomposition into eigen-submodules (see \ref{prop:galois-action}), and the Galois cohomology of each eigen-submodule is computed by the associate Koszul complex (see \ref{lem:p-inf-nu-coh}). Then, we give an elaborate computation of Koszul cohomology and the coboundary maps (see \ref{prop:koszul} and \ref{cor:koszul}), generalizing the computations of Abbes-Gros (\cite[\Luoma{2}.8.1]{abbes2016p}) and Bhatt-Morrow-Scholze (\cite[7.10]{bhattmorrowscholze2018integral}). 
	
	It turns out that the cohomology groups are direct sums of direct summands of $R'/\pi R'$ for some integral closure $R'$ of $R$ in a finite extension of its fraction field and some $\pi\in\ak{m}_{\overline{K}}$ (see \ref{thm:coh}). As such $R'$ is a filtered colimit of Noetherian normal domains, we see that the canonical map from $R'/\pi R'$ to the product of its localizations at generic points of $X_s$ is injective (see \ref{prop:generic-map}). This proves the almost injectivity of \eqref{intro-eq:thm:essential-adequate-coh-1}. Moreover, the study of coboundary maps on Koszul cohomology enables us to bound the cokernel of $R/pR\to H^0(\fal^\et_{X_\eta\to X},\falb/p\falb)$ by $H^1(\fal^\et_{X_\eta\to X},\falb/p\falb)$. This proves the almost exactness of \eqref{intro-eq:thm:essential-adequate-coh-2}.
	
	We remark that our consideration of the scheme $X$ is not limited to the semi-stable case. Indeed, we treat more general schemes that come from log smooth schemes called \emph{essentially adequate schemes} (see \ref{defn:essential-adequate-pair}). Meanwhile, we also allow horizontal divisors (e.g., the homomorphism \eqref{intro-eq:exam:semi-stable-1} is smooth instead of \'etale and the extra coordinates $T_{c+1},\dots,T_d$ cut out a normal crossings divisor on $X_\eta$, see \ref{exam:semi-stable}).
	
	Passing the results to certain limits of semi-stable schemes, we obtain the following perfectoidness criterion.
\end{mypara}

\begin{mythm}[{Purity for perfectoidness, see \ref{cor:purity}}]\label{intro-cor:purity}
	Let $(X_\lambda=\spec(R_\lambda))_{\lambda\in\Lambda}$ be a directed inverse system of affine schemes semi-stable over $S$, $R=\colim_{\lambda\in\Lambda}R_\lambda$ and $X=\spec(R)$. Assume that for any indexes $\lambda\leq \mu$ in $\Lambda$, the transition morphism $X_\mu\to X_\lambda$ induces a map between the sets $\ak{G}(X_{\mu,s})\to \ak{G}(X_{\lambda,s})$ of generic points of their special fibres. Then, the following conditions are equivalent: 
	\begin{enumerate}
		\renewcommand{\labelenumi}{{\rm(\theenumi)}}
		\item(Perfectoidness) The $\ca{O}_{\overline{K}}$-algebra $R$ is pre-perfectoid.
		\item(Generic perfectoidness) For any $x\in \ak{G}(X_s)=\lim_{\lambda\in\Lambda}\ak{G}(X_{\lambda,s})$, the valuation field $\ca{O}_{X,x}[1/p]$ is a pre-perfectoid field.
	\end{enumerate}
\end{mythm}

\begin{mypara}
	This statement shares a similar philosophy with Zariski-Nagata purity theorem, that's why we call it \emph{purity for perfectoidness}. It follows directly from a similar cohomological result \ref{intro-thm:essential-adequate-coh} for $X$. However, the latter does not follow directly from \ref{intro-thm:essential-adequate-coh} by taking colimits even though the Faltings cohomology is limit-preserving. The difficulty arises from the fact that $\ak{G}(X_s)$ is pro-finite rather than finite in general and we don't know how to exchange products over infinite sets with colimits. But one can use the non-emptiness of cofiltered limit of non-empty finite sets (or more generally, spectral spaces) to show that a section of Faltings cohomology non-zero on each $\ak{G}(X_{\lambda,s})$ is also non-zero on $\ak{G}(X_s)$. We shall use the same argument for later generalizations.
\end{mypara}

\begin{mypara}
	In general, the formation of taking the set $\ak{G}(X_s)$ of generic points of the special fibre $X_s$ is not functorial in $X$. Hence, we switch to a functorial space $\rz_{X_\eta}(X)$ associated to $X$, called the \emph{Riemann-Zariski space} of $X$. It is the cofiltered limit of all $X_\eta$-modifications $X'$ of $X$ in the category of locally ringed spaces (see \ref{para:riemann-zariski}). Remarkably, it is a spectral space, which makes our arguments for passing the cohomological results \ref{intro-thm:essential-adequate-coh} to the limit still valid (see \ref{thm:rz-continue}). On the other hand, there is a full description of its points and stalks by Temkin \cite{temkin2011rz}, especially the $p$-adic completions $\widehat{\ca{O}_{\rz_{X_\eta}(X),x}}$ of its stalks are valuation rings (see \ref{thm:riemann-zariski}). Then, we can upgrade the previous perfectoidness criterion for more general limits of semi-stable schemes.
\end{mypara}

\begin{mythm}[{Valuative criterion for perfectoidness, see \ref{thm:criterion}}]\label{intro-thm:criterion-1}
	Let $(X_\lambda=\spec(R_\lambda))_{\lambda\in\Lambda}$ be a directed inverse system of affine schemes semi-stable over $S$, $R=\colim_{\lambda\in\Lambda}R_\lambda$ and $X=\spec(R)$. Then, the following conditions are equivalent: 
	\begin{enumerate}
		\renewcommand{\labelenumi}{{\rm(\theenumi)}}
		\item The $\ca{O}_{\overline{K}}$-algebra $R$ is pre-perfectoid.
		\item For any $x\in \rz_{X_\eta}(X)$, the stalk $\ca{O}_{\rz_{X_\eta}(X),x}$ is pre-perfectoid.
	\end{enumerate}
\end{mythm}

	Combining with Gabber-Ramero's computation of differentials of valuation rings, we deduce the following theorem.

\begin{mythm}[{Differential criterion for perfectoidness, see \ref{thm:criterion}}]\label{intro-thm:criterion-2}
	Let $(X_\lambda=\spec(R_\lambda))_{\lambda\in\Lambda}$ be a directed inverse system of affine schemes semi-stable over $S$, $R=\colim_{\lambda\in\Lambda}R_\lambda$ and $X=\spec(R)$. Then, the following conditions are equivalent: 
	\begin{enumerate}
		\renewcommand{\labelenumi}{{\rm(\theenumi)}}
		\item The $\ca{O}_{\overline{K}}$-algebra $R$ is pre-perfectoid.
		\item The canonical morphism of modules of differentials $\Omega^1_{R/\ca{O}_{\overline{K}}}\to \Omega^1_{R/\ca{O}_{\overline{K}}}[1/p]$ is almost surjective.
	\end{enumerate}
\end{mythm}

\begin{mypara}
	There are enough many $\ca{O}_{\overline{K}}$-algebras which can be written as a filtered colimit of semi-stable algebras. Indeed, for any $\ca{O}_{\overline{K}}$-algebra $R$ which can be written as a countable colimit $R=\colim_{n\in\bb{N}}R_n$ of $\ca{O}_{\overline{K}}$-algebras of finite presentation, applying de Jong's alteration theorem \cite[6.5]{dejong1996alt}, we obtain an inductive system of semi-stable $\ca{O}_{\overline{K}}$-algebras $(R'_n)_{n\in\bb{N}}$ over $(R_n)_{n\in\bb{N}}$ with colimit $R'$ such that the induced morphism $\spec(\overline{R'})\to \spec(\overline{R})$ is a v-covering, where $\overline{R}$ denotes the integral closure of $R$ in $R[1/p]$. 
	
	To apply our previous criteria, we would like to impose a stronger alteration so-called \emph{poly-stable modifications} to those $R_n$. More precisely, for each $R_n$ we want a $\spec(R_n[1/p])$-modification $X'_n$ of $\spec(R_n)$ (instead of alteration) that is poly-stable over $S$, i.e., \'etale locally $X'_n$ is the fibred product of semi-stable schemes over $S$. If it is the case, then the Riemann-Zariski spaces of $X'_n$ and $\spec(R_n)$ are the same, their Faltings cohomologies are almost the same (see \ref{prop:modification-invariance}), and our previous results apply to $X'_n$. In fact, the existence of such modifications is expected by the poly-stable modification conjecture (see \ref{conj:poly-stable}) and proved by Temkin for curves (see \ref{thm:temkin}). Assuming the existence of such modifications, we can construct poly-stable integral models for countable limits of smooth $\overline{K}$-varieties and we obtain the following result.
\end{mypara}

\begin{mythm}[{see \ref{thm:completed}}]\label{intro-thm:completed}
	Let $(Y_n)_{n\in\bb{N}}$ be a directed inverse system of separated smooth $\overline{K}$-schemes of finite type such that $d=\limsup_{n\to\infty}\{\dim(Y_n)\}<+\infty$. Assume that the following conditions hold:
	\begin{enumerate}
		\renewcommand{\labelenumi}{{\rm(\theenumi)}}
		\item Either the poly-stable modification conjecture {\rm\ref{conj:poly-stable}} is true, or $Y_n$ is equidimensional of dimension $1$ for any $n\in\bb{N}$.
		\item For any point $y$ of the inverse limit $\lim_{n\in\bb{N}}Y_n$ of locally ringed spaces, its residue field $\kappa(y)$ is a pre-perfectoid field with respect to any valuation ring $V$ of height $1$ extension of $\ca{O}_{\overline{K}}$ with fraction field $V[1/p]=\kappa(y)$.
	\end{enumerate}
	Then, for any integer $q>d$ and any integer $k\in\bb{N}$, we have
	\begin{align}
		\colim_{n\in\bb{N}}H^q_\et(Y_n,\bb{Z}/p^k\bb{Z})=0.
	\end{align} 
\end{mythm}

\begin{mypara}\label{para:intro-shimura}
	This result is inspired by Scholze's work on Calegari-Emerton's conjecture on the vanishing of higher completed cohomology groups of Shimura varieties \cite{scholze2015torsion}. More precisely, let $G$ be a reductive group over $\bb{Q}$, $(G,X)$ a Shimura datum, $(X_K)_K$ the associated directed inverse system of Shimura varieties, where $K\subseteq G(\bb{A}_f)$ runs through sufficiently small compact open subgroups (\cite[2.1.2]{deligne1979shimura}, see also \cite[5.14]{milne2005shimura}). Fixing a field isomorphism $\bb{C}\cong \overline{\bb{Q}_p}$, we regard each $X_K$ as a separated smooth $\overline{\bb{Q}_p}$-scheme of finite type with common dimension $d$. If $(G,X)$ is of Hodge type, then for a sufficiently small compact open subgroup $K^p\subseteq G(\bb{A}_f^p)$, Scholze proved that the colimit of compactly supported \'etale cohomology groups vanish for any integer $q>d$ \cite[\Luoma{4}.2.2]{scholze2015torsion},
	\begin{align}
		\colim_{K_p\subseteq G(\bb{Q}_p)} H^q_{\mrm{c}}(X_{K^pK_p},\bb{Z}/p^k\bb{Z})=0,
	\end{align}
	where $K_p\subseteq G(\bb{Q}_p)$ runs through sufficiently small compact open subgroups.
	His proof can be divided into two steps: firstly establishing the perfectoidness of the limit of the minimal compactifications of $X_{K^pK_p}$ as adic spaces, and then using his primitive $p$-adic comparison theorem to compute the colimit of \'etale cohomology groups by sheaf cohomology of the perfectoid Shimura variety in the analytic topology, which vanishes in degrees $>d$ by Grothendieck's vanishing of sheaf cohomology on limits of Noetherian spectral spaces.
	
	For general Shimura datum, the perfectoidness is not known for the limit of Shimura varieties so that the argument above faces an obstacle. Therefore, we weaken the requirement for perfectoidness of the limit to stalk-wise perfectoidness in our analogous vanishing result \ref{intro-thm:completed}, meanwhile a purely scheme theoretic argument similar as Scholze's still works since we can control the Faltings cohomology of the limit by the Fatings cohomology of its stalks and we can apply Faltings' main $p$-adic comparison theorem in this situation. In this way, we actually represents the colimit of \'etale cohomology groups (almost) as a countable colimit of Zariski cohomology groups of coherent sheaves (see \ref{thm:vanishing}), which enables us to explicitly relate the completed \'etale cohomology group in the top degree with the countable colimit of the Zariski cohomology groups of the structural sheaves on the generic fibres:
\end{mypara}

\begin{mycor}[{see \ref{cor:completed}}]\label{intro-cor:completed}
	Following {\rm\ref{intro-thm:completed}}, let $(Y_n\to \overline{Y_n})_{n\in\bb{N}}$ a directed inverse system of open immersions of coherent $\overline{K}$-schemes such that each $\overline{Y_n}$ is proper smooth and $\overline{Y_n}\setminus Y_n$ is the support of a normal crossings divisor on $\overline{Y_n}$. Assume moreover that
	\begin{enumerate}
		\setcounter{enumi}{2}
		\renewcommand{\labelenumi}{{\rm(\theenumi)}}
		\item the completed \'etale cohomology group in the top degree $\lim_{k\in\bb{N}}\colim_{n\in\bb{N}}H^d(Y_{n,\et},\bb{Z}/p^k\bb{Z})$ has bounded $p$-power torsion.\label{item:intro-cor:completed-3}
	\end{enumerate}
	Then, there is a canonical homomorphism of $\widetilde{\overline{K}}$-modules,
	\begin{align}\label{eq:intro-cor:completed-1}
		\colim_{n\in\bb{N}}H^d(\widetilde{\overline{Y_n}},\ca{O}_{\widetilde{\overline{Y_n}}})\longrightarrow \left(\lim_{k\in\bb{N}}\colim_{n\in\bb{N}}H^d(Y_{n,\et},\ca{O}_{\widetilde{\overline{K}}}/p^k\ca{O}_{\widetilde{\overline{K}}})\right)[1/p],
	\end{align}
	where $\widetilde{\overline{K}}$ is the maximal completion (or equivalently, spherical completion) of $\overline{K}$ (see {\rm\ref{rem:max-val-exist}}) and $\widetilde{\overline{Y_n}}=\spec(\widetilde{\overline{K}})\times_{\spec(\overline{K})}\overline{Y_n}$, fitting into the following commutative diagram
	\begin{align}\label{eq:intro-cor:completed-2}
		\xymatrix{
			\colim_{n\in\bb{N}}H^d(\widetilde{\overline{Y_n}},\ca{O}_{\widetilde{\overline{Y_n}}})\ar[r]\ar@{->>}[d]& \left(\lim_{k\in\bb{N}}\colim_{n\in\bb{N}}H^d(Y_{n,\et},\ca{O}_{\widetilde{\overline{K}}}/p^k\ca{O}_{\widetilde{\overline{K}}})\right)[1/p]\\
			\bigoplus_J\widetilde{\overline{K}}\ar[r]&\widehat{\bigoplus}_J\widetilde{\overline{K}}\ar[u]_-{\wr}
		}
	\end{align}
	where $J$ is a countable set, the horizontal arrow on the bottom is induced from the $p$-adic completion map $\bigoplus_J\ca{O}_{\widetilde{\overline{K}}}\to\widehat{\bigoplus}_J\ca{O}_{\widetilde{\overline{K}}}$ by inverting $p$, the left vertical arrow is surjective and the right vertical arrow is an isomorphism. In particular, \eqref{eq:intro-cor:completed-1} has dense image.
\end{mycor}

\begin{mypara}
	We note that the extra assumption (\ref{item:intro-cor:completed-3}) holds for the inverse system of Shimura varieties $(X_{K^pK_p})_{K_p\subseteq G(\bb{Q}_p)}$ by \cite[2.2]{emerton2006interpolation} (see \ref{rem:cor:completed}.(\ref{item:rem:cor:completed-2})). We remark that the poly-stable modification conjecture is not necessary for Theorem \ref{intro-thm:completed} if each $Y_n$ is proper (i.e., $Y_n=\overline{Y_n}$, see \ref{rem:thm:completed}.(\ref{item:rem:thm:completed-1})), while on the other hand, it seems inevitable to give such a concrete relation \eqref{eq:intro-cor:completed-2} between the cohomologies in the top degree. We hope that our results will contribute to a better understanding of completed cohomology of Shimura varieties in the future.
\end{mypara}

\begin{mypara}
	The article is structured as follows. In Section \ref{sec:generic}, we discuss functoriality of taking the set of generic points of schemes and recall a dominant property for Noetherian normal domains by its localizations at height-$1$ primes. We conduct an elaborate computation of the cohomology of Koszul complexes and their coboundary maps in Section \ref{sec:koszul}. We introduce the terminology of ``essentially adequate algebras" in Section \ref{sec:essential-adequate-alg} and compute their Galois cohomology and coboundary maps by the previous results on Koszul complexes. In Section \ref{sec:faltings-site}, we summarize some cohomological descent results for Faltings ringed topos proved in my thesis, especially the relation between ``Faltings acyclicity" and ``perfectoidness" is discussed. Then, we translate the results for Galois cohomology to Faltings cohomology of essentially adequate schemes in Section \ref{sec:essential-adequate-scheme}, which enables us to prove purity for perfectoidness in Section \ref{sec:purity}. To overcome the non-functoriality of the set of generic points, we turn to use the functorial Riemann-Zariski spaces. Their basic properties are reviewed in Section \ref{sec:riemann-zariski} with a particular eye on their limit behavior. Switching to stalks of Riemann-Zariski spaces from localizations at generic points of special fibres, we obtain valuative and differential criteria for Faltings acyclicity and perfectoidness for general inverse limits of essentially adequate schemes in Section \ref{sec:local-faltings}. We end this article by the study of the completed \'etale cohomology group in the top degree in Section \ref{sec:polystable}, where we need to apply Kaplansky's structure theorem on torsion-free modules reviewed in Section \ref{sec:torsion-free}.
\end{mypara}

\subsection*{Acknowledgements}
I would like to thank Michael Temkin for answering my questions on modification theorems and conjectures. I am grateful to Lue Pan and Juan Esteban Rodr\'{i}guez Camargo for their questions on perfectoidness after my Luminy lecture in June 2022, which motivated Section \ref{sec:polystable}. I extend my appreciation to thank Bhargav Bhatt, Yuanyang Jiang, Lue Pan, Emanuel Reinecke, and Mingjia Zhang for helpful discussions.

I express my sincere gratitude to Ahmed Abbes, Ofer Gabber, Kiran S. Kedlaya, Peter Scholze, and Takeshi Tsuji for their invaluable feedback and suggestions on the initial draft of this manuscript.

This material is based upon work supported by the National Science Foundation under Grant No. DMS-1926686 during my membership at Institute for Advanced Study in the ``$p$-adic Arithmetic Geometry" special year.

\section{Notation and Conventions}

\begin{mypara}
	All monoids and rings considered in this article are unitary and commutative, and we fix a prime number $p$ throughout this article.
\end{mypara}

\begin{mypara}
	A \emph{valuation field} is a pair $(K,\ca{O}_K)$ where $\ca{O}_K$ is a valuation ring with fraction field $K$ (\cite[\Luoma{6}.\textsection1.2, D\'efinition 2]{bourbaki2006commalg5-7}). We denote by $\ak{m}_K$ the maximal ideal of $\ca{O}_K$ and call the number of the nonzero prime ideals of $\ca{O}_K$ the \emph{height} of $K$ (\cite[\Luoma{6}.\textsection4.4, Proposition 5]{bourbaki2006commalg5-7}). We also refer to \cite[\href{https://stacks.math.columbia.edu/tag/00I8}{00I8}]{stacks-project} for basic properties on valuation rings. For a non-discrete valuation field $K$ of height $1$ (so that $\ak{m}_K=\ak{m}_K^2$), when referring to almost modules over $\ca{O}_K$ we always take $(\ca{O}_K,\ak{m}_K)$ as the basic setup (\cite[2.1.1]{gabber2003almost}).
\end{mypara}

\begin{mypara}\label{para:notation-perfd}
	Following \cite[\textsection5]{he2024coh}, a \emph{pre-perfectoid field} is a valuation field $K$ of height $1$ with non-discrete valuation of residue characteristic $p$ such that the Frobenius map on $\ca{O}_K/p\ca{O}_K$ is surjective (\cite[5.1]{he2024coh}). We note that for any nonzero element $\pi\in\ak{m}_K$ the fraction field $\widehat{K}$ of the $\pi$-adic completion of $\ca{O}_K$ is a perfectoid field in the sense of \cite[3.1]{scholze2012perfectoid} (see \cite[5.2]{he2024coh}). Given a pre-perfectoid field $K$, we say that an $\ca{O}_K$-algebra $R$ is (resp. \emph{almost}) \emph{pre-perfectoid} if there is a nonzero element $\pi\in\ak{m}_K$ such that $p\in \pi^p\ca{O}_K$, the $\pi$-adic completion $\widehat{R}$ is (resp. almost) flat over $\ca{O}_{\widehat{K}}$ and the Frobenius induces an (resp. almost) isomorphism $R/\pi R\to R/\pi^p R$ (see \cite[5.19]{he2024coh}, and this definition does not depend on the choice of $\pi$ by \cite[5.23]{he2024coh}). This is equivalent to that the $\ca{O}_{\widehat{K}}$-algebra $\widehat{R}$ (resp. almost $\ca{O}_{\widehat{K}}$-algebra $\widehat{R}^{\al}$ associated to $\widehat{R}$) is perfectoid in the sense of \cite[3.10.(\luoma{2})]{bhattmorrowscholze2018integral} (resp. \cite[5.1.(\luoma{2})]{scholze2012perfectoid}, see \cite[5.18]{he2024coh}).
\end{mypara}

\begin{mypara}\label{para:notation-intclos}
	Following \cite[\Luoma{6}.1.22]{sga4-2}, a \emph{coherent} scheme (resp. morphism of schemes) stands for a quasi-compact and quasi-separated scheme (resp. morphism of schemes). For a coherent morphism $Y \to X$ of schemes, we denote by $X^Y$ the integral closure of $X$ in $Y$ (\cite[\href{https://stacks.math.columbia.edu/tag/0BAK}{0BAK}]{stacks-project}). For an $X$-scheme $Z$, we say that $Z$ is \emph{$Y$-integrally closed} if $Z=Z^{Y \times_X Z}$.
\end{mypara}

\section{Preliminaries on Extension of Valuation Rings and Generic Points}\label{sec:generic}

\begin{mylem}\label{lem:val-ext}
	Let $f:V\to W$ be a homomorphism of valuation rings. Then, the following statements are equivalent:
	\begin{enumerate}
		\renewcommand{\labelenumi}{{\rm(\theenumi)}}
		\item $f$ is faithfully flat.\label{item:lem:val-ext-1}
		\item $f$ is injective and local.\label{item:lem:val-ext-2}
		\item $f$ is injective and $V=K\cap W\subseteq L$, where $K$ (resp. $L$) is the fraction field of $V$ (resp. $W$). \label{item:lem:val-ext-3}
	\end{enumerate}
\end{mylem}
\begin{proof}
	(\ref{item:lem:val-ext-1})$\Rightarrow$(\ref{item:lem:val-ext-2}) Since $f$ is faithfully flat, we see that $f$ is injective. Moreover, as $\spec(W)\to \spec(V)$ is surjective, there exists a point of $\spec(W)$ mapping to the closed point of $\spec(V)$. As $W$ is local, the closed point of $\spec(W)$ also maps to that of $\spec(V)$. Hence, $f$ is local.
	
	(\ref{item:lem:val-ext-2})$\Rightarrow$(\ref{item:lem:val-ext-1}) Since $f$ is injective homomorphism of domains, we see that the $V$-module $W$ is torsion-free. Hence, $f$ is flat (\cite[\href{https://stacks.math.columbia.edu/tag/0539}{0539}]{stacks-project}). Therefore, $f$ is faithfully flat as $f$ is flat and local (\cite[\href{https://stacks.math.columbia.edu/tag/00HR}{00HR}]{stacks-project}).
	
	(\ref{item:lem:val-ext-1})(\ref{item:lem:val-ext-2})$\Leftrightarrow$(\ref{item:lem:val-ext-3}) Firstly, assume that $f$ is injective. Then, $V'=K\cap W\subseteq L$ is a valuation ring with fraction field $K$ (\cite[\href{https://stacks.math.columbia.edu/tag/00IB}{00IB}, \href{https://stacks.math.columbia.edu/tag/052K}{052K}]{stacks-project}) and we have $V\subseteq V'\subseteq K$. We see that $V'$ is flat over $V$, since the $V$-module $V'$ is torsion-free. Moreover, we claim that $W$ is faithfully flat over $V'$. Indeed, let $\ak{q}\in \spec(V')$ be the image of the closed point $\ak{m}\in\spec(W)$. Then, $V'_{\ak{q}}\subseteq K$ is still contained in $K\cap W$. Hence, $V'_{\ak{q}}=V'=K\cap W$, i.e., $\ak{q}$ is the closed point of $\spec(V')$, which proves the claim by the equivalence (\ref{item:lem:val-ext-1})$\Leftrightarrow$(\ref{item:lem:val-ext-2}).
	
	Then, we see that $f:V\to W$ is faithfully flat if and only if $V\to V'$ is faithfully flat. Since $V$ and $V'$ has the same fraction field, $V\to V'$ is local (i.e., faithfully flat) if and only if $V=V'$, since $V$ itself is the maximal local ring contained in $K$ dominating $V$ (\cite[\href{https://stacks.math.columbia.edu/tag/00I9}{00I9}]{stacks-project}). This proves the equivalence (\ref{item:lem:val-ext-1})(\ref{item:lem:val-ext-2})$\Leftrightarrow$(\ref{item:lem:val-ext-3}).
\end{proof}

\begin{mydefn}\label{defn:val-ext}
	Let $V\to W$ be a homomorphism of valuation rings. We say that it is an (resp. \emph{algebraic}) \emph{extension of valuation rings} if it satisfies the equivalent conditions in \ref{lem:val-ext} (resp. and if it is integral). In this case, we also say that the homomorphism of their fraction fields $\mrm{Frac}(V)\to \mrm{Frac}(W)$ is an (resp. \emph{algebraic}) \emph{extension of valuation fields}.
\end{mydefn}

We note that if $V\to W$ is an algebraic extension of valuation rings, then $W$ is the integral closure of $V$ in the fraction field of $W$.

\begin{mydefn}\label{defn:generic}
	Let $X$ be a scheme. We denote by $\ak{G}(X)$ the set of generic points of irreducible components of $X$. For $X=\spec(A)$ affine, we put $\ak{G}(\spec(A))=\ak{G}(A)$, which is the set of minimal prime ideals of $A$.
\end{mydefn}

\begin{mylem}\label{lem:generic-map}
	Let $f:Y\to X$ be a morphism of schemes. 
	\begin{enumerate}
		\renewcommand{\labelenumi}{{\rm(\theenumi)}}
		\item If generalizations lift along $f$ {\rm(\cite[\href{https://stacks.math.columbia.edu/tag/0063}{0063}]{stacks-project})}, then $f(\ak{G}(Y))\subseteq \ak{G}(X)$.\label{item:lem:generic-map-1}
		\item If $\ak{G}(X)\subseteq f(Y)$, then $\ak{G}(X)\subseteq f(\ak{G}(Y))$.\label{item:lem:generic-map-2}
		\item If the underlying topological space of the fibre of $f$ over any point in $\ak{G}(X)$ is totally disconnected {\rm(\cite[{\href{https://stacks.math.columbia.edu/tag/04MC}{04MC}}]{stacks-project})}, then $f^{-1}(\ak{G}(X))\subseteq \ak{G}(Y)$.\label{item:lem:generic-map-3}
	\end{enumerate}
\end{mylem}
\begin{proof}
	(\ref{item:lem:generic-map-1}) For any $y\in\ak{G}(Y)$, if $x=f(y)\notin \ak{G}(X)$ then there exists a generalization $x'\leadsto x$ of $x$ in $X$ with $x'\neq x$. As generalizations lift along $f$, there exists a generalization $y'\leadsto y$ of $y$ in $Y$ with $f(y')=x'$. Hence, $y'\neq y$, which contradicts the fact that $y$ is the generic point of an irreducible component of $Y$.
	
	(\ref{item:lem:generic-map-2}) For any $x\in\ak{G}(X)$ with $x=f(y)$ for some $y\in Y$, we take $y'\in\ak{G}(Y)$ specializing to $y$. As $f(y')\leadsto f(y)$ and $f(y)$ is a generic point, we have $x=f(y)=f(y')\in f(\ak{G}(Y))$.
	
	(\ref{item:lem:generic-map-3}) For any $y\in Y$ such that $x=f(y)\in \ak{G}(X)$, if $y\notin \ak{G}(Y)$ then there exists a generalization $y'\leadsto y$ of $y$ in $Y$ with $y'\neq y$. As $f(y')\leadsto f(y)$ and $f(y)$ is a generic point, we have $f(y')=f(y)=x$. Then, the connected component of $y'$ in $f^{-1}(x)$ contains $y$, which contradicts the assumption that $f^{-1}(x)$ is a totally disconnected topological space.
\end{proof}
\begin{myrem}\label{rem:generic-map}
	We use \ref{lem:generic-map} to analyze the behavior on generic points for some common types of morphisms of schemes $f:Y\to X$ which we shall encounter in the rest of this article.
	\begin{enumerate}
		\renewcommand{\labelenumi}{{\rm(\theenumi)}}
		\item If $f$ is flat, then $f(\ak{G}(Y))\subseteq \ak{G}(X)$. Indeed, generalizations lift along $f$ by \cite[\href{https://stacks.math.columbia.edu/tag/03HV}{03HV}]{stacks-project}.\label{item:rem:generic-map-1}
		\item If $f$ is pro-\'etale (\cite[7.13]{he2024coh}), then $f^{-1}(\ak{G}(X))= \ak{G}(Y)$. Indeed, the fibres of $f$ are pro-finite sets.\label{item:rem:generic-map-2}
		\item If there exists an injective integral homomorphism $A\to B$ of domains with $A$ normal and if $f$ is the base change of $g:\spec(B)\to \spec(A)$ along a closed immersion $X\to \spec(A)$, then $f^{-1}(\ak{G}(X))= \ak{G}(Y)$ and $f(\ak{G}(Y))=\ak{G}(X)$. Indeed, generalizations lift along $g$ by \cite[\href{https://stacks.math.columbia.edu/tag/00H8}{00H8}]{stacks-project}, $g$ is surjective and each fibre along $g$ is discrete (\cite[\href{https://stacks.math.columbia.edu/tag/00GQ}{00GQ}, \href{https://stacks.math.columbia.edu/tag/00GT}{00GT}]{stacks-project}). By base change along closed immersion, we see that $f$ satisfies the same properties.\label{item:rem:generic-map-3}
	\end{enumerate}
\end{myrem}

\begin{mylem}\label{lem:generic-discrete}
	Let $X$ be a scheme. Assume that the subset $\ak{G}(X)$ of $X$ is locally finite, i.e., any quasi-compact open subset of $X$ has finitely many irreducible components (e.g., when the underlying topological space of $X$ is locally Noetherian, \cite[\href{https://stacks.math.columbia.edu/tag/0052}{0052}]{stacks-project}). Then, the topology on $\ak{G}(X)$ induced from $X$ is discrete.
\end{mylem}
\begin{proof}
	The problem is local on $X$ by \ref{rem:generic-map}.(\ref{item:rem:generic-map-2}). Thus, we may assume that $\ak{G}(X)$ is finite. Then, any $x\in\ak{G}(X)$ is not contained in the irreducible components of other generic points. Thus, there is an open subset of $X$ whose intersection with $\ak{G}(X)$ is $\{x\}$.
\end{proof}

\begin{mylem}\label{lem:generic-limit}
	Let $(X_\lambda)_{\lambda\in\Lambda}$ be a directed inverse system of coherent schemes satisfying the following conditions:
	\begin{enumerate}
		\renewcommand{\labelenumi}{{\rm(\theenumi)}}
		\item For any $\lambda\in\Lambda$, $\ak{G}(X_\lambda)$ is locally finite in $X_\lambda$.
		\item For any indexes $\lambda\leq \mu$ in $\Lambda$, the transition morphism $f_{\lambda\mu}:X_\mu\to X_\lambda$ is affine and we have $f_{\lambda\mu}(\ak{G}(X_\mu))\subseteq \ak{G}(X_\lambda)$.
	\end{enumerate}
	If we put $X=\lim_{\lambda\in\Lambda}X_\lambda$ the inverse limit of schemes, then we have
	\begin{align}
		\ak{G}(X)=\lim_{\lambda\in\Lambda}\ak{G}(X_\lambda).
	\end{align}
\end{mylem}
\begin{proof}
	The problem is local on $X_{\lambda_0}$ for some $\lambda_0\in\Lambda$ by \ref{rem:generic-map}.(\ref{item:rem:generic-map-2}). Thus, we may assume that $X_\lambda$ is affine and thus $\ak{G}(X_\lambda)$ is finite for any $\lambda\in\Lambda$. Recall that the underlying topological space of $X$ is also the limit of the underlying topological spaces of $X_\lambda$ by \cite[8.2.9]{ega4-3} (see also \cite[\href{https://stacks.math.columbia.edu/tag/0CUF}{0CUF}]{stacks-project}). Hence, $\ak{G}_\infty=\lim_{\lambda\in\Lambda}\ak{G}(X_\lambda)$ is a topological subspace of $X$.
	
	Firstly, we show that $\ak{G}_\infty\subseteq \ak{G}(X)$. For any $x\in \ak{G}_\infty$ with image $x_\lambda\in\ak{G}(X_\lambda)$, we take the generic point $y$ of the irreducible component containing $x$ in $X$. Then, the image $y_\lambda\in X_\lambda$ still specializes to $x_\lambda$. But $x_\lambda$ is a generic point, we have $x_\lambda=y_\lambda$. By taking limits, we see that $x=y\in\ak{G}(X)$. 
	
	Conversely, we show that $\ak{G}(X)\subseteq \ak{G}_\infty$. For any $x\in X$ with image $x_\lambda\in X_\lambda$, consider the finite subset
	\begin{align}
		\ak{G}_{x_\lambda}(X_\lambda)=\{y_\lambda\in \ak{G}(X_\lambda)\ |\ y_\lambda\trm{ specializes to } x_\lambda\}.
	\end{align}
	For any indexes $\lambda\leq \mu$ in $\Lambda$, as $f_{\lambda\mu}(\ak{G}(X_\mu))\subseteq \ak{G}(X_\lambda)$, we see that $f_{\lambda\mu}(\ak{G}_{x_\mu}(X_\mu))\subseteq \ak{G}_{x_\lambda}(X_\lambda)$. Since each $\ak{G}_{x_\lambda}(X_\lambda)$ is a non-empty finite set, $\lim_{\lambda\in\Lambda}\ak{G}_{x_\lambda}(X_\lambda)$ is also non-empty (\cite[\href{https://stacks.math.columbia.edu/tag/0A2W}{0A2W}]{stacks-project}). We take $y\in \lim_{\lambda\in\Lambda}\ak{G}_{x_\lambda}(X_\lambda)\subseteq \ak{G}_\infty$ with image $y_\lambda\in \ak{G}_{x_\lambda}(X_\lambda)$. Then, $x=\lim_{\lambda\in\Lambda}x_\lambda\in \lim_{\lambda\in\Lambda}\overline{\{y_\lambda\}}=\overline{\{y\}}$, where the last equality follows from \cite[\href{https://stacks.math.columbia.edu/tag/0CUG}{0CUG}]{stacks-project}. In other words, $y$ specializes to $x$. In particular, if $x\in\ak{G}(X)$, then $x=y\in \ak{G}_\infty$.
\end{proof}

\begin{myprop}[{\cite[4.3]{he2022sen}}]\label{prop:generic-map}
	Let $A\to B$ be an injective and integral homomorphism of normal domains with $A$ Noetherian, $\pi$ a nonzero element of $A$. Assume that $B$ is the union of a directed system $(B_\lambda)_{\lambda\in\Lambda}$ of Noetherian normal $A$-subalgebras. 
	\begin{enumerate}
		\renewcommand{\labelenumi}{{\rm(\theenumi)}}
		\item The set $\ak{G}(B/\pi B)$ coincides with the set $\ak{S}_\pi(B)$ of height-$1$ prime ideals of $B$ containing $\pi$.\label{item:prop:generic-map-1}
		\item For each $\ak{q}\in \ak{G}(B/\pi B)$, if we denote by $\ak{q}_\lambda\in\ak{G}(B_\lambda/\pi B_\lambda)$ its image, then $(B_{\lambda,\ak{q}_\lambda})_{\lambda\in\Lambda}$ is a directed system of discrete valuation rings with faithfully flat transition maps, whose colimit is $B_{\ak{q}}$, a valuation ring of height $1$.\label{item:prop:generic-map-2}
		\item For any integer $n>0$, the natural map
		\begin{align}
			B/\pi^n B\longrightarrow \prod_{\ak{q}\in\ak{S}_\pi(B)}B_{\ak{q}}/\pi^nB_{\ak{q}} 
		\end{align} 
		is injective.\label{item:prop:generic-map-3}
	\end{enumerate}
\end{myprop}
\begin{proof}
	We only need to prove (\ref{item:prop:generic-map-1}) by \cite[4.3]{he2022sen}. Firstly, we note that $\ak{G}(B/\pi B)=\lim_{\lambda\in\Lambda}\ak{G}(B_\lambda/\pi B_\lambda)$ by \ref{lem:generic-limit} (whose assumptions are satisfied by the Noetherianess of $B_\lambda$ and \ref{rem:generic-map}.(\ref{item:rem:generic-map-3})). For $\ak{q}\in \ak{S}_\pi(B)$, it is clear that $\ak{q}/\pi B$ is a minimal prime ideal of $B/\pi B$, namely, $\ak{S}_\pi(B)\subseteq \ak{G}(B/\pi B)$. Conversely, for $\ak{q}\in\ak{G}(B/\pi B)$ regarded as a prime ideal of $B$ containing $\pi$, let $\ak{p}\subseteq \ak{q}$ be a nonzero prime ideal of $B$. Since $\ak{G}(B_\lambda/\pi B_\lambda)=\ak{S}_\pi(B_\lambda)$ by Noetherianess of $B_\lambda$ and by $\pi\neq 0$  (\cite[\href{https://stacks.math.columbia.edu/tag/00KV}{00KV}]{stacks-project}), $\ak{q}_\lambda\in \ak{G}(B_\lambda/\pi B_\lambda)$ is a height-$1$ prime ideal of $B$ so that $\ak{p}_\lambda=\ak{p}\cap B_\lambda$ is equal to $0$ or $\ak{q}_\lambda$. Since $\ak{p}\neq 0$, for $\lambda\in \Lambda$ large enough, we have $\ak{p}_\lambda=\ak{q}_\lambda$. Thus, $\ak{p}=\ak{q}$ by taking colimits, which shows that $\ak{q}$ is of height $1$, i.e., $\ak{q}\in\ak{S}_\pi(B)$.
\end{proof}

\begin{myrem}\label{rem:prop:generic-map}
	In \ref{prop:generic-map}, any \'etale $B$-algebra $C$ also satisfies the assumptions of \ref{prop:generic-map}. More precisely, as $C$ is a finite product of normal domains \'etale over $B$ (\cite[\href{https://stacks.math.columbia.edu/tag/033C}{033C}, \href{https://stacks.math.columbia.edu/tag/030C}{030C}]{stacks-project}, \ref{rem:generic-map}.(\ref{item:rem:generic-map-2})), we focus on the case where $C$ is a domain. Then, there exists $\lambda\in\Lambda$ and an \'etale $B_\lambda$-algebra $C_\lambda$ with $C=B\otimes_{B_\lambda}C_\lambda$ by \cite[8.8.2, 8.10.5]{ega4-3}. As $B_\lambda\to B$ is injective and integral, so is $C_\lambda\to C$. In conclusion, $(B_\mu\otimes_{B_\lambda}C_\lambda)_{\mu\in\Lambda_{\geq \lambda}}$ is a directed system of Noetherian normal $A$-algebras with injective and integral transition morphisms, and its colimit coincides with $C$.
\end{myrem}

\section{Complements on Koszul Complexes}\label{sec:koszul}
In \cite[7.10]{bhattmorrowscholze2018integral}, Bhatt-Morrow-Scholze computed the isomorphism class of the cohomology of Koszul complex. We give explicit computations of the cohomology of Koszul complex in \ref{prop:koszul} and its coboundary map in \ref{cor:koszul}, which form one of the technical cores of this article.

\begin{mypara}\label{para:notation-koszul-0}
	Let $R$ be a ring, $d\in \bb{N}_{>0}$. We denote by $(e_1,\dots,e_n)$ the standard basis of the $R$-module $R^{\oplus d}$. Then, for any integer $0\leq n\leq d$, $\wedge^nR^{\oplus d}$ is a finite free $R$-module with basis $(e_{i_1}\wedge\cdots \wedge e_{i_n})_{1\leq i_1<\cdots<i_n\leq d}$. 
	
	For any integer $1\leq k\leq d$, the $n-1$ elements $(e_1,\dots,\widehat{e_k},\dots,e_d)$ of $R^{\oplus d}$ (where the notation ``$\widehat{e_k}$" means ``removing $e_k$") generate a finite free submodule, which we still denote by $R^{\oplus (d-1)}\subseteq R^{\oplus d}$ if this inclusion is clear in the context. Then, we obtain a canonical isomorphism (by a combinatorial argument)
	\begin{align}\label{eq:para:notation-koszul-0-1}
		(\iota^k,\iota'^k):\wedge^nR^{\oplus (d-1)}\oplus\wedge^{n-1}R^{\oplus (d-1)} \iso \wedge^nR^{\oplus d},
	\end{align}
	where $\iota^k$ sends $e_{i_1}\wedge \cdots\wedge e_{i_n}$ to $e_{i_1}\wedge \cdots\wedge e_{i_n}$ for any $\{i_1,\dots,i_n\}\subseteq \{1,\dots,d\}\setminus\{k\}$(taking $n$ elements out of $d$ without the $k$-th), and $\iota'^k$ sends $e_{i_1}\wedge \cdots\wedge e_{i_{n-1}}$ to $e_k\wedge e_{i_1}\wedge \cdots\wedge e_{i_{n-1}}$ for any $\{i_1,\dots,i_{n-1}\}\subseteq \{1,\dots,d\}\setminus\{k\}$  (taking $n$ elements out of $d$ with the $k$-th). We denote by
	\begin{align}\label{eq:para:notation-koszul-0-2}
		(\pi^k,\pi'^k):\wedge^nR^{\oplus d}\iso \wedge^nR^{\oplus (d-1)}\oplus\wedge^{n-1}R^{\oplus (d-1)}
	\end{align}
	the inverse of $(\iota^k,\iota'^k)$.
\end{mypara}

\begin{mypara}\label{para:notation-koszul-1}
	Let $R$ be a ring, $M$ an $R$-module with commuting $R$-linear endomorphisms $f_1,\dots, f_d$, where $d\in \bb{N}_{>0}$. Recall that the \emph{Koszul complex} $K^\bullet(M;f_1,\dots,f_d)$ (\cite[7.1]{bhattmorrowscholze2018integral}, see also \cite[\Luoma{2}.2.7]{abbes2016p}) is a complex concentrated in degrees $[0,d]$ whose $n$-th component ($0\leq n\leq d$) is
	\begin{align}
		K^n(M;f_1,\dots,f_d)=M\otimes_R\wedge^nR^{\oplus d},
	\end{align}
	and whose $n$-th differential $\df^n:K^n(M;f_1,\dots,f_d)\to K^{n+1}(M;f_1,\dots,f_d)$ sends $x=\sum_{1\leq i_1<\cdots<i_n\leq d} x_{i_1,\dots,i_n}\otimes e_{i_1}\wedge\cdots \wedge e_{i_n}$ to
	\begin{align}
		\df^n(x)=\sum_{k=1}^d(f_k\otimes e_k)\wedge x=\sum_{1\leq i_0< i_1<\cdots<i_n\leq d} \left(\sum_{l=0}^{n}(-1)^lf_{i_l}(x_{i_0,\dots,\widehat{i_l},\dots,i_n})\right)\otimes e_{i_0}\wedge \cdots\wedge e_{i_n},
	\end{align}
	where $x_{i_1,\dots,i_n}\in M$, and $(e_1,\dots,e_d)$ is the standard $R$-basis of $R^{\oplus d}$. We call $x_{i_1,\dots,i_n}$ the \emph{$(i_1,\dots,i_n)$-component} of $x\in M\otimes_R\wedge^nR^{\oplus d}$.
\end{mypara}

\begin{mylem}\label{lem:koszul-basic}
	With the notation in {\rm\ref{para:notation-koszul-0}} and {\rm\ref{para:notation-koszul-1}}, for any integer $1\leq k\leq d$, there is a canonical exact sequence of complexes
	\begin{align}\label{eq:lem:koszul-basic-1}
		0\to K^\bullet(M;f_1,\dots,\widehat{f_k},\dots,f_d)[-1]\to  K^\bullet(M;f_1,\dots,f_d) \to K^\bullet(M;f_1,\dots,\widehat{f_k},\dots,f_d)\to 0,
	\end{align}
	where the morphisms on degree $n$ ($0\leq n\leq d$) are given by the exact sequence of $R$-modules
	\begin{align}\label{eq:lem:koszul-basic-2}
		\xymatrix{
			0\ar[r]&M\otimes_R\wedge^{n-1}R^{\oplus (d-1)}\ar[r]^-{\iota'^k}& M\otimes_R\wedge^nR^{\oplus d}\ar[r]^-{\pi^k}&M\otimes_R\wedge^nR^{\oplus (d-1)}\ar[r]&0,
		}
	\end{align}
	which splits by the morphisms $\pi'^k$ and $\iota^k$.
	
	Moreover, the morphism $K^\bullet(M;f_1,\dots,\widehat{f_k},\dots,f_d)\to K^\bullet(M;f_1,\dots,\widehat{f_k},\dots,f_d)$ given in the triangle associated to the termwise split exact sequence \eqref{eq:lem:koszul-basic-1} {\rm\cite[\href{https://stacks.math.columbia.edu/tag/014I}{014I}]{stacks-project}} is induced by the endomorphism $f_k$ on $M$.
\end{mylem}
\begin{proof}
	Firstly, we check that the morphisms in \eqref{eq:lem:koszul-basic-2} commute with the differential maps of Koszul complexes. Indeed, we regard $R^{\oplus(d-1)}$ as the submodule of $R^{\oplus d}$ generated by $e_1,\dots,\widehat{e_k},\dots,e_d$. Then, for any $x\in M\otimes_R\wedge^{n-1}R^{\oplus (d-1)}$,
	\begin{align}
		\df^n(\iota'^k(x))=&\df^n(e_k\wedge x)=\sum_{l=1}^d (f_l\otimes e_l)\wedge (e_k\wedge x),\\
		\iota'^k(\df^n(x))=&\iota^k(-\sum_{\substack{l=1\\l\neq k}}^d (f_l\otimes e_l)\wedge x)=-\sum_{\substack{l=1\\l\neq k}}^d e_k\wedge ((f_l\otimes e_l)\wedge x),
	\end{align}
	where we remind the readers that the minus sign appearing in the middle comes from the definition of the differential maps of (-1)-shifted complex (\cite[\href{https://stacks.math.columbia.edu/tag/011A}{011A}]{stacks-project}). Thus, $\df^n\circ\iota'^k=\iota'^k\circ\df^n$. On the other hand, for any $x\in M\otimes_R\wedge^nR^{\oplus d}$,
	\begin{align}
		\pi^k(\df^n(x))=&\pi^k(\sum_{l=1}^d (f_l\otimes e_l)\wedge x)=\sum_{\substack{l=1\\l\neq k}}^d (f_l\otimes e_l)\wedge \pi^k(x)=\df^n(\pi^k(x)),
	\end{align}
	where the equality in the middle follows from the fact that $\pi^k(e_k\wedge x)=0$. Thus, $\df^n\circ \pi^k= \pi^k\circ\df^n$.
	
	It remains to verify that 
	\begin{align}
		\pi'^k\circ\df^n\circ\iota^k:M\otimes_R\wedge^nR^{\oplus (d-1)}\longrightarrow M\otimes_R\wedge^nR^{\oplus (d-1)}
	\end{align}
	is induced by $f_k$ (\cite[\href{https://stacks.math.columbia.edu/tag/014I}{014I}]{stacks-project}). Indeed, for any $x\in M\otimes_R\wedge^nR^{\oplus (d-1)}$,
	\begin{align}
		\pi'^k(\df^n(\iota^k(x)))=\pi'^k(\sum_{l=1}^d(f_l\otimes e_l)\wedge \iota^k(x))=\pi'^k((f_k\otimes e_k)\wedge \iota^k(x))=f_k(x),
	\end{align}
	where the second equality follows from the fact that there is no $e_k$ in the nontrivial components of $\iota^k(x)$.
\end{proof}

\begin{mypara}\label{para:notation-koszul-2}
	Let $R$ be a ring, $M$ an $R$-module, $d\in \bb{N}_{>0}$, $g_1,\dots,g_d\in R$. We also regard $\{g_1,\dots,g_d\}$ as commuting endomorphisms of $M$ by multiplication. We note that
	\begin{align}
		K^\bullet(M;g_1,\dots,g_d)=M\otimes_R K^\bullet(R;g_1,\dots,g_d).
	\end{align}
\end{mypara}

\begin{mylem}\label{lem:koszul}
	With the notation in {\rm\ref{para:notation-koszul-2}}, assume that $g_i=a_i\cdot g_1$ for some $a_i\in R$ for any $2\leq i\leq d$ and we put $a_1=1$. For any integer $0\leq n\leq d$, we denote by $Z^n(K^\bullet(M;g_1,\dots,g_d))$ (resp. $B^n(K^\bullet(M;g_1,\dots,g_d))$) the $R$-submodule of $n$-cocycles (resp. $n$-coboundaries) of the Koszul complex $K^\bullet(M;g_1,\dots,g_d)$, and by $H^n(K^\bullet(M;g_1,\dots,g_d))=Z^n(K^\bullet(M;g_1,\dots,g_d))/B^n(K^\bullet(M;g_1,\dots,g_d))$ the $n$-th cohomology group. We regard $R^{\oplus (d-1)}$ as the submodule of $R^{\oplus d}$ generated by $e_2,\dots,e_d$, where $(e_1,\dots,e_d)$ is the standard basis of $R^{\oplus d}$.
	\begin{enumerate}
		\renewcommand{\labelenumi}{{\rm(\theenumi)}}
		\item For any integer $0\leq n\leq d$, the inclusion $\iota^1:M\otimes_R \wedge^n R^{\oplus (d-1)}\to M\otimes_R \wedge^n R^{\oplus d}$ \eqref{eq:para:notation-koszul-0-1} induces an inclusion
		\begin{align}\label{eq:lem:koszul-1}
			\varphi_n: M[g_1]\otimes_R\wedge^n R^{\oplus (d-1)}\longrightarrow Z^n(K^\bullet(M;g_1,\dots,g_d)),
		\end{align}
		sending $x\otimes e_{i_1}\wedge\cdots \wedge e_{i_n}$ to $x\otimes e_{i_1}\wedge\cdots \wedge e_{i_n}$ for any integers $2\leq i_1<\cdots<i_n\leq d$ and any $x\in M[g_1]$, where $M[g_1]=\ke(g_1:M\to M)$. In particular, it induces an $R$-linear homomorphism
		\begin{align}\label{eq:lem:koszul-3}
			\varphi_n: M[g_1]\otimes_R\wedge^n R^{\oplus (d-1)}\longrightarrow H^n(K^\bullet(M;g_1,\dots,g_d)).
		\end{align}\label{item:lem:koszul-1}
		\item For any integer $1\leq n\leq d$, there is an $R$-linear homomorphism
		\begin{align}\label{eq:lem:koszul-4}
			\psi_n: M\otimes_R \wedge^{n-1} R^{\oplus (d-1)}\longrightarrow Z^n(K^\bullet(M;g_1,\dots,g_d))
		\end{align}
		sending $x=\sum_{2\leq i_2<\cdots<i_n\leq d} x_{i_2,\dots,i_n}\otimes e_{i_2}\wedge\cdots \wedge e_{i_n}$ to
		\begin{align}\label{eq:lem:koszul-5}
			\psi_n(x)=&\sum_{k=1}^d(a_k\otimes e_k)\wedge x\\
			=&\sum_{2\leq i_2<\cdots<i_n\leq d} x_{i_2,\dots,i_n}\otimes e_1\wedge e_{i_2}\wedge\cdots \wedge e_{i_n},\nonumber\\
			&+\sum_{2\leq i_1<i_2<\cdots<i_n\leq d} \left(\sum_{l=1}^n (-1)^{l-1}a_{i_l}x_{i_1,\dots,\widehat{i_l},\dots,i_n}\right)\otimes e_{i_1}\wedge e_{i_2}\wedge\cdots \wedge e_{i_n}\nonumber
		\end{align}
		where $x_{i_2,\dots,i_n}\in M$. Moreover, $\psi_n$ maps $g_1M\otimes_R \wedge^{n-1} R^{\oplus (d-1)}$ into $B^n(K^\bullet(M;g_1,\dots,g_d))$, and thus induces an $R$-linear homomorphism
		\begin{align}\label{eq:lem:koszul-6}
			\psi_n: M/g_1M\otimes_R \wedge^{n-1} R^{\oplus (d-1)}\longrightarrow H^n(K^\bullet(M;g_1,\dots,g_d)).
		\end{align}\label{item:lem:koszul-2}
	\end{enumerate}
\end{mylem}
\begin{proof} 
	(\ref{item:lem:koszul-1}) For any $x\in M[g_1]\otimes_R\wedge^n R^{\oplus (d-1)}$, 
	\begin{align}
		d^n(\varphi_n(x))=\sum_{k=1}^d(g_k\otimes e_k)\wedge x
	\end{align}
	which vanishes since the each component of $x$ lies in $M[g_1]$ and thus is killed by anyone of $g_1,\dots,g_d$.
	
	(\ref{item:lem:koszul-2}) For any $x\in M\otimes_R \wedge^{n-1} R^{\oplus (d-1)}$,
	\begin{align}
		\df^n(\psi_n(x))=&\df^n(\sum_{k=1}^d(a_k\otimes e_k)\wedge x)\\
		=&\sum_{l=1}^d(g_l\otimes e_l)\wedge \left(\sum_{k=1}^d(a_k\otimes e_k)\wedge x\right)\nonumber\\
		=&\sum_{1\leq l,k\leq d}(a_la_kg_1\otimes e_l\wedge e_k)\wedge x\nonumber\\
		=&0,\nonumber
	\end{align}
	which shows that $\psi_n(x)\in Z^n(K^\bullet(M;g_1,\dots,g_d))$. On the other hand, 
	\begin{align}
		\psi_n(g_1x)=\sum_{k=1}^d(a_kg_1\otimes e_k)\wedge x=\sum_{k=1}^d(g_k\otimes e_k)\wedge x=\df^n(x),
	\end{align}
	which shows that $\psi_n(g_1x)\in B^n(K^\bullet(M;g_1,\dots,g_d))$.
\end{proof}

\begin{myrem}\label{rem:koszul}
	In \ref{lem:koszul}, if $R$ is a valuation ring, then there are canonical choices for $a_1,\dots,a_d$ as follows: if $g_1\neq 0$, then we take $a_i=g_i/g_1$ for any $1\leq i\leq d$; if $g_1=0$, then we take $a_1=1$ and $a_2=\cdots=a_d=0$ in \ref{lem:koszul}. In particular, with these canonical choices, if $g_1=\cdots=g_d=0$, then $\varphi_n\oplus \psi_n$ coincides with the canonical isomorphism \eqref{eq:para:notation-koszul-0-1},
	\begin{align}
		(\iota^1,\iota'^1):M\otimes_R\wedge^nR^{\oplus (d-1)}\oplus M\otimes_R\wedge^{n-1}R^{\oplus (d-1)} \iso M\otimes_R\wedge^nR^{\oplus d},
	\end{align}
	 by the explicit definitions of \eqref{eq:lem:koszul-1} and \eqref{eq:lem:koszul-4}.
\end{myrem}

\begin{myprop}[{cf. \cite[7.10]{bhattmorrowscholze2018integral}}]\label{prop:koszul}
	Under the same assumptions in {\rm\ref{lem:koszul}} and with the same notation, for any integer $0\leq n\leq d$, the homomorphisms \eqref{eq:lem:koszul-1} and \eqref{eq:lem:koszul-4} induce an $R$-linear isomorphism 
	\begin{align}
		\varphi_n\oplus \psi_n: (M[g_1]\otimes_R\wedge^n R^{\oplus (d-1)})\oplus (M/g_1M\otimes_R \wedge^{n-1} R^{\oplus (d-1)})\iso H^n(K^\bullet(M;g_1,\dots,g_d)).
	\end{align}
\end{myprop}
\begin{proof}
	We follow closely the proof of \cite[7.10]{bhattmorrowscholze2018integral}. We take induction on $d$. For $d=1$, $K^\bullet(M;g_1,\dots,g_d)$ is just the complex $M\stackrel{g_1}{\longrightarrow}M$ and the statement is clear. In general, assume that the statement holds for $d-1$. Then, for any integer $0\leq n\leq d$, there is a diagram
	\begin{align}
		\xymatrix{
			M[g_1]\otimes\wedge^{n-1} R^{\oplus (d-2)}\ar[r]^-{\iota'^{d-1}}\ar@{}[d]|-{\oplus}&M[g_1]\otimes\wedge^n R^{\oplus (d-1)}\ar[r]^-{\pi^{d-1}}\ar@{}[d]|-{\oplus}&M[g_1]\otimes\wedge^n R^{\oplus (d-2)}\ar@{}[d]|-{\oplus}\\
			M/g_1M\otimes \wedge^{n-2} R^{\oplus (d-2)}\ar[r]^-{-\iota'^{d-1}}\ar[d]^-{\varphi_{n-1}\oplus\psi_{n-1}}&M/g_1M\otimes \wedge^{n-1} R^{\oplus (d-1)}\ar[r]^-{\pi^{d-1}}\ar[d]^-{\varphi_n\oplus\psi_n}&M/g_1M\otimes \wedge^{n-1} R^{\oplus (d-2)}\ar[d]^-{\varphi_n\oplus\psi_n}\\
			H^{n-1}(K^\bullet(M;g_1,\dots,g_{d-1}))\ar[r]^-{\iota'^d}&H^n(K^\bullet(M;g_1,\dots,g_d))\ar[r]^-{\pi^d}&H^n(K^\bullet(M;g_1,\dots,g_{d-1}))
		}
	\end{align}
	where the first and second rows are the split exact sequences induced by \eqref{eq:para:notation-koszul-0-1}, and the third row is the exact sequence associated to \eqref{eq:lem:koszul-basic-1}. By induction, the vertical arrows on the left and right are isomorphisms. Thus, it remains to check the commutativity of this diagram.
	
	We regard $R^{\oplus(d-1)}$ (resp. $R^{\oplus(d-2)}$) as the $R$-submodule of $R^{\oplus d}$ generated by $e_2,\dots,e_d$ (resp. $e_2,\dots,e_{d-1}$).
	
	For any integers $2\leq i_1<\cdots<i_{n-1}\leq d-1$ and any $x\in M[g_1]$, we have
	\begin{align}
		\iota'^d(\varphi_{n-1}(x\otimes e_{i_1}\wedge\cdots \wedge e_{i_{n-1}}))&=\iota'^d(x\otimes e_{i_1}\wedge\cdots \wedge e_{i_{n-1}})=x\otimes e_d\wedge e_{i_1}\wedge\cdots \wedge e_{i_{n-1}},\\
		\varphi_n(\iota'^{d-1}(x\otimes e_{i_1}\wedge\cdots \wedge e_{i_{n-1}}))&=\varphi_n(x\otimes e_d\wedge e_{i_1}\wedge\cdots \wedge e_{i_{n-1}})=x\otimes e_d\wedge e_{i_1}\wedge\cdots \wedge e_{i_{n-1}}.
	\end{align}
	Thus, $\iota'^d\circ\varphi_{n-1}=\varphi_n\circ\iota'^{d-1}$.
	
	For any integers $2\leq i_1<\cdots<i_{n-2}\leq d-1$ and any $x\in M$, we have
	\begin{align}
		\iota'^d(\psi_{n-1}(x\otimes e_{i_1}\wedge\cdots \wedge e_{i_{n-2}}))&=e_d\wedge \psi_{n-1}(x\otimes e_{i_1}\wedge\cdots \wedge e_{i_{n-2}}) \\
		&=e_d\wedge \sum_{k=1}^{d-1}(a_kx\otimes e_k)\wedge e_{i_1}\wedge\cdots \wedge e_{i_{n-2}},\nonumber\\
		\psi_n(-\iota'^{d-1}(x\otimes e_{i_1}\wedge\cdots \wedge e_{i_{n-2}}))&=\psi_n(-x\otimes e_d\wedge e_{i_1}\wedge\cdots \wedge e_{i_{n-2}})\\
		&=-\sum_{k=1}^d (a_kx\otimes e_k)\wedge e_d\wedge e_{i_1}\wedge\cdots \wedge e_{i_{n-2}}.\nonumber
	\end{align}
	Thus, $\iota'^d\circ\psi_{n-1}=\psi_n\circ(-\iota'^{d-1})$.
	
	For any integers $2\leq i_1<\cdots<i_n\leq d$ and any $x\in M[g_1]$, we have
	\begin{align}
		\pi^d(\varphi_n(x\otimes e_{i_1}\wedge\cdots\wedge e_{i_n}))=&\pi^d(x\otimes e_{i_1}\wedge\cdots\wedge e_{i_n})=\left\{\begin{array}{ll}
			0,&\trm{if }i_n= d,\\
			x\otimes e_{i_1}\wedge\cdots\wedge e_{i_n},&\trm{if }i_n\neq d,
		\end{array}\right.\\
		\varphi_n(\pi^{d-1}(x\otimes e_{i_1}\wedge\cdots\wedge e_{i_n}))=&\left\{\begin{array}{ll}
			0,&\trm{if }i_n= d,\\
			\varphi_n(x\otimes e_{i_1}\wedge\cdots\wedge e_{i_n}),&\trm{if }i_n\neq d.
		\end{array}\right.
	\end{align}
	Thus, $\pi^d\circ\varphi_n=\varphi_n\circ\pi^{d-1}$.
	
	For any integers $2\leq i_1<\cdots<i_{n-1}\leq d$ and any $x\in M$, we have
	\begin{align}
		\pi^d(\psi_n(x\otimes e_{i_1}\wedge\cdots\wedge e_{i_{n-1}}))=&\pi^d(\sum_{k=1}^d(a_kx\otimes e_k)\wedge e_{i_1}\wedge\cdots\wedge e_{i_{n-1}})\\
		=&\left\{\begin{array}{ll}
			0,&\trm{if }i_{n-1}= d,\\
			\sum_{k=1}^{d-1}(a_kx\otimes e_k)\wedge e_{i_1}\wedge\cdots\wedge e_{i_{n-1}},&\trm{if }i_{n-1}\neq d,
		\end{array}\right.\nonumber\\
		\psi_n(\pi^{d-1}(x\otimes e_{i_1}\wedge\cdots\wedge e_{i_{n-1}}))=&\left\{\begin{array}{ll}
			0,&\trm{if }i_{n-1}= d,\\
			\psi_n(x\otimes e_{i_1}\wedge\cdots\wedge e_{i_{n-1}}),&\trm{if }i_{n-1}\neq d.
		\end{array}\right.
	\end{align}
	Thus, $\pi^d\circ\psi_n=\psi_n\circ\pi^{d-1}$.
\end{proof}

\begin{mycor}\label{cor:koszul}
	Under the same assumptions in {\rm\ref{lem:koszul}} and with the same notation, let $g\in R$ such that the multiplication by $g$ on $M$ is injective and assume moreover that $g=a\cdot g_1$ for some $a\in R$. Then, for any integer $0\leq n\leq d$, the following diagram is commutative
	\begin{align}\label{diam:cor:koszul}
		\xymatrix{
			aM/gM\otimes_R\wedge^nR^{\oplus (d-1)}\ar[d]_-{\varphi_n}\ar[r]^-{\sim}&M/g_1M\otimes_R\wedge^nR^{\oplus (d-1)}\ar[d]^-{\psi_{n+1}}\\
			H^n(K^\bullet(M/gM;g_1,\dots,g_d))\ar[r]^-{\delta^n}&H^{n+1}(K^\bullet(M/gM;g_1,\dots,g_d))
		}
	\end{align}
	where the horizontal isomorphism is the division by $a$, and $\delta^n$ is the $n$-th coboundary map associated to the exact sequence 
	\begin{small}\begin{align}
		\xymatrix{
			0\ar[r]&K^\bullet(M/gM;g_1,\dots,g_d)\ar[r]^-{\cdot g}&K^\bullet(M/g^2M;g_1,\dots,g_d)\ar[r]&K^\bullet(M/gM;g_1,\dots,g_d)\ar[r]&0.
		}
	\end{align}\end{small}In particular, the sequence
	\begin{tiny}\begin{align}
		\xymatrix{
			\cdots\ar[r]^-{\delta^{n-2}}&H^{n-1}(K^\bullet(M/gM;g_1,\dots,g_d))\ar[r]^-{\delta^{n-1}}&H^n(K^\bullet(M/gM;g_1,\dots,g_d))\ar[r]^-{\delta^n}&H^{n+1}(K^\bullet(M/gM;g_1,\dots,g_d))\ar[r]^-{\delta^{n+1}}&\cdots
		}
	\end{align}\end{tiny}is exact.
\end{mycor}
\begin{proof}
	For any integers $2\leq i_1<\cdots<i_n\leq d$ and any $x\in M$, we have
	\begin{align}
		\delta^n(\varphi_n(ax\otimes e_{i_1}\wedge\cdots \wedge e_{i_n}))=&\delta^n(ax\otimes e_{i_1}\wedge\cdots \wedge e_{i_n})\\
		=&g^{-1}\sum_{k=1}^d(g_k\otimes e_k)\wedge (ax\otimes e_{i_1}\wedge\cdots \wedge e_{i_n})\nonumber\\
		=&\sum_{k=1}^d(a_kx\otimes e_k)\wedge e_{i_1}\wedge\cdots \wedge e_{i_n}\nonumber\\
		=&\psi_{n+1}(x\otimes e_{i_1}\wedge\cdots \wedge e_{i_n})\nonumber.
	\end{align}
	Thus, the diagram \eqref{diam:cor:koszul} is commutative. Then, the conclusion follows directly from \ref{prop:koszul}.
\end{proof}

\section{Perfectoid Towers over Essentially Adequate Algebras}\label{sec:essential-adequate-alg}
Firstly, we take a brief review on the definitions and basic properties of essentially adequate algebras following Tsuji \cite[\textsection4]{tsuji2018localsimpson}, and then we compute its Galois cohomology by Koszul complexes using a specific perfectoid tower following Abbes-Gros \cite[\Luoma{2}.6, \Luoma{2}.8]{abbes2016p}. Finally, we specialize the computation of the cohomology of Koszul complexes in the last section to this situation (see \ref{thm:coh}), and show that the Galois cohomology of an essentially adequate algebra is controlled by its localizations at generic points of the special fibre (see \ref{cor:coh}).

\begin{mypara}\label{para:zeta}
	In this section, we fix a complete discrete valuation field $K$ of characteristic $0$ with perfect residue field of characteristic $p>0$, an algebraic closure $\overline{K}$ of $K$, and a compatible system of primitive $n$-th roots of unity $(\zeta_n)_{n\in\bb{N}}$ in $\overline{K}$. Sometimes we denote $\zeta_n$ by $t_{0,n}$. For any rational number $r$, we put
	\begin{align}
		\zeta^r=\zeta_n^m,
	\end{align}
	where $n\in\bb{N}_{>0}$, $m\in\bb{Z}$ coprime to $n$. We note that for any $n\in\bb{N}_{>0}$ and any integer $m$ coprime to $p$, the valuation of $\zeta_{p^n}^m-1$ in $K$ is
	\begin{align}\label{eq:para:zeta-2}
		v_K(\zeta_{p^n}^m-1)=\frac{1}{p^{n-1}(p-1)}v_K(p).
	\end{align} 
	We put $\bb{Z}_p(1)=\lim_{n\to\infty}\mu_{p^n}(\overline{K})$, where $\mu_{p^n}(\overline{K})$ is the subgroup of $p^n$-th roots of unity and the transition map $\mu_{p^{n+1}}(\overline{K})\to \mu_{p^n}(\overline{K})$ is taking the $p$-th power.
\end{mypara}

\begin{mypara}\label{para:product}
	Let $d\in\bb{N}$ be a natural number. We endow the set $(\bb{N}\cup\{\infty\})^d$ with the partial order defined by $\underline{m}\leq \underline{m'}$ if ether $m'_i=\infty$ or $m_i\leq m'_i<\infty$ for any $1\leq i\leq d$, where $\underline{m}=(m_1,\dots,m_d)$ and $\underline{m'}=(m_1',\dots,m_d')$. For any $r\in \bb{N}\cup\{\infty\}$, we set $\underline{r}=(r,\dots,r)\in (\bb{N}\cup\{\infty\})^d$.
	
	On the other hand, we endow the set $(\bb{N}_{>0}\cup\{\infty\})^d$ with the partial order defined by $\underline{N}| \underline{N'}$ if either $N'_i=\infty$ or $N_i$ divides $N'_i<\infty$ for any $1\leq i\leq d$, where $\underline{N}=(N_1,\dots,N_d)$ and $\underline{N'}=(N_1',\dots,N_d')$.
\end{mypara}

\begin{mydefn}[{\cite[9.2]{he2022sen}}]\label{defn:triple}
	A \emph{$(K,\ca{O}_K,\ca{O}_{\overline{K}})$-triple} is a triple $(A_{\triv},A,\overline{A})$ where
	\begin{enumerate}
		\renewcommand{\labelenumi}{{\rm(\theenumi)}}
		\item $A$ is a Noetherian normal domain flat over $\ca{O}_K$ with $A/pA\neq 0$,
		\item $A_\triv$ is a $K$-algebra that is a localization of $A$ with respect to a nonzero element of $pA$,
		\item $\overline{A}$ is an $\ca{O}_{\overline{K}}$-algebra that is the integral closure of $A$ in a maximal unramified extension $\ca{K}_{\mrm{ur}}$ of the fraction field $\ca{K}$ of $A$ with respect to $(A_{\triv},A)$, i.e., $\ca{K}_{\mrm{ur}}$ is the union of all finite field extension $\ca{K}'$ of $\ca{K}$ contained in an algebraic closure $\overline{\ca{K}}$ such that the integral closure of $A_{\triv}$ in $\ca{K}'$ is \'etale over $A_{\triv}$,
	\end{enumerate}
	such that the diagram
	\begin{align}
		\xymatrix{
			A_{\triv}&A\ar[l]\ar[r]&\overline{A}\\
			K\ar[u]&\ca{O}_K\ar[l]\ar[r]\ar[u]&\ca{O}_{\overline{K}}\ar[u]
		}
	\end{align}
	formed by the structural morphisms is commutative. 
\end{mydefn}

\begin{mydefn}[{\cite[page 797]{tsuji2018localsimpson}, cf. \cite[9.5]{he2022sen}}]\label{defn:essential-adequate-alg}
	A $(K,\ca{O}_K,\ca{O}_{\overline{K}})$-triple $(A_{\triv},A,\overline{A})$ is called (resp. \emph{essentially}) \emph{adequate} if there exists a commutative diagram of monoids
	\begin{align}
		\xymatrix{
			A& P\ar[l]_-{\beta}\\
			\ca{O}_K\ar[u]& \bb{N}\ar[l]_-{\alpha}\ar[u]_-{\gamma}
		}
	\end{align} 
	satisfying the following conditions:
	\begin{enumerate}
		\renewcommand{\labelenumi}{{\rm(\theenumi)}}
		\item The element $\alpha(1)$ is a uniformizer of $\ca{O}_K$.\label{item:defn:essential-adequate-alg-1}
		\item The monoid $P$ is fs (i.e., fine and saturated), and if we denote by $\gamma_\eta:\bb{Z}\to P_\eta=\bb{Z}\oplus_{\bb{N}}P$ the pushout of $\gamma$ by the inclusion $\bb{N}\to \bb{Z}$, then there exists an isomorphism for some $c, d\in \bb{N}$ with $c\leq d$,
		\begin{align}\label{eq:monoid-str}
			P_\eta\cong \bb{Z}\oplus \bb{Z}^c\oplus \bb{N}^{d-c},
		\end{align}
		identifying $\gamma_\eta$ with the inclusion of $\bb{Z}$ into the first component on the right hand side.\label{item:defn:essential-adequate-alg-2}
		\item The homomorphism $\beta$ induces an $\ca{O}_K$-algebra homomorphism $\ca{O}_K\otimes_{\bb{Z}[\bb{N}]}\bb{Z}[P]\to A$ which is (resp. \emph{a localization of}) an \'etale homomorphism such that $A\otimes_{\bb{Z}[P]}\bb{Z}[P^{\mrm{gp}}]=A_{\triv}$, where $P^{\mrm{gp}}$ is the associated group of $P$.\label{item:defn:essential-adequate-alg-3}
	\end{enumerate}
	We usually denote $(A_{\triv},A,\overline{A})$ by $A$, and call it an (resp. \emph{essentially}) \emph{adequate $\ca{O}_K$-algebra} for simplicity. The triple $(\alpha:\bb{N}\to \ca{O}_K,\ \beta:P\to A,\ \gamma:\bb{N}\to P)$ is called an (resp. \emph{essentially}) \emph{adequate chart} of $A$. If we fix an isomorphism \eqref{eq:monoid-str}, then we call the images $t_1,\dots,t_d\in A[1/p]$ of the standard basis of $\bb{Z}^c\oplus \bb{N}^{d-c}$ a \emph{system of coordinates} of the chart. We call $d$ the \emph{relative dimension} of $A$ over $\ca{O}_K$ (i.e., the Krull dimension of $A_{\triv}$).
\end{mydefn}

\begin{myrem}\label{rem:quasi-adequate}
	In \ref{defn:essential-adequate-alg}, $B=\ca{O}_K\otimes_{\bb{Z}[\bb{N}]}\bb{Z}[P]$ is a Noetherian normal domain and $B_{\triv}=\ca{O}_K\otimes_{\bb{Z}[\bb{N}]}\bb{Z}[P^{\mrm{gp}}]$ is the localization of $\ca{O}_K\otimes_{\bb{Z}[\bb{N}]}\bb{Z}[P]$ by inverting $p\cdot t_1\cdots t_d$ (see \cite[8.9]{he2022sen}). Thus, if $\overline{B}$ is the integral closure of $B$ in a maximal unramified extension of the fraction field of $B$ with respect to $(B_{\triv},B)$ (see \ref{defn:triple}), then the $(K,\ca{O}_K,\ca{O}_{\overline{K}})$-triple $(B_{\triv},B,\overline{B})$ is adequate.
\end{myrem}

\begin{myexample}[Semi-stable charts]\label{exam:semi-stable-chart}
	We fix integers $0\leq b\leq c\leq d$ and $e\in\bb{N}_{>0}$. Let $P$ be the submonoid of $\bb{N}^{1+b}\oplus\bb{Z}^{c-b}\oplus\bb{N}^{d-c}$ generated by $(e\bb{N})^{1+b}\oplus \bb{Z}^{c-b}\oplus\bb{N}^{d-c}$ and $(1,1,\dots,1,0,\dots,0)\in \{1\}^{1+b}\times\{0\}^{d-b}$. In other words,
	\begin{align}\label{eq:exam:semi-stable-chart-1}
		P=\{(a_0,\dots,a_d)\in \bb{N}^{1+b}\oplus\bb{Z}^{c-b}\oplus\bb{N}^{d-c}\ |\ a_0\equiv a_1\equiv\cdots\equiv a_b\mod e\}.
	\end{align}
	Let $\gamma:\bb{N}\to P$ be the homomorphism of monoids sending $1$ to $(1,1,\dots,1,0,\dots,0)\in \{1\}^{1+b}\times\{0\}^{d-b}\subseteq P$. If we denote by $\gamma_\eta:\bb{Z}\to P_\eta=\bb{Z}\oplus_{\bb{N}}P$ the pushout of $\gamma$ by the inclusion $\bb{N}\to \bb{Z}$, then we see that
	\begin{align}
		P_\eta&=\{(a_0,\dots,a_d)\in \bb{Z}^{1+c}\oplus\bb{N}^{d-c}\ |\ a_0\equiv a_1\equiv\cdots\equiv a_b\mod e\},\\
		P^{\mrm{gp}}&=\{(a_0,\dots,a_d)\in \bb{Z}^{1+d}\ |\ a_0\equiv a_1\equiv\cdots\equiv a_b\mod e\}.
	\end{align}
	In particular, $P=P^{\mrm{gp}}\cap (\bb{N}^{1+b}\oplus\bb{Z}^{c-b}\oplus\bb{N}^{d-c})$ is saturated in $P^{\mrm{gp}}$ and thus is an fs monoid. Notice that there is an isomorphism of monoids
	\begin{align}\label{eq:exam:semi-stable-chart-4}
		P_\eta\iso \bb{Z}\oplus\bb{Z}^c\oplus\bb{N}^{d-c},\ (a_0,\dots,a_d)\mapsto (a_0,(a_1-a_0)/e,\dots,(a_b-a_0)/e,a_{b+1},\dots,a_d),
	\end{align}
	identifying $\gamma_\eta$ with the inclusion of $\bb{Z}$ into the first component on the right hand side. We take a uniformizer $\pi_K$ of $\ca{O}_K$ and let $\alpha:\bb{N}\to \ca{O}_K$ be the homomorphism of monoids sending $1$ to $\pi_K$. Then, $\ca{O}_K\otimes_{\bb{Z}[\bb{N}]}\bb{Z}[P]$ is an adequate $\ca{O}_K$-algebra by \ref{rem:quasi-adequate}, and there is an isomorphism of $\ca{O}_K$-algebras
	\begin{align}\label{eq:exam:semi-stable-chart-5}
		\ca{O}_K[T_0,T_1,\dots,T_b,T_{b+1}^{\pm 1},\dots,T_c^{\pm 1},T_{c+1},\dots,T_d]/(T_0\cdots T_b-\pi_K^e)\iso \ca{O}_K\otimes_{\bb{Z}[\bb{N}]}\bb{Z}[P]
	\end{align}
	sending $T_i$ to $1\otimes(0,\dots,e,\dots,0)$ where $e$ appears on the $i$-th position for any $0\leq i\leq d$. 
\end{myexample}

\begin{mypara}\label{para:cover}
	Let $\gamma:\bb{N}\to P$ be an injective homomorphism of fs monoids such that there exists an isomorphism for some $c, d\in \bb{N}$ with $c\leq d$,
	\begin{align}\label{eq:monoid-str-P}
		P_\eta\cong \bb{Z}\oplus \bb{Z}^c\oplus \bb{N}^{d-c}
	\end{align}
	identifying $\gamma_\eta:\bb{Z}\to P_\eta=\bb{Z}\oplus_{\bb{N}}P$ with the inclusion of $\bb{Z}$ into the first component of right hand side as in \ref{defn:essential-adequate-alg}.(\ref{item:defn:essential-adequate-alg-2}). We identify $P^{\mrm{gp}}$ with $\bb{Z}^{1+d}$ and $\bb{N}^{\mrm{gp}}$ with the first component of $\bb{Z}^{1+d}$. For any $e\in\bb{N}_{>0}$, we define a submonoid of $\bb{Q}^{1+d}$ by 
	\begin{align}\label{eq:para:cover-2}
		P_{e,\underline{\infty}}=\{x\in e^{-1}\bb{Z}\oplus \bb{Q}^d\ |\ \exists k\in\bb{N}_{>0}\trm{ s.t. }kx\in P\}.
	\end{align}
	For any $\underline{r}=(r_1,\dots,r_d)\in \bb{N}_{>0}^d$, we put
	\begin{align}\label{eq:para:cover-3}
		P_{e,\underline{r}}=(e^{-1}\bb{Z}\oplus r_1^{-1}\bb{Z}\oplus\cdots \oplus r_d^{-1}\bb{Z})\cap P_{e,\underline{\infty}}.
	\end{align}
	It is an fs monoid (\cite[3.2]{tsuji2018localsimpson}), and if we denote by $P_{e,\underline{r},\eta}$ the pushout $e^{-1}\bb{Z}\oplus_{e^{-1}\bb{N}}P_{e,\underline{r}}$, then there is an isomorphism
	\begin{align}\label{eq:monoid-tower-str-P}
		P_{e,\underline{r},\eta}\cong e^{-1}\bb{Z}\oplus r_1^{-1}\bb{Z}\oplus\cdots\oplus r_c^{-1}\bb{Z}\oplus r_{c+1}^{-1}\bb{N}\oplus \cdots\oplus r_d^{-1}\bb{N}
	\end{align}
	induced by \eqref{eq:monoid-str-P} (\cite[page 810, equation (2)]{tsuji2018localsimpson}). 
\end{mypara}

\begin{mylem}\label{lem:decomposition}
	With the notation in {\rm\ref{para:cover}}, let $e\in\bb{N}_{>0}$, $\underline{r}=(r_1,\dots,r_d)\in \bb{N}_{>0}^d$. 
	\begin{enumerate}
		\renewcommand{\labelenumi}{{\rm(\theenumi)}}
		\item There is a decomposition of sets
		\begin{align}\label{eq:lem:decomposition-1}
			P_{e,\underline{r}}=\coprod_{\underline{\rho}\in \prod_{i=1}^d(r_i^{-1}\bb{Z})\cap [0,1)}P_{e,\underline{\infty}}^{(\underline{\rho})},
		\end{align}
		where $P_{e,\underline{\infty}}^{(\underline{\rho})}=P_{e,\underline{\infty}}\cap(e^{-1}\bb{Z}\oplus (\underline{\rho}+\bb{Z}^d))$ and $[0,1)=\{x\in\bb{R}\ |\ 0\leq x<1\}$.\label{item:lem:decomposition-1}
		\item Each $P_{e,\underline{\infty}}^{(\underline{\rho})}$ is stable under the addition by $P_{e,\underline{1}}$.\label{item:lem:decomposition-2}
		\item Let $\bb{Z}[P_{e,\underline{\infty}}^{(\underline{\rho})}]$ be the free $\bb{Z}$-submodule of $\bb{Z}[P_{e,\underline{r}}]$ with basis formed by the elements of $P_{e,\underline{\infty}}^{(\underline{\rho})}$ (see \cite[\Luoma{1}.3.1]{ogus2018log}). Then, it is a $\bb{Z}[P_{e,\underline{1}}]$-module of finite presentation, and \eqref{eq:lem:decomposition-1} induces a decomposition of $\bb{Z}[P_{e,\underline{r}}]$ into its $\bb{Z}[P_{e,\underline{1}}]$-submodules of finite presentation
		\begin{align}\label{eq:lem:decomposition-2}
			\bb{Z}[P_{e,\underline{r}}]=\bigoplus_{\underline{\rho}\in \prod_{i=1}^d(r_i^{-1}\bb{Z})\cap [0,1)}\bb{Z}[P_{e,\underline{\infty}}^{(\underline{\rho})}]. 
		\end{align}
		\label{item:lem:decomposition-3}
	\end{enumerate}
\end{mylem}
\begin{proof}
	(\ref{item:lem:decomposition-1}) Notice that there exists a decomposition of $e^{-1}\bb{Z}\oplus \bb{Q}^d$,
	\begin{align}
		e^{-1}\bb{Z}\oplus \bb{Q}^d=\coprod_{\underline{\rho}\in (\bb{Q}\cap [0,1))^d}e^{-1}\bb{Z}\oplus (\underline{\rho}+\bb{Z}^d).
	\end{align}
	Thus, $P_{e,\underline{r}}=\coprod_{\underline{\rho}\in (\bb{Q}\cap [0,1))^d}(e^{-1}\bb{Z}\oplus r_1^{-1}\bb{Z}\oplus\cdots \oplus r_d^{-1}\bb{Z})\cap P_{e,\underline{\infty}}\cap (e^{-1}\bb{Z}\oplus (\underline{\rho}+\bb{Z}^d))$ \eqref{eq:para:cover-3}. The conclusion follows from the fact that 
	\begin{align}
		&(e^{-1}\bb{Z}\oplus r_1^{-1}\bb{Z}\oplus\cdots \oplus r_d^{-1}\bb{Z})\cap (e^{-1}\bb{Z}\oplus (\underline{\rho}+\bb{Z}^d))\\
		=&\left\{\begin{array}{ll}
			e^{-1}\bb{Z}\oplus (\underline{\rho}+\bb{Z}^d),&\trm{if } \underline{\rho}\in \prod_{i=1}^d(r_i^{-1}\bb{Z})\cap [0,1),\\
			\emptyset,&\trm{otherwise}.
		\end{array}\right.\nonumber
	\end{align}
	
	(\ref{item:lem:decomposition-2}) It follows from the fact that $P_{e,\underline{1}}$ is a submonoid of $P_{e,\underline{\infty}}\cap(e^{-1}\bb{Z}\oplus \bb{Z}^d)$.
	
	(\ref{item:lem:decomposition-3})  By the definition of the algebra structure on $\bb{Z}[P_{e,\underline{r}}]$ (\cite[\Luoma{1}.3.1]{ogus2018log}), $\bb{Z}[P_{e,\underline{\infty}}^{(\underline{\rho})}]$ is a $\bb{Z}[P_{e,\underline{1}}]$-submodule as $P_{e,\underline{\infty}}^{(\underline{\rho})}$ is stable under the addition by $P_{e,\underline{1}}$, and we also get the decomposition \eqref{eq:lem:decomposition-2}. Finally, since $\bb{Z}[P_{e,\underline{r}}]$ is finite over the finitely generated $\bb{Z}$-algebra $\bb{Z}[P_{e,\underline{1}}]$ by definition \eqref{eq:para:cover-2}, it is of finite presentation as a module and so are its direct summands.
\end{proof}

\begin{mypara}\label{para:adequate-setup}
	In the rest of this section, we fix an essentially adequate $\ca{O}_K$-algebra $A$, an essentially adequate chart $(\alpha:\bb{N}\to \ca{O}_K,\ \beta:P\to A,\ \gamma:\bb{N}\to P)$ of $A$, and an isomorphism
	$P_\eta\cong \bb{Z}\oplus \bb{Z}^c\oplus \bb{N}^{d-c}$ \eqref{eq:monoid-str}. Let $t_1,\dots,t_d\in A[1/p]$ be the associated system of coordinates of this chart.
	
	Following \cite[page 797]{tsuji2018localsimpson}, let $\ca{A}$ be one of following three $A$-algebras satisfying the corresponding assumptions:
	
	(Case I) The algebra $A$ itself.
	
	(Case II) The Henselization of $(A,pA)$, and we assume that $\spec(A/pA)$ is connected.
	
	(Case III) The $p$-adic completion of $A$, and we assume that $\spec(A/pA)$ is connected.
	
	Then, $\ca{A}$ is still a Noetherian normal domain flat over $\ca{O}_K$ with $\ca{A}/p\ca{A}\neq 0$ (\cite[4.6]{tsuji2018localsimpson}). We put $\ca{A}_{\triv}=A_{\triv}\otimes_A\ca{A}$, $\ca{K}$ the fraction field of $\ca{A}$, and we fix a maximal unramified extension $\ca{K}_{\mrm{ur}}$ of $\ca{K}$ with respect to $(\ca{A}_{\triv},\ca{A})$ containing $\overline{K}$, and we put $\overline{\ca{A}}$ the integral closure of $\ca{A}$ in $\ca{K}_{\mrm{ur}}$. Then, $(\ca{A}_{\triv},\ca{A},\overline{\ca{A}})$ is a $(K,\ca{O}_K,\ca{O}_{\overline{K}})$-triple (\ref{defn:triple}).
	
	Following \cite[\textsection8]{tsuji2018localsimpson} (cf. \cite[9.13]{he2022sen}), for any integer $1\leq i\leq d$, we fix a compatible system of $k$-th roots $(t_{i,k})_{k\in \bb{N}}$ of $t_i$ in $\overline{\ca{A}}[1/p]$. For any field extension $E'/E$, let $\scr{F}_{E'/E}$ (resp. $\scr{F}^{\mrm{fini}}_{E'/E}$) be the set of algebraic (resp. finite) field extensions of $E$ contained in $E'$, and we endow it with the partial order defined by the inclusion relation. For any $L\in \ff{K}$ and any $\underline{r}=(r_1,\dots,r_d)\in\bb{N}^d_{>0}$, we set
	\begin{align}\label{eq:para:adequate-setup}
		\ca{K}^L_{\underline{r}}=L\ca{K}(t_{i,r_i}\ |\ 1\leq i\leq d)
	\end{align}
	where the composites of fields are taken in $\ca{K}_{\mrm{ur}}$. It is clear that $\ca{K}^L_{\underline{r}}$ forms an inductive system of fields over the directed partially ordered set $\ff{K}\times \bb{N}^d_{>0}$ (see \ref{para:product}). Let $\ca{A}^L_{\underline{r}}$ be the integral closure of $\ca{A}$ in $\ca{K}^L_{\underline{r}}$, $\ca{A}^L_{\underline{r},\triv}=\ca{A}_{\triv}\otimes_{\ca{A}}\ca{A}^L_{\underline{r}}$. Note that for any finite subextensions $L, L'$ of $\overline{K}/K$ with $L\subseteq L'$ and any $\underline{r},\underline{r'}\in\bb{N}^d_{>0}$ with $\underline{r}|\underline{r'}$ (\ref{para:product}), the transition morphism 
	\begin{align}
		\ca{A}^L_{\underline{r}}\longrightarrow\ca{A}^{L'}_{\underline{r'}}
	\end{align}
	is an injective finite homomorphism of Noetherian normal domains (\cite[\href{https://stacks.math.columbia.edu/tag/032L}{032L}]{stacks-project}). We extend the notation above to any $(L,\underline{r})\in\scr{F}_{\overline{K}/K}\times (\bb{N}_{>0}\cup\{\infty\})^d$ by taking filtered colimits, and we omit the index $L$ or $\underline{r}$ if $L=K$ or $\underline{r}=\underline{1}$ respectively.
\end{mypara}

We collect some basic properties of the system $(\ca{A}^L_{\underline{r}})_{(L,\underline{r})\in\scr{F}_{\overline{K}/K}\times (\bb{N}_{>0}\cup\{\infty\})^d}$ proved by Tsuji in \cite[\textsection4]{tsuji2018localsimpson} in the following proposition.

\begin{myprop}[{\cite[\textsection4]{tsuji2018localsimpson}}]\label{prop:tower-basic}
	Let $L\in \ff{K}$, $\underline{r}\in\bb{N}^d_{>0}$, $e\in\bb{N}_{>0}$ the integer associated to the morphism $\gamma:\bb{N}\to P$  defined in \cite[3.9.(2)]{tsuji2018localsimpson}.
	\begin{enumerate}
		\renewcommand{\labelenumi}{{\rm(\theenumi)}}
		\item The $(L,\ca{O}_L,\ca{O}_{\overline{K}})$-triple $(A^L_{\underline{r},\triv},A^L_{\underline{r}},\overline{A})$ is essentially adequate with a chart (with the notation in {\rm\ref{para:cover}})
		\begin{align}\label{eq:prop:tower-basic-1}
			(\alpha^L: e_L^{-1}\bb{N}\to \ca{O}_L,\ \beta^L_{\underline{r}}:P_{e_L,\underline{r}}\to A^L_{\underline{r}},\ \gamma^L_{\underline{r}}:e_L^{-1}\bb{N}\to P_{e_L,\underline{r}})
		\end{align}
		where $e_L$ is the ramification index of $L/K$, $\alpha^L$ is a homomorphism of monoids sending $e_L^{-1}$ to a uniformizer $\pi_L$ of $L$, $\beta^L_{\underline{r}}(k_0/e_L,k_1/r_1,\dots,k_d/r_d)=\pi_L^{k_0}\cdot t_{1,r_1}^{k_1}\cdots t_{d,r_d}^{k_d}$, and $\gamma^L_{\underline{r}}$ is induced by the inclusion of the first component of \eqref{eq:monoid-tower-str-P}.\label{item:prop:tower-basic-1}
		\item If $\ca{A}$ is the Henselization (resp. completion) of $(A,pA)$, then $\spec((A^L_{\underline{r}},pA^L_{\underline{r}})^h)$ (resp. $\spec((A^L_{\underline{r}},pA^L_{\underline{r}})^\wedge)$) is a finite disjoint union of Noetherian normal integral schemes, one of whose components identifies with $\spec(\ca{A}^L_{\underline{r}})$. In particular, $\spec(\ca{A}^L_{\underline{r}}/p\ca{A}^L_{\underline{r}})$ is an open and closed subscheme of $\spec(A^L_{\underline{r}}/pA^L_{\underline{r}})$.\label{item:prop:tower-basic-2}
		\item Consider the adequate $\ca{O}_K$-algebra $B=\ca{O}_K\otimes_{\bb{Z}[\bb{N}]}\bb{Z}[P]$ {\rm(\ref{rem:quasi-adequate})}. Then, we have
		\begin{align}\label{eq:prop:tower-basic-2}
			B^L_{\underline{r}}=\ca{O}_L\otimes_{\bb{Z}[e_L^{-1}\bb{N}]}\bb{Z}[P_{e_L,\underline{r}}],
		\end{align}
		where $B^L_{\underline{r}}$ is defined as in {\rm\ref{para:cover}} with respect to the same isomorphism $P_\eta\cong \bb{Z}\oplus \bb{Z}^c\oplus \bb{N}^{d-c}$ \eqref{eq:monoid-str} fixed in {\rm\ref{para:cover}} for $A$. Moreover, $\spec(\ca{A}\otimes_BB^L_{\underline{r}})$ is a finite disjoint union Noetherian normal integral schemes, one of whose components identifies with $\spec(\ca{A}^L_{\underline{r}})$ via \eqref{eq:prop:tower-basic-1}. \label{item:prop:tower-basic-3}
		\item The number of connected components of $\spec(\ca{A}\otimes_BB^L_{\underline{r}})$ and the number of generic points of $\spec(\ca{A}^L_{\underline{r}}/p\ca{A}^L_{\underline{r}})$ are bounded when $L$ and $\underline{r}$ vary. 
		 \label{item:prop:tower-basic-4}
		\item For any $\underline{r'}\in \bb{N}^d_{>0}$ with $\underline{r}|\underline{r'}$ such that $r_i'/r_i$ is a power of $p$ for any $1\leq i\leq d$, the morphism $\spec(\ca{A}\otimes_BB^L_{\underline{r'}})\to \spec(\ca{A}\otimes_BB^L_{\underline{r}})$ induces a bijection between the sets of connected components
		 \begin{align}\label{eq:prop:tower-basic-3}
		 	\pi_0(\spec(\ca{A}\otimes_BB^L_{\underline{r'}}))\iso \pi_0(\spec(\ca{A}\otimes_BB^L_{\underline{r}})).
		 \end{align} 
		 In particular, we have
		 \begin{align}\label{eq:prop:tower-basic-4}
		 	\ca{A}^L_{\underline{r'}}=\ca{A}^L_{\underline{r}}\otimes_{\bb{Z}[P_{e_L,\underline{r}}]}\bb{Z}[P_{e_L,\underline{r'}}].
		 \end{align}
		 Moreover, if $er_i|e_L$ for any $1\leq i\leq d$, then the morphism $\spec(\ca{A}^L_{\underline{r'}})\to \spec(\ca{A}^L_{\underline{r}})$ induces a bijection between the subsets of generic points of their special fibres {\rm(\ref{defn:generic})},
		 \begin{align}\label{eq:prop:tower-basic-5}
		 	\ak{G}(\ca{A}^L_{\underline{r'}}/p\ca{A}^L_{\underline{r'}})\iso \ak{G}(\ca{A}^L_{\underline{r}}/p\ca{A}^L_{\underline{r}}).
		 \end{align}
		 \label{item:prop:tower-basic-5}
		 \item Assume that $er_i|e_L$ for any $1\leq i\leq d$. Then, for any $L'\in\scr{F}_{\overline{K}/L}$, we have
		 \begin{align}\label{eq:prop:tower-basic-6}
		 	B^{L'}_{\underline{r}}=\ca{O}_{L'}\otimes_{\ca{O}_L}B^L_{\underline{r}}.
		 \end{align}
		 In particular, there exists $M\in\scr{F}^{\mrm{fini}}_{\overline{K}/L}$ such that for any $M'\in \scr{F}_{\overline{K}/M}$, we have
		 \begin{align}\label{eq:prop:tower-basic-7}
		 	\ca{A}^{M'}_{\underline{r}}=\ca{O}_{M'}\otimes_{\ca{O}_M}\ca{A}^M_{\underline{r}}.
		 \end{align}\label{item:prop:tower-basic-6}
	\end{enumerate} 
	
\end{myprop}
\begin{proof}
	(\ref{item:prop:tower-basic-1})  By \cite[9.15.(1)]{he2022sen}, we have $B^L_{\underline{r}}=\ca{O}_L\otimes_{\bb{Z}[e_L^{-1}\bb{N}]}\bb{Z}[P_{e_L,\underline{r}}]$, which is an adequate $\ca{O}_L$-algebra (see \ref{rem:quasi-adequate}). Since $B\to A$ is a localization of an \'etale ring homomorphism, $\spec(A\otimes_B B^L_{\underline{r}})$ is a finite disjoint union of Noetherian normal integral schemes, one of whose components is $\spec(A^L_{\underline{r}})$, and the conclusion follows.
	
	(\ref{item:prop:tower-basic-2}) As $A^L_{\underline{r}}$ is a Noetherian normal domain finite over $A$, $(A^h\otimes_AA^L_{\underline{r}},pA^h\otimes_AA^L_{\underline{r}})$ is still a Noetherian normal Henselian pair (\cite[\href{https://stacks.math.columbia.edu/tag/0AGV}{0AGV}, \href{https://stacks.math.columbia.edu/tag/0DYE}{0DYE}]{stacks-project}). Since its reduction modulo $p$ is $A^L_{\underline{r}}/pA^L_{\underline{r}}$ which is Noetherian,  $\spec(A^h\otimes_AA^L_{\underline{r}})$ is a finite disjoint union of normal integral schemes and $\spec((A^h)^L_{\underline{r}})$ is one of the components by construction.
	
	Since $A$ is an excellent ring, $A\to \widehat{A}$ is a regular ring homomorphism (\cite[7.8.3]{ega4-2}). Thus, $\widehat{A}\otimes_AA^L_{\underline{r}}=\widehat{A^L_{\underline{r}}}$ (as $A^L_{\underline{r}}$ is finite over $A$) is a Noetherian normal ring (\cite[6.5.2]{ega4-2}). As above, $\spec(\widehat{A^L_{\underline{r}}})$ is a finite disjoint union of normal integral schemes and $\spec((\widehat{A})^L_{\underline{r}})$ is one of the components by construction.
	
	(\ref{item:prop:tower-basic-3}) It follows from the arguments above.
	
	(\ref{item:prop:tower-basic-4}) The first part is proved in \cite[4.7]{tsuji2018localsimpson} whose arguments actually proves the second part (see \cite[9.3]{tsuji2018localsimpson} and \ref{prop:generic-map}.(\ref{item:prop:generic-map-1})).
	
	(\ref{item:prop:tower-basic-5}) The bijectivity of \eqref{eq:prop:tower-basic-3} is proved in \cite[4.8]{tsuji2018localsimpson} whose arguments actually proves \eqref{eq:prop:tower-basic-5}.
	
	(\ref{item:prop:tower-basic-6}) The equality \eqref{eq:prop:tower-basic-6} is proved in \cite[4.5]{tsuji2018localsimpson}, and thus \eqref{eq:prop:tower-basic-7} follows by combining with (\ref{item:prop:tower-basic-3}) and (\ref{item:prop:tower-basic-4}).
\end{proof}

\begin{myrem}\label{rem:sncd}
	For any $L\in \scr{F}_{\overline{K}/K}$ and $\underline{r}\in\bb{N}^d_{>0}$, $t_{c+1,r_{c+1}},\dots,t_{d,r_d}$ defines a strict normal crossings divisor on the regular scheme $\spec(\ca{A}^L_{\underline{r}}[1/p])$ with complement $\spec(\ca{A}^L_{\underline{r},\triv})$, i.e., in the localization of $\ca{A}^L_{\underline{r}}[1/p]$ at any point, those elements $t_i$ contained in the maximal ideal form a subset of a regular system of parameters. This follows from immediately the fact that $B^L_{\underline{r}}\to \ca{A}^L_{\underline{r}}$ is a regular ring homomorphism. We refer to \cite[10.3]{tsuji2018localsimpson} for a detailed proof.
\end{myrem}

\begin{mycor}[{\cite[page 840, equation (17)]{tsuji2018localsimpson}}]\label{cor:galois-group}
	Let $L\in \scr{F}_{\overline{K}/K}$ containing $\{\zeta_{p^n}\}_{n\in\bb{N}}$, $\underline{N}\in(\bb{N}\setminus p\bb{N})^d\subseteq \bb{N}^d_{>0}$. Then, there is an isomorphism of groups
	\begin{align}\label{eq:cor:galois-group-1}
		\xi:\gal(\ca{K}^L_{p^{\underline{\infty}}\underline{N}}/\ca{K}^L_{\underline{N}})\iso \prod_{i=1}^d\bb{Z}_p(1),\ \tau\mapsto ((\tau(t_{i,p^n})/t_{i,p^n})_{n\in\bb{N}})_i,
	\end{align}
	which identifies $\gal(\ca{K}^L_{p^{\underline{n}}\underline{N}}/\ca{K}^L_{\underline{N}})$ with $\prod_{i=1}^d(\bb{Z}/p^n\bb{Z})(1)=\prod_{i=1}^d\mu_{p^n}(\overline{K})$ for any $n\in\bb{N}$.
\end{mycor}
\begin{proof}
	For any $n\in\bb{N}$ and $L\in \ff{K}$ containing $\zeta_{p^n}$, we see that $\ca{K}^L_{p^{\underline{n}}\underline{N}}=\ca{K}^L_{\underline{N}}(t_{1,p^n},\dots,t_{d,p^n})$ \eqref{eq:para:adequate-setup} and it is a finite Galois extension of $\ca{K}^L_{\underline{N}}$. Thus, the map $\xi$ induces an injection
	\begin{align}
		\gal(\ca{K}^L_{p^{\underline{n}}\underline{N}}/\ca{K}^L_{\underline{N}})\longrightarrow \prod_{i=1}^d(\bb{Z}/p^n\bb{Z})(1).
	\end{align}
	It is actually an isomorphism, since the degree $[\ca{K}^L_{p^{\underline{n}}\underline{N}}:\ca{K}^L_{\underline{N}}]$ is equal to that of $\bb{Z}[P_{e_L,p^{\underline{n}}\underline{N}}^{\mrm{gp}}]=\bb{Z}[T_0^{\pm 1/e_L},T_1^{\pm 1/p^nN_1},\dots,T_d^{\pm 1/p^nN_d}]$ over ${\bb{Z}[P_{e_L,\underline{N}}^{\mrm{gp}}]}=\bb{Z}[T_0^{\pm 1/e_L},T_1^{\pm 1/N_1},\dots,T_d^{\pm 1/N_d}]$ by \eqref{eq:prop:tower-basic-4}, i.e., $p^{nd}$. The conclusion follows from taking cofiltered limit over $L$ and $n$.
\end{proof}

\begin{myprop}[{cf. \cite[\Luoma{2}.8.9]{abbes2016p}}]\label{prop:galois-action}
	Let $L\in \scr{F}_{\overline{K}/K}$ containing $\{\zeta_{p^n}\}_{n\in\bb{N}}$, $\underline{N}\in(\bb{N}\cup\{\infty\}\setminus p\bb{N})^d\subseteq (\bb{N}_{>0}\cup\{\infty\})^d$. Then, there is a canonical $\Delta_{p^\infty}=\gal(\ca{K}^L_{p^{\underline{\infty}}\underline{N}}/\ca{K}^L_{\underline{N}})$-equivariant decomposition of $\ca{A}^L_{p^{\underline{\infty}}\underline{N}}$ into its $\ca{A}^L_{\underline{N}}$-submodules,
	\begin{align}\label{eq:cor:galois-action-1}
		\ca{A}^L_{p^{\underline{\infty}}\underline{N}}=\bigoplus_{\nu\in\ho(\Delta_{p^\infty},\mu_{p^\infty}(\overline{K}))}\ca{A}^{L,(\nu)}_{p^{\underline{\infty}}\underline{N}},
	\end{align}
	where $\ca{A}^{L,(\nu)}_{p^{\underline{\infty}}\underline{N}}=\{x\in\ca{A}^L_{p^{\underline{\infty}}\underline{N}} \ |\ \tau(x)=\nu(\tau)x,\ \forall\ \tau\in\Delta_{p^\infty}\}$, $\mu_{p^\infty}(\overline{K})=\bigcup_{n\in\bb{N}}\mu_{p^n}(\overline{K})\subseteq \overline{K}$ the subgroup of $p$-power roots of unity viewed as a $\bb{Z}_p$-module (isomorphic to $\bb{Q}_p/\bb{Z}_p(1)$), and $\nu$ is a $\bb{Z}_p$-linear homomorphism. Moreover, for any $n\in\bb{N}$, we have
	\begin{align}\label{eq:cor:galois-action-2}
		\ca{A}^L_{p^{\underline{n}}\underline{N}}=\bigoplus_{\nu\in\ho(\Delta_{p^\infty},\mu_{p^n}(\overline{K}))}\ca{A}^{L,(\nu)}_{p^{\underline{\infty}}\underline{N}}.
	\end{align}
\end{myprop}
\begin{proof}
	For any $n\in\bb{N}$ and $L\in \ff{K}$ containing $\zeta_{p^n}$, we have $\ca{A}^L_{p^{\underline{n}}\underline{N}}=\ca{A}^L_{\underline{N}}\otimes_{\bb{Z}[P_{e_L,\underline{N}}]}\bb{Z}[P_{e_L,p^{\underline{n}}\underline{N}}]$ by \eqref{eq:prop:tower-basic-4}. By \ref{lem:decomposition}.(\ref{item:lem:decomposition-3}), there is a decomposition of $\ca{A}^L_{p^{\underline{n}}\underline{N}}$ into its $\ca{A}^L_{\underline{N}}$-submodules of finite presentation
	\begin{align}
		\ca{A}^L_{p^{\underline{n}}\underline{N}}=\bigoplus_{\underline{\rho}\in\prod_{i=1}^d(p^{-n}N_i^{-1}\bb{Z})\cap [0,N_i^{-1})}\ca{A}^L_{\underline{N}}\otimes_{\bb{Z}[P_{e_L,\underline{N}}]} \bb{Z}[P_{e_L,\underline{\infty}}^{(\underline{\rho})}],
	\end{align}
	where $P_{e_L,\underline{\infty}}^{(\underline{\rho})}=P_{e_L,\underline{\infty}}\cap (e_L^{-1}\bb{Z}\oplus(\rho_1+N_1^{-1}\bb{Z})\oplus\cdots\oplus(\rho_d+N_d^{-1}\bb{Z}))$. Notice that for any element $\tau\in \Delta_{p^\infty}$ and any $x_{\underline{\rho}}=(k_0/e_L,\rho_1+k_1/N_1,\dots,\rho_d+k_d/N_d)\in P_{e_L,\underline{\infty}}^{(\underline{\rho})}$ (where $k_0,\dots,k_d\in\bb{Z}$), we deduce from \ref{prop:tower-basic}.(\ref{item:prop:tower-basic-1}) that
	\begin{align}
		\tau(1\otimes x_{\underline{\rho}})=(\prod_{i=1}^d\tau(t_{i,p^nN_i}^{(\rho_i+k_i/N_i)p^nN_i})/t_{i,p^nN_i}^{(\rho_i+k_i/N_i)p^nN_i})\otimes x_{\underline{\rho}}=(\prod_{i=1}^d\tau(t_{i,p^nN_i}^{\rho_ip^nN_i})/t_{i,p^nN_i}^{\rho_ip^nN_i})\otimes x_{\underline{\rho}}.
	\end{align}
	We define a $\bb{Z}_p$-linear homomorphism
	\begin{align}
		\nu_{\underline{\rho}}:\Delta_{p^\infty}\longrightarrow \mu_{p^n}(\overline{K}),\ \tau\mapsto \prod_{i=1}^d\tau(t_{i,p^nN_i}^{\rho_ip^nN_i})/t_{i,p^nN_i}^{\rho_ip^nN_i}.
	\end{align}
	Combining with \eqref{eq:cor:galois-group-1}, one can check that this defines a bijection
	\begin{align}
		\prod_{i=1}^d(p^{-n}N_i^{-1}\bb{Z})\cap [0,N_i^{-1})\iso \ho(\Delta_{p^\infty},\mu_{p^n}(\overline{K})),\ \underline{\rho}\mapsto \nu_{\underline{\rho}}.
	\end{align}
	Indeed, its composition with the isomorphism induced by \eqref{eq:cor:galois-group-1} is
	\begin{align}
		\prod_{i=1}^d(p^{-n}N_i^{-1}\bb{Z})\cap [0,N_i^{-1})\longrightarrow \prod_{i=1}^d\ho(\bb{Z}_p(1),\mu_{p^n}(\overline{K})),\ \underline{\rho}\mapsto (\zeta\mapsto \zeta^{p^n\rho_i})_i,
	\end{align}
	which is clearly a bijection (\ref{para:zeta}).
	On the other hand, we see that 
	\begin{align}
		\ca{A}^L_{\underline{N}}\otimes_{\bb{Z}[P_{e_L,\underline{N}}]} \bb{Z}[P_{e_L,\underline{\infty}}^{(\underline{\rho})}]=\{x\in \ca{A}^L_{p^{\underline{n}}\underline{N}}\ |\  \tau(x)=\nu_{\underline{\rho}}(\tau)x,\ \forall\ \tau\in\Delta_{p^\infty}\},
	\end{align}
	which we denote by $\ca{A}^{L,(\nu_{\underline{\rho}})}_{p^{\underline{n}}\underline{N}}$. In other words, we decompose $\ca{A}^L_{p^{\underline{n}}\underline{N}}$ into its eigen-submodules under the action of $\Delta_{p^\infty}$,
	\begin{align}\label{eq:cor:galois-action-8}
		\ca{A}^L_{p^{\underline{n}}\underline{N}}=\bigoplus_{\nu\in\ho(\Delta_{p^\infty},\mu_{p^n}(\overline{K}))} \ca{A}^{L,(\nu)}_{p^{\underline{n}}\underline{N}}.
	\end{align}
	As the formation of taking eigen-submodules commutes with filtered colimit, we see that \eqref{eq:cor:galois-action-8} holds for any $n\in \bb{N}\cup\{\infty\}$ and any $L\in \scr{F}_{\overline{K}/K}$ containing $\{\zeta_{p^k}\}_{k\in\bb{N}_{\leq n}}$. 
	
	Finally, for any $L\in \scr{F}_{\overline{K}/K}$ containing $\{\zeta_{p^k}\}_{k\in\bb{N}}$, any $n\in\bb{N}$ and any $\nu\in \ho(\Delta_{p^\infty},\mu_{p^n}(\overline{K}))$, since $\ca{A}^{L,(\nu)}_{p^{\underline{m}}\underline{N}}$ is $p^n\Delta_{p^\infty}$-invariant for any $m\in\bb{N}_{\geq n}\cup\{\infty\}$, it lies in $(\ca{A}^L_{p^{\underline{m}}\underline{N}})^{p^n\Delta_{p^\infty}}=\ca{A}^L_{p^{\underline{n}}\underline{N}}$ by \ref{cor:galois-group}, and thus we have
	\begin{align}
		\ca{A}^{L,(\nu)}_{p^{\underline{m}}\underline{N}}=\ca{A}^{L,(\nu)}_{p^{\underline{n}}\underline{N}},
	\end{align}
	which verifies \eqref{eq:cor:galois-action-2} by \eqref{eq:cor:galois-action-8}.
\end{proof}
\begin{myrem}\label{rem:galois-action}
	In \ref{prop:galois-action}, if $L=\overline{K}$, then for any $\nu\in\ho(\Delta_{p^\infty},\mu_{p^\infty}(\overline{K}))$, $\ca{A}^{\overline{K},(\nu)}_{p^{\underline{\infty}}\underline{N}}$ is an $\ca{A}^{\overline{K}}_{\underline{N}}$-module of finite presentation, since $\ca{A}^{\overline{K}}_{p^{\underline{n}}\underline{N}}$ is an $\ca{A}^{\overline{K}}_{\underline{N}}$-module of finite presentation by \ref{prop:tower-basic}.(\ref{item:prop:tower-basic-6}).
\end{myrem}

\begin{mycor}\label{cor:galois-action-1}
	Let $L\in \scr{F}_{\overline{K}/K}$ containing $\{\zeta_{p^n}\}_{n\in\bb{N}}$, $\underline{N}\in(\bb{N}\cup\{\infty\}\setminus p\bb{N})^d$, $\nu:\Delta_{p^\infty}\to\mu_{p^\infty}(\overline{K})$ a $\bb{Z}_p$-linear homomorphism. Then, for any nonzero element $\pi\in\ak{m}_L$, the natural map
	\begin{align}
		\ca{A}^{L,(\nu)}_{p^{\underline{\infty}}\underline{N}}/\pi\ca{A}^{L,(\nu)}_{p^{\underline{\infty}}\underline{N}}\longrightarrow \prod_{\ak{p}\in\ak{G}(\ca{A}^L_{\underline{N}}/\pi\ca{A}^L_{\underline{N}})} (\ca{A}^{L,(\nu)}_{p^{\underline{\infty}}\underline{N}}/\pi\ca{A}^{L,(\nu)}_{p^{\underline{\infty}}\underline{N}})_{\ak{p}}
	\end{align}
	is injective, where $\ak{G}(\ca{A}^L_{\underline{N}}/\pi\ca{A}^L_{\underline{N}})$ is the finite set of generic points of $\spec(\ca{A}^L_{\underline{N}}/\pi\ca{A}^L_{\underline{N}})$ {\rm(\ref{defn:generic}, \ref{prop:tower-basic}.(\ref{item:prop:tower-basic-4}))}.
\end{mycor}
\begin{proof}
	We take $n\in\bb{N}$ such that $\nu$ factors through $\mu_{p^n}(\overline{K})$. Since $\ca{A}^L_{p^{\underline{n}}\underline{N}}$ is the union of a directed system of Noetherian normal $\ca{A}$-subalgebras, the natural map 
	\begin{align}
		\ca{A}^L_{p^{\underline{n}}\underline{N}}/\pi\ca{A}^L_{p^{\underline{n}}\underline{N}}\longrightarrow \prod_{\ak{q}\in\ak{G}(\ca{A}^L_{p^{\underline{n}}\underline{N}}/\pi\ca{A}^L_{p^{\underline{n}}\underline{N}})} (\ca{A}^L_{p^{\underline{n}}\underline{N}}/\pi\ca{A}^L_{p^{\underline{n}}\underline{N}})_{\ak{q}}
	\end{align}
	is injective by \ref{prop:generic-map}.(\ref{item:prop:generic-map-3}). Since the morphism $\spec(\ca{A}^L_{p^{\underline{n}}\underline{N}})\to\spec(\ca{A}^L_{\underline{N}})$ induces a map $\ak{G}(\ca{A}^L_{p^{\underline{n}}\underline{N}}/\pi\ca{A}^L_{p^{\underline{n}}\underline{N}})\to \ak{G}(\ca{A}^L_{\underline{N}}/\pi\ca{A}^L_{\underline{N}})$ (\ref{rem:generic-map}.(\ref{item:rem:generic-map-3})), the above injection factors through the natural map
	\begin{align}
		\ca{A}^L_{p^{\underline{n}}\underline{N}}/\pi\ca{A}^L_{p^{\underline{n}}\underline{N}}\longrightarrow \prod_{\ak{p}\in\ak{G}(\ca{A}^L_{\underline{N}}/\pi\ca{A}^L_{\underline{N}})} (\ca{A}^L_{p^{\underline{n}}\underline{N}}/\pi\ca{A}^L_{p^{\underline{n}}\underline{N}})_{\ak{p}},
	\end{align}
	which is thus also injective. Therefore, the conclusion follows from the decomposition \eqref{eq:cor:galois-action-2}.
\end{proof}

\begin{mylem}[{cf. \cite[9.16]{he2022sen}}]\label{lem:A-perfd}
	Let $L\in\scr{F}_{\overline{K}/K}$ be a pre-perfectoid field {\rm(\ref{para:notation-perfd})}. Then, the $\ca{O}_L$-algebra $\ca{A}^L_{\underline{\infty}}$ is pre-perfectoid {\rm(\ref{para:notation-perfd})}.
\end{mylem}
\begin{proof}
	It suffices to check that the Frobenius induces an isomorphism $\ca{A}^L_{\underline{\infty}}/p_1\ca{A}^L_{\underline{\infty}}\to \ca{A}^L_{\underline{\infty}}/p\ca{A}^L_{\underline{\infty}}$, where $p_1$ is a $p$-th root of $p$ up to a unit in $L$ (\cite[5.4]{he2024coh}). Thus, it suffices to check that the Frobenius induces an isomorphism $A^L_{\underline{\infty}}/p_1A^L_{\underline{\infty}}\to A^L_{\underline{\infty}}/pA^L_{\underline{\infty}}$ by \ref{prop:tower-basic}.(\ref{item:prop:tower-basic-2}), which is proved in \cite[9.16]{he2022sen}.
\end{proof}

\begin{myprop}[{cf. \cite[9.23]{he2022sen}}]\label{prop:almost-purity}
	Let $L\in\scr{F}_{\overline{K}/K}$ be a pre-perfectoid field. Then, for any $\ca{K}'\in \scr{F}^{\mrm{fini}}_{\ca{K}_{\mrm{ur}}/\ca{K}^L_{\underline{\infty}}}$, the integral closure $\ca{A}'$ of $\ca{A}$ in $\ca{K}'$ is almost finite \'etale over $\ca{A}^L_{\underline{\infty}}$.
\end{myprop}
\begin{proof}
	By Abhyankar's lemma \cite[10.3]{tsuji2018localsimpson} and a limit argument (see \cite[8.21]{he2024coh}), $\ca{A}'[1/p]$ is finite \'etale over $\ca{A}^L_{\underline{\infty}}[1/p]$. Then, the conclusion follows from the almost purity theorem \cite[7.9]{scholze2012perfectoid} (see \cite[7.12]{he2022sen}) and \ref{lem:A-perfd}.
\end{proof}

\begin{mycor}\label{cor:almost-purity}
	Let $L\in\scr{F}_{\overline{K}/K}$ be a pre-perfectoid field, $M$ an $\overline{\ca{A}}$-module endowed with a continuous $\overline{\ca{A}}$-semi-linear action of the Galois group $\ca{H}^L_{\infty}=\gal(\ca{K}_{\mrm{ur}}/\ca{K}^L_{\underline{\infty}})$ with respect to the discrete topology on $M$. Then, the natural homomorphism
	\begin{align}
		\overline{\ca{A}}\otimes_{\ca{A}^L_{\underline{\infty}}}M^{\ca{H}^L_{\infty}}\longrightarrow M
	\end{align}
	is an almost isomorphism, and the continuous group cohomology $H^n(\ca{H}^L_{\infty},M)$ is almost zero for any integer $n>0$. In particular, $\ca{A}^L_{\underline{\infty}}/\pi\ca{A}^L_{\underline{\infty}}\to (\overline{\ca{A}}/\pi \overline{\ca{A}})^{\ca{H}^L_{\infty}}$ is an almost isomorphism for any nonzero element $\pi\in \ak{m}_L$.
\end{mycor}
\begin{proof}
	It follows from \ref{prop:almost-purity} and almost Galois descent, see the arguments of \cite[\Luoma{2}.6.19, \Luoma{2}.6.22]{abbes2016p}.
\end{proof}

\begin{mylem}\label{lem:coprime-p-coh}
	Let $L\in\scr{F}_{\overline{K}/K}$ containing $\{\zeta_n\}_{n\in\bb{N}\setminus p\bb{N}}$, $M$ an $\ca{A}^L_{\underline{\infty}}$-module consisting of $p$-power torsion elements endowed with an $\ca{A}^L_{\underline{\infty}}$-semi-linear action of the Galois group $\ca{H}^L_{p^\infty}/\ca{H}^L_{\infty}=\gal(\ca{K}^L_{\underline{\infty}}/\ca{K}^L_{p^{\underline{\infty}}})$ with respect to the discrete topology on $M$. Then, the natural homomorphism
	\begin{align}
		\ca{A}^L_{\underline{\infty}}\otimes_{\ca{A}^L_{p^{\underline{\infty}}}}M^{\ca{H}^L_{p^\infty}/\ca{H}^L_{\infty}}\longrightarrow M
	\end{align}
	is an isomorphism, and the continuous group cohomology $H^n(\ca{H}^L_{p^\infty}/\ca{H}^L_{\infty},M)$ is zero for any integer $n>0$. Moreover, $\ca{A}^L_{p^{\underline{\infty}}}/\pi\ca{A}^L_{p^{\underline{\infty}}}\to (\ca{A}^L_{\underline{\infty}}/\pi \ca{A}^L_{\underline{\infty}})^{\ca{H}^L_{p^\infty}/\ca{H}^L_{\infty}}$ is an isomorphism for any element $\pi\in \ak{m}_L$.
\end{mylem}
\begin{proof}
	Note that for any $N\in\bb{N}\setminus p\bb{N}$, there is an injective homomorphism of groups
	\begin{align}
		\ca{H}^L_{p^\infty}/\ca{H}^L_{p^\infty N}=\gal(\ca{K}^L_{p^{\underline{\infty}}\underline{N}}/\ca{K}^L_{p^{\underline{\infty}}})\longrightarrow \prod_{i=1}^d\mu_N(\overline{K}),\ \tau\mapsto (\tau(t_{i,N})/t_{i,N})_i,
	\end{align}
	where the target is a finite abelian group with order coprime to $p$, and thus so is the source. Then, the conclusion follows from the general arguments of \cite[\Luoma{2}.6.11, \Luoma{2}.6.13]{abbes2016p}.
\end{proof}

\begin{mylem}\label{lem:coh-p-inf}
	Let $L\in\scr{F}_{\overline{K}/K}$ be a pre-perfectoid field containing $\{\zeta_n\}_{n\in\bb{N}_{>0}}$, $\pi$ a nonzero element of $\ak{m}_L$, $\ca{H}^L=\gal(\ca{K}_{\mrm{ur}}/\ca{K}^L)$, $\Delta_{p^\infty}=\gal(\ca{K}^L_{p^{\underline{\infty}}}/\ca{K}^L)$. Then, for any integer $n$, the canonical morphism of continuous group cohomology groups
	\begin{align}
		H^n(\Delta_{p^\infty}, \ca{A}^L_{p^{\underline{\infty}}}/\pi\ca{A}^L_{p^{\underline{\infty}}})\longrightarrow H^n(\ca{H}^L,\overline{\ca{A}}/\pi\overline{\ca{A}})
	\end{align}
	is an almost isomorphism of $\ca{A}^L$-modules.
\end{mylem}
\begin{proof}
	It follows from \ref{cor:almost-purity} and \ref{lem:coprime-p-coh}, see the proof of \cite[\Luoma{2}.6.24]{abbes2016p}.
\end{proof}

\begin{mylem}\label{lem:p-inf-nu-coh}
	Let $L\in\scr{F}_{\overline{K}/K}$ be a pre-perfectoid field containing $\{\zeta_{p^n}\}_{n\in\bb{N}}$, $\pi$ a nonzero element of $\ak{m}_L$, $\Delta_{p^\infty}=\gal(\ca{K}^L_{p^{\underline{\infty}}}/\ca{K}^L)$ with $\bb{Z}_p$-basis $\tau_1,\dots,\tau_d$ {\rm(\ref{cor:galois-group})}, $\nu:\Delta_{p^\infty}\to\mu_{p^\infty}(\overline{K})$ a $\bb{Z}_p$-linear homomorphism. Then, there is a canonical isomorphism in the derived category $\dd(\ca{A}^L\module)$,
	\begin{align}
		\rr\Gamma(\Delta_{p^\infty},\ca{A}^{L,(\nu)}_{p^{\underline{\infty}}}/\pi\ca{A}^{L,(\nu)}_{p^{\underline{\infty}}})\cong K^\bullet(\ca{A}^{L,(\nu)}_{p^{\underline{\infty}}}/\pi\ca{A}^{L,(\nu)}_{p^{\underline{\infty}}};\nu(\tau_1)-1,\dots,\nu(\tau_d)-1),
	\end{align}
	where $\ca{A}^{L,(\nu)}_{p^{\underline{\infty}}}=\{x\in\ca{A}^L_{p^{\underline{\infty}}} \ |\ \tau(x)=\nu(\tau)x,\ \forall\ \tau\in\Delta_{p^\infty}\}$ {\rm(\ref{prop:galois-action})}, and $K^\bullet$ is the Koszul complex {\rm(\ref{para:notation-koszul-1})}.
\end{mylem}
\begin{proof}
	It is proved in \cite[\Luoma{2}.3.25]{abbes2016p} and also in \cite[7.3]{bhattmorrowscholze2018integral}.
\end{proof}

\begin{mythm}\label{thm:coh}
	Let $L\in\scr{F}_{\overline{K}/K}$ be a pre-perfectoid field containing $\{\zeta_n\}_{n\in\bb{N}_{>0}}$, $\pi$ a nonzero element of $\ak{m}_L$, $\ca{H}^L=\gal(\ca{K}_{\mrm{ur}}/\ca{K}^L)$. 
	\begin{enumerate}
		\renewcommand{\labelenumi}{{\rm(\theenumi)}}
		\item For any integer $n>d$, the continuous group cohomology $H^n(\ca{H}^L,\overline{\ca{A}}/\pi\overline{\ca{A}})$ is almost zero.\label{item:thm:coh-1}
		\item For any integer $0\leq n\leq d$, there exists a canonical $\ca{A}^L$-linear almost isomorphism $\varphi_n\oplus \psi_n=(\varphi_n^{(\nu)}\oplus \psi_n^{(\nu)})_{\nu}$ from
		\begin{small}\begin{align}
			\bigoplus_{\nu\in\ho(\Delta_{p^\infty},\mu_{p^\infty}(\overline{K}))} (\frac{\pi}{\pi^{(\nu)}}\ca{A}^{L,(\nu)}_{p^{\underline{\infty}}}/\pi\ca{A}^{L,(\nu)}_{p^{\underline{\infty}}} \otimes_{\ca{O}_L}\wedge^n \ca{O}_L^{\oplus (d-1)})\oplus (\ca{A}^{L,(\nu)}_{p^{\underline{\infty}}}/\pi^{(\nu)}\ca{A}^{L,(\nu)}_{p^{\underline{\infty}}}\otimes_{\ca{O}_L} \wedge^{n-1} \ca{O}_L^{\oplus (d-1)})
		\end{align}\end{small}to $H^n(\ca{H}^L,\overline{\ca{A}}/\pi\overline{\ca{A}})$, where $\Delta_{p^\infty}=\gal(\ca{K}^L_{p^{\underline{\infty}}}/\ca{K}^L)$ with $\bb{Z}_p$-basis $\tau_1,\dots,\tau_d$, $\nu:\Delta_{p^\infty}\to\mu_{p^\infty}(\overline{K})$ is a $\bb{Z}_p$-linear homomorphism, $\pi^{(\nu)}\in \{\pi, \nu(\tau_1)-1,\dots,\nu(\tau_d)-1\}\subseteq \ak{m}_L$ is the element with the minimal valuation {\rm(\ref{para:zeta})} and $\ca{A}^{L,(\nu)}_{p^{\underline{\infty}}}=\{x\in\ca{A}^L_{p^{\underline{\infty}}} \ |\ \tau(x)=\nu(\tau)x,\ \forall\ \tau\in\Delta_{p^\infty}\}$ {\rm(\ref{prop:galois-action})}.\label{item:thm:coh-2}
		\item If $\pi\in(\zeta_p-1)\ca{O}_L$, then the canonical morphism of complexes of $\ca{A}^L$-modules \label{item:thm:coh-3}
		\begin{small}\begin{align}\label{diam:thm:coh}
			\xymatrix{
				\cdots\ar[r]^-{0}&\ca{A}^L/\pi\ca{A}^L\otimes_{\bb{Z}_p}\wedge^{n-1} \bb{Z}_p^{\oplus d}\ar[d]\ar[r]^-{0}&\ca{A}^L/\pi\ca{A}^L\otimes_{\bb{Z}_p}\wedge^n \bb{Z}_p^{\oplus d}\ar[d]\ar[r]^-{0}&\ca{A}^L/\pi\ca{A}^L\otimes_{\bb{Z}_p}\wedge^{n+1} \bb{Z}_p^{\oplus d}\ar[d]\ar[r]^-{0}&\cdots\\
				\cdots\ar[r]^-{\delta^{n-2}}&H^{n-1}(\ca{H}^L,\overline{\ca{A}}/\pi\overline{\ca{A}})\ar[r]^-{\delta^{n-1}}&H^n(\ca{H}^L,\overline{\ca{A}}/\pi\overline{\ca{A}})\ar[r]^-{\delta^n}&H^{n+1}(\ca{H}^L,\overline{\ca{A}}/\pi\overline{\ca{A}})\ar[r]^-{\delta^{n+1}}&\cdots.
			}
		\end{align}\end{small}is an almost quasi-isomorphism, where the vertical arrows are induced by the canonical isomorphism $\ca{A}^L/\pi\ca{A}^L\otimes_{\bb{Z}_p} \bb{Z}_p^{\oplus d}\iso\ho_{\bb{Z}_p}(\Delta_{p^\infty},\ca{A}^L/\pi\ca{A}^L)$ (induced by the $\bb{Z}_p$-basis $\tau_1,\dots,\tau_d$ of $\Delta_{p^\infty}$) and the canonical maps $\wedge^n\ho_{\bb{Z}_p}(\Delta_{p^\infty},\ca{A}^L/\pi\ca{A}^L)\to H^n(\Delta_{p^\infty},\ca{A}^L/\pi\ca{A}^L)\to H^n(\ca{H}^L,\overline{\ca{A}}/\pi\overline{\ca{A}})$ (induced by cup products, see \cite[\Luoma{2}.3.30]{abbes2016p}), and where $\delta^n$ is the $n$-th coboundary map given in the long exact sequence associated to the exact sequence 
		\begin{align}
			\xymatrix{
				0\ar[r]&\overline{\ca{A}}/\pi\overline{\ca{A}}\ar[r]^-{\cdot \pi}&\overline{\ca{A}}/\pi^2\overline{\ca{A}}\ar[r]&\overline{\ca{A}}/\pi\overline{\ca{A}}\ar[r]&0.
			}
		\end{align}
	\end{enumerate}
\end{mythm}
\begin{proof}
	For any $\bb{Z}_p$-linear homomorphism $\nu: \Delta_{p^\infty}\to\mu_{p^\infty}(\overline{K})$ and any nonzero element $\pi\in\ak{m}_L$, we write $M^{(\nu)}=\ca{A}^{L,(\nu)}_{p^{\underline{\infty}}}/\pi\ca{A}^{L,(\nu)}_{p^{\underline{\infty}}}$ and $\{\nu(\tau_1)-1,\dots,\nu(\tau_d)-1\}=\{g_1,\dots,g_d\}$ such that $v_L(g_1)\leq \cdots\leq v_L(g_d)$, where $v_L:L\to \bb{R}\cup\{\infty\}$ is a valuation map. Then, $H^n(\Delta_{p^\infty},M^{(\nu)})$ is canonically computed by the $n$-th cohomology group of the Koszul complex $K^\bullet(M^{(\nu)};g_1,\dots,g_d)$ by \ref{lem:p-inf-nu-coh}. Thus, there is a canonical $\ca{A}^L$-linear isomorphism by \ref{rem:koszul} and \ref{prop:koszul},
	\begin{small}\begin{align}
		\varphi_n^{(\nu)}\oplus \psi_n^{(\nu)}:(M^{(\nu)}[g_1]\otimes_{\ca{O}_L}\wedge^n \ca{O}_L^{\oplus (d-1)})\oplus (M^{(\nu)}/g_1M^{(\nu)}\otimes_{\ca{O}_L} \wedge^{n-1} \ca{O}_L^{\oplus (d-1)})\iso H^n(\Delta_{p^\infty},M^{(\nu)}).
	\end{align}\end{small}As $\pi^{(\nu)}\in\{\pi,g_1\}$ is the element with minimal valuation, we see that
	\begin{align}
		M^{(\nu)}[g_1]=\frac{\pi}{\pi^{(\nu)}}\ca{A}^{L,(\nu)}_{p^{\underline{\infty}}}/\pi\ca{A}^{L,(\nu)}_{p^{\underline{\infty}}}\quad\trm{ and }\quad M^{(\nu)}/g_1M^{(\nu)}=\ca{A}^{L,(\nu)}_{p^{\underline{\infty}}}/\pi^{(\nu)}\ca{A}^{L,(\nu)}_{p^{\underline{\infty}}}.
	\end{align}
	We remark that if $\nu=1$, then $M^{(0)}=\ca{A}^{L,(0)}_{p^{\underline{\infty}}}/\pi\ca{A}^{L,(0)}_{p^{\underline{\infty}}}=\ca{A}^L/\pi\ca{A}^L$ and the following canonical diagram is commutative by \ref{rem:koszul},
	\begin{align}
		\xymatrix{
			(M^{(0)}\otimes_{\ca{O}_L}\wedge^n \ca{O}_L^{\oplus (d-1)})\oplus (M^{(0)}\otimes_{\ca{O}_L} \wedge^{n-1} \ca{O}_L^{\oplus (d-1)})\ar[r]^-{\sim}_-{\varphi_n^{(0)}\oplus \psi_n^{(0)}}\ar[d]^-{\wr}_{\iota^1\oplus \iota'^1}& H^n(\Delta_{p^\infty},M^{(0)})\\
			M^{(0)}\otimes_{\ca{O}_L}\wedge^n \ca{O}_L^{\oplus d}\ar[r]^-{\sim}&\wedge^n\ho_{\bb{Z}_p}(\Delta_{p^\infty},M^{(0)})\ar[u]_-{\wr}
		}
	\end{align}
	where $\iota^1\oplus \iota'^1$ is the isomorphism \eqref{eq:para:notation-koszul-0-1}, the bottom horizontal isomorphism is induced by the isomorphism $\ca{O}_L^{\oplus d}\iso \ho_{\bb{Z}_p}(\Delta_{p^\infty},\ca{O}_L)$ and the $\bb{Z}_p$-basis $\tau_1,\dots,\tau_d$ of $\Delta_{p^\infty}$, and the right vertical isomorphism is induced by the canonical isomorphism $\ho_{\bb{Z}_p}(\Delta_{p^\infty},M^{(0)})\iso \ho_{\bb{Z}_p}(\Delta_{p^\infty},M^{(0)})$ and the cup product (\cite[\Luoma{2}.3.30]{abbes2016p}).
	
	Moreover, if $\pi\in(\zeta_p-1)\ca{O}_L$ and $\nu\neq 1$, then $\pi$ is divisible by $g_1$ \eqref{eq:para:zeta-2}. Thus, the following diagram is commutative by \ref{cor:koszul},
	\begin{align}
		\xymatrix{
			M^{(\nu)}[g_1]\otimes_{\ca{O}_L}\wedge^n \ca{O}_L^{\oplus (d-1)}\ar[d]_-{\varphi_n^{(\nu)}}\ar[r]^-{\sim}&M^{(\nu)}/g_1M^{(\nu)}\otimes_{\ca{O}_L} \wedge^n \ca{O}_L^{\oplus (d-1)}\ar[d]^-{\psi_{n+1}^{(\nu)}}\\
			H^n(\Delta_{p^\infty},M^{(\nu)})\ar[r]^-{\delta^n}&H^{n+1}(\Delta_{p^\infty},M^{(\nu)})
		}
	\end{align}
	where the horizontal isomorphism is the division by $\pi/g_1$, and $\delta^n$ is the $n$-th coboundary map associated to $0\to M^{(\nu)}\stackrel{\cdot \pi}{\longrightarrow} \ca{A}^{L,(\nu)}_{p^{\underline{\infty}}}/\pi^2\ca{A}^{L,(\nu)}_{p^{\underline{\infty}}}\to M^{(\nu)}\to 0$. In particular, we see that the canonical sequence
	\begin{align}\label{eq:thm:coh-9}
		\xymatrix{
			\cdots\ar[r]^-{\delta^{n-2}}&	H^{n-1}(\Delta_{p^\infty},M^{(\nu)})\ar[r]^-{\delta^{n-1}}&	H^n(\Delta_{p^\infty},M^{(\nu)})\ar[r]^-{\delta^n}&	H^{n+1}(\Delta_{p^\infty},M^{(\nu)})\ar[r]^-{\delta^{n+1}}&\cdots
			}
	\end{align}
	is exact. We note that if $\nu=1$, then $\ca{A}^{L,(0)}_{p^{\underline{\infty}}}=\ca{A}^L$ is endowed with trivial $\Delta_{p^\infty}$-action so that the $n$-th coboundary map $\delta^n:H^n(\Delta_{p^\infty},M^{(0)})\to	H^{n+1}(\Delta_{p^\infty},M^{(0)})$ associated to $0\to M^{(0)}\stackrel{\cdot \pi}{\longrightarrow} \ca{A}^{L,(0)}_{p^{\underline{\infty}}}/\pi^2\ca{A}^{L,(0)}_{p^{\underline{\infty}}}\to M^{(0)}\to 0$ is zero. Thus, the canonical sequence \eqref{eq:thm:coh-9} is equal to
	\begin{align}
		\xymatrix{
			\cdots\ar[r]^-{0}&	H^{n-1}(\Delta_{p^\infty},M^{(0)})\ar[r]^-{0}&	H^n(\Delta_{p^\infty},M^{(0)})\ar[r]^-{0}&	H^{n+1}(\Delta_{p^\infty},M^{(0)})\ar[r]^-{0}&\cdots.
		}
	\end{align}
	All in all, the conclusion follows from the fact that the canonical morphism
	\begin{align}
		\bigoplus_{\nu\in\ho(\Delta_{p^\infty},\mu_{p^\infty}(\overline{K}))}H^n(\Delta_{p^\infty},M^{(\nu)})\longrightarrow H^n(\ca{H}^L,\overline{\ca{A}}/\pi\overline{\ca{A}})
	\end{align}
	is an almost isomorphism by \ref{lem:coh-p-inf} and \ref{prop:galois-action}.
\end{proof}

\begin{mycor}\label{cor:coh}
	Let $L\in\scr{F}_{\overline{K}/K}$ be a pre-perfectoid field containing $\{\zeta_n\}_{n\in\bb{N}_{>0}}$, $\pi$ a nonzero element of $\ak{m}_L$, $\ca{H}^L=\gal(\ca{K}_{\mrm{ur}}/\ca{K}^L)$. Then, for any integer $n$, the natural map of $\ca{O}_L$-modules
	\begin{align}\label{eq:cor:coh-1}
		H^n(\ca{H}^L,\overline{\ca{A}}/\pi\overline{\ca{A}})\longrightarrow \prod_{\ak{p}\in\ak{G}(\ca{A}^L/p\ca{A}^L)}H^n(\ca{H}^L,\overline{\ca{A}}/\pi\overline{\ca{A}})_{\ak{p}}
	\end{align}
	is almost injective, where $\ak{G}(\ca{A}^L/p\ca{A}^L)$ is the set of generic points of $\spec(\ca{A}^L/p\ca{A}^L)$ {\rm(\ref{defn:generic})}. Moreover, if $\pi\in(\zeta_p-1)\ca{O}_L$, then the canonical exact sequence $0\to\overline{\ca{A}}/\pi\overline{\ca{A}}\stackrel{\cdot \pi}{\longrightarrow}\overline{\ca{A}}/\pi^2\overline{\ca{A}}\to\overline{\ca{A}}/\pi\overline{\ca{A}}\to 0$ induces an almost exact sequence of $\ca{A}^L$-modules
	\begin{align}\label{eq:cor:coh-2}
		\xymatrix{
			0\ar[r]&\ca{A}^L/\pi\ca{A}^L\ar[r]&H^0(\ca{H}^L,\overline{\ca{A}}/\pi\overline{\ca{A}})\ar[r]^-{\delta^0}&H^1(\ca{H}^L,\overline{\ca{A}}/\pi\overline{\ca{A}}).
		}
	\end{align}
\end{mycor}
\begin{proof}
	The almost injectivity of \eqref{eq:cor:coh-1} follows directly from \ref{thm:coh} ((\ref{item:thm:coh-1}) and (\ref{item:thm:coh-2})) and \ref{cor:galois-action-1}, and the almost exactness of \eqref{eq:cor:coh-2} follows directly from \ref{thm:coh}.(\ref{item:thm:coh-3}).
\end{proof}

\section{Brief Review on Cohomological Descent for Faltings Ringed Topos}\label{sec:faltings-site}
In this section, we briefly review some results on the relation between ``Faltings acyclicity" and ``perfectoidness" following \cite{he2024coh}.

\begin{mypara}\label{para:faltings-site}
	Let $Y\to X$ be a morphism of coherent schemes, and let $\fal^{\et}_{Y\to X}$ (resp. $\fal^{\proet}_{Y\to X}$) be the category of morphisms $V\to U$ of coherent schemes \emph{\'etale} (resp. \emph{pro-\'etale}) over the morphism $Y\to X$, namely, the category of commutative diagrams of coherent schemes
	\begin{align}
		\xymatrix{
			V\ar[r]\ar[d]&U\ar[d]\\
			Y\ar[r]&X
		}
	\end{align}
	such that $U\to X$ is \'etale (resp. pro-\'etale) and that $V\to U_Y=Y\times_XU$ is finite \'etale (resp. pro-finite \'etale) (\cite[7.2 (resp. 7.20)]{he2024coh}). The functor
	\begin{align}
		\phi^+: \fal^{\et}_{Y\to X}&\longrightarrow X_\et,\ (V\to U)\mapsto U,\\
		(\trm{resp. } \phi^+: \fal^{\proet}_{Y\to X}&\longrightarrow X_\proet,\ (V\to U)\mapsto U),
	\end{align}
	endows $\fal^{\et}_{Y\to X}/X_\et$ (resp. $\fal^{\proet}_{Y\to X}/X_\proet$) with a structure of fibred sites, whose fibre site over $U$ is the finite \'etale site $U_{Y,\fet}$ (resp. pro-finite \'etale site $U_{Y,\profet}$). Here, the each site is formed only by coherent schemes. We endow $\fal^{\et}_{Y\to X}$ (resp. $\fal^{\proet}_{Y\to X}$) with the associated covanishing topology (\cite[7.4 (resp. 7.23)]{he2024coh}). We define a presheaf of rings $\falb$ on it by
	\begin{align}
		\falb(V\to U)=\Gamma(U^V,\ca{O}_{U^V}),
	\end{align}
	where $U^V$ is the integral closure of $U$ in $V$. It is actually a sheaf with respect to the covanishing topology (\cite[7.32]{he2024coh}). We call $(\fal^{\et}_{Y\to X},\falb)$ (resp. $(\fal^{\proet}_{Y\to X},\falb)$) the (resp. \emph{pro-\'etale}) \emph{Faltings ringed site} associated to $Y\to X$. We note that for any morphism of coherent schemes $V\to U$ \'etale (resp. pro-\'etale) over $Y\to X$, $(\fal^{\et}_{V\to U},\falb)$ (resp. $(\fal^{\proet}_{V\to U},\falb)$) identifies canonically with the localization of ringed site $(\fal^{\et}_{Y\to X},\falb)_{/(V\to U)}$ (resp. $(\fal^{\proet}_{Y\to X},\falb)_{/(V\to U)}$) (see \cite[7.3 (resp. 7.21)]{he2024coh} and remarks after \cite[7.7 (resp. 7.24)]{he2024coh}). 
\end{mypara}

\begin{mydefn}[{\cite[8.1]{he2024coh}}]\label{defn:faltings-acyclic}
	Let $L$ be a pre-perfectoid field extension of $\bb{Q}_p$, $Y\to X$ a morphism of coherent schemes such that $Y\to X^Y$ is over $\spec(L)\to \spec(\ca{O}_L)$, where $X^Y$ denotes the integral closure of $X$ in $Y$. We say that $Y\to X$ is \emph{Faltings acyclic} if $X$ is affine and if the canonical morphism
	\begin{align}\label{eq:defn:faltings-acyclic}
		B/p B\longrightarrow \rr\Gamma(\fal^\proet_{Y\to X},\falb/p\falb)
	\end{align}
	is an almost isomorphism (see \cite[5.7]{he2024coh}), where $B$ denotes the $\ca{O}_L$-algebra $\falb(Y\to X)$ (i.e., $X^Y=\spec(B)$).
\end{mydefn}
\begin{myrem}\label{rem:faltings-acyclic}
	In \ref{defn:faltings-acyclic}, the canonical morphism
	\begin{align}\label{eq:rem:faltings-acyclic-1}
		\rr\Gamma(\fal_{Y\to X}^\et,\falb/\pi\falb)\longrightarrow \rr\Gamma(\fal_{Y\to X}^\proet,\falb/\pi\falb)
	\end{align}
	is an isomorphism for any nonzero element $\pi\in \ak{m}_L$ by \cite[7.32]{he2024coh}, and the condition \eqref{eq:defn:faltings-acyclic} is equivalent to that the canonical morphism
		\begin{align}\label{eq:rem:faltings-acyclic-2}
		B/\pi B\longrightarrow \rr\Gamma(\fal^\et_{Y\to X},\falb/\pi\falb)
	\end{align}
	is an almost isomorphism for any nonzero element $\pi\in \ak{m}_L$ by \cite[8.3]{he2024coh}.
\end{myrem}

\begin{mylem}\label{lem:al-perfd}
	Let $L$ be a pre-perfectoid field extension of $\bb{Q}_p$, $B$ a flat $\ca{O}_L$-algebra. Assume that for any $x\in B[1/p]$ if $x^p\in B$ then $x\in B$. Then, $B$ is almost pre-perfectoid if and only if $B$ is pre-perfectoid {\rm(\ref{para:notation-perfd})}.
\end{mylem}
\begin{proof}
	We only need to prove that almost pre-perfectoidness implies pre-perfectoidness. As the valuation on $L$ is non-discrete, we take a nonzero element $\pi\in\ak{m}_L$ such that $p\in \pi^{p^2}\ca{O}_L$. The assumption on $B$ is equivalent to that the Frobenius induces an injection $B/\pi B\to B/\pi^p B$ by \cite[5.21]{he2024coh}. It remains to check that the Frobenius induces a surjection $B/\pi B\to B/\pi^p B$ (as $\widehat{B}$ is flat over $\ca{O}_{\widehat{L}}$ by \cite[5.20]{he2024coh}). Since $B$ is almost pre-perfectoid, the Frobenius induces an almost surjection $B/\pi^pB\to B/\pi^{p^2}B$ (\cite[5.23]{he2024coh}). Thus, for any $x\in B$, there exists $y,z\in B$ such that $\pi^px=y^p+\pi^{p^2}z$. As $B$ is flat, we see that $(y/\pi)^p=x-\pi^{p^2-p}z\in B$. By assumption, $y/\pi\in B$. Hence, $x$ lies in the image of the Frobenius map $B/\pi B\to B/\pi^p B$, which completes the proof.
\end{proof}

\begin{mylem}[{cf. \cite[8.10]{he2024coh}}]\label{lem:perfd-basis}
	Let $L$ be a pre-perfectoid field extension of $\bb{Q}_p$, $Y\to X$ a morphism of coherent schemes such that $Y\to X^Y$ is over $\spec(L)\to \spec(\ca{O}_L)$. Then, the objects $V\to U$ of $\fal^\proet_{Y\to X}$ such that $U$ is affine and that $U^V=\spec(C)$ is the spectrum of an $\ca{O}_L$-algebra $C$ which is pre-perfectoid form a topological generating family.
\end{mylem}
\begin{proof}
	We follow the proof of \cite[8.10]{he2024coh} (where the case $Y=(X^Y)_L$ is proved). After replacing $X$ by an affine open covering, we may assume that $X=\spec(R)$. Consider the the localization $R'=(1+pR)^{-1}R$ of $R$. We have $R'/pR'=R/pR$. It is clear that $\spec(R') \coprod \spec(R[1/p]) \to \spec(R)$ is pro-\'etale and surjective, thus a covering in $X_\proet$. So we reduce to the situation where $p$ lies in the Jacobson radical $J(R)$ or $p \in R^\times$. The latter case is trivial, since the $p$-adic completion of $B= \Gamma(X^Y,\ca{O}_{X^Y})$ is zero if $p$ is invertible in $R$. Therefore, we may assume that $p \in J(R)$ in the following.
	
	Since $B= \Gamma(X^Y,\ca{O}_{X^Y})$ is integral over $R$, we also have $p \in J(B)$. Applying \cite[8.8]{he2024coh} to the $\ca{O}_L$-algebra $B$, we obtain a covering $V_\infty=\spec (B_\infty[1/p]) \to V=(X^Y)_L=\spec (B[1/p])$ in $V_\profet$ such that $B_\infty = \Gamma(X^{V_\infty},\ca{O}_{X^{V_\infty}})$ is an $\ca{O}_L$-algebra which is almost pre-perfectoid. Since $B_\infty$ is integrally closed in $B_\infty[1/p]$, we see that it is pre-perfectoid by \ref{lem:al-perfd}. As $V$ is integrally closed in $Y$, $V_\infty$ is integrally closed in $Y\times_VV_\infty$ by \cite[3.17, 3.18]{he2024coh} ($V_\infty$ is pro-\'etale over $V$). Thus, $X^{V_\infty}$ is integrally closed in $Y\times_VV_\infty$.
	\begin{align}
		\xymatrix{
			Y\times_VV_\infty\ar[r]\ar[d]&V_\infty\ar[d]\ar[r]&X^{V_\infty}\ar[d]&\\
			Y\ar[r]&V\ar[r]&X^Y\ar[r]&X
		}
	\end{align}
	Thus, $\falb(Y\times_VV_\infty\to X)=\Gamma(X^{Y\times_VV_\infty},\ca{O}_{X^{Y\times_VV_\infty}})=\Gamma(X^{V_\infty},\ca{O}_{X^{V_\infty}})=B_\infty$ is pre-perfectoid. As $Y\times_VV_\infty\to Y$ is a covering in $Y_\profet$, the conclusion follows.
\end{proof}

\begin{mylem}[{cf. \cite[8.24]{he2024coh}}]\label{lem:fal-acyc-perfd}
	Let $L$ be a pre-perfectoid field extension of $\bb{Q}_p$, $Y\to X$ a morphism of coherent schemes such that $Y\to X^Y$ is over $\spec(L)\to \spec(\ca{O}_L)$, $B=\falb(Y\to X)$. If $B/\pi B\to H^0(\fal^{\proet}_{Y\to X},\falb/\pi\falb)$ is an almost isomorphism for any nonzero element $\pi\in\ak{m}_L$ (e.g., if $Y\to X$ is Faltings acyclic), then the $\ca{O}_L$-algebra $B$ is pre-perfectoid.
\end{mylem}
\begin{proof}
	We follow the proof of \cite[8.24]{he2024coh}. Let $p_1\in \ca{O}_K$ be a $p$-th root of $p$ up to a unit (\cite[5.4]{he2024coh}). Then, for any object $V\to U$ in \ref{lem:perfd-basis}, the Frobenius induces an almost isomorphism $\falb(V\to U)/p_1\falb(V\to U)\to \falb(V\to U)/p\falb(V\to U)$ by construction. We see that the Frobenius induces an almost isomorphism $\falb/p_1\falb\to \falb/p\falb$ by sheafification. Taking global section, it also induces an almost isomorphism $B/p_1B\to B/pB$ (applying the assumption for $\pi=p_1$ and $\pi=p$). Since $B$ is flat over $\ca{O}_L$, it is pre-perfectoid by \ref{lem:al-perfd}.
\end{proof}

\begin{mypara}\label{para:notation-open}
	Let $L$ be a pre-perfectoid field extension of $\bb{Q}_p$, $\eta = \spec(L)$, $S = \spec (\ca{O}_L)$, $Y \to X$ a morphism of coherent schemes such that  $Y\to X^Y$ is over $\eta\to S$. Consider the following condition:
	\begin{enumerate}
		\renewcommand{\labelenumi}{{\rm(\theenumi)}}
		\item[($\star$)] The morphism $Y\to X^Y$ is an open immersion and there exist finitely many nonzero divisors $t_1,\dots, t_r$ of $\Gamma(X^Y_\eta,\ca{O}_{X^Y_\eta})$ such that the divisor $D=\sum_{i=1}^r \mrm{div}(t_i)$ on $X^Y_\eta$ has support $X^Y_\eta\setminus Y$ and that at each strict henselization of $X^Y_\eta$ those elements $t_i$ contained in the maximal ideal form a subset of a regular system of parameters (in particular, $D$ is a normal crossings divisor on $X^Y_\eta$, and we allow $D$ to be empty, i.e. $r=0$).
	\end{enumerate}
	Under this assumption, we set
	\begin{align}
		Y_{\infty}&=\lim_{n\in\bb{N}_{>0}} \spec\nolimits_{\ca{O}_Y}(\ca{O}_Y[T_1,\dots,T_r]/(T_1^{n}-t_1,\dots,T_r^{n}-t_r)),
	\end{align}
	where the limit is taken over $\bb{N}_{>0}$ with respect to the division relation. We see that $Y_\infty$ is faithfully flat and pro-finite \'etale over $Y$.
	
	We remark that with the notation in \ref{para:adequate-setup} and assume that $L$ is a subextension of $\overline{K}/K$. Then, for any $\underline{r}\in\bb{N}^d_{>0}$, $Y=\spec(\ca{A}^L_{\underline{r},\triv})$ and $X=\spec(\ca{A}^L_{\underline{r}})$ satisfy the above assumptions ($\star$) with $X^Y=X$ and $r=d-c$ by \ref{rem:sncd}. Moreover, the canonical morphism $\spec(\ca{A}^L_{\underline{\infty},\triv})\to Y$ factors through $Y_\infty$.
\end{mypara}

\begin{mythm}[{\cite[8.24]{he2024coh} and \ref{lem:al-perfd}}]\label{thm:acyclic}
	With the notation in {\rm\ref{para:notation-open}}, assume that $Y\to X$ satisfies the condition $(\star)$ and let $V\to U$ be an object of $\fal_{Y_{\infty}\to X}^\proet$. Then, the following statements are equivalent:
	\begin{enumerate}
		\renewcommand{\labelenumi}{{\rm(\theenumi)}}
		\item The morphism $V\to U$ is Faltings acyclic.\label{chap1-item:thm-acyclic1}
		\item The scheme $U$ is affine and $U^V=\spec(A)$ is the spectrum of a pre-perfectoid $\ca{O}_L$-algebra $A$.\label{chap1-item:thm-acyclic2}
	\end{enumerate} 
\end{mythm}

\begin{mycor}\label{cor:profet-coh}
	With the notation in {\rm\ref{para:notation-open}}, assume that $Y\to X$ satisfies the condition $(\star)$ and that there exists a covering $V\to Y$ in $Y_\profet$ factoring through $Y_\infty$ such that $V\to X$ is Faltings acyclic. Then, for any nonzero element $\pi\in\ak{m}_L$ and any integer $n$, the natural morphism
	\begin{align}
		H^n(Y_\profet,\falb/\pi\falb)\longrightarrow H^n(\fal_{Y\to X}^\proet,\falb/\pi\falb)
	\end{align}
	is an almost isomorphism, where we view $\falb$ as a sheaf over $Y_\profet$ by restriction via the functor $\beta^+:Y_\profet\to \fal_{Y\to X}^\proet,\ V\mapsto (V\to X)$ {\rm(\cite[7.26]{he2024coh})}.
\end{mycor}
\begin{proof}
	Recall that $\beta^+$ is left exact and continuous, and thus defines a morphism of ringed sites $\beta:(\fal_{Y\to X}^\proet,\falb)\to (Y_\profet,\falb)$. Notice that for any object $V'$ of $Y_{\profet/V}$, $X^{V'}=\spec(B')$ is still affine and $B'$ is pre-perfectoid by \ref{thm:acyclic} and \cite[8.23]{he2024coh}. Thus, we see that the images of objects of $Y_{\profet/V}$ by $\beta^+$ are Faltings acyclic (\ref{thm:acyclic}). Therefore, $\falb/\pi\falb\to \rr\beta_*(\falb/\pi\falb)$ is an almost isomorphism, which completes the proof.
\end{proof}
\begin{myrem}\label{rem:profet-coh}
	Under the assumptions in {\rm\ref{cor:profet-coh}} and with the same notation, assume moreover that $Y$ is connected and fix a geometric point $\overline{y}\to Y$. Let $(Y_i)_{i\in I}$ be a normalized universal cover of $Y$ at $\overline{y}$ (\cite[\Luoma{6}.9.8]{abbes2016p}) and we put $\overline{Y}=\lim_{i\in I}Y_i\in\ob(Y_{\profet})$. Then, $(\falb/\pi\falb)(\overline{Y})$ is endowed with a canonical continuous action of the fundamental group $\pi_1(Y,\overline{y})$ with respect to the discrete topology (\cite[\Luoma{6}.9.8.4]{abbes2016p}). Moreover, there is a canonical isomorphism for any integer $n$ (see \cite[\Luoma{6}.9.8.6]{abbes2016p} and \cite[7.32]{he2024coh})
	\begin{align}
		 H^n(\pi_1(Y,\overline{y}),(\falb/\pi\falb)(\overline{Y}))\iso H^n(Y_\fet,\falb/\pi\falb)=H^n(Y_\profet,\falb/\pi\falb).
	\end{align}
	By construction, $\overline{Y}\to Y$ factors through $V$, thus $\overline{Y}\to X$ is Faltings acyclic by the arguments of  \ref{cor:profet-coh}. In particular, the canonical morphism
	\begin{align}
		\falb(\overline{Y})/\pi\falb(\overline{Y})\longrightarrow(\falb/\pi\falb)(\overline{Y}) 
	\end{align}
	is an almost isomorphism.
	
	Moreover, if $Y$ is integral with fraction field $\ca{K}$, then we fix an algebraic closure $\overline{\ca{K}}$ of $\ca{K}$ as the geometric point $\overline{y}=\spec(\overline{\ca{K}})\to Y$. We denote by $\ca{K}_{\mrm{ur}}$ the fraction field of $\overline{Y}$, and we put $\ca{H}=\gal(\ca{K}_{\mrm{ur}}/\ca{K})=\pi_1(Y,\overline{y})$, $\falb(\overline{Y})=\overline{B}$. Then, we conclude that there is a canonical almost isomorphism (combining with \ref{rem:faltings-acyclic} and \ref{cor:profet-coh})
	\begin{align}
		H^n(\ca{H}, \overline{B}/\pi\overline{B})\longrightarrow H^n(\fal_{Y\to X}^\et,\falb/\pi\falb).
	\end{align}
\end{myrem}

\begin{mypara}\label{para:notation-open-2}
	With the notation in {\rm\ref{para:notation-open}}, we consider the following condition:
	\begin{enumerate}
		\renewcommand{\labelenumi}{{\rm(\theenumi)}}
		\item[($\star\star$)] There exists a covering family $\{(V_i\to U_i)\to (Y\to X)\}_{i\in I}$ in $\fal_{Y\to X}^\proet$ such that each $V_i\to U_i$ satisfies the condition $(\star)$ in \ref{para:notation-open}. More precisely, for any $i\in I$, $V_i\to U_i^{V_i}$ is an open immersion and there exist finitely many nonzero divisors $t_{i,1},\dots, t_{i,r_i}$ of $\Gamma(U_{i,\eta}^{V_i},\ca{O}_{U_{i,\eta}^{V_i}})$ such that the divisor $D_i=\sum_{j=1}^{r_i} \mrm{div}(t_{i,j})$ on $U_{i,\eta}^{V_i}$ has support $U_{i,\eta}^{V_i}\setminus V_i$ and that at each strict henselization of $U_{i,\eta}^{V_i}$ those elements $t_{i,j}$ ($1\leq j\leq r_i$) contained in the maximal ideal form a subset of a regular system of parameters.
	\end{enumerate}
\end{mypara}

\begin{mycor}\label{cor:acyclic}
	With the notation in {\rm\ref{para:notation-open}}, assume that $Y\to X$ satisfies the condition $(\star\star)$ in {\rm\ref{para:notation-open-2}}. Let $U=\spec(A')$ be an object of $X_\proet$ and assume moreover that $X=\spec(A)$ is affine. Then, for any nonzero element $\pi\in\ak{m}_L$ and any integer $n$, the natural morphism
	\begin{align}
		H^n(\fal_{Y\to X}^\proet,\falb/\pi\falb)\otimes_AA'\longrightarrow H^n(\fal_{U_Y\to U}^\proet,\falb/\pi\falb)
	\end{align}
	is an almost isomorphism.
\end{mycor}
\begin{proof}
	Firstly, we note that finite limits and finite coproducts exist in $\fal_{Y\to X}^\proet$ and any object is quasi-compact (\cite[7.23]{he2024coh}). In particular, we may assume that the index set $I$ is finite in $(\star\star)$. Moreover, it follows directly from the definition that $\coprod_{i\in I}V_i\to \coprod_{i\in I}U_i$ also satisfies the condition $(\star)$. In other words, we may assume that there exists a covering $(Y'\to X')\to (Y\to X)$ in $\fal_{Y\to X}^{\proet}$ such that $Y'\to X'$ satisfies the condition $(\star)$ in \ref{para:notation-open}. 
	
	Now, any object in $\fal_{Y\to X}^\proet$ can be covered by an object of $\fal_{Y'_\infty\to X'}^{\proet}=(\fal_{Y\to X}^{\proet})_{/(Y'_\infty\to X')}$, and we can require such an object of $\fal_{Y'_\infty\to X'}^{\proet}$ to be Faltings acyclic by \ref{lem:perfd-basis} and \ref{thm:acyclic}. Thus, we can take a hypercovering $(Y_\bullet\to X_\bullet)\to (Y\to X)$ in $\fal_{Y\to X}^\proet$ with $Y_i\to X_i$ Faltings acyclic and $Y_i\to Y$ factors through $Y'_\infty$ for any object $i$ of the simplicial category (\cite[\href{https://stacks.math.columbia.edu/tag/094K}{094K}, \href{https://stacks.math.columbia.edu/tag/0DB1}{0DB1}]{stacks-project}, see also \cite[8.17]{he2024coh}). We put $\falb(Y_i\to X_i)=B_i$. Then, $\rr\Gamma(\fal_{Y\to X}^\proet,\falb/\pi\falb)$ is almost quasi-isomorphic to the complex $B_0/\pi B_0\to B_1/\pi B_1\to\cdots$ (see \cite[\href{https://stacks.math.columbia.edu/tag/0D8H}{0D8H}, \href{https://stacks.math.columbia.edu/tag/0D7E}{0D7E}]{stacks-project}). Since $U\to X$ is pro-\'etale, $\falb(U_{Y_i}\to U_{X_i})=\falb(Y\to X)\otimes_AA'=B_i\otimes_AA'$ by \cite[3.17, 3.18]{he2024coh}. Moreover, $B_i\otimes_AA'$ is still pre-perfectoid by \cite[5.37]{he2024coh} and \ref{lem:al-perfd}. Thus, $U_{Y_i}\to U_{X_i}$ is again Faltings acyclic by \ref{thm:acyclic}. As $(U_{Y_\bullet}\to U_{X_\bullet})\to (U_Y\to U)$ is a hypercovering in $\fal_{Y\to X}^\proet$, we know that $\rr\Gamma(\fal_{U_Y\to U}^\proet,\falb/\pi\falb)$ is almost quasi-isomorphic to the complex $(B_0/\pi B_0)\otimes_AA'\to (B_1/\pi B_1)\otimes_AA'\to\cdots$. The conclusion follows from the flatness of $A'$ over $A$.
\end{proof}

\begin{myrem}\label{rem:quasi-coherent}
	With the notation in {\rm\ref{para:notation-open}}, assume that $Y\to X$ satisfies the condition $(\star\star)$ in {\rm\ref{para:notation-open-2}}. Let $\sigma:(\fal^{\proet}_{Y\to X},\falb)\to (X_\proet,\ca{O}_{X_\proet})$ be the morphism of ringed sites defined by the left exact continuous functor $\sigma^+:X_\proet\to \fal^{\proet}_{Y\to X},\ U\mapsto (U_Y\to U)$ (\cite[(7.26.5)]{he2024coh}). Then, for any nonzero element $\pi\in\ak{m}_L$ and any integer $n$, the $\ca{O}_{X_\proet}$-module $\ak{m}_L\otimes_{\ca{O}_L}\rr^n\sigma_*(\falb/\pi\falb)$ is quasi-coherent by \ref{cor:acyclic} (see \cite[2.8.2, 4.7.1]{abbes2020suite}).
\end{myrem}

\begin{myprop}\label{prop:modification-invariance}
	With the notation in {\rm\ref{para:notation-open}}, assume that $Y\to X$ satisfies the condition $(\star\star)$ in {\rm\ref{para:notation-open-2}}. Let $f:X'\to X$ be a morphism of coherent schemes under $Y$ such that $g:X'^Y\to X^Y$ is a separated v-covering {\rm(\cite[3.1.(1)]{he2024coh})} and that $g_\eta:X'^Y_\eta\to X^Y_\eta$ is an isomorphism. 
	\begin{align}\label{eq:prop:modification-invariance-1}
		\xymatrix{
			Y\ar[r]\ar@{=}[d]&X'^Y_\eta\ar[d]^-{g_\eta}_-{\wr}\ar[r]&X'^Y\ar[r]\ar[d]^-{g}&X'\ar[d]^-{f}\\
			Y\ar[r]&X^Y_\eta\ar[r]&X^Y\ar[r]&X
		}
	\end{align}
	If we denote by $u:(\fal^{\proet}_{Y\to X'},\falb)\to (\fal^{\proet}_{Y\to X},\falb)$ the canonical morphism of ringed sites, then for any nonzero element $\pi\in\ak{m}_L$, the canonical morphism
	\begin{align}
		\falb/\pi\falb\longrightarrow \rr u_*(\falb/\pi\falb)
	\end{align} 
	is an almost isomorphism {\rm(\cite[5.7]{he2024coh})}.
\end{myprop}
\begin{proof}
	We shall reduce to the case where $Y=X'^Y_\eta= X^Y_\eta$ by Abhyankar's lemma (which is proved in \cite[8.16]{he2024coh}). For any object $U$ of $X_\proet$, we put $U'=X'\times_XU$. Since taking integral closure commutes with pro-\'etale base change \cite[3.17, 3.18]{he2024coh}, the base change of \eqref{eq:prop:modification-invariance-1} along $U\to X$ is the following commutative diagram
	\begin{align}\label{eq:prop:modification-invariance-3}
		\xymatrix{
			U'_Y\ar[r]\ar@{=}[d]&U'^{U'_Y}_\eta\ar[d]^-{g_\eta}_-{\wr}\ar[r]&U'^{U'_Y}\ar[r]\ar[d]^-{g}&U'\ar[d]^-{f}\\
			U_Y\ar[r]&U^{U_Y}_\eta\ar[r]&U^{U_Y}\ar[r]&U
		}
	\end{align}
	Taking integral closures in an object $V$ of $U_{Y,\profet}$, we obtain a commutative diagram
	\begin{align}\label{eq:prop:modification-invariance-4}
		\xymatrix{
			V\ar[r]\ar@{=}[d]&U'^V_\eta\ar[d]^-{g_\eta}_-{\wr}\ar[r]&U'^V\ar[r]\ar[d]^-{g}&U'\ar[d]^-{f}\\
			V\ar[r]&U^V_\eta\ar[r]&U^V\ar[r]&U
		}
	\end{align}
	This shows that the problem is local in $\fal^{\proet}_{Y\to X}$. After replacing $Y\to X$ by $V\to U$, we may assume that $Y\to X$ satisfies the condition $(\star)$ in \ref{para:notation-open}. Since $g_\eta:X'^Y_\eta\to X^Y_\eta$ is an isomorphism, we see that $Y\to X'$ also satisfies the condition $(\star)$ in \ref{para:notation-open} with respect to the same divisors $t_1,\dots,t_r$. Consider the canonical commutative diagram of ringed sites
	\begin{align}
		\xymatrix{
			(\fal^{\proet}_{Y_\infty\to X'},\falb)\ar[r]^-{u}\ar[d]&(\fal^{\proet}_{Y_\infty\to X},\falb)\ar[d]\\
			(\fal^{\proet}_{X'^{Y_\infty}_\eta\to X'},\falb)\ar[r]^-{u'}&(\fal^{\proet}_{X^{Y_\infty}_\eta\to X},\falb)
		}
	\end{align}
	where the vertical arrows are equivalences by Abhyankar's lemma \cite[8.22]{he2024coh}. Since $\falb/\pi\falb\to \rr u'_*(\falb/\pi\falb)$ is an almost isomorphism proved in \cite[8.16]{he2024coh}, we see that so is $\falb/\pi\falb\to \rr u_*(\falb/\pi\falb)$. This completes the proof as $Y_\infty\to Y$ is a covering in $Y_{\profet}$.
\end{proof}

\section{Faltings Cohomology of Essentially Adequate Schemes}\label{sec:essential-adequate-scheme}

In this section, we translate the previous results on the Galois cohomology of essentially adequate algebras to the Faltings cohomology of essentially adequate schemes. Especially, we show that the Faltings cohomology of an essentially adequate scheme is controlled by the Faltings cohomology of the localizations at generic points of the special fibre (see \ref{thm:essential-adequate-coh}). 

\begin{mydefn}[{cf. \ref{defn:essential-adequate-alg}}]\label{defn:essential-adequate-pair}
	Let $L$ be a valuation field of height $1$ extension of $\bb{Q}_p$, $\eta=\spec(L)$, $S=\spec(\ca{O}_L)$, $X^{\triv}\to X$ be an open immersion of coherent schemes over $\eta\to S$. We say that $X^{\triv}\to X$ is (resp. \emph{essentially}) \emph{adequate} if for any geometric point $\overline{x}$ of $X$, there exists 
	\begin{enumerate}
		\renewcommand{\labelenumi}{{\rm(\theenumi)}}
		\item a complete discrete valuation field $K$ extension of $\bb{Q}_p$ with perfect residue field and a local, injective and integral homomorphism $\ca{O}_K\to \ca{O}_L$ (i.e., $L$ is an algebraic valuation extension of $K$),
		\item an (resp. essentially) adequate $(K,\ca{O}_K,\ca{O}_{\overline{K}})$-triple $(A_{\triv},A,\overline{A})$ with $\ca{A}$ being one of the three $A$-algebras satisfying the corresponding assumptions in \ref{para:adequate-setup},
		\item an \'etale neighborhood $U\to X$ of $\overline{x}$ with an \'etale morphism $U\to \spec(\ca{A}^L)$ over $S$ such that $X^{\triv}\times_XU=\spec(\ca{A}^L_{\triv})\times_{\spec(\ca{A}^L)}U$, where $\ca{A}^L$ is the integral closure of $\ca{A}$ in $L\ca{K}$, $\ca{K}$ is the fraction field of $\ca{A}$, and $\ca{A}^L_{\triv}=A_{\triv}\otimes_A\ca{A}^L$.
	\end{enumerate} 
	We usually omit $X^{\triv}$ and simply call $X$ an (resp. \emph{essentially}) \emph{adequate scheme} over $S$.
\end{mydefn}

\begin{myrem}\label{rem:essential-adequate-log-smooth}
	Log smooth morphisms will produce adequate schemes as follows: given the following data (see \cite[\textsection8]{he2022sen} for the terminologies on logarithmic geometry)
	\begin{enumerate}
		\renewcommand{\labelenumi}{{\rm(\theenumi)}}
		\item $K$ a complete discrete valuation field extension of $\bb{Q}_p$, $\overline{K}$ an algebraic closure of $K$,
		\item $(\spec(\ca{O}_K),\scr{M}_K)$ the log scheme with underlying scheme $\spec(\ca{O}_K)$ endowed with the compactifying log structure defined by the closed point,
		\item $f:(X,\scr{M}_X)\to (\spec(\ca{O}_K),\scr{M}_K)$ a smooth morphism of fs log schemes whose underlying generic fibre $X_K$ is regular,
	\end{enumerate}
	let $\overline{X}$ be the integral closure of $X$ in $\spec(\overline{K})\times_{\spec(\ca{O}_K)} X$ and we put $\overline{X}^{\triv}=X^{\triv}\times_X\overline{X}$. Then, the open immersion $\overline{X}^{\triv}\to \overline{X}$ over $\spec(\overline{K})\to\spec(\ca{O}_{\overline{K}})$ is adequate. 
	
	Indeed, by a theorem of Abbes-Gros-Tsuji on the local description of log smooth morphisms (see \cite[8.11]{he2022sen}), every geometric point of $X$ admits an \'etale neighborhood $U$ such that $f|_U:(U,\scr{M}_X|_U)\to (\spec(\ca{O}_K),\scr{M}_K)$ factors through $(\spec(\ca{O}_{K'}),\scr{M}_{K'})$ for a tamely ramified finite extension $K'$ of $K$ and that the induced morphism of fs log schemes $(U,\scr{M}_X|_U)\to (\spec(\ca{O}_{K'}),\scr{M}_{K'})$ admits an adequate chart in the sense of \cite[8.8]{he2022sen}. After shrinking $U$, we may assume that $U$ is the spectrum of an adequate $\ca{O}_{K'}$-algebra $A$ with $U^{\triv}=\spec(A_{\triv})$. Thus, the integral closure $\overline{U}$ of $U$ in $\spec(\overline{K})\times_{\spec(\ca{O}_K)}U$ is the disjoint union of finitely many copies of the integral closure of $U$ in $\spec(\overline{K})\times_{\spec(\ca{O}_{K'})}U$, which is a finite disjoint union of $\spec(A^{\overline{K}})$ with the notation in \ref{para:adequate-setup}. As such $\overline{U}$ forms an \'etale covering of $\overline{X}$ (\cite[3.17]{he2024coh}), we see that $\overline{X}$ is an adequate scheme over $\ca{O}_{\overline{K}}$.
\end{myrem}

\begin{myexample}[Semi-stable case]\label{exam:semi-stable}
	Let $L$ be a valuation field of height $1$ extension of $\bb{Q}_p$, $X^{\triv}\to X$ be an open immersion of coherent schemes over $\spec(L)\to \spec(\ca{O}_L)$. We say that $X^{\triv}\to X$ is \emph{semi-stable} if for any geometric point $\overline{x}$ of $X$, there exists an \'etale neighborhood $U\to X$ of $\overline{x}$ with an \'etale morphism over $\ca{O}_L$,
	\begin{align}\label{eq:exam:semi-stable-1}
		U\longrightarrow \spec(\ca{O}_L[T_0,T_1,\dots,T_b,T_{b+1}^{\pm 1},\dots,T_c^{\pm 1},T_{c+1},\dots,T_d]/(T_0\cdots T_b-a))
	\end{align}
	for some integers $0\leq b\leq c\leq d$ and $a\in\ak{m}_L\setminus \{0\}$, such that $U^{\triv}=X^{\triv}\times_XU$ is the complement of the divisor $T_0\cdots T_d$.
	
	We remark that if there exists a complete discrete valuation field $K$ extension of $\bb{Q}_p$ with perfect residue field and a local, injective and integral homomorphism $\ca{O}_K\to \ca{O}_L$ (i.e., $L$ is an algebraic valuation extension of $K$), then $X^{\triv}\to X$ is adequate over $\eta\to S$.
	
	Indeed, it suffices to show that $U$ is adequate. After shrinking $U$, we may assume that $U$ is affine. After enlarging $K$, we may assume that $a=\pi_K^e$ for a uniformizer $\pi_K$ of $K$ and an integer $e\in\bb{N}_{>0}$, and that there exists an \'etale ring homomorphism
	\begin{align}\label{eq:exam:semi-stable-2}
		\ca{O}_K[T_0,T_1,\dots,T_b,T_{b+1}^{\pm 1},\dots,T_c^{\pm 1},T_{c+1},\dots,T_d]/(T_0\cdots T_b-\pi_K^e)\longrightarrow A
	\end{align}
	whose base change along $\ca{O}_K\to \ca{O}_L$ gives \eqref{eq:exam:semi-stable-1} (\cite[8.8.2, 8.10.5]{ega4-3}). Notice that for any finite extension $K'$ of $K$, $\ca{O}_{K'}[T_0,T_1,\dots,T_b,T_{b+1}^{\pm 1},\dots,T_c^{\pm 1},T_{c+1},\dots,T_d]/(T_0\cdots T_b-\pi_K^e)$ is an adequate $\ca{O}_{K'}$-algebra by \ref{exam:semi-stable-chart}. In particular, it is a Noetherian normal domain. Thus, we see that $U=\spec(\ca{O}_L\otimes_{\ca{O}_K}A)$ is a finite disjoint union of $\spec(A^L)$ with the notation in \ref{para:adequate-setup}, which shows that $U$ is adequate.
\end{myexample}

\begin{mypara}\label{para:zeta-2}
	In the rest of this section, we fix a pre-perfectoid field $L$ extension of $\bb{Q}_p$ containing a compatible system of primitive $n$-th roots of unity $\{\zeta_n\}_{n\in\bb{N}_{>0}}$. We put
	\begin{align}
		\eta=\spec(L), \quad S=\spec(\ca{O}_L),\quad s=\spec(\ca{O}_L/\ak{m}_L).
	\end{align}
	We fix an open immersion of coherent schemes $X^{\triv}\to X$ essentially adequate over $\eta\to S$. For any $x\in X_s$, we put $X_{(x)}=\spec(\ca{O}_{X,x})$ the localization of $X$ at $x$, and we put $X^{\triv}_{(x)}= X^{\triv}\times_XX_{(x)}$.
\end{mypara}

\begin{mylem}\label{lem:essential-adequate-open}
	The open immersion $X^{\triv}\to X$ over $\eta\to S$ is affine, dominant, and satisfies the condition $(\star\star)$ in {\rm\ref{para:notation-open-2}}. In particular, $X_\eta$ is a regular scheme with normal crossings divisor $X_\eta\setminus X^{\triv}$. 
\end{mylem}
\begin{proof}
	With the notation in \ref{defn:essential-adequate-pair}, we take a system of coordinates $t_1,\dots,t_d\in A[1/p]$ of the chart of the essentially adequate algebra $A$ (\ref{defn:essential-adequate-alg}). Then, by \ref{rem:sncd}, we see that the generic fibre $U_\eta$ of the \'etale neighborhood $U$ is regular, and that $U^{\triv}$ is the complement of a normal crossings divisor defined by $t_1\cdots t_d$ and that $U^{\triv}\to U$ satisfies the condition $(\star)$ in \ref{para:notation-open} (as $U$ is normal).
\end{proof}

\begin{mylem}\label{lem:essential-adequate-generic}
	The scheme $X$ is a finite disjoint union of normal integral schemes flat over $S$, and the set of generic points $\ak{G}(X_s)$ of its special fibre $X_s$ is finite.
\end{mylem}
\begin{proof}
	As $X$ is quasi-compact, it suffices to prove the statements for an \'etale covering of $X$ (\ref{rem:generic-map}.(\ref{item:rem:generic-map-2})). With the notation in \ref{defn:essential-adequate-pair}, we may assume that the \'etale neighborhood $U$ is affine. Recall that $\spec(\ca{A}^L)$ is a normal integral scheme flat over $S$ with $\ak{G}(\ca{A}^L/p\ca{A}^L)$ finite by \ref{prop:tower-basic}.(\ref{item:prop:tower-basic-4}). Thus, $U$ is a finite disjoint union of normal integral schemes with $\ak{G}(U_s)$ finite by \ref{rem:generic-map}.(\ref{item:rem:generic-map-2}) as $U\to \spec(\ca{A}^L)$ is \'etale and quasi-compact.
\end{proof}

\begin{mylem}\label{lem:weak-dim}
	Let $A\to B$ be a faithfully flat ring homomorphism, $d\in\bb{N}$. If every $B$-module has Tor-dimension $\leq d$, then every $A$-module has Tor-dimension $\leq d$. In particular, if every local ring of $B$ is a valuation ring, then so does $A$.
\end{mylem}
\begin{proof}
	For any $A$-module $M$, we take a free resolution $F_\bullet\to M$. As $B\otimes_AM$ has Tor-dimension $\leq d$, $B\otimes_A\cok(F_{d+1}\to F_d)=\cok(B\otimes_AF_{d+1}\to B\otimes_AF_d)$ is a flat $B$-module (\cite[\href{https://stacks.math.columbia.edu/tag/0653}{0653}]{stacks-project}). Thus, $\cok(F_{d+1}\to F_d)$ is a flat $A$-module as $A\to B$ is faithfully flat. Hence, $0\to \cok(F_{d+1}\to F_d)\to F_{d-1}\to \cdots\to F_0\to M$ is a flat resolution of $M$, which implies that the Tor-dimension of $M$ $\leq d$. The ``in particular" part follows from the statement for $d=1$ and \cite[\href{https://stacks.math.columbia.edu/tag/092S}{092S}]{stacks-project}.
\end{proof}

\begin{mylem}\label{lem:essential-adequate-generic-loc}
	For any $x\in \ak{G}(X_s)$, the stalk $\ca{O}_{X,x}$ is a valuation ring of height $1$ extension of $\ca{O}_L$. In particular, $X^{\triv}_{(x)}=\spec(\ca{O}_{X,x}[1/p])$ is the generic point of $X_{(x)}$.
\end{mylem}
\begin{proof}
	With the notation in \ref{defn:essential-adequate-pair}, we may assume that the \'etale neighborhood $U$ is affine whose image contains $x$. We take $y\in\ak{G}(U_s)$ over $x\in\ak{G}(X_s)$ by \ref{rem:generic-map}.(\ref{item:rem:generic-map-2}). As $U$ is \'etale over $\spec(\ca{A}^L)$, we see that $\ca{O}_{U,y}$ is a valuation ring of height $1$ extension of $\ca{O}_L$ by \ref{prop:generic-map} and \ref{rem:prop:generic-map}. Then, $\ca{O}_{X,x}\to \ca{O}_{U,y}$ is a faithfully flat \'etale local ring homomorphism. Thus, $\ca{O}_{X,x}$ is a valuation ring of height $1$ extension of $\ca{O}_L$ by \ref{lem:weak-dim}.
\end{proof}

\begin{mythm}\label{thm:essential-adequate-coh}
	With the notation in {\rm\ref{para:zeta-2}}, assume that $X=\spec(B)$ is affine. Then, for any nonzero element $\pi$ of $\ak{m}_L$ and any integer $n$, the natural map of $A$-modules 
	\begin{align}\label{eq:thm:essential-adequate-coh-1}
		H^n(\fal^\et_{X^{\triv}\to X},\falb/\pi\falb)\longrightarrow \prod_{x\in\ak{G}(X_s)}H^n(\fal^\et_{X^{\triv}_{(x)}\to X_{(x)}},\falb/\pi\falb)
	\end{align}
	is almost injective, where $\ak{G}(X_s)$ is the finite set of generic points of $X_s$ {\rm(\ref{defn:generic})}. Moreover, if $\pi\in(\zeta_p-1)\ca{O}_L$, then the canonical exact sequence $0\to \falb/\pi\falb\stackrel{\cdot \pi}{\longrightarrow}\falb/\pi^2\falb\to \falb/\pi\falb\to 0$ induces an almost exact sequence of $B$-modules
	\begin{align}\label{eq:thm:essential-adequate-coh-2}
		\xymatrix{
			0\ar[r]&B/\pi B\ar[r]&H^0(\fal^\et_{X^{\triv}\to X},\falb/\pi\falb)\ar[r]^-{\delta^0}&H^1(\fal^\et_{X^{\triv}\to X},\falb/\pi\falb).
		}
	\end{align}
\end{mythm}
\begin{proof}
	For any $x\in \ak{G}(X_s)$, notice that the localization morphism $X_{(x)}\to X$ is pro-\'etale. Thus, the canonical morphism
	\begin{align}
		H^n(\fal^\et_{X^{\triv}\to X},\falb/\pi\falb)\otimes_BB_{\ak{p}_x}\longrightarrow H^n(\fal^\et_{X^{\triv}_{(x)}\to X_{(x)}},\falb/\pi\falb)
	\end{align}
	is an almost isomorphism by \ref{rem:faltings-acyclic} and \ref{cor:acyclic} (whose assumptions are satisfied by \ref{lem:essential-adequate-open}), where $\ak{p}_x$ is the prime ideal of $B$ corresponding to the point $x$. With the notation in \ref{defn:essential-adequate-pair}, we may assume that the \'etale neighborhood $U=\spec(B')$ is affine whose image contains $x$. It suffices to prove that \eqref{eq:thm:essential-adequate-coh-1} is almost injective and that \eqref{eq:thm:essential-adequate-coh-2} is almost exact after tensoring with $B'$ over $B$. Again by \ref{cor:acyclic}, it suffices to prove that the canonical map
	\begin{align}
		H^n(\fal^\et_{U^{\triv}\to U},\falb/\pi\falb)\longrightarrow \prod_{x\in\ak{G}(U_s)}H^n(\fal^\et_{U^{\triv}\to U},\falb/\pi\falb)_{\ak{p}_x}
	\end{align}
	is almost injective and that (if $\pi\in(\zeta_p-1)\ca{O}_L$)
	\begin{align}
		\xymatrix{
			0\ar[r]&B'/\pi B'\ar[r]&H^0(\fal^\et_{U^{\triv}\to U},\falb/\pi\falb)\ar[r]^-{\delta^0}&H^1(\fal^\et_{U^{\triv}\to U},\falb/\pi\falb)
		}
	\end{align}
	is almost exact (we used the finiteness of $\ak{G}(X_s)$ by \ref{lem:essential-adequate-generic}). Apply the same reduction to the \'etale morphism $U\to \spec(\ca{A}^L)$, we can reduce further to the case where $U=\spec(\ca{A}^L)$, which is proved in \ref{cor:coh} by \ref{rem:profet-coh}.
\end{proof}

\section{Purity for Perfectoidness of Certain Limits of Essentially Adequate Schemes}\label{sec:purity}
The property that Faltings cohomology of an essentially adequate scheme is controlled by the Faltings cohomology of the localizations at generic points of the special fibre is preserved after taking certain inverse limit of essentially adequate schemes (see \ref{thm:purity}). As a corollary, we show that being pre-perfectoid for such a limit can be checked on the localization at the generic points of its special fibre, which is a purity-type criterion (see \ref{cor:purity}).

\begin{mypara}\label{para:zeta-4}
	In this section, we fix a pre-perfectoid field $L$ extension of $\bb{Q}_p$ containing a compatible system of primitive $n$-th roots of unity $\{\zeta_n\}_{n\in\bb{N}_{>0}}$. We put
	\begin{align}
		\eta=\spec(L), \quad S=\spec(\ca{O}_L),\quad s=\spec(\ca{O}_L/\ak{m}_L).
	\end{align}
	We fix a directed inverse system $(X^{\triv}_\lambda\to X_\lambda)_{\lambda\in\Lambda}$ of open immersion of coherent schemes essentially adequate over $\eta\to S$ (\ref{defn:essential-adequate-pair}) with affine transition morphisms. We put 
	\begin{align}\label{para:zeta-4-2}
		(X^{\triv}\to X)=\lim_{\lambda\in\Lambda} (X^{\triv}_\lambda\to X_\lambda).
	\end{align}
	Recall that the Faltings ringed site $(\fal^{\et}_{X^{\triv}\to X},\falb)$ (\ref{para:faltings-site}) is canonically equivalent to the limit of ringed sites defined in \cite[8.2.3, 8.6.2]{sga4-2} by \cite[7.12]{he2024coh},
	\begin{align}\label{para:zeta-4-3}
		(\fal^{\et}_{X^{\triv}\to X},\falb)&=\lim_{\lambda\in\Lambda}(\fal^{\et}_{X^{\triv}_\lambda\to X_\lambda},\falb).
	\end{align}
	We remark that $X$ is integrally closed in $X^{\triv}$ as each $X_\lambda$ is integrally closed in $X^{\triv}_{\lambda}$ by \ref{lem:essential-adequate-open} and \ref{lem:essential-adequate-generic} (\cite[3.18]{he2024coh}).
	
	For any $x\in X_s$, we put $X_{(x)}=\spec(\ca{O}_{X,x})$ the localization of $X$ at $x$, and we put $X^{\triv}_{(x)}= X^{\triv}\times_XX_{(x)}$. Then, there is a canonical morphism $(X^{\triv}_{(x)}\to X_{(x)})\to (X^{\triv}\to X)$ of open immersions of coherent schemes over $\eta\to S$, which induces a canonical morphism of Faltings ringed sites
	\begin{align}\label{para:zeta-4-4}
		(\fal^{\et}_{X^{\triv}_{(x)}\to X_{(x)}},\falb)\longrightarrow(\fal^{\et}_{X^{\triv}\to X},\falb).
	\end{align}
\end{mypara}

\begin{mypara}\label{para:purity-setup}
	Consider the following condition on the inverse system $(X^{\triv}_\lambda\to X_\lambda)_{\lambda\in\Lambda}$: 
	\begin{enumerate}
		\renewcommand{\labelenumi}{{\rm(\theenumi)}}
		\item[$(\mbf{G})$]  For any indexes $\lambda\leq \mu$ in $\Lambda$, the transition morphism
		\begin{align}
			f_{\lambda\mu}:X_\mu\longrightarrow X_\lambda
		\end{align}
		sends the set $\ak{G}(X_{\mu,s})$ of generic points of the special fibre of $X_\mu$ to the set $\ak{G}(X_{\lambda,s})$ of generic points of the special fibre of $X_\lambda$.
	\end{enumerate} 
	In this case, $(\ak{G}(X_{\lambda,s}))_{\lambda\in\Lambda}$ forms an inverse system of finite discrete subspaces (\ref{lem:essential-adequate-generic} and \ref{lem:generic-discrete}) whose limit coincides with the subspace $\ak{G}(X_s)$ of generic points of $X_s$ by \ref{lem:generic-limit},
	\begin{align}\label{eq:para:purity-setup-2}
		\ak{G}(X_s)=\lim_{\lambda\in\Lambda}\ak{G}(X_{\lambda,s}),
	\end{align}
	which is pro-finite.
\end{mypara}

\begin{myrem}\label{rem:cond-G}
	The condition $(\mbf{G})$ in {\rm\ref{para:purity-setup}} holds in each of the following situations:
	\begin{enumerate}
		\renewcommand{\labelenumi}{{\rm(\theenumi)}}
		\item The transition morphism $f_{\lambda\mu}$ is flat for any indexes $\lambda\leq \mu$ in $\Lambda$. Indeed, this follows from \ref{rem:generic-map}.(\ref{item:rem:generic-map-1}).\label{item:rem:cond-G-1}
		\item The transition morphism $f_{\lambda\mu}$ is integral and sends the set $\ak{G}(X_\mu)$ of generic points of $X_\mu$ to the set $\ak{G}(X_\lambda)$ of generic points of $X_\lambda$ for any indexes $\lambda\leq \mu$ in $\Lambda$. Indeed, this follows from \ref{rem:generic-map}.(\ref{item:rem:generic-map-3}) and \ref{lem:essential-adequate-generic}.\label{item:rem:cond-G-2}
	\end{enumerate}
\end{myrem}

\begin{mylem}\label{lem:val-transit}
	Assume that the condition $(\mbf{G})$ in {\rm\ref{para:purity-setup}} holds. 
	\begin{enumerate}
		\renewcommand{\labelenumi}{{\rm(\theenumi)}}
		\item For any $\lambda\in\Lambda$ and any point $x_\lambda\in\ak{G}(X_{\lambda,s})$, the local ring $\ca{O}_{X_\lambda,x_\lambda}$ is a valuation ring of height $1$ extension of $\ca{O}_L$. Moreover, $X^{\triv}_{\lambda,(x_\lambda)}=\spec(\ca{O}_{X_\lambda,x_\lambda}[1/p])$ is the generic point of $X_{\lambda,(x_\lambda)}$.\label{item:lem:val-transit-1}
		\item For any indexes $\lambda\leq \mu$ in $\Lambda$ and any point $x_\mu\in \ak{G}(X_\mu)$ with image $x_\lambda\in\ak{G}(X_\lambda)$, the canonical morphism of local rings induced by $f_{\lambda\mu}$,
		\begin{align}\label{eq:lem:val-transit-1}
			\ca{O}_{X_\lambda,x_\lambda}\longrightarrow \ca{O}_{X_\mu,x_\mu}
		\end{align} 
		is an extension of valuation rings of height $1$.\label{item:lem:val-transit-2}
		\item For any $x\in\ak{G}(X_s)$, if we denote by $x_\lambda$ its image in $\ak{G}(X_{\lambda,s})$ for any $\lambda\in\Lambda$, then
		\begin{align}\label{eq:lem:val-transit-2}
			\ca{O}_{X,x}=\colim_{\lambda\in\Lambda}\ca{O}_{X_\lambda,x_\lambda}
		\end{align}
		is a valuation ring of height $1$ extension of $\ca{O}_L$. Moreover, $X^{\triv}_{(x)}=\spec(\ca{O}_{X,x}[1/p])$ is the generic point of $X_{(x)}$.\label{item:lem:val-transit-3}
	\end{enumerate}
\end{mylem}
\begin{proof}
	(\ref{item:lem:val-transit-1}) is proved in \ref{lem:essential-adequate-generic-loc}. Then, the morphism \eqref{eq:lem:val-transit-1} has to be local and injective, i.e., an extension of valuation rings, which proves (\ref{item:lem:val-transit-2}).  By taking filtered colimit, we get (\ref{item:lem:val-transit-3}) (\cite[\href{https://stacks.math.columbia.edu/tag/0AS4}{0AS4}]{stacks-project}).
\end{proof}

\begin{mythm}\label{thm:purity}
	Assume that the condition $(\mbf{G})$ in {\rm\ref{para:purity-setup}} holds, and that $X_\lambda=\spec(B_\lambda)$ is affine for any $\lambda\in\Lambda$ and we put $B=\colim_{\lambda\in\Lambda}B_\lambda$ (so that $X=\spec(B)$). Let $\pi$ be a nonzero element of $\ak{m}_L$.
	\begin{enumerate}
		\renewcommand{\labelenumi}{{\rm(\theenumi)}}
		\item For any integer $n$, the natural map of $B$-modules induced by \eqref{para:zeta-4-4},
		\begin{align}\label{eq:thm:purity-1}
			H^n(\fal^\et_{X^{\triv}\to X},\falb/\pi\falb)\longrightarrow \prod_{x\in \ak{G}(X_s)}H^n(\fal^\et_{X^{\triv}_{(x)}\to X_{(x)}},\falb/\pi\falb)
		\end{align}
		is almost injective.\label{item:thm:purity-1}
		\item If $\pi\in(\zeta_p-1)\ca{O}_L$, then the canonical exact sequence $0\to \falb/\pi\falb\stackrel{\cdot \pi}{\longrightarrow}\falb/\pi^2\falb\to \falb/\pi\falb\to 0$ induces an almost exact sequence of $B$-modules
		\begin{align}\label{eq:thm:purity-2}
			\xymatrix{
				0\ar[r]&B/\pi B\ar[r]&H^0(\fal^\et_{X^{\triv}\to X},\falb/\pi\falb)\ar[r]^-{\delta^0}&H^1(\fal^\et_{X^{\triv}\to X},\falb/\pi\falb).
			}
		\end{align} \label{item:thm:purity-2}
	\end{enumerate}
\end{mythm}
\begin{proof}
	(\ref{item:thm:purity-1}) Note that for any $\lambda\leq\mu$ in $\Lambda$ the assumption that $f_{\lambda\mu}(\ak{G}(X_{\mu,s}))\subseteq \ak{G}(X_{\lambda,s})$ induces a canonical morphism
	\begin{align}\label{eq:thm:purity-3}
		f_{\lambda\mu}^*:\prod_{x_\lambda\in \ak{G}(X_{\lambda,s})}H^n(\fal^\et_{X^{\triv}_{\lambda,(x_\lambda)}\to X_{\lambda,(x_\lambda)}},\falb/\pi\falb)\longrightarrow\prod_{x_\mu\in \ak{G}(X_{\mu,s})}H^n(\fal^\et_{X^{\triv}_{\mu,(x_\mu)}\to X_{\mu,(x_\mu)}},\falb/\pi\falb)
	\end{align}
	such that for any element $(\xi_{x_\lambda})_{x_\lambda\in \ak{G}(X_{\lambda,s})}$ on the left, the $x_\mu$-component of its image on the right is the image of $\xi_{f_{\lambda\mu}(x_\mu)}$ via the canonical morphism
	\begin{align}\label{eq:thm:purity-4}
		H^n(\fal^\et_{X^{\triv}_{\lambda,(f_{\lambda\mu}(x_\mu))}\to X_{\lambda,(f_{\lambda\mu}(x_\mu))}},\falb/\pi\falb)\longrightarrow H^n(\fal^\et_{X^{\triv}_{\mu,(x_\mu)}\to X_{\mu,(x_\mu)}},\falb/\pi\falb)
	\end{align}
	induced by the canonical morphism $(X_{\mu,(x_\mu)}^{\triv}\to X_{\mu,(x_\mu)})\to (X^{\triv}_{\lambda,(f_{\lambda\mu}(x_\mu))}\to X_{\lambda,(f_{\lambda\mu}(x_\mu))})$. Therefore, there is a canonical commutative diagram
	\begin{align}\label{eq:thm:purity-5}
		\xymatrix{
			\colim_{\lambda\in\Lambda}H^n(\fal^\et_{X^{\triv}_\lambda\to X_\lambda},\falb/\pi\falb)\ar@{=}[d]\ar[r]^-{\alpha}& \colim_{\lambda\in\Lambda}\prod_{x_\lambda\in \ak{G}(X_{\lambda,s})}H^n(\fal^\et_{X^{\triv}_{\lambda,(x_\lambda)}\to X_{\lambda,(x_\lambda)}},\falb/\pi\falb)\ar[d]\\
			H^n(\fal^\et_{X^{\triv}\to X},\falb/\pi\falb)\ar[r]^-{\beta}&\prod_{x\in \ak{G}}H^n(\fal^\et_{X^{\triv}_{(x)}\to X_{(x)}},\falb/\pi\falb)
		}
	\end{align}
	where the left vertical equality follows from $(X^{\triv}\to X)=\lim_{\lambda\in\Lambda} (X^{\triv}_\lambda\to X_\lambda)$, \cite[7.12]{he2024coh} and \cite[8.7.7]{sga4-2}, and where $\alpha$ is almost injective as each 
	\begin{align}\label{eq:thm:purity-6}
		\alpha_\lambda:H^n(\fal^\et_{X^{\triv}_\lambda\to X_\lambda},\falb/\pi\falb)\longrightarrow \prod_{x_\lambda\in \ak{G}(X_{\lambda,s})}H^n(\fal^\et_{X^{\triv}_{\lambda,(x_\lambda)}\to X_{\lambda,(x_\lambda)}},\falb/\pi\falb)
	\end{align}
	is almost injective by \ref{thm:essential-adequate-coh}. In order to check the almost injectivity of $\beta$, it suffices to show that for any $\xi\in H^n(\fal^\et_{X^{\triv}\to X},\falb/\pi\falb)$, if $\beta(\xi)=0$ then $\alpha(\xi)=0$.
	
	Assume that $\alpha(\xi)\neq 0$. We take $\lambda_0\in\Lambda$ large enough such that there exists $\xi_{\lambda_0}\in H^n(\fal^\et_{X^{\triv}_{\lambda_0}\to X_{\lambda_0}},\falb/\pi\falb)$ whose image in $H^n(\fal^\et_{X^{\triv}\to X},\falb/\pi\falb)$ is $\xi$. We denote by $\xi_\lambda=f_{\lambda_0\lambda}^*(\xi_{\lambda_0})$ the image of $\xi_{\lambda_0}$ in $H^n(\fal^\et_{X^{\triv}_\lambda\to X_\lambda},\falb/\pi\falb)$ for any $\lambda\in\Lambda_{\geq\lambda_0}$. Consider the subset
	\begin{align}
		\mrm{Supp}_\lambda(\xi_\lambda)=\{x_\lambda\in \ak{G}(X_{\lambda,s})\ |\ \alpha_\lambda(\xi_\lambda)_{x_\lambda}\neq 0\}\subseteq \ak{G}(X_{\lambda,s}),
	\end{align}
	where $\alpha_\lambda(\xi_\lambda)_{x_\lambda}$ is the $x_\lambda$-component of $\alpha_\lambda(\xi_\lambda)$ \eqref{eq:thm:purity-6}. For any $\lambda\leq\mu$ in $\Lambda_{\geq \lambda_0}$, as $f_{\lambda\mu}^*(\xi_\lambda)=\xi_{\mu}$, we see that $f_{\lambda\mu}:\ak{G}(X_{\mu,s})\to \ak{G}(X_{\lambda,s})$ sends $\mrm{Supp}_\mu(\xi_\mu)$ into $\mrm{Supp}_\lambda(\xi_\lambda)$ by \eqref{eq:thm:purity-3} and \eqref{eq:thm:purity-4}. The assumption $\alpha(\xi)\neq 0$ implies that $\mrm{Supp}_\lambda(\xi_\lambda)$ is non-empty for any $\lambda\in\Lambda_{\geq \lambda_0}$. Thus,
	\begin{align}
		\lim_{\lambda\in\Lambda_{\geq \lambda_0}}\mrm{Supp}_\lambda(\xi_\lambda)
	\end{align}
	is also non-empty as a cofiltered limit of non-empty finite sets by \cite[\href{https://stacks.math.columbia.edu/tag/0A2W}{0A2W}]{stacks-project} (as each $\ak{G}(X_{\lambda,s})$ is finite by \ref{lem:essential-adequate-generic}). We take $x\in \lim_{\lambda\in\Lambda_{\geq \lambda_0}}\mrm{Supp}_\lambda(\xi_\lambda)\subseteq \ak{G}(X_s)=\lim_{\lambda\in\Lambda_{\geq\lambda_0}}\ak{G}(X_{\lambda,s})\subseteq X_s=\lim_{\lambda\in\Lambda_{\geq\lambda_0}}X_{\lambda,s}$ with image $x_\lambda\in \mrm{Supp}_\lambda(\xi_\lambda)$. Then, $\alpha_\lambda(\xi_\lambda)_{x_\lambda}\neq 0$ for any $\lambda\in \Lambda_{\geq \lambda_0}$. Thus, the $x$-component of $\beta(\xi)$, as an element of
	\begin{align}
		\colim_{\lambda\in\Lambda}H^n(\fal^\et_{X^{\triv}_{\lambda,(x_\lambda)}\to X_{\lambda,(x_\lambda)}},\falb/\pi\falb)=H^n(\fal^\et_{X^{\triv}_{(x)}\to X_{(x)}},\falb/\pi\falb),
	\end{align} 
	does not vanish, where the equality above follows from  $X_{(x)}=\lim_{\lambda\in\Lambda} X_{\lambda,(x_\lambda)}$ \eqref{eq:lem:val-transit-2}, \cite[7.12]{he2024coh} and \cite[8.7.7]{sga4-2}. This shows that $\beta(\xi)\neq 0$.
	
	(\ref{item:thm:purity-2}) Consider the canonical morphism of sequences
	\begin{small}\begin{align}
			\xymatrix{
				0\ar[r]&\colim_{\lambda\in\Lambda}B_\lambda/\pi B_\lambda\ar[d]\ar[r]&\colim_{\lambda\in\Lambda}H^0(\fal^\et_{X^{\triv}_{\lambda}\to X_\lambda},\falb/\pi\falb)\ar[d]\ar[r]^-{\delta^0}& \colim_{\lambda\in\Lambda}H^1(\fal^\et_{X^{\triv}_\lambda\to X_\lambda},\falb/\pi\falb)\ar[d]\\
				0\ar[r]&B/\pi B\ar[r]&H^0(\fal^\et_{X^{\triv}\to X},\falb/\pi\falb)\ar[r]^-{\delta^0}& H^1(\fal^\et_{X^{\triv}\to X},\falb/\pi\falb)
			}
	\end{align}\end{small}where the vertical maps are isomorphisms by $(X^{\triv}\to X)=\lim_{\lambda\in\Lambda} (X^{\triv}_\lambda\to X_\lambda)$, \cite[7.12]{he2024coh} and \cite[8.7.7]{sga4-2}. The first row is almost exact by \ref{thm:essential-adequate-coh}. Thus, so is the second row.
\end{proof}

\begin{mycor}[Purity for perfectoidness]\label{cor:purity}
	Assume that the condition $(\mbf{G})$ in {\rm\ref{para:purity-setup}} holds, and that $X_\lambda=\spec(B_\lambda)$ is affine for any $\lambda\in\Lambda$ and we put $B=\colim_{\lambda\in\Lambda}B_\lambda$ (so that $X=\spec(B)$). Then, the following conditions are equivalent: 
	\begin{enumerate}
		\renewcommand{\labelenumi}{{\rm(\theenumi)}}
		\item For any $x\in \ak{G}(X_s)$, the valuation field $\ca{O}_{X,x}[1/p]$ is a pre-perfectoid field.\label{item:cor:purity-1}
		\item The $\ca{O}_L$-algebra $B$ is pre-perfectoid.\label{item:cor:purity-2}
	\end{enumerate}
\end{mycor}
\begin{proof}
	(\ref{item:cor:purity-1})$\Rightarrow$(\ref{item:cor:purity-2}) Since $\ca{O}_{X,x}$ is pre-perfectoid, $X^{\triv}_{(x)}\to X_{(x)}$ is Faltings acyclic by \ref{thm:acyclic}. Thus, $H^n(\fal^\et_{X^{\triv}_{(x)}\to X_{(x)}},\falb/p\falb)$ is almost zero for any integer $n>0$, and so is $H^n(\fal^\et_{X^{\triv}\to X},\falb/p\falb)$ by \ref{thm:purity}.(\ref{item:thm:purity-1}). Then, $B/p B\to \rr\Gamma(\fal^\et_{X^{\triv}\to X},\falb/p\falb)$ is an almost isomorphism by \ref{thm:purity}.(\ref{item:thm:purity-2}). In other words, $X^{\triv}\to X$ is Faltings acyclic (\ref{defn:faltings-acyclic}). Thus, $B=\falb(X_{\triv}\to X)$ is pre-perfectoid by \ref{lem:fal-acyc-perfd}.
	
	(\ref{item:cor:purity-2})$\Rightarrow$(\ref{item:cor:purity-1}) Since $\ca{O}_{X,x}$ is a localization of $B$, it is also pre-perfectoid (\cite[5.37]{he2024coh}).
\end{proof}

%
%

\section{Preliminaries on Riemann-Zariski Spaces}\label{sec:riemann-zariski}
The formation of taking the set of generic points of special fibres is not functorial in general. One of the functorial replacements of the set of generic points of special fibres is the so-called Riemann-Zariski space (see \ref{para:riemann-zariski}). We refer to \cite{temkin2011rz} for a systematic development of the theory of Riemann-Zariski spaces, and add some basic properties in this section mainly focusing on their limit behavior (see \ref{prop:riemann-zariski-limit} and \ref{cor:riemann-zariski-limit}).

\begin{mypara}\label{para:valuation-spectrum}
	Let $Y\to X$ be a morphism of coherent schemes. We put $\spa(Y,X)$ the set of isomorphism classes of commutative diagrams of coherent schemes
	\begin{align}\label{eq:para:riemann-zariski-2}
		\xymatrix{
			y\ar[r]\ar[d]&W\ar[d]\\
			Y\ar[r]&X
		}
	\end{align} 
	where $y$ is a point of $Y$ and $W$ is the spectrum of a valuation ring with generic point $y$ (\cite[page 20]{temkin2011rz}). We denote by $\val_Y(X)$ the subset of $\spa(Y,X)$ consisting of the diagrams \eqref{eq:para:riemann-zariski-2} such that the induced morphism $y\to Y\times_XW$ is a closed immersion (\cite[page 9]{temkin2011rz}). When $X=\spec(A)$ and $Y=\spec(B)$ are affine, we simply put $\spa(Y,X)=\spa(B,A)$ and $\val_Y(X)=\val_B(A)$.
\end{mypara}

\begin{mydefn}\label{defn:pro-open}
	A morphism $Y\to X$ of coherent schemes is called \emph{pro-open} if there is a directed inverse system of quasi-compact open subsets $(U_\lambda)_{\lambda\in\Lambda}$ with affine transition inclusion maps such that there is an isomorphism of $X$-schemes $Y\cong \lim_{\lambda\in\Lambda} U_\lambda$. We also call $Y$ a \emph{pro-open subset} of $X$.
\end{mydefn}

We remark that the stalk of a pro-open subset $Y$ of a coherent scheme $X$ at any point $y\in Y$ coincides with the stalk of $X$ at $y$.

\begin{mylem}\label{lem:riemann-zariski-basic-spa}
	Let $(Y'\to X')\to (Y\to X)$ be a morphism of morphisms of coherent schemes. Then, there is a canonical map
	\begin{align}\label{eq:lem:riemann-zariski-basic-spa-1}
		\spa(Y',X')\longrightarrow \spa(Y,X)
	\end{align}
	characterized by the following properties: for any element $y'\to W'$ of $\spa(Y',X')$ with image $y\to W$ in $\spa(Y,X)$, $y=\spec(\kappa(y))\in Y$ is the image of $y'=\spec(\kappa(y'))\in Y'$ and $\ca{O}_W=\kappa(y)\cap\ca{O}_{W'}\subseteq \kappa(y')$ (and thus $\ca{O}_W\to \ca{O}_{W'}$ is an extension of valuation rings).
	
	Moreover, if $Y'=Y\times_XX'$ and if $Y\to X$ is pro-open, then the map above induces a map of subsets
	\begin{align}
		\val_{Y'}(X')\longrightarrow\val_Y(X)
	\end{align}
	and we have $y'=y\times_WW'=Y\times_XW'$.
\end{mylem}
\begin{proof}
	Indeed, we define $y=\spec(\kappa(y))\in Y$ to be the image of $y'=\spec(\kappa(y'))\in Y'$ and $\ca{O}_W=\kappa(y)\cap \ca{O}_{W'}$. We see that $\ca{O}_W\to \ca{O}_{W'}$ is an extension of valuation rings by \ref{lem:val-ext}. 
	
	For the ``moreover" part, assume that $y'\to W'$ is an element of $\val_{Y'}(X')$. Notice that $Y'\times_{X'}W'=Y\times_XW'\subseteq W'$ is a pro-open subset by base change. Since $y'$ is the generic point of $W'$, we see that the closed immersion $y'\to Y\times_XW'$ is bijective. Moreover, $y'\to Y\times_XW'$ is an isomorphism, since the stalk of the pro-open subset $Y\times_XW'\subset W'$ at $y'$ coincides with the stalk of $W'$ at $y'$.
	\begin{align}
		\xymatrix{
			y'\ar[r]^-{\sim}\ar[d]&Y\times_XW'\ar[r]\ar[d]&W'\ar[d]\\
			y\ar[r]&Y\times_XW\ar[r]\ar[d]&W\ar[d]\\
			&Y\ar[r]&X\\
		}
	\end{align}
	Since $W'\to W$ is faithfully flat (\ref{lem:val-ext}), so is $Y\times_XW'\to Y\times_XW$. Therefore, $Y\times_XW$ also has only one point. Since $Y\times_XW$ is a pro-open subset of the integral scheme $W$, $y\to Y\times_XW$ is an isomorphism by the same arguments as above. In particular, $y\to W$ is an element of $\val_Y(X)$.
\end{proof}

\begin{mylem}\label{lem:val-spa-inj}
	Let $(Y'\to X')\to (Y\to X)$ be a morphism of morphisms of coherent schemes such that $Y'\to Y$ is an immersion and that $X'\to X$ is separated.
	\begin{enumerate}
		\renewcommand{\labelenumi}{{\rm(\theenumi)}}
		\item The canonical map $\spa(Y',X')\to \spa(Y,X)$ is injective.\label{item:lem:val-spa-inj-1}
		\item Let $(Y''\to X'')\to (Y\to X)$ be a morphism of morphisms of coherent schemes. Then, we have
		\begin{align}
			\spa(Y'\times_YY'', X'\times_XX'')=\spa(Y',X')\times_{\spa(Y,X)}\spa(Y'',X'')\subseteq \spa(Y'',X'').
		\end{align}\label{item:lem:val-spa-inj-2}
	\end{enumerate}
\end{mylem}
\begin{proof}
	(\ref{item:lem:val-spa-inj-1}) follows from directly the valuative criterion for separateness (\cite[\href{https://stacks.math.columbia.edu/tag/01KZ}{01KZ}]{stacks-project}). For (\ref{item:lem:val-spa-inj-2}), note that $Y'\times_YY''\to Y''$ is still an immersion and $X'\times_XX''\to X''$ is still separated. Hence, we view $\spa(Y'\times_YY'', X'\times_XX'')$ as a subset of $\spa(Y'',X'')$. Let $y''\to W''$ be an element of $\spa(Y'',X'')$ with image $y\to W$ in $\spa(Y,X)$.
	\begin{align}
		\xymatrix{
			y''\ar[r]\ar[d]&W''\ar@/^1pc/[rr]\ar[d]\ar@{.>}[r]&X'\times_XX''\ar[r]\ar[d]&X''\ar[d]\\
			y\ar[r]&W\ar@/_1pc/[rr]\ar@{.>}[r]&X'\ar[r]&X
		}
	\end{align}
	Then, $y''\to W''$ lies in the subset $\spa(Y'\times_YY'', X'\times_XX'')$ if and only if $y''\in Y'\times_YY''\subseteq Y''$ (i.e., its image $y$ lies in $Y'\subseteq Y$ by \ref{lem:riemann-zariski-basic-spa}) and $W''\to X''$ (uniquely) factors through $X'\times_XX''$ (i.e., $W\to X$ (uniquely) factors through $X'$ by \ref{lem:riemann-zariski-basic-spa}). Thus, it is equivalent to that $y\to W$ lies in the subset $\spa(Y',X')$. 
\end{proof}

\begin{mypara}\label{para:val-spa-top}
	Let $Y\to X$ be a morphism of coherent schemes. An \emph{affine subset} of $\spa(Y,X)$ is a subset of the form $\spa(Y',X')$ where $Y'\to X'$ is a morphism of affine schemes over $Y\to X$ such that $Y'\to Y$ is an open immersion and $X'\to X$ is of finite type. Note that the intersection of two affine subsets is a finite union of affine subsets (\ref{lem:val-spa-inj}, see also \cite[3.1.1.(\luoma{3})]{temkin2011rz}). We endow $\spa(Y,X)$ with the topology generated by its affine subsets, and we endow $\val_Y(X)$ with the topology induced from $\spa(Y,X)$ (\cite[page 21]{temkin2011rz}).
\end{mypara}

\begin{mylem}[{\cite[3.1.2]{temkin2011rz}}]\label{lem:spa-top}
	Let $A\to B$ be a ring homomorphism. Then, the affine subsets of $\spa(B,A)$,
	\begin{align}\label{eq:para:val-spa-top-2}
		\spa(B[\frac{1}{b_0}],A[\frac{b_1}{b_0},\dots,\frac{b_n}{b_0}])=\{(y\to W)\in \spa(B,A)\ |\ b_0\neq 0\trm{ in }\kappa(y)\trm{ and }b_1,\dots,b_n\in b_0\ca{O}_W\subseteq \kappa(y)\},
	\end{align} 
	 where $\kappa(y)$ is the residue field of $y$ and $b_0,b_1,\dots,b_n\in B$ are finitely many elements with $n\in\bb{N}_{>0}$, form a topological basis. In particular, the topological space $\spa(B,\bb{Z})$ coincides with the the valuation spectrum $\mrm{Spv}(B)$ of $B$ defined in \cite[\textsection2]{huber1993val}.
\end{mylem}
\begin{proof}
	Let $(y\to W)\in\spa(B,A)$ and let $\spa(B',A')$ be an affine subset of $\spa(B,A)$ containing $y\to W$. As $\spec(B')$ is an open neighborhood of $y$ in $\spec(B)$, we take $b_0\in B$ such that $y\in \spec(B[1/b_0])\subseteq \spec(B')$. As $A'$ is of finite type over $A$, let $b_1/b_0^m,\dots,b_n/b_0^m\in B[1/b_0]$ ($m\in\bb{N}_{>0}$) be the images of a finite family of generators of $A'$ over $A$. After replacing $b_0$ by $b_0^m$, we may assume that $m=1$. Thus, we have morphisms $(A\to B)\to (A'\to B')\to (A[\frac{b_1}{b_0},\dots,\frac{b_n}{b_0}]\to B[1/b_0])$. Moreover, the image of each $b_i/b_0$ in $\kappa(y)$ lies in $\ca{O}_W$ as $A'\to \kappa(y)$ factors through $\ca{O}_W$. Hence, $\spa(B[\frac{1}{b_0}],A[\frac{b_1}{b_0},\dots,\frac{b_n}{b_0}])$ is an affine subset containing $y\to W$ and refining $\spa(B',A')$.
	
	Notice that giving an element of $\spa(B,\bb{Z})$ is equivalent to giving a point $y$ of $\spec(B)$ and a valuation on the residue field $\kappa(y)$, which is thus equivalent to giving a valuation of $B$ in the sense of \cite[\Luoma{6}.\textsection3.1, D\'efinition 1]{bourbaki2006commalg5-7}. Hence, $\spa(B,\bb{Z})=\mrm{Spv}(B)$ and their topologies coincide by the discussion above (\cite[page 461]{huber1993val}).
\end{proof}

\begin{myprop}\label{lem:spa-induced-top}
	Let $(Y'\to X')\to (Y\to X)$ be a morphism of morphisms of coherent schemes such that $Y'\to Y$ is an immersion and that $X'\to X$ is separated. Then, the topology of $\spa(Y',X')$ is induced from the topology of $\spa(Y,X)$ by the canonical injection {\rm(\ref{lem:val-spa-inj}.(\ref{item:lem:val-spa-inj-1}))}.
\end{myprop}
\begin{proof}
	Firstly, we consider the case where $Y'$ and $X'$ are affine, $Y'\to Y$ is an open immersion and $X'\to X$ is of finite type. In this case, $\spa(Y',X')$ is an affine subset of $\spa(Y,X)$. Since the intersection of $\spa(Y',X')$ with an affine subset of $\spa(Y,X)$ is a finite union of affine subsets of $\spa(Y,X)$, which are also affine subsets of $\spa(Y',X')$ (\ref{lem:val-spa-inj}.(\ref{item:lem:val-spa-inj-2})). Hence, $\spa(Y',X')\to \spa(Y,X)$ is continuous. Moreover, any affine subset of $\spa(Y',X')$ is also an affine subset of $\spa(Y,X)$ by definition. Hence, $\spa(Y',X')$ is endowed with the subspace topology.
	
	In general, we can replace $\spa(Y',X')$ and $\spa(Y,X)$ by their affine subsets by the discussion above. Thus, we may assume that $Y', X', Y, X$ are affine and that $Y'\to Y$ is a closed immersion. We write $Y=\spec(B)$, $X=\spec(A)$, $Y'=\spec(B/I)$ and $X'=\spec(A')$. Then, for any elements $b_0,b_1,\dots,b_n\in B/I$, we take liftings $\widetilde{b_0},\widetilde{b_1},\dots,\widetilde{b_n}\in B$.
	\begin{align}
		\xymatrix{
			&\kappa(y)&\ca{O}_W\ar[l]\\
			b_0,b_1,\dots,b_n	&B/I\ar[u]&A'\ar[l]\ar[u]\\
			\widetilde{b_0},\widetilde{b_1},\dots,\widetilde{b_n}\ar@{|->}[u]	&B\ar[u]&A\ar[l]\ar[u]
		}
	\end{align}
	Then, for any $(y\to W)\in\spa(Y',X')$, the condition that $b_0\neq 0$ in $\kappa(y)$ and $b_1,\dots,b_n\in b_0\ca{O}_W$ is equivalent to the condition that $\widetilde{b_0}\neq 0$ in $\kappa(y)$ and $\widetilde{b_1},\dots,\widetilde{b_n}\in \widetilde{b_0}\ca{O}_W$. In other words, we have
	\begin{align}
		\spa(B/I[\frac{1}{b_0}],A'[\frac{b_1}{b_0},\dots,\frac{b_n}{b_0}])=\spa(Y',X')\cap \spa(B[\frac{1}{\widetilde{b_0}}],A[\frac{\widetilde{b_1}}{\widetilde{b_0}},\dots,\frac{\widetilde{b_n}}{\widetilde{b_0}}]).
	\end{align}
	Thus, $\spa(Y',X')$ is endowed with the subspace topology by \ref{lem:spa-top}.
\end{proof}

\begin{myprop}\label{prop:spa-spectral}
	Let $Y\to X$ be a morphism of coherent schemes. Then, $\spa(Y,X)$ is a spectral space {\rm(\cite[\href{https://stacks.math.columbia.edu/tag/08YF}{08YF}]{stacks-project})}.
\end{myprop}
\begin{proof}
	Firstly, assume that $X=\spec(A)$ and $Y=\spec(B)$ are affine. Let $A'$ be the image of $A$ in $B$. As $\spa(B,A')=\spa(B,A)$, we may assume that $A\subseteq B$. Recall that for any $b\in B$, $\spa(B,\bb{Z}[b])=\{(y\to W)\in \spa(B,\bb{Z})\ |\ b\in \ca{O}_W\}$ is spectral by \cite[2.2]{huber1993val}. Then, the intersection of quasi-compact open subsets $\bigcap_{a\in A}\spa(B,\bb{Z}[a])$ is again a spectral subspace of $\spa(B,\bb{Z})$ (\cite[\href{https://stacks.math.columbia.edu/tag/0A31}{0A31}]{stacks-project}). Notice that $\spa(B,A)=\bigcap_{a\in A}\{(y\to W)\in \spa(B,\bb{Z})\ |\ a\in \ca{O}_W\}=\bigcap_{a\in A}\spa(B,\bb{Z}[a])$ as a subset of $\spa(B,\bb{Z})$ and that the topology on $\spa(B,A)$ is induced from $\spa(B,\bb{Z})$ by \ref{lem:spa-induced-top}. We conclude that $\spa(B,A)$ is spectral.
	
	In general, since $\spa(Y,X)$ can be covered by finitely many affine subsets and the affine subsets are spectral with respect to the induced topology by \ref{lem:spa-induced-top} and the discussion above, we see that $\spa(Y,X)$ is sober (\cite[\href{https://stacks.math.columbia.edu/tag/06N9}{Lemma 06N9}]{stacks-project}), quasi-compact, and quasi-compact open subsets of $\spa(Y,X)$ form a basis. Moreover, since the intersection of two affine subsets is a finite union of affine subsets, we conclude that $\spa(Y,X)$ is spectral (\cite[\href{https://stacks.math.columbia.edu/tag/08YG}{08YG}]{stacks-project}).
\end{proof}

\begin{mycor}\label{cor:spa-spectral}
	Let $(Y'\to X')\to (Y\to X)$ be a morphism of morphisms of coherent schemes. Then, the canonical map $\spa(Y',X')\to \spa(Y,X)$ \eqref{eq:lem:riemann-zariski-basic-spa-1} is spectral {\rm(\cite[\href{https://stacks.math.columbia.edu/tag/08YF}{08YF}]{stacks-project})}.
\end{mycor}
\begin{proof}
	We can replace $\spa(Y',X')$ and $\spa(Y,X)$ by \ref{lem:spa-induced-top} so that we may assume that $Y', X', Y, X$ are affine. Let $\spa(Y'',X'')$ be an affine subset of $\spa(Y,X)$ such that $Y''$ and $X''$ affine, $Y''\to Y$ is an open immersion, $X''\to X$ is of finite type. Then, we see that $\spa(Y'\times_YY'',X'\times_XX'')$ is also an affine subset of $\spa(Y',X')$ and moreover we have (\ref{lem:val-spa-inj}.(\ref{item:lem:val-spa-inj-2}))
	\begin{align}
		\spa(Y'\times_YY'',X'\times_XX'')=\spa(Y',X')\times_{\spa(Y,X)}\spa(Y'',X''),
	\end{align}
	which shows that $\spa(Y',X')\to \spa(Y,X)$ is continuous and quasi-compact by \ref{prop:spa-spectral} (i.e., spectral).
\end{proof}

\begin{myprop}\label{prop:spa-limit}
	Let $(Y_\lambda\to X_\lambda)_{\lambda\in\Lambda}$ be a directed inverse system of morphisms of coherent schemes with affine transition morphisms, $(Y\to X)=\lim_{\lambda\in\Lambda}(Y_\lambda\to X_\lambda)$. Then, the canonical map of topological spaces
	\begin{align}\label{eq:prop:spa-limit-1}
		\spa(Y,X)\longrightarrow \lim_{\lambda\in\Lambda}\spa(Y_\lambda,X_\lambda)
	\end{align}
	is a homeomorphism.
\end{myprop}
\begin{proof}
	Notice that by \ref{lem:riemann-zariski-basic-spa}, any element $(y_\lambda\to W_\lambda)_{\lambda\in\Lambda}$ of $\lim_{\lambda\in\Lambda}\spa(Y_\lambda,X_\lambda)$ is actually a directed inverse system of morphisms of affine schemes such that $\ca{O}_{W_\lambda}\to \ca{O}_{W_\mu}$ is an extension of valuation rings for any indexes $\lambda\leq \mu$ in $\Lambda$. Thus, we set
	\begin{align}
		y=\lim_{\lambda\in\Lambda}y_\lambda\in Y=\lim_{\lambda\in\Lambda}Y_\lambda,\quad \trm{ and }\quad W=\lim_{\lambda\in\Lambda}W_\lambda.
	\end{align}
	Then, we still have $\ca{O}_{W_\lambda} \to \ca{O}_W$ is an extension of valuation rings for any $\lambda\in\Lambda$ (\cite[\href{https://stacks.math.columbia.edu/tag/0AS4}{0AS4}]{stacks-project}). Thus, $y\to W$ is an element of $\spa(Y,X)$. This defines a set-theoretic section 
	\begin{align}\label{eq:prop:spa-limit-3}
		\lim_{\lambda\in\Lambda}\spa(Y_\lambda,X_\lambda)\longrightarrow \spa(Y,X)
	\end{align}
	of the canonical map \eqref{eq:prop:spa-limit-1} (see \ref{lem:val-ext}). 
	
	Conversely, for any element $y\to W$ in $\spa(Y,X)$ with image $(y_\lambda\to W_\lambda)_{\lambda\in\Lambda}$ in $\lim_{\lambda\in\Lambda}\spa(Y_\lambda,X_\lambda)$, we deduce from the fact $Y=\lim_{\lambda\in\Lambda}Y_\lambda$ that $y=\lim_{\lambda\in\Lambda}y_\lambda$. Moreover, we have $\ca{O}_{W_\lambda}=\kappa(y_\lambda)\cap \ca{O}_W\subseteq \kappa(y)$ by \ref{lem:riemann-zariski-basic-spa}. We see that $y\to W$ is the image of $(y_\lambda\to W_\lambda)_{\lambda\in\Lambda}$ under the section \eqref{eq:prop:spa-limit-3}, which implies that the \eqref{eq:prop:spa-limit-1} is a bijection.
	
	It remains to show that the bijection \eqref{eq:prop:spa-limit-3} is continuous (as we already know that \eqref{eq:prop:spa-limit-1} is continuous by \ref{cor:spa-spectral}). Indeed, it suffices to check the case where $Y=\spec(B),\ X=\spec(A),\ Y_\lambda=\spec(B_\lambda),\  X_\lambda=\spec(A_\lambda)$ are affine by \ref{lem:spa-induced-top}. We put $f_\lambda:\spa(B,A)\to \spa(B_\lambda,A_\lambda)$ the canonical map. For any finitely many elements $b_0,\dots,b_n$ of $B$, there exists $\lambda\in \Lambda$ and $b_{\lambda,0},\dots,b_{\lambda,n}\in B_\lambda$ whose images in $B$ are $b_0,\dots,b_n$. For for any $\mu\in\Lambda_{\geq\lambda}$, we denote by $b_{\mu,0},\dots,b_{\mu,n}\in B_\mu$ the images of $b_{\lambda,0},\dots,b_{\lambda,n}\in B_\lambda$. Consider the open subset \eqref{eq:para:val-spa-top-2}
	\begin{align}
		U=\{(y\to W)\in \spa(B,A)\ |\ b_0\neq 0\trm{ in }\kappa(y)\trm{ and }b_1,\dots,b_n\in b_0\ca{O}_W\subseteq \kappa(y)\}
	\end{align}
	 of $\spa(B,A)$ and the open subset
	 \begin{align}
	 	U_\mu=\{(y_\mu\to W_\mu)\in \spa(B_\mu,A_\mu)\ |\ b_{0,\mu}\neq 0\trm{ in }\kappa(y_\mu)\trm{ and }b_{1,\mu},\dots,b_{n,\mu}\in b_{0,\mu}\ca{O}_{W_\mu}\subseteq \kappa(y_\mu)\}
	 \end{align}
	 of $\spa(B_\mu,A_\mu)$, we see that $f_\mu(U)\subseteq U_\mu$ and $f_\mu^{-1}(U_\mu)\subseteq U$ by \ref{lem:riemann-zariski-basic-spa}. Thus, we have $U=f_\mu^{-1}(U_\mu)$ and it is an open subset of $\lim_{\lambda\in\Lambda}\spa(Y_\lambda,X_\lambda)$ with respect to the limit topology, which verifies the continuity of \eqref{eq:prop:spa-limit-3} by \ref{lem:spa-top}.
\end{proof}

\begin{mylem}[{\cite[3.1.10]{temkin2011rz}}]\label{lem:spa-X}
	Let $Y\to X$ be a morphism of coherent schemes. Then, there is a canonical continuous map
	\begin{align}\label{eq:lem:spa-X-1}
		\spa(Y,X)\longrightarrow X
	\end{align}
	sending $y\to W$ to the image of the closed point of $W$ in $X$.
\end{mylem}
\begin{proof}
	We only need to check the continuity of \eqref{eq:lem:spa-X-1}. Let $U$ be an affine open subset of $X$. Then, the closed point of $W$ lies in $U$ if and only if $W\to X$ factors through $U$, i.e., $y\to W$ lies in the subspace $\spa(Y\times_XU,U)\subseteq \spa(Y,X)$ (\ref{lem:spa-induced-top}). In other words, the inverse image of $U$ by \eqref{eq:lem:spa-X-1} is the subspace $\spa(Y\times_XU,U)$, which is open as a union of the affine subsets $\spa(V_i,U)$ of $\spa(Y,X)$ where $\{V_i\to Y\times_XU\}_{i\in I}$ is an affine open covering.
\end{proof}

\begin{mypara}\label{para:limit-ringed-space}
	Let $(X_\lambda)_{\lambda\in\Lambda}$ be a directed inverse system of (resp. locally) ringed spaces with spectral underlying topological spaces and spectral transition maps. Recall that the the limit of $(X_\lambda)_{\lambda\in\Lambda}$ in the category of (resp. locally) ringed spaces is represented by $X$, whose underlying topological space is the spectral space $\lim_{\lambda\in\Lambda}X_\lambda$ and whose structural sheaf is $\colim_{\lambda\in\Lambda}f_\lambda^{-1}\ca{O}_{X_\lambda}$, where $f_\lambda:X\to X_\lambda$ is the canonical spectral map (\cite[\href{https://stacks.math.columbia.edu/tag/0A2Z}{0A2Z}]{stacks-project}). We note that any quasi-compact open subset of $X$ is the pullback of a quasi-compact open subset of $X_\lambda$ for some $\lambda\in\Lambda$, and any finite covering of quasi-compact open subsets of $X$ is the pullback of a finite covering of quasi-compact open subsets of $X_\lambda$ for some $\lambda\in\Lambda$ (\cite[\href{https://stacks.math.columbia.edu/tag/0A30}{0A30}]{stacks-project}).
	
	On the other hand, let $X_{\lambda,\mrm{Zar}}$ be the site formed by quasi-compact open subsets of $X_\lambda$ with open coverings (for any $\lambda\in\Lambda$ or $\lambda$ the empty index). Then, we obtain a directed inverse system of ringed sites $(X_{\lambda,\mrm{Zar}},\ca{O}_{X_\lambda})_{\lambda\in\Lambda}$, and let $\lim_{\lambda\in\Lambda}(X_{\lambda,\mrm{Zar}},\ca{O}_{X_\lambda})$ be the limit of ringed sites defined in \cite[8.2.3, 8.6.2]{sga4-2}. We deduce from \cite[8.3.3]{sga4-2} and the discussion above that there is a canonical equivalence of ringed sites
	\begin{align}
		(X_{\mrm{Zar}},\ca{O}_X)\iso \lim_{\lambda\in\Lambda}(X_{\lambda,\mrm{Zar}},\ca{O}_{X_\lambda}).
	\end{align}
	In the following, we shall omit the subscript ``Zar" and regard $X$ as the limit in both senses of (resp. locally) ringed spaces or ringed sites.
\end{mypara}

\begin{mypara}\label{para:riemann-zariski}
	Let $Y\to X$ be a morphism of coherent schemes. A \emph{$Y$-modification of $X$} is a factorization of $Y\to X$ into a composition of a scheme theoretically dominant morphism of coherent schemes $Y\to X'$ (\cite[5.4.2]{ega1-2}) with a proper morphism $X'\to X$. We denote by $\modf_Y(X)$ the category of $Y$-modifications of $X$. It is a cofiltered category and there exists at most one arrow between two objects (\cite[3.3]{temkin2010stablecurve}), and thus we also regard its opposite category as a directed set. We note that if $Y\to X$ is an open immersion of coherent schemes, then any $Y$-modification $f:X'\to X$ is an isomorphism over $Y$, i.e., $f|_{f^{-1}(Y)}:f^{-1}(Y)\iso Y$ is an isomorphism (\cite[\href{https://stacks.math.columbia.edu/tag/0CNG}{0CNG}]{stacks-project}).
	
	We define the \emph{Riemann-Zariski space} of $X$ with respect to $Y$ to be the cofiltered limit of locally ringed spaces (\cite[3.3]{temkin2010stablecurve}, \cite[2.1.1]{temkin2011rz})
	\begin{align}\label{eq:para:riemann-zariski-1}
		\rz_Y(X)=\lim_{X'\in \modf_Y(X)}X'.
	\end{align}
	In other words, $\rz_Y(X)$ is a locally ringed space whose underlying topological space is the limit of topological spaces $\lim X'$ and whose structure sheaf is the colimit of sheaves $\colim \ca{O}_{X'}|_{\rz_Y(X)}$, where $\ca{O}_{X'}|_{\rz_Y(X)}$ denotes the pullback of $\ca{O}_{X'}$ along the continuous map $\rz_Y(X')\to X'$ (\ref{para:limit-ringed-space}). For any $x\in \rz_Y(X)$, the stalk $\ca{O}_{\rz_Y(X),x}=\colim \ca{O}_{X',x'}$, where $x'$ is the image of $x$ under $\rz_Y(X')\to X'$, which is a local ring.
	
	By the valuative criterion for proper morphism (\cite[\href{https://stacks.math.columbia.edu/tag/0BX5}{0BX5}]{stacks-project}), for any $Y$-modification $X'$ of $X$, the canonical map $\spa(Y,X')\to \spa(Y,X)$ \eqref{eq:lem:riemann-zariski-basic-spa-1} is a homeomorphism (\ref{lem:spa-induced-top}). Taking cofiltered limit of the canonical map $\spa(Y,X')\to X'$ \eqref{eq:lem:spa-X-1} over $X'\in\modf_Y(X)$, we obtain a canonical continuous map 
	\begin{align}\label{eq:para:riemann-zariski-3}
		\spa(Y,X)\longrightarrow \rz_Y(X),
	\end{align}
	sending $y\to W$ to the inverse system of the images of the closed point of $W$ in $X'$. In particular, $\spa(Y,X)$ canonically identifies with the set of isomorphism classes of commutative diagram of locally ringed spaces
	\begin{align}\label{eq:para:riemann-zariski-4}
		\xymatrix{
			y\ar[r]\ar[d]&W\ar[d]\\
			Y\ar[r]&\rz_Y(X)
		}
	\end{align} 
	where $y$ is a point of $Y$ and $W$ is the spectrum of a valuation ring with generic point $y$, and the canonical map \eqref{eq:para:riemann-zariski-3} sends such a diagram \eqref{eq:para:riemann-zariski-4} to the image of the closed point of $W$ in $\rz_Y(X)$.
\end{mypara}

\begin{mythm}[{\cite[2.2.1, 3.4.7]{temkin2011rz}}]\label{thm:riemann-zariski}
	Let $Y\to X$ be a separated morphism of coherent schemes. Then, the canonical map \eqref{eq:para:riemann-zariski-3} induces a homeomorphism
	\begin{align}\label{eq:thm:riemann-zariski}
		\val_Y(X)\iso \rz_Y(X).
	\end{align}
	Moreover, for any $(y\to W)\in\val_Y(X)$ with image $x\in\rz_Y(X)$, the canonical commutative diagram of local rings induced by \eqref{eq:para:riemann-zariski-4},
	\begin{align}\label{eq:thm:riemann-zariski-2}
		\xymatrix{
			\kappa(y)&\ca{O}_W\ar[l]\\
			\ca{O}_{Y,y}\ar@{->>}[u]&\ca{O}_{\rz_Y(X),x}\ar[l]\ar@{->>}[u]
		}
		\end{align} 
	is Cartesian, i.e., $\ca{O}_{\rz_Y(X),x}\subseteq \ca{O}_{Y,y}$ is the preimage of $\ca{O}_W\subseteq \kappa(y)$ along the surjection $\ca{O}_{Y,y}\surj \kappa(y)$, where $\kappa(y)$ is the residue field of $y$ and $W=\spec(\ca{O}_W)$. In particular, for any element $\pi\in \ca{O}_{\rz_Y(X),x}\cap \ca{O}_{Y,y}^\times$, the canonical morphism
	\begin{align}\label{eq:thm:riemann-zariski-3}
		\ca{O}_{\rz_Y(X),x}/\pi\ca{O}_{\rz_Y(X),x}\longrightarrow \ca{O}_W/\pi\ca{O}_W
	\end{align}
	is an isomorphism.
\end{mythm}
\begin{proof}
	After \cite[2.2.1, 3.4.7]{temkin2011rz}, it remains to check the ``in particular" part. Let $\ak{m}_{Y,y}$ be the maximal ideal of $\ca{O}_{Y,y}$. Then, there is a canonical exact sequence
	\begin{align}\label{eq:thm:riemann-zariski-4}
		\xymatrix{
			0\ar[r]&\ak{m}_{Y,y}\ar[r]&\ca{O}_{\rz_Y(X),x}\ar[r]&\ca{O}_W\ar[r]&0.
		}
	\end{align}
	As $\pi$ is invertible in $\ca{O}_{Y,y}$, we have $\ak{m}_{Y,y}=\pi\ak{m}_{Y,y}\subseteq \pi\ca{O}_{\rz_Y(X),x}$. Hence, $\ca{O}_{\rz_Y(X),x}/\pi\ca{O}_{\rz_Y(X),x}= \ca{O}_W/\pi\ca{O}_W$.
\end{proof}

\begin{mycor}\label{cor:riemann-zariski}
	Let $Y\to X$ be an affine pro-open morphism of coherent schemes. Then, for any $(y\to W)\in\val_Y(X)$ with image $x\in\rz_Y(X)$, the canonical commutative diagrams of schemes
	\begin{align}\label{eq:cor:riemann-zariski-1}
		\xymatrix{
			y\ar[r]\ar[d]&W\ar[d]\\
			\spec(\ca{O}_{Y,y})\ar[r]\ar[d]&\spec(\ca{O}_{\rz_Y(X),x})\ar[d]\\
			Y\ar[r]&X
		}
	\end{align} 
	are Cartesian.
\end{mycor}
\begin{proof}
	As $Y\to X$ is affine, we put $Y\times_X\spec(\ca{O}_{\rz_Y(X),x})=\spec(B)$. As $Y\to X$ is pro-open, we know that $Y\times_XW=y$ by \ref{lem:riemann-zariski-basic-spa}. It remains to check that the canonical morphism $B\to\ca{O}_{Y,y}$ is an isomorphism. Notice that it induces a canonical morphism of exact sequences
	\begin{align}
		\xymatrix{
			0\ar[r]&\ak{m}_{Y,y}\ar@{=}[d]\ar[r]&B\ar[d]\ar[r]&\kappa(y)\ar[r]\ar@{=}[d]&0\\
			0\ar[r]&\ak{m}_{Y,y}\ar[r]&\ca{O}_{Y,y}\ar[r]&\kappa(y)\ar[r]&0
		}
		\end{align}
	where $\kappa(y)$ is the residue field of $y$, $\ak{m}_{Y,y}$ is the maximal ideal of $\ca{O}_{Y,y}$, and the first row is induced from \eqref{eq:thm:riemann-zariski-4} by base change. Thus, we see that $B=\ca{O}_{Y,y}$.
\end{proof}

\begin{mycor}\label{cor:ht-1-spa}
	Let $A$ be a ring, $\pi$ an element of $A$, $S_\pi(A)$ the subset of prime ideals $\ak{p}$ of $A$ such that $A_{\ak{p}}$ is a valuation ring of height $1$ with $\pi\in\ak{p}A_{\ak{p}}\setminus \{0\}$. If we put $Y=\spec(A[1/\pi])$ and $X=\spec(A)$, then there is a canonical injective map
	\begin{align}
		S_\pi(A)\longrightarrow \rz_Y(X)
	\end{align}
	such that for any $\ak{p}\in S_\pi(A)$ with image $x\in \rz_Y(X)$, there is a canonical isomorphism
	\begin{align}
		\ca{O}_{\rz_Y(X),x}\iso A_{\ak{p}}.
	\end{align}
\end{mycor}
\begin{proof}
	For any $\ak{p}\in S_\pi(A)$, $A_{\ak{p}}$ is a valuation ring with fraction field $A_{\ak{p}}[1/\pi]$ by assumption. In particular, $\spec(A_{\ak{p}}[1/\pi])\to Y\times_X\spec(A_{\ak{p}})$ is an isomorphism (thus, a closed immersion). Hence, $\spec(A_{\ak{p}}[1/\pi]) \to \spec(A_{\ak{p}})$ is an element of $\val_Y(X)$. This defines a canonical injection $S_\pi(A)\to \val_Y(X)$. Moreover, as the fraction field $A_{\ak{p}}[1/\pi]$ is also a localization of $A[1/\pi]$, it coincides with the stalk $\ca{O}_{Y,y}$ of $\ca{O}_Y$ at $y=\spec(A_{\ak{p}}[1/\pi])\in Y$ (actually $y$ is a generic point of $Y$). Therefore, the conclusion follows from \ref{thm:riemann-zariski}.
\end{proof}

\begin{mylem}[{\cite[3.2.2]{temkin2010stablecurve}}]\label{lem:riemann-zariski-descent}
	Let $Y\to X$ be a morphism of coherent schemes. Then, for any $x\in \rz_Y(X)$, there exists a directed inverse system $(U_\lambda)_{\lambda\in\Lambda}$ of affine schemes such that $\spec(\ca{O}_{\rz_Y(X),x})=\lim_{\lambda\in\Lambda}U_\lambda$ and that each $U_\lambda$ is an open subset of a $Y$-modification $X'$ of $X$ fitting into the following commutative diagram of locally ringed spaces
	\begin{align}
		\xymatrix{
			\spec(\ca{O}_{\rz_Y(X),x})\ar[r]\ar[d]&\rz_Y(X)\ar[d]\\
			U_\lambda\ar[r]&X'
		}
	\end{align}
\end{mylem}
\begin{proof}
	For any $Y$-modification $Z$ of $X$, let $(U_{Z,\lambda})_{\lambda\in\Lambda_Z}$ be the directed inverse system of all affine open neighborhoods of the image $z$ of $x$ in $Z$. Then, we have $\spec(\ca{O}_{Z,z})=\lim_{\lambda\in\Lambda_Z}U_{Z,\lambda}$. We endow a pre-order on the set $I=\{(Z,\lambda)\ |\ Z\in\modf_Y(X),\ \lambda\in\Lambda_Z\}$ as follows: $(Z,\lambda)\leq (Z',\lambda')$ if and only if there exists a morphism $Z'\to Z$ in $\modf_Y(X)$ (thus the unique one) sending $U_{Z',\lambda'}$ into $U_{Z,\lambda}$. As $\modf_Y(X)^{\oppo}$ is a directed set, one check easily that so is $I$. By universal properties, we obtain two canonical maps
	\begin{align}
		\xymatrix{
			\lim_{(Z,\lambda)\in I}U_{Z,\lambda}\ar@<0.6ex>[r]&\lim_{Z\in\modf_Y(X)}\lim_{\lambda\in\Lambda_Z}U_{Z,\lambda}\ar@<0.6ex>[l]
		}
	\end{align}
	mutual inverse to each other. As the right hand side is equal to $\spec(\ca{O}_{\rz_Y(X),x})=\lim_{Z\in\modf_Y(X)}\spec(\ca{O}_{Z,z})$ by construction, we complete the proof.
\end{proof}

\begin{mylem}\label{lem:riemann-zariski-basic}
	Let $Y\to X$ be a morphism of coherent schemes. Then, the underlying topological space of $\rz_Y(X)$ is spectral, and the canonical map $\rz_Y(X)\to X$ is spectral.
\end{mylem}
\begin{proof}
	It follows directly from the fact that $\rz_Y(X)$ is the cofiltered limit of spectral spaces with spectral transition maps and \cite[\href{https://stacks.math.columbia.edu/tag/0A2Z}{0A2Z}]{stacks-project}.
\end{proof}

\begin{myprop}\label{prop:riemann-zariski-basic}
	Let $(Y'\to X')\to (Y\to X)$ be a morphism of morphisms of coherent schemes. Then, there is a canonical morphism of locally ringed spaces
	\begin{align}\label{eq:prop:riemann-zariski-basic-0}
		\rz_{Y'}(X')\longrightarrow \rz_{Y}(X)
	\end{align}
	which is spectral as a map of spectral spaces and fits into the following commutative diagram of spectral spaces
	\begin{align}\label{eq:prop:riemann-zariski-basic-1}
		\xymatrix{
			\spa(Y',X')\ar[r]\ar[d]&\rz_{Y'}(X')\ar[d]\\
			\spa(Y,X)\ar[r]&\rz_{Y}(X)
		}
	\end{align}
	where the left vertical map is \eqref{eq:lem:riemann-zariski-basic-spa-1} and the horizontal maps are \eqref{eq:para:riemann-zariski-3}.
\end{myprop}
\begin{proof}
	For any $Y$-modification $Z$ of $X$, let $Z'$ be the scheme theoretical image of the induced morphism $Y'\to Z\times_XX'$ of coherent schemes. Then, we see that $Y'\to Z'$ is scheme theoretically dominant and that $Z'\to X'$ is proper, i.e., $Z'$ is a $Y'$-modification of $X'$. This natural process defines a functor $\modf_Y(X)\to \modf_{Y'}(X')$ and thus a canonical morphism of locally ringed spaces
	\begin{align}\label{eq:prop:riemann-zariski-basic-2}
		\rz_{Y'}(X')=\lim_{Z'\in \modf_{Y'}(X')}Z'\longrightarrow\lim_{Z\in \modf_Y(X)}Z =\rz_{Y}(X),
	\end{align}
	which is clearly compatible with \eqref{eq:lem:riemann-zariski-basic-spa-1} by the valuative criterion for proper morphisms \eqref{eq:para:riemann-zariski-3}. Moreover, it is spectral, i.e., quasi-compact. Indeed, for a quasi-compact open subset $U$ of $\rz_Y(X)$, there exists $Z\in \modf_Y(X)$ and a quasi-compact open subset $V$ of $Z$ whose inverse image in $\rz_Y(X)$ is $U$ (\cite[\href{https://stacks.math.columbia.edu/tag/0A30}{0A30}]{stacks-project}). Then, for any $Z'\in \modf_{Y'}(X')$ over $Z$, the inverse image of $V$ in $Z'$ is also a quasi-compact open subset as $Z'\to Z$ is quasi-compact. Thus, the inverse image of $V$ in $\rz_{Y'}(X')$ (equal to the inverse image of $U$ in $\rz_{Y'}(X')$) is a quasi-compact open subset (\cite[\href{https://stacks.math.columbia.edu/tag/0A2V}{0A2V}]{stacks-project}). This completes the proof.
\end{proof}

\begin{myprop}\label{prop:riemann-zariski-limit}
	Let $(Y_\lambda\to X_\lambda)_{\lambda\in\Lambda}$ be a directed inverse system of pro-open morphisms of coherent schemes with affine transition morphisms such that $Y_\mu=Y_\lambda\times_{X_\lambda}X_\mu$ for any indexes $\lambda\leq \mu$ in $\Lambda$, $(Y\to X)=\lim_{\lambda\in\Lambda}(Y_\lambda\to X_\lambda)$. Then, the canonical morphism of locally ringed spaces
	\begin{align}\label{eq:prop:riemann-zariski-limit-1}
		\rz_Y(X)\longrightarrow \lim_{\lambda\in\Lambda}\rz_{Y_\lambda}(X_\lambda)
	\end{align} 
	is an isomorphism which fits into the following commutative diagram of spectral spaces
	\begin{align}\label{eq:prop:riemann-zariski-limit-1-2}
		\xymatrix{
			\val_Y(X)\ar[r]^-{\sim}\ar[d]_-{\wr}&\rz_Y(X)\ar[d]^-{\wr}\\
			\lim_{\lambda\in\Lambda}\val_{Y_\lambda}(X_\lambda)\ar[r]^-{\sim}&\lim_{\lambda\in\Lambda}\rz_{Y_\lambda}(X_\lambda)
		}
	\end{align}
	where the horizontal homeomorphisms are induced by \eqref{eq:thm:riemann-zariski} and the left vertical homeomorphism is induced by \eqref{eq:prop:spa-limit-1}.
\end{myprop}
\begin{proof}
	It follows from \ref{lem:riemann-zariski-basic-spa} and \ref{prop:riemann-zariski-basic} that there is a canonical commutative diagram of spectral spaces
	\begin{align}\label{eq:prop:riemann-zariski-limit-3}
		\xymatrix{
			\val_Y(X)\ar[r]\ar[d]&\spa(Y,X)\ar[r]\ar[d]&\rz_Y(X)\ar[d]\\
			\lim_{\lambda\in\Lambda}\val_{Y_\lambda}(X_\lambda)\ar[r]&\lim_{\lambda\in\Lambda}\spa(Y_\lambda,X_\lambda)\ar[r]&\lim_{\lambda\in\Lambda}\rz_{Y_\lambda}(X_\lambda)
		}
	\end{align}
	where the two compositions of the horizontal arrows are both homeomorphisms by \ref{thm:riemann-zariski}, and the vertical arrow in the middle is a homeomorphism by \ref{prop:spa-limit}. Notice that by \ref{lem:riemann-zariski-basic-spa}, any element $(y_\lambda\to W_\lambda)_{\lambda\in\Lambda}$ of $\lim_{\lambda\in\Lambda}\val_{Y_\lambda}(X_\lambda)$ is actually a directed inverse system of morphisms of affine schemes such that $\ca{O}_{W_\lambda}\to \ca{O}_{W_\mu}$ is an extension of valuation rings for any indexes $\lambda\leq \mu$ in $\Lambda$ and that $y_\lambda=Y_\lambda\times_{X_\lambda}W_\lambda$ for any $\lambda\in\Lambda$. Thus, we set
	\begin{align}
		y=\lim_{\lambda\in\Lambda}y_\lambda\in Y=\lim_{\lambda\in\Lambda}Y_\lambda,\quad \trm{ and }\quad\ca{O}_W=\colim_{\lambda\in\Lambda}\ca{O}_{W_\lambda}.
	\end{align}
	Then, we still have $\ca{O}_{W_\lambda} \to \ca{O}_W$ is an extension of valuation rings for any $\lambda\in\Lambda$ and $y=Y\times_XW$. Thus, $y\to W$ is an element of $\val_Y(X)$. Arguing as in \ref{prop:spa-limit}, we see that the left vertical arrow in \eqref{eq:prop:riemann-zariski-limit-3} is bijective. Moreover, it is a homeomorphism since it is a bijection between topological subspaces of $\spa(Y,X)=\lim_{\lambda\in\Lambda}\spa(Y_\lambda,X_\lambda)$. Therefore, \eqref{eq:prop:riemann-zariski-limit-1} is a homeomorphism. The conclusion follows from the comparison between the local rings via the horizontal arrows by \ref{thm:riemann-zariski}.
\end{proof}

\begin{myrem}\label{rem:prop:riemann-zariski-limit}
	It seems that we couldn't prove \ref{prop:riemann-zariski-limit} by the standard limit argument in \cite{ega4-3}, since without Noetherian or finite presented assumptions, we couldn't directly approximate $Y$-modifications of $X$ by $Y_{\lambda}$-modifications of $X_\lambda$.
\end{myrem}

\begin{mypara}\label{para:riemann-zariski-openopen}
	Let $Y\to X$ be a pro-open morphism of coherent schemes. Note that for any quasi-compact open subset $U$ of $Y$, $\spa(U,X)$ is a quasi-compact open subset of $\spa(Y,X)$ (\ref{lem:val-spa-inj}, \ref{lem:spa-induced-top}, \ref{prop:spa-spectral}). Let $Z$ be a pro-open subset of $Y$ (\ref{defn:pro-open}). We deduce from \ref{prop:spa-limit} that
	\begin{align}
		\spa(Z,X)=\bigcap_{\lambda\in\Lambda}\spa(U_\lambda,X),
	\end{align}
	where $Z=\bigcap_{\lambda\in\Lambda}U_\lambda$ for a directed inverse system of quasi-compact open subsets $(U_\lambda)_{\lambda\in\Lambda}$ of $Y$ with affine transition inclusion maps.	Then, $\spa(Z,X)\cap \val_Y(X)$ defines a subspace $\rz^Z_Y(X)$ of $\rz_Y(X)$ via the homeomorphism \eqref{eq:thm:riemann-zariski}. It is a spectral space as an intersection of quasi-compact open subsets (\cite[\href{https://stacks.math.columbia.edu/tag/0A31}{0A31}]{stacks-project}). Unwinding the definitions, we see that
	\begin{align}\label{eq:para:riemann-zariski-openopen}
		\rz^Z_Y(X)=\{x\in \rz_Y(X)\ |\ y\in Z,\trm{ where }(y\to W)\in \val_Y(X) \trm{ maps to $x$ via }\eqref{eq:thm:riemann-zariski}\}.
	\end{align}
\end{mypara}

\begin{mycor}\label{cor:riemann-zariski-limit}
	Under the assumptions in {\rm\ref{prop:riemann-zariski-limit}} and with the same notation, let $(Z_\lambda)_{\lambda\in\Lambda}$ be an inverse directed system of pro-open subsets of $(Y_\lambda)_{\lambda\in\Lambda}$, $Z=\lim_{\lambda\in\Lambda}Z_\lambda\subseteq Y$. Then, $Z$ is a pro-open subset of $Y$, and the canonical map \eqref{eq:prop:riemann-zariski-limit-1} induces a homeomorphism of spectral subspaces
	\begin{align}\label{eq:cor:riemann-zariski-limit}
		\rz_Y^Z(X)\iso \lim_{\lambda\in\Lambda}\rz_{Y_\lambda}^{Z_\lambda}(X_\lambda).
	\end{align}
\end{mycor}
\begin{proof}
	One checks easily that $Z$ is a cofiltered intersection of quasi-compact open subsets of $Y$ with affine inclusion maps. Then, for an element $x=(x_\lambda)_{\lambda\in\Lambda}\in \lim_{\lambda\in\Lambda}\rz_{Y_\lambda}(X_\lambda)$, let $(y\to W)=(y_\lambda\to W_\lambda)_{\lambda\in\Lambda}\in \val_Y(X)=\lim_{\lambda\in\Lambda}\val_{Y_\lambda}(X_\lambda)$ be the corresponding element via \eqref{eq:prop:riemann-zariski-limit-1-2}. Then, $x$ lies in the subset $\lim_{\lambda\in\Lambda}\rz_{Y_\lambda}^{Z_\lambda}(X_\lambda)$ if and only if the map of inverse system of schemes $(y_\lambda)_{\lambda\in\Lambda}\to (Y_\lambda)_{\lambda\in\Lambda}$ factors through $(Z_\lambda)_{\lambda\in\Lambda}$. This is equivalent to that $\lim_{\lambda\in\Lambda}y_\lambda\to \lim_{\lambda\in\Lambda}Y_\lambda$ factors through $\lim_{\lambda\in\Lambda}Z_\lambda$, i.e., $y\in Y$ lies in $Z$. In other words, this is saying that $x\in \rz_Y^Z(X)$. This shows that \eqref{eq:prop:riemann-zariski-limit-1} induces a bijection between subsets \eqref{eq:cor:riemann-zariski-limit}. Since each side of \eqref{eq:cor:riemann-zariski-limit} is endowed with the subspace topology of the corresponding side of \eqref{eq:prop:riemann-zariski-limit-1}, \eqref{eq:cor:riemann-zariski-limit} is a homeomorphism.
\end{proof}

\begin{mylem}\label{lem:riemann-zariski-openopen}
	Let $Y\to X$ be an affine pro-open morphism of coherent schemes, $Z$ a pro-open subset of $Y$. Then, we have
	\begin{align}\label{eq:lem:riemann-zariski-openopen}
		\rz^Z_Y(X)=\{x\in \rz_Y(X)\ |\ Y\times_X\spec(\ca{O}_{\rz_Y(X),x})\to Y\trm{ factors through }Z\}.
	\end{align}
\end{mylem}
\begin{proof}
	For any $x\in \rz_Y(X)$, let $(y\to W)\in\val_Y(X)$ be the corresponding element via \eqref{eq:thm:riemann-zariski}. Since $Y\to X$ is affine and pro-open, we have $\spec(\ca{O}_{Y,y})=Y\times_X\spec(\ca{O}_{\rz_Y(X),x})$ by \ref{cor:riemann-zariski}. Since $Z$ is a pro-open subset of $Y$, a point $y$ of $Y$ lies in $Z$ if and only if the morphism $\spec(\ca{O}_{Y,y})\to Y$ factors through $Z$. Hence, we obtain \eqref{eq:lem:riemann-zariski-openopen} from \eqref{eq:para:riemann-zariski-openopen}.
\end{proof}

\section{Criteria for Perfectoidness of General Limits of Essentially Adequate Schemes}\label{sec:local-faltings}
We generalize the main results of Section \ref{sec:purity} replacing the set of generic points of the special fibre by the Riemann-Zariski space. In a similar manner, we prove that the property that Faltings cohomology of an essentially adequate scheme is controlled by the Faltings cohomology of stalks of the Riemann-Zariski space (called \emph{Riemman-Zariski faithfulness}, see \ref{defn:rz-faithful}) is preserved after taking general inverse limit of essentially adequate schemes (see \ref{cor:rz-continue}). As a corollary, we show that the Faltings acyclicity for such a limit can be checked on the stalks of the Riemann-Zariski space (see \ref{thm:val-criterion-faltings-acyclic}). Combining with Gabber-Ramero's computation of differentials of valuation rings, we deduce a differential criterion for Faltings acyclicity (see \ref{thm:diff-criterion-faltings-acyclic}). Restricting to the affine case, we obtain valuative and differential criteria for perfectoidness (see \ref{thm:criterion}). For applications in Section \ref{sec:polystable}, we study general limits (pro-objects) and relative version of Faltings cohomology in this section.

\begin{mypara}\label{para:zeta-3}
	In this section, we fix a pre-perfectoid field $L$ extension of $\bb{Q}_p$. We put
	\begin{align}\label{eq:para:zeta-3-1}
		\eta=\spec(L), \quad S=\spec(\ca{O}_L),\quad s=\spec(\ca{O}_L/\ak{m}_L).
	\end{align}
	Let $\fal^{\mrm{open}}_{\eta\to S}$ be the category of open immersions of coherent schemes over $\eta\to S$, and let $\pro(\fal^{\mrm{open}}_{\eta\to S})$ the category of pro-objects of $\fal^{\mrm{open}}_{\eta\to S}$ (\cite[\Luoma{1}.8.10]{sga4-1}). In other words, an object of $\pro(\fal^{\mrm{open}}_{\eta\to S})$ is a directed inverse system $(X^{\triv}_\lambda\to X_\lambda)_{\lambda\in\Lambda}$ of open immersions of coherent schemes over $\eta\to S$, and the set of morphisms of such two objects is given by
	\begin{align}\label{eq:para:zeta-3-2}
		\mo_{\pro(\fal^{\mrm{open}}_{\eta\to S})}((Y^{\triv}_\xi,Y_\xi)_{\xi\in\Xi},(X^{\triv}_\lambda\to X_\lambda)_{\lambda\in\Lambda})=\lim_{\lambda\in\Lambda}\colim_{\xi\in\Xi}\mo_{\fal^{\mrm{open}}_{\eta\to S}}((Y^{\triv}_\xi,Y_\xi),(X^{\triv}_\lambda\to X_\lambda)).
	\end{align}
	We regard $\fal^{\mrm{open}}_{\eta\to S}$ as a full subcategory of $\pro(\fal^{\mrm{open}}_{\eta\to S})$ (\cite[\Luoma{1}.8.10.6]{sga4-1}). 
	
	For any object $(X^{\triv}_\lambda\to X_\lambda)_{\lambda\in\Lambda}$ in $\pro(\fal^{\mrm{open}}_{\eta\to S})$, we put
	\begin{align}\label{eq:fal-limit}
		(\fal^{\et}_{X^{\triv}\to X},\falb)&=\lim_{\lambda\in\Lambda}(\fal^{\et}_{X^{\triv}_\lambda\to X_\lambda},\falb)
	\end{align}
	the cofiltered limit of ringed sites defined in \cite[8.2.3, 8.6.2]{sga4-2}. Note that $\falb$ is flat over $\ca{O}_L$. We remark that if the transition morphisms of $(X^{\triv}_\lambda\to X_\lambda)_{\lambda\in\Lambda}$ are affine, then $(\fal^{\et}_{X^{\triv}\to X},\falb)$ is canonically equivalent to the Faltings ringed site associated to the morphism of coherent schemes $\lim_{\lambda\in\Lambda} X^{\triv}_\lambda\to \lim_{\lambda\in\Lambda} X_\lambda$ by \cite[7.12]{he2024coh}.
	
	Similarly, we put
	\begin{align}\label{eq:rz-limit}
		X=\lim_{\lambda\in\Lambda}X_\lambda,\quad X_\eta=\lim_{\lambda\in\Lambda}X_{\lambda,\eta},\quad X^{\triv}=\lim_{\lambda\in\Lambda}X^{\triv}_\lambda
	\end{align}
	the cofiltered limits of locally ringed spectral spaces, which are also the cofiltered limits of ringed sites (\ref{para:limit-ringed-space}). The canonical morphisms of ringed sites 
	\begin{align}\label{eq:sigma_lambda}
		\sigma_\lambda:(\fal^{\et}_{X^{\triv}_\lambda\to X_\lambda},\falb)\longrightarrow (X_\lambda,\ca{O}_{X_\lambda}),
	\end{align}
	defined by the left exact continuous functor $\sigma^+_\lambda:X_{\lambda,\mrm{Zar}}\to \fal^{\et}_{X^{\triv}_\lambda\to X_\lambda}$ sending each quasi-compact open subset $U$ of $X_\lambda$ to $U^{\triv}=X^{\triv}_\lambda\times_{X_\lambda}U\to U$ (cf. \cite[(7.8.5)]{he2024coh}), define a canonical morphism of ringed sites by taking cofiltered limits,
	\begin{align}\label{eq:sigma}
		\sigma: (\fal^{\et}_{X^{\triv}\to X},\falb)\longrightarrow (X,\ca{O}_X).
	\end{align}
	We remark that any construction above is functorial in $\pro(\fal^{\mrm{open}}_{\eta\to S})$, any site above is coherent (\cite[\Luoma{6}.2.3]{sga4-2}), and thus any morphism of sites above is coherent (\cite[\Luoma{6}.3.1]{sga4-2}).
\end{mypara}

\begin{mypara}\label{para:limit-rz}
	Following \ref{para:zeta-3}, for any object $(X^{\triv}_\lambda\to X_\lambda)_{\lambda\in\Lambda}$ in $\pro(\fal^{\mrm{open}}_{\eta\to S})$, we put (see \ref{prop:riemann-zariski-basic})
	\begin{align}\label{eq:para:limit-rz-1}
		\rz_{X_\eta}(X)=\lim_{\lambda\in\Lambda}\rz_{X_{\lambda,\eta}}(X_\lambda)
	\end{align}
	the cofiltered limit of locally ringed spectral spaces, which are also the cofiltered limit of ringed sites (\ref{para:limit-ringed-space}). 
	
	For any $\lambda\in\Lambda$, consider the spectral subspace $\rz^{X^{\triv}_\lambda}_{X_{\lambda,\eta}}(X_\lambda)$ of $\rz_{X_{\lambda,\eta}}(X_\lambda)$ defined in \ref{para:riemann-zariski-openopen}. For simplicity, we denote it by $X^{\rz}_\lambda$ and we put
	\begin{align}\label{eq:para:limit-rz-2}
		X^{\rz}=\lim_{\lambda\in\Lambda}X^{\rz}_\lambda,
	\end{align}
	which is a spectral subspace of $\rz_{X_\eta}(X)$ (\cite[\href{https://stacks.math.columbia.edu/tag/0A2V}{0A2V}]{stacks-project}). 
	
	Moreover, it follows from \ref{lem:riemann-zariski-basic-spa} and \ref{prop:riemann-zariski-basic} that there is a canonical commutative diagram of inverse systems of spectral spaces
	\begin{align}\label{eq:para:limit-rz-3}
		\xymatrix{
			(\val_{X_{\lambda,\eta}}(X_\lambda)\cap\spa(X^{\triv}_\lambda,X_\lambda) )_{\lambda\in\Lambda}\ar[r]\ar[d]&(\spa(X^{\triv}_\lambda,X_\lambda))_{\lambda\in\Lambda}\ar[r]\ar[d]&(\rz^{X^{\triv}_\lambda}_{X_{\lambda,\eta}}(X_\lambda))_{\lambda\in\Lambda}\ar[d]\\
			(\val_{X_{\lambda,\eta}}(X_\lambda))_{\lambda\in\Lambda}\ar[r]&(\spa(X_{\lambda,\eta},X_\lambda))_{\lambda\in\Lambda}\ar[r]&(\rz_{X_{\lambda,\eta}}(X_\lambda))_{\lambda\in\Lambda}
		}
	\end{align}
	where for each $\lambda\in\Lambda$, the two compositions of the horizontal arrows are both homeomorphisms by \ref{thm:riemann-zariski} and the vertical arrows are inclusions of spectral subspaces. We put 
	\begin{align}\label{eq:para:limit-rz-4}
		\val_{X_\eta}(X)=\lim_{\lambda\in\Lambda}\val_{X_{\lambda,\eta}}(X_\lambda),\  \spa(X^{\triv},X)=\lim_{\lambda\in\Lambda}\spa(X^{\triv}_\lambda,X_\lambda),\  \spa(X_\eta,X)=\lim_{\lambda\in\Lambda}\spa(X_{\lambda,\eta},X_\lambda)
	\end{align}
	the cofiltered limits of spectral spaces. Thus, taking cofiltered limits of \eqref{eq:para:limit-rz-3}, we obtain a commutative diagram of spectral spaces
	\begin{align}\label{eq:para:limit-rz-5}
		\xymatrix{
			\val_{X_\eta}(X)\cap\spa(X^{\triv},X)\ar[r]\ar[d]&\spa(X^{\triv},X)\ar[r]\ar[d]&X^{\rz}\ar[d]\\
			\val_{X_\eta}(X)\ar[r]&\spa(X_\eta,X)\ar[r]&\rz_{X_\eta}(X)
		}
	\end{align}
	where the two compositions of the horizontal arrows are both homeomorphisms and the vertical arrows are inclusions of spectral subspaces.
	
	We remark that if the transition morphisms of $(X^{\triv}_\lambda\to X_\lambda)_{\lambda\in\Lambda}$ are affine, then $\rz_{X_\eta}(X)$ is canonically equivalent to the Riemann-Zariski space associated to the pro-open morphism of coherent schemes $\lim_{\lambda\in\Lambda}X_{\lambda,\eta}\to \lim_{\lambda\in\Lambda}X_\lambda$ by \ref{prop:riemann-zariski-limit}, and that $X^{\rz}$ is the spectral subspace $\rz_{\lim_{\lambda\in\Lambda}X_{\lambda,\eta}}^{\lim_{\lambda\in\Lambda}X^{\triv}_\lambda} (\lim_{\lambda\in\Lambda}X_\lambda)$ by \ref{cor:riemann-zariski-limit}. Similar statements hold for $\val$ and $\spa$. We can extend properties in the last section formally by taking cofiltered limits to this situation. We only list some properties in need and sketch the proof below.
\end{mypara}

\begin{myprop}\label{prop:limit-rz}
	Let $(X^{\triv}_\lambda\to X_\lambda)_{\lambda\in\Lambda}$ be an object in $\pro(\fal^{\mrm{open}}_{\eta\to S})$.
	\begin{enumerate}
		\renewcommand{\labelenumi}{{\rm(\theenumi)}}
		\item The underlying set of $\spa(X_\eta,X)$ naturally identifies with the set of isomorphism classes of commutative diagrams of locally ringed spectral spaces
		\begin{align}\label{eq:prop:limit-rz-1}
			\xymatrix{
				y\ar[r]\ar[d]&W\ar[d]\\
				X_\eta\ar[r]&X
			}
		\end{align} 
		where $y$ is a point of $X_\eta$ and $W$ is the spectrum of a valuation ring with generic point $y$. \label{item:prop:limit-rz-1}
		\item The subset $\spa(X^{\triv},X)\subset \spa(X_\eta,X)$ identifies with those $y\to W$ such that $y\in X^{\triv}\subseteq X_\eta$.\label{item:prop:limit-rz-2}
		\item The subset $\val_{X_\eta}(X)\subset \spa(X_\eta,X)$ identifies with those $y\to W$ such that $\kappa(y)=\ca{O}_W[1/p]$, where $\kappa(y)$ is the residue field of $y$ and $W=\spec(\ca{O}_W)$.\label{item:prop:limit-rz-3}
		\item Let $(X'^{\triv}_\xi,X'_\xi)_{\xi\in\Xi}\to (X^{\triv}_\lambda\to X_\lambda)_{\lambda\in\Lambda}$ be a morphism in $\pro(\fal^{\mrm{open}}_{\eta\to S})$. Then, the canonical spectral map of spectral spaces $\spa(X'_\eta,X')\to \spa(X_\eta,X)$ sends an element $y'\to W'$ to $y\to W$, where $y$ is the image of $y'$ under $X'_\eta\to X_\eta$ and $\ca{O}_W=\kappa(y)\cap \ca{O}_{W'}\subseteq \kappa(y')$. In particular, $\ca{O}_W\to \ca{O}_{W'}$ is an extension of valuation rings.\label{item:prop:limit-rz-4}
		\item For any $(y\to W)\in \spa(X_\eta,X)$ with image $x\in\rz_{X_\eta}(X)$, there is a canonical commutative Cartesian and co-Cartesian diagram of local rings
		\begin{align}\label{eq:prop:limit-rz-2}
			\xymatrix{
				\kappa(y)&\ca{O}_W\ar[l]\\
				\ca{O}_{\rz_{X_\eta}(X),x}[1/p]\ar@{->>}[u]&\ca{O}_{\rz_{X_\eta}(X),x}\ar[l]\ar@{->>}[u]
			}
		\end{align} 
		which induces an isomorphism of $p$-adic completions
		\begin{align}\label{eq:prop:limit-rz-3}
			\widehat{\ca{O}_{\rz_{X_\eta}(X),x}}\iso \widehat{\ca{O}_W}.
		\end{align}\label{item:prop:limit-rz-5}
		\item We have $X^{\rz}=\{x\in \rz_{X_{\eta}}(X)\ |\ \spec(\ca{O}_{\rz_{X_{\eta}}(X),x}[1/p])\to X_{\eta}\trm{ factors through }X^{\triv}\}$.\label{item:prop:limit-rz-6}
	\end{enumerate}
\end{myprop}
\begin{proof}
	(\ref{item:prop:limit-rz-1}) It follows from the same arguments as in \ref{prop:spa-limit}. We remind the readers that each $(y\to W)\in \spa(X_\eta,X)$ is actually the cofiltered limit of the associated directed inverse system $(y_\lambda\to W_\lambda)_{\lambda\in\Lambda}$ of morphisms of affine schemes such that $\ca{O}_{W_\lambda}\to \ca{O}_{W_\mu}$ is an extension of valuation rings with inclusion of fraction fields $\kappa(y_\lambda)\subseteq \kappa(y_\mu)$ for any indexes $\lambda\leq \mu$ in $\Lambda$.
	
	(\ref{item:prop:limit-rz-2}) It follows from 	(\ref{item:prop:limit-rz-1}) and the fact that $y\in X^{\triv}$ if and only if $y_\lambda\in X^{\triv}_\lambda$ for any $\lambda\in\Lambda$ (see \ref{lem:val-spa-inj}.(\ref{item:lem:val-spa-inj-1})).
	
	(\ref{item:prop:limit-rz-3}) It follows from (\ref{item:prop:limit-rz-1}) and the fact that $\kappa(y)=\ca{O}_W[1/p]$ if and only if $\kappa(y_\lambda)=\ca{O}_{W_\lambda}[1/p]$ for any $\lambda\in\Lambda$ (see \ref{lem:riemann-zariski-basic-spa} and \ref{lem:val-ext}).
	
	(\ref{item:prop:limit-rz-4}) The morphism $(X'^{\triv}_\xi,X'_\xi)_{\xi\in\Xi}\to (X^{\triv}_\lambda\to X_\lambda)_{\lambda\in\Lambda}$ is given by a compatible family of morphisms $(X'^{\triv}_{\xi_\lambda},X'_{\xi_\lambda})\to(X^{\triv}_\lambda\to X_\lambda)$ for each $\lambda\in\Lambda$ and for some $\xi_\lambda\in\Xi$ large enough \eqref{eq:para:zeta-3-2}. Thus, for any $\lambda\in\Lambda$, the associated homomorphism $\ca{O}_{W_\lambda}\to \ca{O}_{W'_{\xi_\lambda}}$ is an extension of valuation rings by \ref{lem:riemann-zariski-basic-spa}. Taking filtered colimits, we see that $\ca{O}_W=\colim_{\lambda\in\Lambda}\ca{O}_{W_\lambda}\to \ca{O}_{W'}=\colim_{\xi\in\Xi}\ca{O}_{W'_\xi}$ is still an extension of valuation rings (see \ref{lem:val-ext}).
	
	(\ref{item:prop:limit-rz-5}) Let $x_\lambda\in \rz_{X_{\lambda,\eta}}(X_\lambda)$ be the image of $(y_\lambda\to W_\lambda)\in \spa(X_{\lambda,\eta},X_\lambda)$. Then, the canonical commutative diagram \eqref{eq:thm:riemann-zariski-2}
	\begin{align}
		\xymatrix{
			\kappa(y_\lambda)&\ca{O}_{W_\lambda}\ar[l]\\
			\ca{O}_{X_{\lambda,\eta},y_\lambda}\ar@{->>}[u]&\ca{O}_{\rz_{X_{\lambda,\eta}}(X_\lambda),x_\lambda}\ar[l]\ar@{->>}[u]
		}
	\end{align}
	is Cartesian by \ref{thm:riemann-zariski} and we have $\ca{O}_{\rz_{X_{\lambda,\eta}}(X_\lambda),x_\lambda}/p^n\ca{O}_{\rz_{X_{\lambda,\eta}}(X_\lambda),x_\lambda}=\ca{O}_{W_\lambda}/p^n\ca{O}_{W_\lambda}$ for any $n\in\bb{N}$. Moreover, it follows from \ref{cor:riemann-zariski} that $\ca{O}_{X_{\lambda,\eta},y_\lambda}=\ca{O}_{\rz_{X_{\lambda,\eta}}(X_\lambda),x_\lambda}[1/p]$ and that this diagram is also co-Cartesian. Taking filtered colimits, we see that \eqref{eq:prop:limit-rz-2} is Cartesian and co-Cartesian, that $\ca{O}_{\rz_{X_\eta}(X),x}[1/p]$ is a local ring and that $\ca{O}_{\rz_{X_\eta}(X),x}/p^n\ca{O}_{\rz_{X_\eta}(X),x}=\ca{O}_W/p^n\ca{O}_W$. 
	
	(\ref{item:prop:limit-rz-6}) Notice that $X^{\triv}\to X_\eta$ induces isomorphisms on stalks. Thus, $\spec(\ca{O}_{\rz_{X_{\eta}}(X),x}[1/p])\to X_{\eta}$ factors through $X^{\triv}$ if and only if $y\in X^{\triv}$, since $\ca{O}_{\rz_{X_{\eta}}(X),x}[1/p]$ is a local ring whose closed point is $y$. The conclusion follows from the homeomorphism $\val_{X_\eta}(X)\cap\spa(X^{\triv},X)\iso X^{\rz}$ and (\ref{item:prop:limit-rz-2}).
\end{proof}

\begin{mylem}\label{lem:val-micro}
	Let $V$ be an $\ca{O}_L$-algebra which is a valuation ring with fraction field $V[1/p]$. Then, either $V$ is a field extension of $L$ or its localization $V_{\sqrt{pV}}$ at the radical ideal generated by $p$ is a valuation ring of height $1$ extension of $\ca{O}_L$. In both cases, $V\to V_{\sqrt{pV}}$ is an almost isomorphism.
\end{mylem}
\begin{proof}
	If $V$ is not a field (i.e., $V\neq V[1/p]$), then the radical ideal $\sqrt{pV}$ is a prime ideal of $V$, and $V_{\sqrt{pV}}$ is a valuation ring of height $1$ extension of $\ca{O}_L$ (\cite[2.1]{he2024coh}). Note that $V\to V_{\sqrt{pV}}$ is injective whose cokernel is killed by $\sqrt{pV}$ (\cite[4.7]{he2024coh}). As $\sqrt{pV}$ contains $\ak{m}_L$, we see that $V\to V_{\sqrt{pV}}$ is an almost isomorphism.
\end{proof}
	
\begin{mypara}\label{para:limit-rz-summary}
	Following \ref{para:limit-rz}, for any $x\in X^{\rz}$, we put 
	\begin{align}\label{eq:para:zeta-3-3}
		X^{\rz}_{(x)}=\spec(\ca{O}_{\rz_{X_\eta}(X),x})\quad \trm{and}\quad X^{\rz}_{(x),\eta}=\spec(\ca{O}_{\rz_{X_\eta}(X),x}[1/p]).
	\end{align}
	We note that there is a canonical morphism $(X^{\rz}_{(x),\eta}\to X^{\rz}_{(x)})\to (X^{\triv}\to X)$ of morphisms of locally ringed spaces by \ref{prop:limit-rz}.(\ref{item:prop:limit-rz-6}), and (equivalently) a canonical morphism $(X^{\rz}_{(x),\eta}\to X^{\rz}_{(x)})\to (X^{\triv}_\lambda\to X_\lambda)_{\lambda\in\Lambda}$ in $\pro(\fal^{\mrm{open}}_{\eta\to S})$. In particular, there is a canonical commutative diagram of ringed sites by the functoriality of \eqref{eq:sigma},
	\begin{align}\label{diam:sigma}
		\xymatrix{
			(\fal^\et_{X^{\rz}_{(x),\eta}\to X^{\rz}_{(x)}},\falb)\ar[d]_-{\sigma_x}\ar[r]&(\fal^{\et}_{X^{\triv}\to X},\falb)\ar[d]^-{\sigma} \\
			(X^{\rz}_{(x)},\ca{O}_{\rz_{X_\eta}(X),x})\ar[r]^-{f_x}&(X,\ca{O}_X).
		}
	\end{align}
\end{mypara}

\begin{mydefn}\label{defn:rz-faithful}
	Let $(X^{\triv}_\lambda\to X_\lambda)_{\lambda\in\Lambda}$ be an object in $\pro(\fal^{\mrm{open}}_{\eta\to S})$. We say that its \emph{Faltings cohomology is Riemann-Zariski faithful} if the following conditions hold:
	\begin{enumerate}
		\renewcommand{\labelenumi}{{\rm(\theenumi)}}
		\item For any nonzero element $\pi\in\ak{m}_L$ and any integer $n$, the natural map of $\ca{O}_X$-modules over $X$ induced by \eqref{diam:sigma},
		\begin{align}\label{eq:defn:rz-faithful-1}
			\rr^n\sigma_*(\falb/\pi\falb)\longrightarrow \prod_{x\in X^{\rz}}f_{x*}\rr^n\sigma_{x*}(\falb/\pi\falb)
		\end{align}
		is almost injective.\label{item:defn:rz-faithful-1}
		\item For any nonzero element $\pi\in(\zeta_p-1)\ca{O}_L$, the canonical exact sequence $0\to \falb/\pi\falb\stackrel{\cdot \pi}{\longrightarrow}\falb/\pi^2\falb\to \falb/\pi\falb\to 0$ induces an almost exact sequence of $\ca{O}_X$-modules
		\begin{align}\label{eq:defn:rz-faithful-2}
			\xymatrix{
				0\ar[r]&\ca{O}_X/\pi \ca{O}_X\ar[r]&\sigma_*(\falb/\pi\falb)\ar[r]^-{\delta^0}&\rr^1\sigma_*(\falb/\pi\falb).
			}
		\end{align}
		where $\sigma: (\fal^{\et}_{X^{\triv}\to X},\falb)\to (X,\ca{O}_X)$ is the canonical morphism of ringed sites \eqref{eq:sigma}.\label{item:defn:rz-faithful-2}
	\end{enumerate}
\end{mydefn}

\begin{myprop}\label{prop:rz-faithful-adequate}
	Let $X^{\triv}\to X$ be an open immersion of coherent schemes essentially adequate over $\eta\to S$ {\rm(\ref{defn:essential-adequate-pair})}. Assume that $L$ contains a compatible system of primitive $n$-th roots of unity $\{\zeta_n\}_{n\in\bb{N}_{>0}}$. Then, the Faltings cohomology of $X^{\triv}\to X$ is Riemann-Zariski faithful.
\end{myprop}
\begin{proof}
	Recall that for the morphism $\sigma:\fal^\et_{X^{\triv}\to X}\to X$ \eqref{eq:sigma}, $\rr^n\sigma_*(\falb/\pi\falb)$ is the sheaf associated to the presheaf on $X$ sending a quasi-compact open subset $U$ to $H^n(\fal^{\et}_{U^{\triv}\to U},\falb/\pi\falb)$, where $U^{\triv}=X^{\triv}\times_XU$. On the other hand, if $U=\spec(B)$ is affine, then by \ref{thm:essential-adequate-coh} the canonical morphism \eqref{eq:thm:essential-adequate-coh-1}
	\begin{align}
		H^n(\fal^\et_{U^{\triv}\to U},\falb/\pi\falb)\longrightarrow \prod_{x\in\ak{G}(U_s)}H^n(\fal^\et_{U^{\triv}_{(x)}\to U_{(x)}},\falb/\pi\falb)
	\end{align}
	is almost injective and the canonical sequence \eqref{eq:thm:essential-adequate-coh-2} (if $\pi\in(\zeta_p-1)\ca{O}_L$)
	\begin{align}
		\xymatrix{
			0\ar[r]&B/\pi B\ar[r]&H^0(\fal^\et_{U^{\triv}\to U},\falb/\pi\falb)\ar[r]^-{\delta^0}&H^1(\fal^\et_{U^{\triv}\to U},\falb/\pi\falb)
		}
	\end{align}
	is almost exact. Note that there are canonical injections $\ak{G}(U_s)\to U^{\rz}\to X^{\rz}$ identifying the corresponding local rings by \ref{cor:ht-1-spa} (whose assumptions are satisfied by \ref{lem:essential-adequate-generic-loc}) and \ref{prop:limit-rz}. Thus, there are natural maps
	\begin{align}
		\prod_{x\in\ak{G}(U_s)}H^n(\fal^\et_{U^{\triv}_{(x)}\to U_{(x)}},\falb/\pi\falb)&\subseteq \prod_{x\in U^{\rz}}H^n(\fal^\et_{U^{\rz}_{(x),\eta}\to U^{\rz}_{(x)}},\falb/\pi\falb)\\
		&\to \Gamma(U, \prod_{x\in U^{\rz}}f_{x*}\rr^n\sigma_{x*}(\falb/\pi\falb))\nonumber\\
		&\subseteq \Gamma(U, \prod_{x\in X^{\rz}}f_{x*}\rr^n\sigma_{x*}(\falb/\pi\falb)).\nonumber
	\end{align}
	where the second map is an almost injection, since $\rr^n\sigma_{x*}(\falb/\pi\falb)$ is canonically almost isomorphic to the quasi-coherent $\ca{O}_{\rz_{X_\eta}(X),x}$-module associated to $H^n(\fal^\et_{U^{\rz}_{(x),\eta}\to U^{\rz}_{(x)}},\falb/\pi\falb)$ by \ref{cor:acyclic}.
	After sheafification over $X$, we see that \eqref{eq:defn:rz-faithful-1} is almost injective and that \eqref{eq:defn:rz-faithful-2} is almost exact.
\end{proof}

\begin{mylem}\label{lem:rz-faithful-supp}
	Let $X^{\triv}\to X$ be an open immersion of coherent schemes over $\eta\to S$. Then, for any quasi-compact open subset $U$ of $X$, any nonzero element $\pi\in\ak{m}_L$, any integer $n$ and any element $\xi\in  \Gamma(U,	\rr^n\sigma_*(\falb/\pi\falb))$, the subset of $X^{\rz}$,
	\begin{align}\label{eq:lem:rz-faithful-supp-1}
		\supprz(\xi)=\{x\in X^{\rz}\ |\ \xi\neq 0\trm{ in } \Gamma(U,f_{x*}\rr^n\sigma_{x*}(\falb/\pi\falb))\trm{ via }\eqref{eq:defn:rz-faithful-1}\},
	\end{align}
	is a closed subspace.
\end{mylem}
\begin{proof}
	For any $x\notin \supprz(\xi)$, there exists a directed inverse system $(U_\lambda)_{\lambda\in\Lambda}$ of affine schemes such that $X^{\rz}_{(x)}=\lim_{\lambda\in\Lambda}U_\lambda$ and that each $U_\lambda$ is an open subset of an $X_\eta$-modification $X'$ of $X$ fitting into the following commutative diagram of locally ringed spaces (\ref{lem:riemann-zariski-descent})
	\begin{align}
		\xymatrix{
			X^{\rz}_{(x)}\ar[r]\ar[d]&\rz_{X_\eta}(X)\ar[d]\\
			U_\lambda\ar[r]&X'
		}
	\end{align}
	Note that $\sigma_x:(\fal^\et_{X^{\rz}_{(x),\eta}\to X^{\rz}_{(x)}},\falb)\to(X^{\rz}_{(x)},\ca{O}_{\rz_{X_\eta}(X),x})$ is the cofiltered limit of $\sigma_\lambda:(\fal^{\et}_{U^{\triv}_\lambda\to U_\lambda},\falb)\to (U_\lambda,\ca{O}_{U_\lambda})$ (where $U^{\triv}_\lambda=X^{\triv}_\lambda\times_{X_\lambda}U_\lambda$).
	\begin{align}
		\xymatrix{
			\fal^\et_{X^{\rz}_{(x),\eta}\to X^{\rz}_{(x)}}\ar[r]\ar[d]^-{\sigma_x}&\fal^{\et}_{U^{\triv}_\lambda\to U_\lambda}\ar[r]\ar[d]^-{\sigma_\lambda}&\fal^{\et}_{X^{\triv}\to X}\ar[d]^-{\sigma}&\\
			X^{\rz}_{(x)}\ar[r]\ar@/_1pc/[rr]_-{f_x}&U_\lambda\ar[r]^-{f_\lambda}&X&U\ar[l]
		}
	\end{align}
	By \ref{para:limit-ringed-space} and \cite[8.7.5, 8.7.7]{sga4-2} we have
	\begin{align}
		\Gamma(U,f_{x*}\rr^n\sigma_{x*}(\falb/\pi\falb))=\colim_{\lambda\in\Lambda}\Gamma(U,f_{\lambda*}\rr^n\sigma_{\lambda*}(\falb/\pi\falb)).
	\end{align}
	As $x\notin \supprz(\xi)$, there exists $\lambda\in\Lambda$ large enough such that the image of $\xi$ vanishes under the canonical morphism
	\begin{align}\label{eq:lem:rz-faithful-supp-3}
		\Gamma(U,\rr^n\sigma_*(\falb/\pi\falb))\longrightarrow  \Gamma(U,f_{\lambda*}\rr^n\sigma_{\lambda*}(\falb/\pi\falb)).
	\end{align}
	Therefore, for any point $y\in X^{\rz}$ lies in the preimage of the quasi-compact open subset $U_\lambda$ along the composition $X^{\rz}\to \rz_{X_\eta}(X)\to X'$, the image of $\xi$ vanishes under the canonical morphism
	\begin{align}
		\Gamma(U,\rr^n\sigma_*(\falb/\pi\falb))\longrightarrow  \Gamma(U,f_{y*}\rr^n\sigma_{y*}(\falb/\pi\falb))
	\end{align} 
	as it factors through \eqref{eq:lem:rz-faithful-supp-3}. In other words, we have $y\notin \supprz(\xi)$, which shows that $\supprz(\xi)$ is a closed subset.
\end{proof}

\begin{mythm}\label{thm:rz-continue}
	Let $(X^{\triv}_\lambda\to X_\lambda)_{\lambda\in\Lambda}$ be an object in $\pro(\fal^{\mrm{open}}_{\eta\to S})$. For any $x\in X^{\rz}$, let $x_\lambda$ be its image in $X^{\rz}_\lambda$ for any $\lambda\in\Lambda$.
	\begin{align}
		\xymatrix{
			\fal^\et_{X^{\rz}_{(x),\eta}\to X^{\rz}_{(x)}}\ar[r]\ar[d]&\fal^{\et}_{X^{\triv}\to X}\ar[d]&X^{\rz}_{(x)}\ar[r]^-{f_x}\ar[d]&X\ar[d]\\
			\fal^\et_{X^{\rz}_{\lambda,(x_\lambda),\eta}\to X^{\rz}_{\lambda,(x_\lambda)}}\ar[r]&\fal^{\et}_{X^{\triv}_\lambda\to X_\lambda}&X^{\rz}_{\lambda,(x_\lambda)}\ar[r]^-{f_{x_\lambda}}&X_\lambda
		}
	\end{align}
	For any nonzero element $\pi$ of $\ak{m}_L$ and any integer $n$, consider the canonical commutative diagram induced by the canonical commutative diagrams above and the canonical morphisms $\sigma$ \eqref{eq:sigma} of ringed sites,
	\begin{align}\label{eq:thm:rz-continue-1}
		\xymatrix{
			\colim_{\lambda\in\Lambda}\rr^n\sigma_{\lambda*}(\falb/\pi\falb)\ar@{=}[d]\ar[r]^-{\alpha}& \colim_{\lambda\in\Lambda}\prod_{x_\lambda\in X^{\rz}_\lambda}f_{x_\lambda*}\rr^n\sigma_{x_\lambda*}(\falb/\pi\falb)\ar[d]\\
			\rr^n\sigma_*(\falb/\pi\falb)\ar[r]^-{\beta}&\prod_{x\in X^{\rz}}f_{x*}\rr^n\sigma_{x*}(\falb/\pi\falb)
		}
	\end{align}
	where the left vertical equality follows from \cite[8.7.7]{sga4-2}. Then, for any local section $\xi$ of $\rr^n\sigma_*(\falb/\pi\falb)$, if $\beta(\xi)=0$, then $\alpha(\xi)=0$.
\end{mythm}
\begin{proof}
	Assume that $\alpha(\xi)\neq 0$ and $\xi\in \Gamma(U,\rr^n\sigma_*(\falb/\pi\falb))$ for some quasi-compact open subset $U$ of $X$. Notice that there exists $\lambda_0\in\Lambda$ and a quasi-compact open subset $U_{\lambda_0}$ of $X_{\lambda_0}$ whose base change along $X\to X_{\lambda_0}$ is $U$ (\cite[\href{https://stacks.math.columbia.edu/tag/0A30}{0A30}]{stacks-project}). For any $\lambda\in\Lambda_{\geq \lambda_0}$, let $U_\lambda$ be the base change of $U_{\lambda_0}$ along $X_\lambda\to X_{\lambda_0}$. Since $\sigma:(\fal^{\et}_{X^{\triv}\to X},\falb)\longrightarrow (X,\ca{O}_X)$ is the cofiltered limit of $\sigma_\lambda:(\fal^{\et}_{X^{\triv}_\lambda\to X_\lambda},\falb)\longrightarrow (X_\lambda,\ca{O}_{X_\lambda})$ by \eqref{eq:fal-limit} and \eqref{eq:rz-limit}, by \cite[8.7.5, 8.7.7]{sga4-2} we have 
	\begin{align}\label{eq:lem:rz-faithful-supp-2}
		\Gamma(U,\rr^n\sigma_*(\falb/\pi\falb))=\colim_{\lambda\in\Lambda}\Gamma(U_\lambda,\rr^n\sigma_{\lambda*}(\falb/\pi\falb))
	\end{align}
	After enlarging $\lambda_0$, we may assume that there exists $\xi_{\lambda_0}\in \Gamma(U_{\lambda_0},\rr^n\sigma_{\lambda_0*}(\falb/\pi\falb))$ whose image in $\Gamma(U,\rr^n\sigma_*(\falb/\pi\falb))$ is $\xi$. For any $\lambda\in\Lambda_{\geq\lambda_0}$, let $\xi_\lambda\in \Gamma(U_\lambda,\rr^n\sigma_{\lambda*}(\falb/\pi\falb))$ be the image of $\xi_{\lambda_0}$. 
	
	Note that for any $\lambda\in\Lambda_{\geq\lambda_0}$ the subset of $X^{\rz}_\lambda$, 
	\begin{align}\label{eq:thm:rz-continue-2}
		\supprz_\lambda(\xi_\lambda)=\{x_\lambda\in X^{\rz}_\lambda\ |\ \xi_\lambda\neq 0\trm{ in }\Gamma(U_{\lambda},f_{x_\lambda*}\rr^n\sigma_{x_\lambda*}(\falb/\pi\falb))\}
	\end{align}
	is a closed subset by \ref{lem:rz-faithful-supp}. Recall that for any indexes $\lambda\leq \mu$ in $\Lambda$, the canonical spectral map of spectral spaces (cf. \ref{prop:riemann-zariski-basic})
	\begin{align}\label{eq:thm:rz-continue-3}
		X_\mu^{\rz}\longrightarrow X_\lambda^{\rz}
	\end{align}
	induces a natural map
	\begin{align}
		\prod_{x_\lambda\in X^{\rz}_\lambda}\Gamma(U_{\lambda},f_{x_\lambda*}\rr^n\sigma_{x_\lambda*}(\falb/\pi\falb))\longrightarrow \prod_{x_\mu\in X^{\rz}_\mu}\Gamma(U_{\mu},f_{x_\mu*}\rr^n\sigma_{x_\mu*}(\falb/\pi\falb)).
	\end{align}
	Then, we see that $X_\mu^{\rz}\to X_\lambda^{\rz}$ \eqref{eq:thm:rz-continue-3} sends $\supprz_\mu(\xi_\mu)$ into $\supprz_\lambda(\xi_\lambda)$ for any indexes $\lambda\leq \mu$ in $\Lambda_{\geq\lambda_0}$. The assumption $\alpha(\xi)\neq 0$ implies that $\supprz_\lambda(\xi_\lambda)$ is non-empty for any $\lambda\in\Lambda_{\geq\lambda_0}$. Then,
	\begin{align}
		\lim_{\lambda\in\Lambda_{\geq\lambda_0}}\supprz_\lambda(\xi_\lambda)
	\end{align}
	is also non-empty as a cofiltered limit of non-empty spectral spaces with spectral transition maps (\cite[\href{https://stacks.math.columbia.edu/tag/0A2W}{0A2W}]{stacks-project}). We take $x\in \lim_{\lambda\in\Lambda_{\geq\lambda_0}}\supprz_\lambda(\xi_\lambda)$ and let $x_\lambda$ be its image in $\supprz_\lambda(\xi_\lambda)\subseteq X^{\rz}_\lambda$. Then, the image of $\xi$ in 
	\begin{align}
		\colim_{\lambda\in\Lambda}\Gamma(U_{\lambda},f_{x_\lambda*}\rr^n\sigma_{x_\lambda*}(\falb/\pi\falb))=\Gamma(U,f_{x*}\rr^n\sigma_{x*}(\falb/\pi\falb))
	\end{align}
	does not vanish by construction, where the equality above follows from $X^{\rz}_{(x)}=\lim_{\lambda\in\Lambda}X^{\rz}_{\lambda,(x_\lambda)}$ (\ref{prop:riemann-zariski-limit}), \cite[7.12]{he2024coh} and \cite[8.7.5, 8.7.7]{sga4-2}. This shows that $\beta(\xi)\neq 0$.
\end{proof}

\begin{mycor}\label{cor:rz-continue}
	Let $(X^{\triv}_\lambda\to X_\lambda)_{\lambda\in\Lambda}$ be an object in $\pro(\fal^{\mrm{open}}_{\eta\to S})$. If the Faltings cohomology of $X^{\triv}_\lambda\to X_\lambda$ is Riemann-Zariski faithful for any $\lambda\in\Lambda$, then so is $(X^{\triv}_\lambda\to X_\lambda)_{\lambda\in\Lambda}$.
\end{mycor}
\begin{proof}
	For any integer $n$, consider the canonical commutative diagram \eqref{eq:thm:rz-continue-1}. The map $\alpha$ is almost injective by assumption so that $\beta$ is also almost injective by \ref{thm:rz-continue}. This verifies the condition \ref{defn:rz-faithful}.(\ref{item:defn:rz-faithful-1}) for $X$. The condition \ref{defn:rz-faithful}.(\ref{item:defn:rz-faithful-2}) follows simply from taking colimits (\cite[8.7.5]{sga4-2}).
\end{proof}

\begin{mydefn}[{cf. \ref{defn:faltings-acyclic}}]\label{defn:locally-faltings-acyclic}
	Let $(Y_\lambda\to X_\lambda)_{\lambda\in\Lambda}$ be a directed inverse system of morphisms of coherent schemes such that $Y_\lambda\to X_\lambda^{Y_\lambda}$ is over $\eta\to S$, where $X_\lambda^{Y_\lambda}$ is the integral closure of $X_\lambda$ in $Y_\lambda$. We say that $(Y_\lambda\to X_\lambda)_{\lambda\in\Lambda}$ is \emph{locally Faltings acyclic} if the canonical morphism
	\begin{align}\label{eq:defn:locally-faltings-acyclic}
		\ca{O}_{X^Y}/p\ca{O}_{X^Y}\longrightarrow \rr\sigma_*(\falb/p\falb)
	\end{align}
	is an almost isomorphism (see \cite[5.7]{he2024coh}), where $X^Y=\lim_{\lambda\in\Lambda}X_\lambda^{Y_\lambda}$ is the cofiltered limit of locally ringed spaces (\ref{para:limit-ringed-space}) and $\sigma:(\fal^{\et}_{Y\to X},\falb)\to (X^Y,\ca{O}_{X^Y})$ is the cofiltered limit of the canonical morphisms of ringed sites $\sigma_\lambda:(\fal^{\et}_{Y_\lambda\to X_\lambda},\falb)\to (X_\lambda^{Y_\lambda},\ca{O}_{X_\lambda^{Y_\lambda}})$ (cf. \eqref{eq:sigma_lambda}).
\end{mydefn}
\begin{myrem}\label{rem:locally-faltings-acyclic}
	Notice that $\ca{O}_{X^Y}$ and $\falb$ are flat over $\ca{O}_L$. We can replace modulo $p$ in \eqref{eq:defn:locally-faltings-acyclic} by modulo any nonzero element $\pi\in \ak{m}_L$. More precisely, by the same d\'evissage arguments as in \cite[8.3]{he2024coh}, we can prove that the condition \eqref{eq:defn:locally-faltings-acyclic} is equivalent to that the canonical morphism
	\begin{align}\label{eq:rem:locally-faltings-acyclic}
		\ca{O}_{X^Y}/\pi\ca{O}_{X^Y}\longrightarrow \rr\sigma_*(\falb/\pi\falb)
	\end{align}
	is an almost isomorphism.
\end{myrem}

\begin{mylem}\label{lem:faltings-acyclic}
	Let $Y\to X$ be a morphism of coherent schemes such that $Y\to X^Y$ is over $\eta\to S$ and that $X$ is affine. 
	\begin{enumerate}
		\renewcommand{\labelenumi}{{\rm(\theenumi)}}
		\item If $Y\to X$ is locally Faltings acyclic in the sense of {\rm\ref{defn:locally-faltings-acyclic}} (regarded as a constant inverse system), then $Y\to X$ is Faltings acyclic in the sense of {\rm\ref{defn:faltings-acyclic}}.\label{item:lem:faltings-acyclic-1}
		\item If $Y\to X$ is is Faltings acyclic in the sense of {\rm\ref{defn:faltings-acyclic}} and if $Y\to X$ satisfies the condition $(\star\star)$ in {\rm\ref{para:notation-open-2}}, then $Y\to X$ is locally Faltings acyclic in the sense of {\rm\ref{defn:locally-faltings-acyclic}}.\label{item:lem:faltings-acyclic-2}
	\end{enumerate}
\end{mylem}
\begin{proof}
	(\ref{item:lem:faltings-acyclic-1})	It follows directly from taking $\rr\Gamma(X^Y,-)$ on the almost isomorphism \eqref{eq:defn:locally-faltings-acyclic}.
	
	(\ref{item:lem:faltings-acyclic-2}) Notice that for any integer $n$, $\rr^n\sigma_*(\falb/p\falb)$ is canonically almost isomorphic to the quasi-coherent $\ca{O}_{X^Y}$-module associated to $H^n(\fal^\et_{Y\to X},\falb/p\falb)$ by \ref{cor:acyclic}. On the other hand, $\ca{O}_{X^Y}/p\ca{O}_{X^Y}$ is the quasi-coherent $\ca{O}_{X^Y}$-module associated to $B/p B$, where $X^Y=\spec(B)$. Hence, we see that if $B/p B\to \rr\Gamma(\fal^\et_{Y\to X},\falb/p\falb)$ is an almost isomorphism, then so is $\ca{O}_{X^Y}/p\ca{O}_{X^Y}\to \rr\sigma_*(\falb/p\falb)$.
\end{proof}

\begin{mythm}[Valuative criterion for local Faltings acyclicity]\label{thm:val-criterion-faltings-acyclic}
	Let $(X^{\triv}_\lambda\to X_\lambda)_{\lambda\in\Lambda}$ be an object in $\pro(\fal^{\mrm{open}}_{\eta\to S})$ whose Faltings cohomology is Riemann-Zariski faithful {\rm(\ref{defn:rz-faithful})}. Assume that the following conditions hold:
	\begin{enumerate}
		\renewcommand{\labelenumi}{{\rm(\theenumi)}}
		\item For any $\lambda\in\Lambda$, $X_\lambda$ is integrally closed in $X^{\triv}_\lambda$.\label{item:thm:val-criterion-faltings-acyclic-1}
		\item For any $x\in X^{\rz}$ {\rm(\ref{para:limit-rz})}, the stalk $\ca{O}_{\rz_{X_\eta}(X),x}$ is pre-perfectoid.\label{item:thm:val-criterion-faltings-acyclic-2}
	\end{enumerate}
	Then, $(X^{\triv}_\lambda\to X_\lambda)_{\lambda\in\Lambda}$ is locally Faltings acyclic {\rm(\ref{defn:locally-faltings-acyclic})}.
\end{mythm}
\begin{proof}
	Since $\ca{O}_{\rz_{X_\eta}(X),x}$ is pre-perfectoid, $X^{\rz}_{(x),\eta}\to X^{\rz}_{(x)}$ is Faltings acyclic by \ref{thm:acyclic}. Thus, it is locally Faltings acyclic by \ref{lem:faltings-acyclic}.(\ref{item:lem:faltings-acyclic-2}). Then, for any nonzero element $\pi\in\ak{m}_L$ and any integer $n>0$, $\rr^n\sigma_{x*}(\falb/p\falb)$ is almost zero, and thus so is $\rr^n\sigma_*(\falb/p\falb)$ by \ref{defn:rz-faithful}.(\ref{item:defn:rz-faithful-1}). Then, $\ca{O}_X/p \ca{O}_X\to \rr\sigma_*(\falb/p\falb)$ is an almost isomorphism by \ref{defn:rz-faithful}.(\ref{item:defn:rz-faithful-2}). Notice that $X_\lambda=X_\lambda^{Y_\lambda}$ by assumption so that $X=X^Y$ by taking cofiltered limits. Thus, $(X^{\triv}_\lambda\to X_\lambda)_{\lambda\in\Lambda}$ is locally Faltings acyclic by definition \ref{defn:locally-faltings-acyclic}.
\end{proof}

\begin{myrem}\label{rem:val-criterion-faltings-acyclic}
	In \ref{thm:val-criterion-faltings-acyclic}, the condition (\ref{item:thm:val-criterion-faltings-acyclic-2}) is equivalent to the following condition: for any point $y\in X^{\triv}$, its residue field $\kappa(y)$ is a pre-perfectoid field with respect to any valuation ring $V$ of height $1$ extension of $\ca{O}_L$ with fraction field $V[1/p]=\kappa(y)$ and $\spec(V)\to\spec(\ca{O}_L)$ factors through $X$. Indeed, this follows directly from the description of the stalks (\ref{prop:limit-rz}.(\ref{item:prop:limit-rz-5}) and \ref{lem:val-micro}).
\end{myrem}

\begin{myprop}[{\cite[6.5.13.(\luoma{2})]{gabber2003almost}}]\label{prop:perfd-diff}
	Let $B$ be a (resp. almost) flat $\ca{O}_L$-algebra that is (resp. almost) pre-perfectoid. Then, the canonical morphism of cotangent complexes
	\begin{align}
		\bb{L}_{B/\ca{O}_L}\longrightarrow \bb{L}_{B[1/p]/L}
	\end{align}
	is an (resp. almost) isomorphism in the derived category of simplicial $B$-modules.
\end{myprop}
\begin{proof}
	Recall that the canonical morphism $B\to (B^\al)_*=\ho_{\ca{O}_L}(\ak{m}_L,B)$ is an almost isomorphism (\cite[5.5]{he2024coh}). Moreover, $(B^\al)_*$ is flat over $\ca{O}_L$ (\cite[5.12]{he2024coh}). After replacing $B$ by $(B^\al)_*$, we may assume that $B=(B^\al)_*$. Thus, the Frobenius induces an isomorphism $B/p_1B\to B/pB$ by \cite[5.22]{he2024coh}, where $p_1$ is a $p$-th root of $p$ up to a unit (\cite[5.4]{he2024coh}). In particular, the assumptions of \cite[6.5.13.(\luoma{2})]{gabber2003almost} on the flat homomorphism $\ca{O}_L\to B$ are satisfied so that the conclusion follows immediately.
\end{proof}

\begin{myprop}\label{prop:val-diff}
	Let $\ca{O}_F$ be a valuation ring of height $1$ extension of $\ca{O}_L$. Then, the canonical morphism of modules of differentials 
	\begin{align}
		\Omega^1_{\ca{O}_F/\ca{O}_L}\longrightarrow \Omega^1_{F/L}
	\end{align}
	is injective. Moreover, it is almost surjective if and only if $F$ is a pre-perfectoid field, and in this case we have $\Omega^1_{\ca{O}_F/\ca{O}_L}= \Omega^1_{F/L}$.
\end{myprop}
\begin{proof}
	Assume firstly that $F$ is a pre-perfectoid field. Then, we can apply \cite[6.5.13.(\luoma{2})]{gabber2003almost} to the flat homomorphism $\ca{O}_L\to \ca{O}_F$, since the Frobenius endomorphisms on $\ca{O}_L/p\ca{O}_L$ and $\ca{O}_F/p\ca{O}_F$ are surjective by definition (\ref{para:notation-perfd}) and there exists a $p$-th root $p_1$ of $p$ up to a unit (\cite[5.4]{he2024coh}). Therefore, the canonical morphism $\bb{L}_{\ca{O}_F/\ca{O}_L}\to \bb{L}_{\ca{O}_F[1/p]/\ca{O}_L[1/p]}=\bb{L}_{F/L}$ is an isomorphism. Taking the $0$-th homology groups, we get $\Omega^1_{\ca{O}_F/\ca{O}_L}=\Omega^1_{F/L}$. 
	
	In general, let $\overline{F}$ be an algebraic closure of $F$. Then, $\overline{F}$ is a pre-perfectoid field. As $\overline{F}$ is separable over $F$, $H_n(\bb{L}_{\ca{O}_{\overline{F}}/\ca{O}_F})=0$ for any integer $n>0$ by \cite[6.3.32]{gabber2003almost}. Thus, there is a canonical exact sequence
	\begin{align}
		0=H_1(\bb{L}_{\ca{O}_{\overline{F}}/\ca{O}_F})\longrightarrow \ca{O}_{\overline{F}}\otimes_{\ca{O}_F}\Omega^1_{\ca{O}_F/\ca{O}_L} \longrightarrow \Omega^1_{\ca{O}_{\overline{F}}/\ca{O}_L}\longrightarrow \Omega^1_{\ca{O}_{\overline{F}}/\ca{O}_F}\longrightarrow 0.
	\end{align}
	In particular, $\Omega^1_{\ca{O}_F/\ca{O}_L}$ is $p$-torsion-free as a submodule of $\Omega^1_{\ca{O}_{\overline{F}}/\ca{O}_L}=\Omega^1_{\overline{F}/L}$, where the equality follows from the discussion above. This proves the first statement.
	
	It remains to check that if $\Omega^1_{\ca{O}_F/\ca{O}_L}\to \Omega^1_{F/L}$ is almost surjective then $F$ is a pre-perfectoid field. Indeed, we see that $\Omega^1_{\ca{O}_F/\ca{O}_L}\to \Omega^1_{F/L}$ is an almost isomorphism by the discussion above. Consider the morphism of exact sequences given by multiplication by $p$,
	\begin{align}
		\xymatrix{
			\ca{O}_F\otimes_{\ca{O}_L}\Omega^1_{\ca{O}_L/\bb{Z}}\ar[d]^-{\cdot p}\ar[r]&\Omega^1_{\ca{O}_F/\bb{Z}}\ar[r]\ar[d]^-{\cdot p}&\Omega^1_{\ca{O}_F/\ca{O}_L}\ar[r]\ar[d]^-{\cdot p}&0\\
			\ca{O}_F\otimes_{\ca{O}_L}\Omega^1_{\ca{O}_L/\bb{Z}}\ar[r]&\Omega^1_{\ca{O}_F/\bb{Z}}\ar[r]&\Omega^1_{\ca{O}_F/\ca{O}_L}\ar[r]&0
		}
	\end{align}
	where the left vertical arrow is surjective as $L$ is a pre-perfectoid field (\cite[6.6.6]{gabber2003almost}), the right vertical arrow is an almost isomorphism by assumption. Thus, the middle vertical arrow is almost surjective. In other words, $\Omega^1_{\ca{O}_F/\bb{Z}}$ is almost $p$-divisible so that $F$ is a pre-perfectoid field (\cite[6.6.6]{gabber2003almost}).
\end{proof}

\begin{mycor}\label{cor:val-diff}
	Let $B$ be an $\ca{O}_L$-algebra. Assume that $\Omega^1_{B/\ca{O}_L}\to \Omega^1_{B/\ca{O}_L}[1/p]$ is almost surjective (e.g., when $p\cdot \Omega^1_{B/\ca{O}_L}\to \Omega^1_{B/\ca{O}_L}$ is almost surjective). Then, for any valuation ring $V$ of height $1$ extension of $\ca{O}_L$ with an $\ca{O}_L$-homomorphism $B\to V$ such that $V[1/p]$ is a localization of $B[1/p]$, the valuation field $V[1/p]$ is a pre-perfectoid field.
\end{mycor}
\begin{proof}
	Consider the canonical morphism of exact sequences
	\begin{align}
		\xymatrix{
			V\otimes_B\Omega^1_{B/\ca{O}_L}\ar[d]\ar[r]&\Omega^1_{V/\ca{O}_L}\ar[r]\ar[d]&\Omega^1_{V/B}\ar[r]\ar[d]&0\\
			V\otimes_B\Omega^1_{B/\ca{O}_L}[1/p]\ar[r]&\Omega^1_{V/\ca{O}_L}[1/p]\ar[r]&\Omega^1_{V/B}[1/p]\ar[r]&0
		}
	\end{align}
	where the left vertical arrow is almost surjective by assumption. Notice that $\Omega^1_{V/B}[1/p]=\Omega^1_{V[1/p]/B[1/p]}=0$, since $V[1/p]$ is a localization of $B[1/p]$. Hence, the middle vertical arrow $\Omega^1_{V/\ca{O}_L}\to\Omega^1_{V/\ca{O}_L}[1/p]$ is also almost surjective. Then, the conclusion follows from \ref{prop:val-diff}.
\end{proof}

\begin{mylem}[{cf. \cite[7.4.13]{gabber2003almost}}]\label{lem:diff-criterion-faltings-acyclic}
	Let $(X^{\triv}_\lambda\to X_\lambda)_{\lambda\in\Lambda}$ be an object in $\pro(\fal^{\mrm{open}}_{\eta\to S})$. Assume that $\Omega^1_{X/S}\to \Omega^1_{X/S}[1/p]$ is almost surjective, where $\Omega^1_{X/S}$ (resp. $\Omega^1_{X/S}[1/p]$) is the filtered colimit over $\lambda\in\Lambda$ of the pullback of $\Omega^1_{X_\lambda/S}$ (resp. $\Omega^1_{X_\lambda/S}[1/p]$) to $X$. Then, for any $x\in X^{\rz}$, the stalk $\ca{O}_{\rz_{X_\eta}(X),x}$ is pre-perfectoid.
\end{mylem}
\begin{proof}
	It suffices to show that for any point $y\in X^{\triv}$ and any valuation ring $V$ of height $1$ extension of $\ca{O}_L$ with fraction field $V[1/p]=\kappa(y)$ such that $\spec(V)\to \spec(\ca{O}_L)$ factors through $X$, the residue field $\kappa(y)$ is a pre-perfectoid field with respect to $V$ (see \ref{rem:criterion}). For any $\lambda\in\Lambda$, let $y_\lambda$ be the image of $y$ in $X^{\triv}_\lambda$, $z_\lambda$ the image of the closed point $z$ of $\spec(V)$ in $X_\lambda$, $\ak{p}_\lambda$ the prime ideal of the local ring $\ca{O}_{X_\lambda,z_\lambda}$ corresponding to $y_\lambda$. Then, we see that $B=\colim_{\lambda\in\Lambda}\ca{O}_{X_\lambda,z_\lambda}/\ak{p}_\lambda$ is a domain with fraction field $ \colim_{\lambda\in\Lambda}\kappa(y_\lambda)=\kappa(y)=V[1/p]$.
	
	We claim that  $\Omega^1_{B/\ca{O}_L}\to\Omega^1_{B/\ca{O}_L}[1/p]$ is almost surjective. Indeed, for any $\lambda\in\Lambda$, let $g_\lambda:\spec(\ca{O}_{X_\lambda,z_\lambda}/\ak{p}_\lambda)\to X_\lambda$ be the canonical morphism of schemes. Then, the canonical morphism of $\ca{O}_{X_\lambda,z_\lambda}/\ak{p}_\lambda$-modules $g_\lambda^*\Omega^1_{X_\lambda/S}\to \Omega^1_{(\ca{O}_{X_\lambda,z_\lambda}/\ak{p}_\lambda)/\ca{O}_L}$ is surjective. Let $g:\spec(B)\to X$ be the cofiltered limit of $(g_\lambda)_{\lambda\in\Lambda}$ as morphisms of locally ringed spaces. 
	\begin{align}
		\xymatrix{
			\spec(V)\ar[r]&\spec(B)\ar[r]^-{g}\ar[d]&X\ar[d]&\\
			&\spec(\ca{O}_{X_\lambda,z_\lambda}/\ak{p}_\lambda)\ar[r]^-{g_\lambda}& X_\lambda\ar[r]&S=\spec(\ca{O}_L)
		}
	\end{align}
	By taking filtered colimits, we see that $g^*\Omega^1_{X/S}\to \Omega^1_{B/\ca{O}_L}$ is also surjective. Similarly, $g^*\Omega^1_{X/S}[1/p]\to \Omega^1_{B/\ca{O}_L}[1/p]$ is surjective. As $\Omega^1_{X/S}\to \Omega^1_{X/S}[1/p]$ is almost surjective by assumption, so is $g^*\Omega^1_{X/S}\to g^*\Omega^1_{X/S}[1/p]$ and the claim follows.
	
	As $V[1/p]$ is the fraction field of $B$ by construction, we can apply \ref{cor:val-diff} to the $\ca{O}_L$-homomorphism $B\to V$ by the claim above. Thus, $\kappa(y)$ is a pre-perfectoid field with respect to $V$, which completes the proof.
\end{proof}

\begin{mythm}[Differential criterion for local Faltings acyclicity]\label{thm:diff-criterion-faltings-acyclic}
	Let $(X^{\triv}_\lambda\to X_\lambda)_{\lambda\in\Lambda}$ be an object in $\pro(\fal^{\mrm{open}}_{\eta\to S})$ whose Faltings cohomology is Riemann-Zariski faithful {\rm(\ref{defn:rz-faithful})}. Assume that the following conditions hold: 
	\begin{enumerate}
		\renewcommand{\labelenumi}{{\rm(\theenumi)}}
		\item For any $\lambda\in\Lambda$, $X_\lambda$ is integrally closed in $X^{\triv}_\lambda$.
		\item The morphism $\Omega^1_{X/S}\to \Omega^1_{X/S}[1/p]$ is almost surjective, where $\Omega^1_{X/S}$ (resp. $\Omega^1_{X/S}[1/p]$) is the filtered colimit over $\lambda\in\Lambda$ of the pullback of $\Omega^1_{X_\lambda/S}$ (resp. $\Omega^1_{X_\lambda/S}[1/p]$) to $X$.
	\end{enumerate}
	Then, $(X^{\triv}_\lambda\to X_\lambda)_{\lambda\in\Lambda}$ is locally Faltings acyclic {\rm(\ref{defn:locally-faltings-acyclic})}.
\end{mythm}
\begin{proof}
	It follows directly from \ref{thm:val-criterion-faltings-acyclic} and \ref{lem:diff-criterion-faltings-acyclic}.
\end{proof}

\begin{mypara}\label{para:setup-criteria}
	Let $X^{\triv}\to X$ be a pro-open morphism of coherent schemes over $\eta\to S$ (\ref{defn:pro-open}). It defines an isomorphism class of objects in $\pro(\fal^{\mrm{open}}_{\eta\to S})$ given by $(U_\lambda\to X)_{\lambda\in\Lambda}$, where $(U_\lambda)_{\lambda\in\Lambda}$ is a directed inverse system of quasi-compact open subsets of $X_\eta$ with affine transition inclusion maps such that $X^{\triv}=\lim_{\lambda\in\Lambda}U_\lambda=\bigcap_{\lambda\in\Lambda}U_\lambda\subseteq X$ (see \ref{defn:pro-open}). Note that the canonical morphism of Faltings ringed sites 
	\begin{align}\label{eq:para:setup-criteria-2}
		\sigma:(\fal^{\et}_{X^{\triv}\to X},\falb)\longrightarrow (X,\ca{O}_X)
	\end{align}
	defined by the left exact continuous functor $\sigma^+:X\to \fal^{\et}_{X^{\triv}\to X}$ sending each quasi-compact open subset $U$ of $X$ to $U^{\triv}=X^{\triv}\times_{X}U\to U$ (cf. \cite[(7.8.5)]{he2024coh}), is the cofiltered limit of $(\sigma_\lambda:(\fal^{\et}_{U_\lambda\to X},\falb)\longrightarrow (X,\ca{O}_X))$ \eqref{eq:sigma_lambda} by \cite[7.12]{he2024coh}. Therefore, the condition that the Faltings cohomology of $(U_\lambda\to X)_{\lambda\in\Lambda}$ is Riemann-Zariski faithful (\ref{defn:rz-faithful}) if and only if the same conditions in \ref{defn:rz-faithful} hold for \eqref{eq:para:setup-criteria-2}. In this case, we also say that \emph{the Faltings cohomology of $X^{\triv}\to X$ is Riemann-Zariski faithful}.
\end{mypara}

\begin{mythm}\label{thm:criterion}
	Let $\spec(B_{\triv})\to \spec(B)$ be a pro-open morphism of affine schemes over $\eta\to S$ whose Faltings cohomology is Riemann-Zariski faithful {\rm(\ref{para:setup-criteria})}, $X=\spec(B)$. Assume that $B$ is integrally closed in $B_{\triv}$. Then, the following conditions are equivalent: 
	\begin{enumerate}
		\renewcommand{\labelenumi}{{\rm(\theenumi)}}
		\item The $\ca{O}_L$-algebra $B$ is pre-perfectoid.\label{item:thm:criterion-1}
		\item The canonical morphism of modules of differentials $\Omega^1_{B/\ca{O}_L}\to \Omega^1_{B/\ca{O}_L}[1/p]$ is surjective.\label{item:thm:criterion-2}
		\item For any $x\in X^{\rz}$ {\rm(\ref{para:limit-rz})}, the stalk $\ca{O}_{\rz_{X_\eta}(X),x}$ is pre-perfectoid.\label{item:thm:criterion-3}
	\end{enumerate}
\end{mythm}
\begin{proof}
	(\ref{item:thm:criterion-1})$\Rightarrow$(\ref{item:thm:criterion-2}) is a special case of \ref{prop:perfd-diff}. (\ref{item:thm:criterion-2})$\Rightarrow$(\ref{item:thm:criterion-3}) is a special case of \ref{lem:diff-criterion-faltings-acyclic} (see \ref{para:setup-criteria}). (\ref{item:thm:criterion-3})$\Rightarrow$(\ref{item:thm:criterion-1}) follows from \ref{thm:val-criterion-faltings-acyclic}, \ref{lem:faltings-acyclic}.(\ref{item:lem:faltings-acyclic-1}) and \ref{lem:fal-acyc-perfd}.
\end{proof}

\begin{myrem}\label{rem:criterion}
	In \ref{thm:criterion}, the condition (\ref{item:thm:criterion-3}) is equivalent to the following statement: for any point $y\in X^{\triv}=\spec(B_\triv)$, its residue field $\kappa(y)$ is a pre-perfectoid field with respect to any valuation ring $V$ of height $1$ extension of $\ca{O}_L$ with fraction field $V[1/p]=\kappa(y)$ and $\ca{O}_L\to V$ factors through $B$. Indeed, this follows directly from the description of the stalks (\ref{prop:limit-rz}.(\ref{item:prop:limit-rz-5}) and \ref{lem:val-micro}).
\end{myrem}

\begin{myexample}[Pro-semi-stable]\label{exam:pro-semi-stable}
	We say that a pro-open morphism $X^{\triv}\to X$ of coherent schemes over $\eta\to S$ is \emph{pro-semi-stable} if there exists a directed inverse system $(X^{\triv}_\lambda\to X_\lambda)_{\lambda\in\Lambda}$ of open immersions of coherent schemes semi-stable over $\eta\to S$ in the sense of \ref{exam:semi-stable} with affine transition morphisms such that $(X^{\triv}\to X)=\lim_{\lambda\in\Lambda}(X^{\triv}_\lambda\to X_\lambda)_{\lambda\in\Lambda}$. 
	
	We remark that if there exists a complete discrete valuation field $K$ extension of $\bb{Q}_p$ with perfect residue field and a local, injective and integral homomorphism $\ca{O}_{K(\zeta_{\infty})}\to \ca{O}_L$ (where $K(\zeta_\infty)$ is the extension of $K$ in an algebraic closure by adjoining all roots of unity), then the Faltings cohomology of a pro-open morphism of coherent schemes $X^{\triv}\to X$ pro-semi-stable over $\eta\to S$ is Riemann-Zariski faithful by \ref{exam:semi-stable}, \ref{prop:rz-faithful-adequate} and \ref{cor:rz-continue}.
\end{myexample}

\section{Preliminaries on Torsion-Free Modules over Maximal Valuation Rings}\label{sec:torsion-free}

We take a brief review on Kaplansky's structure theorem on countably generated torsion-free modules over maximal valuation rings (see \ref{thm:max-val-mod} and \ref{cor:max-val-mod}), which will be used in the next section (see \ref{rem:cor:vanishing-2}).

\begin{mydefn}[{\cite[page 191]{krull1932val}, see \cite[page 305]{kaplansky1942max}}]\label{defn:max-val}
	Let $L$ be a valuation field. An \emph{immediate extension} of $L$ is an extension of valuation fields $L\to F$ (\ref{defn:val-ext}) such that
	\begin{enumerate}
		\renewcommand{\labelenumi}{{\rm(\theenumi)}}
		\item it induces an isomorphism of value groups, i.e., $L^\times/\ca{O}_L^\times\iso F^\times/\ca{O}_F^\times$, and that
		\item it induces an isomorphism of residue fields, i.e., $\ca{O}_L/\ak{m}_L\iso \ca{O}_F/\ak{m}_F$.
	\end{enumerate}
	We say that $L$ is \emph{maximal} if any immediate extension $L\to F$ is trivial (i.e., $L=F$). We say that an immediate extension $L\to F$ is \emph{maximal} if $F$ is maximal.
\end{mydefn}

We note that the completion of a valuation field is an immediate extension (\cite[\Luoma{6}.\textsection5.3, Proposition 5]{bourbaki2006commalg5-7}). In particular, any maximal valuation field is complete.

\begin{myrem}\label{rem:max-val}
	Let $L$ be a valuation field.
	\begin{enumerate}
		\renewcommand{\labelenumi}{{\rm(\theenumi)}}
		\item There is a necessary and sufficient condition for $L$ to be maximal, due to Kaplansky \cite[Theorem 4]{kaplansky1942max} (see \cite[page 336, footnote 7]{kaplansky1952mod}): for any set $J$, given any element $a_j\in L$ and any fractional ideal $I_j$ of $\ca{O}_L$ for any $j\in J$, if $(a_{j_1}+I_{j_1})\cap (a_{j_2}+I_{j_2})$ is non-empty for any $j_1,j_2\in J$, then $\cap_{j\in J}(a_j+I_j)$ is non-empty.
		\item Assume that $L$ is of height $1$. Then, $L$ is maximal if and only if $L$ is \emph{spherically complete} with respect to a norm defining its topology, i.e., every shrinking sequence of balls $B_{r_1}(a_1)\supseteq B_{r_2}(a_2)\supseteq\cdots$ in $L$ has non-empty intersection (see \cite[Theorem 4.27]{rooij1978analysis}).
		\item Any complete discrete valuation field is maximal. The completion $\widehat{\overline{\bb{Q}_p}}$ of an algebraic closure of $\bb{Q}_p$ is not maximal, see \cite[Chapter 3, \textsection3.4]{robert2000analysis}.
	\end{enumerate}
\end{myrem}

The existence of a maximal immediate extension is proved by Krull, by bounding the cardinalities of immediate extensions and using Zorn's lemma.

\begin{mythm}[{\cite[Satz 24]{krull1932val}}]\label{thm:max-val-exist}
	Any valuation field admits a maximal immediate extension.
\end{mythm}

\begin{myrem}\label{rem:max-val-exist}
	In general, maximal immediate extensions of a valuation field are not unique up to isomorphisms, see \cite[\textsection5]{kaplansky1942max}. However, the uniqueness holds in certain cases by \cite[Theorem 5]{kaplansky1942max}. In particular, if the residue field of $L$ is algebraically closed of characteristic $p>0$ and if the value group $\Gamma=L^\times/\ca{O}_L^\times$ is $p$-divisible (i.e., $\Gamma=p\Gamma$), then the maximal immediate extension of $L$ is unique up to isomorphisms of valuation field extensions of $L$, and we denote it by $\widetilde{L}$, called the \emph{maximal completion} of $L$.
\end{myrem}

\begin{mycor}\label{cor:max-val-alg-clos}
	A maximal valuation field is algebraic closed if and only if its residue field is algebraically closed and its value group is divisible.
\end{mycor}
\begin{proof}
	Let $L$ be a maximal valuation field, $l=\ca{O}_L/\ak{m}_L$, $\Gamma=L^\times/\ca{O}_L^\times$. If $L$ is algebraically closed, then it is clear that $l$ is algebraically closed and $\Gamma$ is divisible. Conversely, let $\overline{L}$ be an algebraic closure of $L$. We endow $\overline{L}$ with a valuation extending that of $L$ (\cite[\Luoma{6}.\textsection8.6, Proposition 6]{bourbaki2006commalg5-7}), and we put $\overline{l}=\ca{O}_{\overline{L}}/\ak{m}_{\overline{L}}$ and $\overline{\Gamma}=\overline{L}^\times/\ca{O}_{\overline{L}}^\times$. Notice that $\overline{l}$ is algebraic over $l$ and $\overline{\Gamma}/\Gamma$ is torsion (\cite[\Luoma{6}.\textsection8.1, Proposition 1]{bourbaki2006commalg5-7}). As $l$ is assumed to be algebraically closed, we have $l=\overline{l}$. As $\Gamma$ is assumed to be divisible (so that injective as a $\bb{Z}$-module), we have $\overline{\Gamma}\cong\Gamma\oplus   \overline{\Gamma}/\Gamma$. But $\overline{\Gamma}$ is a totally ordered abelian group so that torsion-free. Thus, we see that $\Gamma=\overline{\Gamma}$. In particular, $\overline{L}$ is an immediate extension of $L$. The maximality of $L$ implies that $L=\overline{L}$, i.e., $L$ is algebraically closed.
\end{proof}

\begin{mythm}[{\cite[Theorem 12]{kaplansky1952mod}}]\label{thm:max-val-mod}
	Let $L$ be a maximal valuation field, $M$ a countably generated torsion-free $\ca{O}_L$-module. Then, $M=\oplus_{i\in I}M_i$, where $I$ is a countable set and $M_i$ is a torsion-free $\ca{O}_L$-module with rank $\dim_L (L\otimes_{\ca{O}_L}M_i)=1$ for any $i\in I$.
\end{mythm}

\begin{mycor}\label{cor:max-val-mod}
	Let $L$ be a maximal valuation field of height $1$, $M$ a countably generated torsion-free $\ca{O}_L$-module. Then, there exist countable sets $I$ and $J$ such that for any element $\epsilon\in \ak{m}_L$ there exists an $\ca{O}_L$-submodule of $M$ isomorphic to $L^{\oplus I}\oplus \ca{O}_L^{\oplus J}$ containing $\epsilon M$.
\end{mycor}
\begin{proof}
	After replacing $M$ by its direct summands by \ref{thm:max-val-mod}, we may assume that $M$ is a non-zero $\ca{O}_L$-submodule of $L$. Let $v_L:L^\times\to\bb{R}$ be a valuation map. If $\inf_{x\in M}v_L(x)=-\infty$, then $M=\bigcup_{x\in M}\ca{O}_L\cdot x\supseteq L$ and thus $M=L$. If $\inf_{x\in M}v_L(x)>-\infty$, then for any $\epsilon\in \ak{m}_L$ there exists $x_\epsilon\in M$ such that $\epsilon M\subseteq \ca{O}_L\cdot x_\epsilon\subseteq M$.
\end{proof}

\begin{myrem}\label{rem:max-val-mod}
	In \ref{cor:max-val-mod}, if moreover the value group of $L$ is isomorphic to $\bb{R}$, then there exists an almost isomorphism $M\to L^{\oplus I}\oplus \ca{O}_L^{\oplus J}$ by a similar argument.
	
	On the other hand, for any valuation field $L$ of height $1$, there exists an extension of valuation fields $L\to F$ of height $1$ such that $F$ is maximal whose value group is isomorphic to $\bb{R}$ by \cite[Proposition 1]{diarra1984ultra} (see \cite[7.1, 7.3]{escassut1995analysis} and see also \cite[Chapter 3, \textsection2.2]{robert2000analysis} for the special case $L=\bb{Q}_p$).
\end{myrem}

\begin{mylem}\label{lem:torsion-free-bc}
	Let $L$ be a maximal valuation field of height $1$, $(M_n)_{n\in\bb{N}}$ a directed inverse system of $\ca{O}_L$-modules, $N$ a torsion-free $\ca{O}_L$-module. If the torsion submodule of $\lim_{n\in\bb{N}}M_n$ is killed by an element $\pi\in\ca{O}_L$, then the torsion submodule of $\lim_{n\in\bb{N}}(M_n\otimes_{\ca{O}_L}N)$ is killed by $\pi\ak{m}_L^2$.
\end{mylem}
\begin{proof}
	Firstly, assume that $N$ is countably generated. Then, for any $\epsilon\in\ak{m}_L$, we write $N\subseteq N'\subseteq \epsilon^{-1}N$ where $N'=L^{\oplus I}\oplus \ca{O}_L^{\oplus J}$ for some countable sets $I$ and $J$ by \ref{cor:max-val-mod}. In particular, $N\to N'$ is an $\epsilon$-isomorphism, i.e., its kernel and cokernel are killed by $\epsilon$. Hence, $\lim_{n\in\bb{N}}(M_n\otimes_{\ca{O}_L}N)\to \lim_{n\in\bb{N}}(M_n\otimes_{\ca{O}_L}N')$ is an $\epsilon^2$-isomorphism by \cite[7.3.(2)]{he2022sen}. Notice that $\lim_{n\in\bb{N}}(M_n\otimes_{\ca{O}_L}N')\subseteq \lim_{n\in\bb{N}}(\prod_I (M_n\otimes_{\ca{O}_L}L)\times\prod_JM_n)= \prod_I\lim_{n\in\bb{N}}(M_n\otimes_{\ca{O}_L}L)\times \prod_J\lim_{n\in\bb{N}}M_n$, whose torsion is killed by $\pi$ since $\lim_{n\in\bb{N}}M_n$ is so and $M_n\otimes_{\ca{O}_L}L$ is torsion-free. Therefore, the torsion elements of $\lim_{n\in\bb{N}}(M_n\otimes_{\ca{O}_L}N)$ is killed by $\pi\epsilon^2$, which verifies the statement in this case.
	
	In general, consider the directed inverse system $\scr{C}$ of countably generated submodules of $N$. Since any countable internal sum of countably generated submodules of $N$ is still a countably generated submodule, we see that any countable diagram of $\scr{C}$ has a cocone, i.e., $\scr{C}$ is $\aleph_1$-filtered (cf. \cite[5.3.1.7]{lurie2009topos}). One checks easily that $\aleph_1$-filtered colimits commutes with countable limits (cf. \cite[5.3.3.3]{lurie2009topos}). In particular, we have $\lim_{n\in\bb{N}}(M_n\otimes_{\ca{O}_L}N)=\lim_{n\in\bb{N}}\colim_{N'\in \scr{C}}(M_n\otimes_{\ca{O}_L}N')=\colim_{N'\in \scr{C}}\lim_{n\in\bb{N}}(M_n\otimes_{\ca{O}_L}N')$, whose torsion elements are killed by $\pi\ak{m}_L^2$ by the previous discussion.
\end{proof}

\section{Poly-stable Modification Conjecture and Completed Cohomology}\label{sec:polystable}
Assuming the poly-stable modification conjecture or working only with curves, we investigate the relation between the \'etale cohomology and coherent cohomology groups of limits of smooth varieties in the top degree (see \ref{thm:completed} and \ref{cor:completed}).

\begin{mypara}\label{para:poly-stable}
	Let $L$ be a valuation field of height $1$ extension of $\bb{Q}_p$, $\eta=\spec(L)$, $S=\spec(\ca{O}_L)$, $X^{\triv}\to X$ an open immersion of coherent schemes over $\eta\to S$. We say that $X^{\triv}\to X$ is \emph{poly-stable} if for any geometric point $\overline{x}$ of $X$, there exists an \'etale neighborhood $U\to X$ of $\overline{x}$ and finitely many open immersions of coherent schemes $(U^{\triv}_i\to U_i)_{1\leq i\leq n}$ semi-stable over $\eta\to S$ in the sense of \ref{exam:semi-stable} with an \'etale morphism over $S$,
	\begin{align}\label{eq:para:poly-stable}
		U\longrightarrow U_1\times_S\cdots\times_SU_n
	\end{align}
	such that $U^{\triv}=X^{\triv}\times_XU$ is the base change of $U_1^{\triv}\times_\eta\cdots\times_\eta U_n^{\triv}$ along \eqref{eq:para:poly-stable} (see \cite[5.2.15]{temkin2019logval}). 
\end{mypara}

\begin{mylem}\label{lem:poly-stable-basic}
	Let $F/L$ be an extension of valuation fields of height $1$ extensions of $\bb{Q}_p$, $X^{\triv}\to X$ an open immersion of coherent schemes over $\spec(L)\to \spec(\ca{O}_L)$, $X^{\triv}_F\to X_{\ca{O}_F}$ its base change along $\ca{O}_L\to\ca{O}_F$.
	\begin{enumerate}
		\renewcommand{\labelenumi}{{\rm(\theenumi)}}
		\item If $X^{\triv}\to X$ poly-stable over $\spec(L)\to \spec(\ca{O}_L)$, then $X^{\triv}_F\to X_{\ca{O}_F}$ is poly-stable over $\spec(F)\to \spec(\ca{O}_F)$.\label{item:lem:poly-stable-basic-1}
		\item Assume that $\ca{O}_F$ is an algebraic extension of $\ca{O}_L$ {\rm(\ref{defn:val-ext})}. If $X^{\triv}_F\to X_{\ca{O}_F}$ is poly-stable over $\spec(F)\to \spec(\ca{O}_F)$, then there exists a finite subextension $L'$ of $L$ in $F$ such that $X^{\triv}_{L'}\to X_{\ca{O}_{L'}}$ is poly-stable over $\spec(L')\to \spec(\ca{O}_{L'})$.\label{item:lem:poly-stable-basic-2}
	\end{enumerate}
\end{mylem}
\begin{proof}
	(\ref{item:lem:poly-stable-basic-1}) It follows from the fact that semi-stability is stable under base change by the explicit formula \eqref{eq:exam:semi-stable-1}.
	
	(\ref{item:lem:poly-stable-basic-2}) The poly-stability of $X_{\ca{O}_F}$ implies that it is locally of finite presentation over $\ca{O}_F$, and thus it is actually of finite presentation since it is coherent. Then, there exists an \'etale covering by finitely presented $\ca{O}_F$-schemes $\ak{U}=\{U_i\to X_{\ca{O}_F}\}_{i\in I}$ with $I$ finite and for each $i\in I$ an \'etale morphism of finitely presented $\ca{O}_F$-schemes $f_i:U_i\to U_{i,1}\times_{\ca{O}_F}\cdots\times_{\ca{O}_F} U_{i,n_i}$ such that each $U_{i,j}^{\triv}\to U_{i,j}$ over $\spec(F)\to\spec(\ca{O}_F)$ is of the form in \eqref{eq:exam:semi-stable-1}. By \cite[8.8.2, 8.10.5]{ega4-3} and \cite[17.7.8]{ega4-4}, there exists a finite subextension $L'$ of $L$ in $F$ and an \'etale covering by finitely presented $\ca{O}_{L'}$-schemes $\ak{V}=\{V_i\to X_{\ca{O}_{L'}}\}_{i\in I}$ and for each $i\in I$ an \'etale morphism of finitely presented $\ca{O}_{L'}$-schemes $g_i:V_i\to V_{i,1}\times_{\ca{O}_{L'}}\cdots\times_{\ca{O}_{L'}} V_{i,n_i}$ such that $\ak{U}$ and $f_i$ are the base change of $\ak{V}$ and $g_i$ along $\ca{O}_{L'}\to\ca{O}_F$ and that each $V_{i,j}^{\triv}\to V_{i,j}$ over $\spec(L')\to\spec(\ca{O}_{L'})$ is of the form in \eqref{eq:exam:semi-stable-1}. We see that $X^{\triv}_{L'}\to X_{\ca{O}_{L'}}$ is poly-stable over $\spec(L')\to \spec(\ca{O}_{L'})$.
\end{proof}

\begin{mylem}\label{lem:poly-stable-adequate}
	Let $K$ be a complete discrete valuation field extension of $\bb{Q}_p$ with perfect residue field, $\overline{K}$ an algebraic closure of $K$, $\overline{\eta}=\spec(\overline{K})$, $\overline{S}=\spec(\ca{O}_{\overline{K}})$, $\eta=\spec(K)$, $S=\spec(\ca{O}_K)$, $\scr{M}_S\to \ca{O}_{S_\et}$ the compactifying log structure on $S$ associated with the open immersion $\eta\to S$ {\rm(\cite[4.3]{he2024falmain})}, $X^{\triv}\to X$ an open immersion of coherent schemes poly-stable over $\eta\to S$, $\scr{M}_X\to \ca{O}_{X_\et}$ the compactifying log structure on $X$ associated with the open immersion $X^{\triv}\to X$. 
	\begin{enumerate}
		\renewcommand{\labelenumi}{{\rm(\theenumi)}}
		\item The scheme $X$ is normal and of finite type over $S$.\label{item:lem:poly-stable-adequate-1}
		\item The generic fibre $X_\eta$ is smooth over $\eta$ and $X_\eta\setminus X^{\triv}$ is the support of a strict normal crossings divisor on $X_\eta$.\label{item:lem:poly-stable-adequate-2}
		\item The structure morphism $(X,\scr{M}_X)\to (S,\scr{M}_S)$ is a smooth and saturated morphism between fine, saturated and regular log schemes.\label{item:lem:poly-stable-adequate-3}
	\end{enumerate} 
	In particular, $(X,\scr{M}_X)\to (S,\scr{M}_S)$ satisfies the conditions of \cite[\Luoma{3}.4.7]{abbes2016p}, and the morphism $X^{\triv}_{\overline{\eta}}\to X_{\overline{S}}$ is adequate over $\overline{\eta}\to \overline{S}$ in the sense of {\rm\ref{defn:essential-adequate-pair}} by {\rm\ref{rem:essential-adequate-log-smooth}}.
\end{mylem}
\begin{proof}
	With the notation in \ref{para:poly-stable}, after replacing each $U_i$ by an \'etale neighborhood of $\overline{x}$, we may assume that for any $1\leq i\leq n$ there exist homomorphisms of monoids $\alpha:\bb{N}\to \ca{O}_K$ and $\gamma:\bb{N}\to P$ of the forms in \ref{exam:semi-stable-chart} such that $U_i=\spec(\ca{O}_K\otimes_{\bb{Z}[\bb{N}]}\bb{Z}[P])$. Then, (\ref{item:lem:poly-stable-adequate-1}) and (\ref{item:lem:poly-stable-adequate-2}) follows immediately from the isomorphism \eqref{eq:exam:semi-stable-chart-5}. For (\ref{item:lem:poly-stable-adequate-3}), it suffices to check that each $(U_i,\scr{M}_{U_i})\to (S,\scr{M}_S)$ is a smooth and saturated morphism between fine, saturated and regular log schemes (\cite[4.4]{he2024falmain}). 
	
	With the notation in \ref{exam:semi-stable-chart}, since $P$ is an fs monoid, the log scheme $(U_i,\scr{M}'_{U_i})$ defined as $\spec(\ca{O}_K\otimes_{\bb{Z}[\bb{N}]}\bb{Z}[P])$ endowed with the log structure associated to the chart $P\to \bb{Z}[P]$ is fine and saturated (\cite[\Luoma{2}.5.15]{abbes2016p}). Moreover, since $P^{\mrm{gp}}\cong \bb{Z}\oplus\bb{Z}^d$ \eqref{eq:exam:semi-stable-chart-4} where $\gamma:\bb{N}^{\mrm{gp}}\to P^{\mrm{gp}}$ identifies with the inclusion to the first component, we see that $(U_i,\scr{M}'_{U_i})$ is a fine and saturated log scheme that is log smooth over $(S,\scr{M}_S)$ by \cite[\Luoma{2}.5.25]{abbes2016p}. As $(S,\scr{M}_S)$ is regular, so is $(U_i,\scr{M}'_{U_i})$ and thus $\scr{M}'_{U_i}$ coincides with the compactifying log structure $\scr{M}_{U_i}$ (\cite[4.4]{he2024falmain}). 
	
	It remains to check the saturatedness of $\gamma:\bb{N}\to P$. By the criterion \cite[\Luoma{1}.4.1]{tsuji2019saturated}, it suffices to check that for any prime number $q$, any $n\in\bb{N}$ and any $x\in P$ such that $qx-\gamma(n)\in P$, there exists $m\in\bb{N}$ such that $n\leq qm$ and $x-\gamma(m)\in P$. Indeed, we write $x=(a_0,\dots,a_d)\in P\subseteq \bb{N}^{1+b}\oplus\bb{Z}^{c-b}\oplus\bb{N}^{d-c}$ \eqref{eq:exam:semi-stable-chart-1}. Then, $qx-\gamma(n)\in P$ if and only if $n\leq qa_k$ for any $0\leq k\leq b$. We take $m$ to be the minimal element of $\{a_0,\dots,a_b\}\subseteq \bb{N}$. Hence, $n\leq qm$ and $x-\gamma(m)\in P$. This completes the proof.
\end{proof}

\begin{mythm}[{Faltings' main $p$-adic comparison theorem, \cite[Theorem 8 page 223]{faltings2002almost}, see \cite[4.8.13]{abbes2020suite}}]\label{thm:faltings-comparison}
	Let $K$ be a complete discrete valuation field extension of $\bb{Q}_p$ with perfect residue field, $\overline{K}$ an algebraic closure of $K$, $\overline{\eta}=\spec(\overline{K})$, $\overline{S}=\spec(\ca{O}_{\overline{K}})$, $X^{\triv}\to X$ an open immersion of coherent schemes poly-stable over $\overline{\eta}\to \overline{S}$ with $X$ proper over $\overline{S}$. Then, for any $n\in\bb{N}$, there exists a canonical morphism
	\begin{align}
		\rr\Gamma(X^{\triv}_{\et},\bb{Z}/p^n\bb{Z})\otimes_{\bb{Z}}^{\mrm{L}}\ca{O}_{\overline{K}}\longrightarrow \rr\Gamma(\fal^{\et}_{X^{\triv}\to X},\falb/p^n\falb)
	\end{align}
	which is an almost isomorphism {\rm(\cite[5.7]{he2024coh})}.
\end{mythm}
\begin{proof}
	We put $\eta=\spec(K)$ and $S=\spec(\ca{O}_K)$. After enlarging $K$, we may assume that there exists an open immersion of coherent schemes $Y^{\triv}\to Y$ poly-stable over $\eta\to S$ with $Y$ proper over $S$ such that $(X^{\triv}\to X)=(Y^{\triv}_{\overline{\eta}}\to Y_{\overline{S}})$ by \ref{lem:poly-stable-basic}.(\ref{item:lem:poly-stable-basic-2}). Since $(X^{\triv}\to X)=\lim_{K'}(Y^{\triv}_{\overline{\eta}}\to Y_{\ca{O}_{K'}})$ where the limit is taking over all finite subextension $K'$ of $K$ in $\overline{K}$, we have $(\fal^{\et}_{X^{\triv}\to X},\falb)=\lim_{K'}(\fal^{\et}_{Y^{\triv}_{\overline{\eta}}\to Y_{\ca{O}_{K'}}},\falb)$ by \cite[7.12]{he2024coh}. Recall that there is a canonical commutative diagram (\cite[5.5]{he2024falmain})
	\begin{align}
		\xymatrix{
			\rr\Gamma(X^{\triv}_{\et},\bb{Z}/p^n\bb{Z})\otimes_{\bb{Z}}^{\mrm{L}}\ca{O}_{\overline{K}}&\rr\Gamma(\fal^{\et}_{X^{\triv}\to X},\bb{Z}/p^n\bb{Z})\otimes_{\bb{Z}}^{\mrm{L}}\ca{O}_{\overline{K}}\ar[l]_-{\alpha_{\overline{K}}}\ar[r]^-{\beta_{\overline{K}}}& \rr\Gamma(\fal^{\et}_{X^{\triv}\to X},\falb/p^n\falb)\\
			\rr\Gamma(Y^{\triv}_{\overline{\eta},\et},\bb{Z}/p^n\bb{Z})\otimes_{\bb{Z}}^{\mrm{L}}\ca{O}_{\overline{K}}\ar@{=}[u]&\rr\Gamma(\fal^{\et}_{Y^{\triv}_{\overline{\eta}}\to Y_{\ca{O}_{K'}}},\bb{Z}/p^n\bb{Z})\otimes_{\bb{Z}}^{\mrm{L}}\ca{O}_{\overline{K}}\ar[l]_-{\alpha_{K'}}\ar[r]^-{\beta_{K'}}\ar[u]& \rr\Gamma(\fal^{\et}_{Y^{\triv}_{\overline{\eta}}\to Y_{\ca{O}_{K'}}},\falb/p^n\falb)\ar[u]
		}
	\end{align}
	Recall that $\alpha_{K'}$ is an isomorphism by \cite[9.6]{achinger2015kpi1} whose assumptions are satisfied by \ref{lem:poly-stable-basic}.(\ref{item:lem:poly-stable-basic-1}) and \ref{lem:poly-stable-adequate} (see also \cite[4.4.2]{abbes2020suite} and \cite[5.3]{he2024falmain}), and that $\beta_{K'}$ is an almost isomorphism by Faltings' main $p$-adic comparison theorem (\cite[Theorem 8 page 223]{faltings2002almost}, see \cite[4.8.13]{abbes2020suite}). Taking filtered colimit over $K'$ by \cite[\Luoma{6}.8.7.7]{sga4-2}, we see that $\alpha_{\overline{K}}$ is an isomorphism and that $\beta_{\overline{K}}$ is an almost isomorphism.
\end{proof}

\begin{myconj}[{Poly-stable modification conjecture, cf. \cite[1.2.6]{temkin2019logval}}]\label{conj:poly-stable}
	Let $K$ be a complete discrete valuation field extension of $\bb{Q}_p$ with perfect residue field, $\overline{K}$ an algebraic closure of $K$, $\overline{\eta}=\spec(\overline{K})$, $\overline{S}=\spec(\ca{O}_{\overline{K}})$, $X^{\triv}\to X$ an open immersion of $\overline{S}$-schemes of finite presentation over $\overline{\eta}\to\overline{S}$. Assume that $X_{\overline{\eta}}$ is smooth over $\overline{\eta}$ and that $X_{\overline{\eta}}\setminus X^{\triv}$ is the support of a normal crossings divisor on $X_{\overline{\eta}}$. Then, there exists an $X_{\overline{\eta}}$-modification $X'$ of $X$ {\rm(\ref{para:riemann-zariski})} such that $X'^{\triv}\to X'$ is poly-stable over $\overline{\eta}\to \overline{S}$, where $X'^{\triv}=X^{\triv}$.
\end{myconj}

\begin{mythm}[{Temkin's stable modification theorem for curves, \cite[1.5]{temkin2010stablecurve}}]\label{thm:temkin}
	Under the assumptions in {\rm\ref{conj:poly-stable}} and with the same notation, assume moreover that $X^{\triv}$ is equidimensional of dimension $1$. Then, the statement of {\rm\ref{conj:poly-stable}} holds true.
\end{mythm}
\begin{proof}
	We put $\eta=\spec(K)$, $S=\spec(\ca{O}_K)$ and $s=\spec(\ca{O}_K/\ak{m}_K)$. After enlarging $K$ by \cite[8.8.2, 8.10.5]{ega4-3} and \cite[17.7.8]{ega4-4}, we may assume that there exists an open immersion $Y^{\triv}\to Y$ of $S$-schemes of finite type whose base change along $\overline{S}\to S$ is $X^{\triv}\to X$ and $Y_\eta$ is smooth over $\eta$. Note that $X_{\overline{\eta}}$ is equidimension of dimension $1$ (as its dense open subset $X^{\triv}$ is so) so that $Y_\eta$ is also equidimensional of dimension $1$. We claim that it suffices to show that after enlarging $K$ there exists a $Y_\eta$-modification $Y'$ of $Y$ such that $Y'^{\triv}\to Y'$ is poly-stable over $\eta\to S$, where $Y'^{\triv}=Y^{\triv}$. Indeed, $Y'^{\triv}_{\overline{\eta}}\to Y'_{\overline{S}}$ is poly-stable over $\overline{\eta}\to \overline{S}$ by \ref{lem:poly-stable-basic}.(\ref{item:lem:poly-stable-basic-1}) and $Y'_{\overline{S}}$ is an $X_{\overline{\eta}}=Y_{\overline{\eta}}$-modification of $X=Y_{\overline{S}}$.
	
	We focus on finding $Y'$ in the following. After replacing $Y$ by its integral closure in $Y_\eta$ (as $Y$ is Nagata and $Y_\eta$ is reduced, see \cite[\href{https://stacks.math.columbia.edu/tag/03GR}{03GR}]{stacks-project}), we may assume that $Y$ is normal. As $Y_\eta$ is smooth of finite type over $\eta$, $Y$ is a finite disjoint union of Noetherian normal integral schemes flat over $S$ with $\pi_0(Y_\eta)=\pi_0(Y)$ (\cite[\href{https://stacks.math.columbia.edu/tag/035L}{035L}]{stacks-project}). In particular, the special fibre $Y_s$ of $Y$ is either empty or equidimensional of dimension $1$ by \cite[\href{https://stacks.math.columbia.edu/tag/0B2J}{0B2J}]{stacks-project}. Hence, $Y$ is flat, of finite presentation, of relative dimension $1$ over $S$.
	
	We regard $Y_\eta\setminus Y^{\triv}$ as a reduced closed subscheme of $Y_\eta$ consisting of finitely many closed points, and let $D$ be its scheme theoretic closure in $Y$. Then, $D$ is flat of finite presentation over $S$ whose generic fibre is of dimension $0$, and thus its special fibre $D_s$ is also of dimension $0$ (\cite[\href{https://stacks.math.columbia.edu/tag/0B2J}{0B2J}]{stacks-project}). Hence, we can apply \cite[1.5]{temkin2010stablecurve} to the multi-pointed curve $(Y,D)$ over $S$. Then, after enlarging $K$ we obtain a commutative diagram of flat $S$-schemes of finite presentation 
	\begin{align}
		\xymatrix{
			D'\ar[r]\ar[d]&Y'\ar[d]\\
			D\ar[r]&Y
		}
	\end{align}
	where $Y'$ is an $Y_\eta$-modification of $Y$ whose special fibre $Y'_s$ is at-worst-nodal (\cite[\href{https://stacks.math.columbia.edu/tag/0C47}{0C47}]{stacks-project}),  $D'$ is a $D_\eta$-modification of $D$ which is \'etale over $S$, $D'\subseteq Y'$ is a Cartier divisor disjoint from the nodes of $Y'_s$. Note that $Y'^{\triv}=Y'_\eta\setminus D'_\eta=Y_\eta\setminus D_\eta=Y^{\triv}$. It suffices to show that $Y'^{\triv}\to Y'$ is semi-stable over $\eta\to S$ in the sense of \ref{exam:semi-stable}. Let $\overline{y}$ be a geometric point of $Y'$ with image $y\in Y'$. 
	
	If $y\in Y'^{\triv}$, then there exists an \'etale neighborhood $U$ of $\overline{y}$ and an \'etale morphism $U\to \spec(K[T^{\pm 1}])$ as $Y'^{\triv}$ is a smooth curve over $K$.
	
	If $y\in D'_\eta$, then there exists an \'etale neighborhood $U$ of $\overline{y}$ in $Y'_\eta$ whose image in $Y'_\eta$ intersects with $D'_\eta$ only at $y$. Let $y\in U$ be the image of $\overline{y}$. Then, after shrinking $U$, there exists a regular system of parameters $t$ of the $1$-dimensional regular local ring $\ca{O}_{U,y}$ which defines the strict normal crossings divisor $y\in U$ (\cite[\href{https://stacks.math.columbia.edu/tag/0BI9}{0BI9}]{stacks-project}). Then, the completed regular local ring $\widehat{\ca{O}_{U,y}}$ is canonically isomorphic to the algebra of power series $\kappa(y)[[t]]$, where $\kappa(y)$ is the residue field of $y$ finite over $K$ (\cite[\href{https://stacks.math.columbia.edu/tag/01TF}{01TF}, \href{https://stacks.math.columbia.edu/tag/00NO}{00NO}]{stacks-project}). After shrinking $U$, we may assume that $U=\spec(R)$ is affine and $t\in R$. Thus, the $\kappa(y)$-algebra homomorphism $\kappa(y)[T]\to R$ sending $T$ to $t$ is \'etale at $y$ by \cite[17.5.3]{ega4-4}. After shrinking $U$, we may assume that $\kappa(y)[T]\to R$ sending $T$ to $t$ is \'etale and thus so is $K[T]\to R$. Note that $T=0$ defines the strict normal crossings divisor $y=U_\eta\setminus U^{\triv}$.
	
	If $y$ is a non-singular point of $Y'_s$ which does not lie in $D'$, then $Y'\setminus D'$ is smooth at $y$ as it is flat over $S$ (\cite[\href{https://stacks.math.columbia.edu/tag/01V9}{01V9}]{stacks-project}). Hence, there exists an \'etale neighborhood $U$ of $\overline{y}$ and an \'etale morphism $U\to \spec(\ca{O}_K[T^{\pm 1}])$ with $U^{\triv}=U_\eta$ over $Y'^{\triv}$.
	
	If $y$ is a non-singular point of $Y'_s$ which lies in $D'$, then since the Cartier divisor $D'$ is \'etale over $S$, after passing to an \'etale neighborhood of $\overline{y}$, we may assume that $Y=\spec(R)$ is affine and $D'=\spec(\ca{O}_{K'})$ is defined by a nonzero divisor $t\in R$, where $K'$ is an unramified finite extension of $K$. Thus, the local ring $R_y$ of $R$ at $y$ is a $2$-dimensional regular local ring where $t$ and a uniformizer of $K$ form a regular system of parameters (\cite[\href{https://stacks.math.columbia.edu/tag/00NU}{00NU}]{stacks-project}). Then, the completed regular local ring $\widehat{R_y}$ is canonically isomorphic to $\ca{O}_{K'}[[t]]$ (\cite[\href{https://stacks.math.columbia.edu/tag/00NO}{00NO}]{stacks-project}). After shrinking $\spec(R)$ by \cite[17.5.3]{ega4-4}, we may assume that $\ca{O}_{K'}[T]\to R$ sending $T$ to $t$ is \'etale and thus so is $\ca{O}_K[T]\to R$. Note that $T=0$ defines the strict normal crossings divisor $D'$ in $Y'$.
	
	If $y$ is a singular point of $Y'_s$, then after passing to an \'etale neighborhood of $\overline{y}$ by \cite[\href{https://stacks.math.columbia.edu/tag/0CBY}{0CBY}]{stacks-project}, we may assume that there exists an \'etale morphism $Y'\to \spec(\ca{O}_{K'}[u,v]/(uv-a))$ over $\ca{O}_K'$, where $K'$ is an unramified finite extension of $K$ and $a\in\ca{O}_{K'}$. Since $y$ is a singular point of $Y'_s$ and $Y'_\eta$ is smooth, we have $a\in\ak{m}_{K'}\setminus \{0\}$ and the image of $\overline{y}$ in $\spec(\ca{O}_{K'}[u,v]/(uv-a))$ is the point defined by $u=v=0$. After multiplying a unit on $a$ (and also on $u$), we may assume that $a\in \ak{m}_K\setminus\{0\}$ as $K'$ is unramified over $K$. Thus, $Y'\to \spec(\ca{O}_K[u,v]/(uv-a))$ is a semi-stable chart (\ref{exam:semi-stable}).
\end{proof}

\begin{mycor}\label{cor:polystable-model}
	Let $K$ be a complete discrete valuation field extension of $\bb{Q}_p$ with perfect residue field, $\overline{K}$ an algebraic closure of $K$, $\overline{\eta}=\spec(\overline{K})$, $\overline{S}=\spec(\ca{O}_{\overline{K}})$, $(Y_n)_{n\in\bb{N}}$ a directed inverse system of separated smooth $\overline{\eta}$-schemes of finite type. Assume moreover that one of the following conditions holds:
	\begin{enumerate}
		\renewcommand{\labelenumi}{{\rm(\theenumi)}}
		\item The poly-stable modification conjecture {\rm\ref{conj:poly-stable}} is true.
		\item For any $n\in\bb{N}$, $Y_n$ is equidimensional of dimension $1$.
	\end{enumerate}
	Then, there exists a directed inverse system $(X^{\triv}_n\to X_n)_{n\in\bb{N}}$ of open immersions of coherent schemes poly-stable over $\overline{\eta}\to\overline{S}$ with each $X_n$ proper over $\overline{S}$
	such that there is an isomorphism of inverse systems of $\overline{\eta}$-schemes $(Y_n)_{n\in\bb{N}}\iso (X^{\triv}_n)_{n\in\bb{N}}$.
\end{mycor}
\begin{proof}
	We construct it by induction and we put $\eta=\spec(K)$ and $S=\spec(\ca{O}_K)$. Assume that we have constructed a poly-stable morphism $(X^{\triv}_n\to X_n)\to (\overline{\eta}\to\overline{S})$ with $X^{\triv}_n=Y_n$. As $Y_{n+1}$ and $X_n$ are separated of finite presentation over $S$, after enlarging $K$ by \cite[8.8.2, 8.10.5]{ega4-3}, we may assume that there exists a separated morphism $Y'\to X'$ of $S$-schemes of finite type over $\eta\to S$ whose base change along $\overline{S}\to S$ is $Y_{n+1}\to X_n$. By Nagata compactification theorem (\cite[\href{https://stacks.math.columbia.edu/tag/0F41}{0F41}]{stacks-project}), there exists an open immersion $Y'\to Z'$ of $X'$-schemes with $Z'$ proper over $X'$. By Hironaka's resolution of singularities (\cite[Main Theorems I and II]{hironaka1964resolution}), there exists a $Y'$-modification $Z''_\eta$ of $Z'_\eta$ such that $Z''_\eta$ is smooth over $\eta$ and that $Z''_\eta\setminus Y'$ is the support of a normal crossings divisor on $Z''_\eta$. Applying Nagata compactification theorem again, we obtain an open immersion $Z''_\eta\to Z''$ of $Z'$-schemes with $Z''$ proper over $Z'$.
	
	As $S$ is Noetherian, $Z''$ is proper of finite presentation over $X'$. Let $Z_{n+1}$ be the base change of $Z''$ along $\overline{S}\to S$. Then, there is a canonical commutative diagram of $\overline{S}$-schemes of finite presentation
	\begin{align}\label{eq:cor:polystable-model-1}
		\xymatrix{
			Y_{n+1}\ar[d]\ar[r]&Z_{n+1,\overline{\eta}}\ar[r]\ar[d]&Z_{n+1}\ar[d]\\
			Y_n\ar[r]&X_{n,\overline{\eta}}\ar[r]&X_n
		}
	\end{align} 
	where $Y_{n+1}\to Z_{n+1,\overline{\eta}}$ is an open immersion, $Z_{n+1,\overline{\eta}}$ is smooth, $Z_{n+1,\overline{\eta}}\setminus Y_{n+1}$ is the support of a normal crossings divisor on $Z_{n+1,\overline{\eta}}$ (as $\overline{\eta}\to\eta$ is a regular morphism), and $Z_{n+1}$ is proper over $X_n$. Applying \ref{conj:poly-stable} or \ref{thm:temkin} to the morphism $(Y_{n+1}\to Z_{n+1})\to (\overline{\eta}\to\overline{S})$, we obtain a $Z_{n+1,\overline{\eta}}$-modification $X_{n+1}$ of $Z_{n+1}$ such that $(X^{\triv}_{n+1}\to X_{n+1})$ is poly-stable over $\overline{\eta}\to\overline{S}$, where $X^{\triv}_{n+1}=Y_{n+1}$. We see that $X_{n+1}$ is also proper over $X_n$, and thus proper over $\overline{S}$. This completes the induction process.
\end{proof}

\begin{mythm}\label{thm:vanishing}
	Let $K$ be a complete discrete valuation field extension of $\bb{Q}_p$ with perfect residue field, $\overline{K}$ an algebraic closure of $K$, $\overline{\eta}=\spec(\overline{K})$, $\overline{S}=\spec(\ca{O}_{\overline{K}})$, $(X^{\triv}_n\to X_n)_{n\in\bb{N}}$ a directed inverse system of open immersions of coherent schemes poly-stable over $\overline{\eta}\to \overline{S}$ with each $X_n$ proper over $\overline{S}$. Regarding $(X^{\triv}_n\to X_n)_{n\in\bb{N}}$ as an object of $\pro(\fal^{\mrm{open}}_{\overline{\eta}\to \overline{S}})$ {\rm(\ref{para:zeta-3})}, assume that for any $x\in X^{\rz}=\lim_{n\in\bb{N}}X_n^{\rz}$, the stalk $\ca{O}_{\rz_{X_{\overline{\eta}}}(X),x}$ is pre-perfectoid, where $\rz_{X_{\overline{\eta}}}(X)=\lim_{n\in\bb{N}}\rz_{X_{n,\overline{\eta}}}(X_n)$ (see {\rm\ref{para:limit-rz}}). Then, for any integer $k\in\bb{N}$, there exists a canonical isomorphism of homotopy colimits {\rm(\cite[\href{https://stacks.math.columbia.edu/tag/090Z}{090Z}]{stacks-project})} in the derived category of almost $\ca{O}_{\overline{K}}$-modules (see {\rm\cite[5.7]{he2024coh}})
	\begin{align}\label{eq:thm:vanishing-1}
		\op{hocolim}_{n\in\bb{N}}\rr\Gamma(X^{\triv}_{n,\et},\ca{O}_{\overline{K}}/p^k\ca{O}_{\overline{K}})\iso \op{hocolim}_{n\in\bb{N}}\rr\Gamma(X_n,\ca{O}_{X_n}/p^k\ca{O}_{X_n}).
	\end{align}
\end{mythm}
\begin{proof}
	Since each $X^{\triv}_n\to X_n$ is adequate over $\overline{\eta}\to \overline{S}$ by \ref{lem:poly-stable-basic} and \ref{lem:poly-stable-adequate}, its Faltings cohomology is Riemann-Zariski faithful by \ref{prop:rz-faithful-adequate} and \ref{cor:rz-continue}. Thus, $(X^{\triv}_n\to X_n)_{n\in\bb{N}}$ is locally Faltings acyclic by \ref{thm:val-criterion-faltings-acyclic}. Therefore, by \ref{defn:locally-faltings-acyclic}, \ref{rem:locally-faltings-acyclic} and \cite[\Luoma{6}.8.7.7]{sga4-2}, the canonical morphism
	\begin{align}\label{eq:thm:vanishing-3}
		\colim_{n\in\bb{N}}H^q(X_n,\ca{O}_{X_n}/p^k\ca{O}_{X_n})\longrightarrow \colim_{n\in\bb{N}}H^q(\fal^{\et}_{X^{\triv}_n\to X_n},\falb/p^k\falb)
	\end{align}
	is an almost isomorphism for any $q,\ k\in\bb{N}$ (note that $X_n$ is integrally closed in $X^{\triv}_n$). In other words, the canonical morphism of homotopy colimits
	\begin{align}\label{eq:thm:vanishing-4}
		\op{hocolim}_{n\in\bb{N}}\rr\Gamma(X_n,\ca{O}_{X_n}/p^k\ca{O}_{X_n})\longrightarrow \op{hocolim}_{n\in\bb{N}}\rr\Gamma(\fal^{\et}_{X^{\triv}_n\to X_n},\falb/p^k\falb)
	\end{align}
	is an isomorphism in the derived category of almost $\ca{O}_{\overline{K}}$-modules (\cite[\href{https://stacks.math.columbia.edu/tag/0CRK}{0CRK}]{stacks-project}).
	
	On the other hand, by Faltings' main $p$-adic comparison theorem \ref{thm:faltings-comparison}, there is a canonical morphism of homotopy colimits
	\begin{align}\label{eq:thm:vanishing-5}
		\op{hocolim}_{n\in\bb{N}}\rr\Gamma(X^{\triv}_{n,\et},\ca{O}_{\overline{K}}/p^k\ca{O}_{\overline{K}})\longrightarrow \op{hocolim}_{n\in\bb{N}}\rr\Gamma(\fal^{\et}_{X^{\triv}_n\to X_n},\falb/p^k\falb)
	\end{align}
	which is an isomorphism in the derived category of almost $\ca{O}_{\overline{K}}$-modules (\cite[\href{https://stacks.math.columbia.edu/tag/0CRK}{0CRK}]{stacks-project}). Composing \eqref{eq:thm:vanishing-5} with the inverse of \eqref{eq:thm:vanishing-4}, we obtain a canonical isomorphism \eqref{eq:thm:vanishing-1}.
\end{proof}

\begin{mycor}\label{cor:vanishing-1}
	Under the assumptions in {\rm\ref{thm:vanishing}} and with the same notation, for any integer $q>\limsup_{n\to\infty}\{\dim(X^{\triv}_n)\}$, we have
	\begin{align}\label{eq:cor:vanishing-1}
		\colim_{n\in\bb{N}}H^q(X^{\triv}_{n,\et},\bb{Z}/p^k\bb{Z})=0.
	\end{align} 
\end{mycor}
\begin{proof}
	We may assume that $q>\sup_{n\in\bb{N}}\{\dim(X^{\triv}_n)\}$ after discarding finitely many $n\in\bb{N}$. By d\'evissage, it suffices to prove the case where $k=1$. 
	
	We claim that  $H^q(X_n,\ca{O}_{X_n}/p\ca{O}_{X_n})=0$. Indeed, for any $n\in\bb{N}$, $X_n$ is a finite disjoint union of normal integral schemes flat over $\overline{S}$ by \ref{lem:essential-adequate-generic} (whose assumptions are satisfied by \ref{lem:poly-stable-basic} and \ref{lem:poly-stable-adequate}). Thus, the dimension of the special fibre of $X_n$ is bounded by the dimension of the generic fibre (\cite[\href{https://stacks.math.columbia.edu/tag/0B2J}{0B2J}]{stacks-project}) and thus strictly less than $q$. Let $f_n:X_n\to \overline{S}$ denote the canonical morphism of schemes. As $f_n$ is proper whose special fibre has dimension strictly less than $q$, we have
	\begin{align}\label{eq:thm:vanishing-6}
		H^q(X_n,\ca{O}_{X_n}/p\ca{O}_{X_n})=\rr^q f_{n*}(\ca{O}_{X_n}/p\ca{O}_{X_n})=0
	\end{align}
	by Grothendieck's vanishing \cite[\href{https://stacks.math.columbia.edu/tag/0E7D}{0E7D}]{stacks-project}.
	
	The claim implies that $\colim_{n\in\bb{N}}H^q(X^{\triv}_{n,\et},\bb{F}_p)\otimes_{\bb{Z}_p}\ca{O}_{\overline{K}}$ is almost zero by \ref{thm:vanishing}. As $\colim_{n\in\bb{N}}H^q(X^{\triv}_{n,\et},\bb{F}_p)$ is an (free) $\bb{F}_p$-module, its base change to $\ca{O}_{\overline{K}}/p\ca{O}_{\overline{K}}$ is almost zero if and only if itself is zero. This completes the proof.
\end{proof}

\begin{mycor}\label{cor:vanishing-2}
	Under the assumptions in {\rm\ref{thm:vanishing}} and with the same notation, let $F$ be a valuation field extension of $\overline{K}$ of height $1$. Then, there is a canonical exact sequence of almost $\ca{O}_F$-modules for any $q\in\bb{N}$,
	\begin{align}\label{eq:cor:vanishing-2-1}
		\xymatrix{
			0\ar[r]&\widehat{\ca{M}^q}\ar[r]&\ca{H}^q\ar[r]&T_p(\ca{M}^{q+1})\ar[r]&0,
		}
	\end{align}
	where $\ca{H}^q=\lim_{k\in\bb{N}}\colim_{n\in\bb{N}}H^q(X^{\triv}_{n,\et},\bb{Z}/p^k\bb{Z})\otimes_{\bb{Z}_p}\ca{O}_F$, and  $\ca{M}^q=\colim_{n\in\bb{N}}H^q(X_n,\ca{O}_{X_n})\otimes_{\ca{O}_{\overline{K}}}\ca{O}_F$ with $\widehat{\ca{M}^q}=\lim_{k\in\bb{N}}\ca{M}^q/p^k\ca{M}^q$ its $p$-adic completion and $T_p(\ca{M}^q)=\ho_{\bb{Z}_p}(\bb{Q}_p/\bb{Z}_p,\ca{M}^q)=\lim_{k\in\bb{N}}\ca{M}^q[p^k]$ its Tate module.
\end{mycor}
\begin{proof}
	Taking cohomologies of the exact sequence $0\to \ca{O}_{X_n}\stackrel{\cdot p^k}{\longrightarrow}\ca{O}_{X_n}\longrightarrow \ca{O}_{X_n}/p^k\ca{O}_{X_n}\longrightarrow 0$ of coherent modules over $X_n$, we obtain an exact sequence for any $k\in\bb{N}$,
	\begin{align}
		0\longrightarrow H^q(X_n,\ca{O}_{X_n})/p^k\longrightarrow H^q(X_n,\ca{O}_{X_n}/p^k\ca{O}_{X_n})\longrightarrow H^{q+1}(X_n,\ca{O}_{X_n})[p^k]\longrightarrow 0.
	\end{align}
	Taking the flat base change along $\ca{O}_{\overline{K}}\to \ca{O}_F$ and taking filtered colimit over $n\in\bb{N}$, we obtain an exact sequence
	\begin{align}
		0\longrightarrow \ca{M}^q/p^k\ca{M}^q\longrightarrow \colim_{n\in\bb{N}}H^q(X_n,\ca{O}_{X_n}/p^k\ca{O}_{X_n})\otimes_{\ca{O}_{\overline{K}}}\ca{O}_F\longrightarrow \ca{M}^{q+1}[p^k]\longrightarrow 0.
	\end{align}
	As the inverse system $(\ca{M}^q/p^k\ca{M}^q)_{k\in\bb{N}}$ satisfies the Mittag-Leffler condition, we have $\rr^1\lim_{k\in\bb{N}}\ca{M}^q/p^k\ca{M}^q=0$ (\cite[\href{https://stacks.math.columbia.edu/tag/07KW}{07KW}]{stacks-project}). Thus, taking inverse limit over $k\in\bb{N}$, we obtain an exact sequence
	\begin{align}
		0\longrightarrow \widehat{\ca{M}^q}\longrightarrow \lim_{k\in\bb{N}}\colim_{n\in\bb{N}}H^q(X_n,\ca{O}_{X_n}/p^k\ca{O}_{X_n})\otimes_{\ca{O}_{\overline{K}}}\ca{O}_F\longrightarrow T_p(\ca{M}^{q+1})\longrightarrow 0.
	\end{align}
	The conclusion follows from \ref{thm:vanishing}.
\end{proof}

\begin{myrem}\label{rem:cor:vanishing-2}
	In \ref{cor:vanishing-2}, $\ca{M}^q$ is a countably generated $\ca{O}_F$-module for any $q\in\bb{N}$. Indeed, since $X_n$ is the base change of a proper $\ca{O}_{K'}$-scheme along $\overline{S}\to\spec(\ca{O}_{K'})$ for a finite subextension $K'$ of $K$ in $\overline{K}$ by \ref{lem:poly-stable-basic}.(\ref{item:lem:poly-stable-basic-2}), we see that $H^q(X_n,\ca{O}_{X_n})$ is a $\ca{O}_{\overline{K}}$-module of finite presentation (\cite[\href{https://stacks.math.columbia.edu/tag/02O5}{02O5}]{stacks-project}). 
	
	In particular, $\ca{M}^q/\ca{M}^q[p^\infty]$ is a countably generated torsion-free $\ca{O}_F$-module. We take $F$ to be the maximal completion $\widetilde{\overline{K}}$ of the valuation field $\overline{K}$ (see \ref{rem:max-val-exist}), which is still algebraically closed by \ref{cor:max-val-alg-clos}. Then, we apply Kaplansky's structure theorem \ref{cor:max-val-mod} to it so that there exists countable sets $I_q$ and $J_q$ such that for any $\epsilon\in\ak{m}_{\widetilde{\overline{K}}}$ we can write
	\begin{align}\label{eq:rem:cor:vanishing-2-1}
		\epsilon \cdot (\ca{M}^q/\ca{M}^q[p^\infty])\subseteq \widetilde{\overline{K}}^{\oplus I_q}\oplus \ca{O}_{\widetilde{\overline{K}}}^{\oplus J_q} \subseteq \ca{M}^q/\ca{M}^q[p^\infty].
	\end{align}
	In particular, we have
	\begin{align}\label{eq:rem:cor:vanishing-2-2}
		 \colim_{n\in\bb{N}}H^q(X_{n,\overline{\eta}},\ca{O}_{X_{n,\overline{\eta}}})\otimes_{\overline{K}}\widetilde{\overline{K}}= \ca{M}^q[1/p]=\widetilde{\overline{K}}^{\oplus I_q}\oplus \widetilde{\overline{K}}^{\oplus J_q}.
	\end{align}
\end{myrem}

\begin{mylem}\label{lem:torsion-completion}
	Let $M$ be a torsion $\bb{Z}_p$-module, $\widehat{M}$ its $p$-adic completion. Assume that $\widehat{M}[p^\infty]=\widehat{M}[p^m]$ for some $m\in\bb{N}$. Then, $\widehat{M}=\widehat{M}[p^m]$ and $\bigcap_{k\in\bb{N}}p^kM=p\cdot \bigcap_{k\in\bb{N}}p^kM$.
\end{mylem}
\begin{proof}
	As $M$ is torsion, $M/\bigcap_{k\in\bb{N}}p^kM$ is contained in $\widehat{M}[p^\infty]$. By assumption, it is killed by $p^m$, and thus so is its $p$-adic completion $(M/\bigcap_{k\in\bb{N}}p^kM)^\wedge=\widehat{M}$. 
	
	On the other hand, taking $p$-adic completion of the exact sequence, $0\to \bigcap_{k\in\bb{N}}p^kM\to M\to M/\bigcap_{k\in\bb{N}}p^kM\to 0$, we obtain an exact sequence $0\to (\bigcap_{k\in\bb{N}}p^kM)^\wedge\to \widehat{M}\iso (M/\bigcap_{k\in\bb{N}}p^kM)^\wedge\to 0$ as $M/\bigcap_{k\in\bb{N}}p^kM$ is killed by $p^m$ (\cite[\href{https://stacks.math.columbia.edu/tag/0BNG}{0BNG}]{stacks-project}). In particular, $(\bigcap_{k\in\bb{N}}p^kM)^\wedge=0$, which implies that $\bigcap_{k\in\bb{N}}p^kM=p\cdot \bigcap_{k\in\bb{N}}p^kM$.
\end{proof}

\begin{mycor}\label{cor:vanishing-3}
	Under the assumptions in {\rm\ref{thm:vanishing}} and with the same notation, assume moreover that $d=\limsup_{n\to\infty}\{\dim(X^{\triv}_n)\}<+\infty$ and that the $p$-power torsion of $\lim_{k\in\bb{N}}\colim_{n\in\bb{N}}H^d(X^{\triv}_{n,\et},\bb{Z}/p^k\bb{Z})$ is bounded. Then, for any valuation field $F$ extension of $\overline{K}$ of height $1$, there is a canonical isomorphism of almost $\ca{O}_F$-modules 
	\begin{align}\label{eq:cor:vanishing-3-1}
		(\ca{M}^d/\ca{M}^d[p^\infty])^\wedge\iso \ca{H}^d/\ca{H}^d[p^\infty]
	\end{align}
	where  $\ca{M}^d=\colim_{n\in\bb{N}}H^d(X_n,\ca{O}_{X_n})\otimes_{\ca{O}_{\overline{K}}}\ca{O}_F$ and $\ca{H}^d=\lim_{k\in\bb{N}}\colim_{n\in\bb{N}}H^d(X^{\triv}_{n,\et},\bb{Z}/p^k\bb{Z})\otimes_{\bb{Z}_p}\ca{O}_F$.
\end{mycor}
\begin{proof}
	We take again the notation in \ref{cor:vanishing-2} and we put $f_n:X_n\to \overline{S}$ the canonical morphism. As in the proof of \eqref{eq:thm:vanishing-6}, $f_n$ is proper whose special fibre has dimension no more than $d$, thus $H^{d+1}(X_n,\ca{O}_{X_n})=\rr^{d+1}f_{n*}\ca{O}_{X_n}=0$ by Grothendieck's vanishing \cite[\href{https://stacks.math.columbia.edu/tag/0E7D}{0E7D}]{stacks-project}. In particular, we get $\ca{M}^{d+1}=0$. Thus, $T_p(\ca{M}^{d+1})=0$ and $\widehat{\ca{M}^d}$ is almost isomorphic to $\ca{H}^d$ by \ref{cor:vanishing-2}. Since $\bb{Z}_p$ is a maximal valuation ring and the $p$-power torsion of $\lim_{k\in\bb{N}}\colim_{n\in\bb{N}}H^d(X^{\triv}_{n,\et},\bb{Z}/p^k\bb{Z})$ is bounded by assumption, the $p$-power torsion of $\ca{H}^d=\lim_{k\in\bb{N}}\colim_{n\in\bb{N}}H^d(X^{\triv}_{n,\et},\bb{Z}/p^k\bb{Z})\otimes_{\bb{Z}_p}\ca{O}_F$ is also bounded by \ref{lem:torsion-free-bc}. Hence, the $p$-power torsion of $\widehat{\ca{M}^d}$ is bounded. 
	
	Consider the exact sequence $0\to \ca{M}^d[p^\infty]\to \ca{M}^d\to \ca{M}^d/\ca{M}^d[p^\infty]\to 0$. Passing to $p$-adic completion, the sequence
	\begin{align}
		\xymatrix{
			0\ar[r]&\widehat{\ca{M}^d[p^\infty]}\ar[r]&\widehat{\ca{M}^d}\ar[r]&(\ca{M}^d/\ca{M}^d[p^\infty])^\wedge\ar[r]&0
		}
	\end{align}
	is still exact as $\ca{M}^d/\ca{M}^d[p^\infty]$ is torsion-free (\cite[\href{https://stacks.math.columbia.edu/tag/0315}{0315}]{stacks-project}). The previous discussion shows that the $p$-power torsion of $\widehat{\ca{M}^d[p^\infty]}$ is bounded. Hence, $\widehat{\ca{M}^d[p^\infty]}$ is killed by a certain power of $p$ by \ref{lem:torsion-completion}. Since $(\ca{M}^d/\ca{M}^d[p^\infty])^\wedge$ is still torsion-free (\cite[5.20]{he2024coh}), we see that $\widehat{\ca{M}^d[p^\infty]}=\widehat{\ca{M}^d}[p^\infty]$. Therefore, we deduce from the isomorphism of almost $\ca{O}_{\overline{K}}$-modules $\widehat{\ca{M}^d}\iso \ca{H}^d$ \eqref{eq:cor:vanishing-2-1} an isomorphism of their torsion-free quotients (as almost modules) $(\ca{M}^d/\ca{M}^d[p^\infty])^\wedge\iso \ca{H}^d/\ca{H}^d[p^\infty]$.
\end{proof}

\begin{myrem}\label{rem:cor:vanishing-3}
	We take again the notation in \ref{rem:cor:vanishing-2}.
	\begin{enumerate}
		\renewcommand{\labelenumi}{{\rm(\theenumi)}}
		\item Assume that there is an isomorphism of almost $\ca{O}_F$-modules for some $q\in\bb{N}$,
		\begin{align}\label{eq:rem:cor:vanishing-3-1}
			(\ca{M}^q/\ca{M}^q[p^\infty])^\wedge\iso \ca{H}^q/\ca{H}^q[p^\infty],
		\end{align}
		where  $\ca{M}^q=\colim_{n\in\bb{N}}H^q(X_n,\ca{O}_{X_n})\otimes_{\ca{O}_{\overline{K}}}\ca{O}_F$ and $\ca{H}^q=\lim_{k\in\bb{N}}\colim_{n\in\bb{N}}H^q(X^{\triv}_{n,\et},\bb{Z}/p^k\bb{Z})\otimes_{\bb{Z}_p}\ca{O}_F$. Then, it induces a homomorphism of $\widetilde{\overline{K}}$-modules by composing with $\ca{M}^q\to (\ca{M}^q/\ca{M}^q[p^\infty])^\wedge$ and inverting $p$,
		\begin{align}\label{eq:rem:cor:vanishing-3-2}
			\ca{M}^q[1/p]\longrightarrow \ca{H}^q[1/p].
		\end{align}
		If we take $F=\widetilde{\overline{K}}$, then \eqref{eq:rem:cor:vanishing-3-2} fits into a commutative diagram induced by \eqref{eq:rem:cor:vanishing-2-1},
		\begin{align}\label{eq:rem:cor:vanishing-3-3}
			\xymatrix{
				\ca{M}^q[1/p]\ar[r]\ar@{->>}[d]& \ca{H}^q[1/p]\\
				\bigoplus_{J_q}\widetilde{\overline{K}}\ar[r]&\widehat{\bigoplus}_{J_q}\widetilde{\overline{K}}\ar[u]_-{\wr}
			}
		\end{align}
		where the horizontal arrow on the bottom is induced from the $p$-adic completion map $\bigoplus_{J_q}\ca{O}_{\widetilde{\overline{K}}}\to\widehat{\bigoplus}_{J_q}\ca{O}_{\widetilde{\overline{K}}}$ by inverting $p$, the left vertical arrow is surjective and the right vertical arrow is an isomorphism. In particular, \eqref{eq:rem:cor:vanishing-3-2} has dense image with respect to the topology induced by the $p$-adic topology of $\ca{H}^q/\ca{H}^q[p^\infty]$ (which coincides with the topology quotient of the $p$-adic topology of $\ca{H}^q$ if the $p$-power torsion of $\ca{H}^q$ is bounded).\label{item:rem:cor:vanishing-3-1}
		\item Instead of the extra assumption in \ref{cor:vanishing-3}, if we assume that the $p$-power torsion of $\colim_{n\in\bb{N}}H^q(X_n,\ca{O}_{X_n})$ is bounded for any $q\in\bb{N}$, then one can prove that for any valuation field $F$ extension of $\overline{K}$ of height $1$, \eqref{eq:cor:vanishing-2-1} induces a canonical isomorphism of almost $\ca{O}_F$-modules  
		\begin{align}\label{eq:rem:cor:vanishing-3-4}
			(\ca{M}^q/\ca{M}^q[p^\infty])^\wedge\iso \ca{H}^q/\ca{H}^q[p^\infty].
		\end{align}
		\label{item:rem:cor:vanishing-3-2}
	\end{enumerate}	
\end{myrem}

\begin{mythm}\label{thm:completed}
	Let $K$ be a complete discrete valuation field extension of $\bb{Q}_p$ with perfect residue field, $\overline{K}$ an algebraic closure of $K$, $\overline{\eta}=\spec(\overline{K})$, $(Y_n)_{n\in\bb{N}}$ a directed inverse system of separated smooth $\overline{\eta}$-schemes of finite type such that $d=\limsup_{n\to\infty}\{\dim(Y_n)\}<+\infty$. Assume that the following conditions hold:
	\begin{enumerate}
		\renewcommand{\labelenumi}{{\rm(\theenumi)}}
		\item Either the poly-stable modification conjecture {\rm\ref{conj:poly-stable}} is true, or $Y_n$ is equidimensional of dimension $1$ for any $n\in\bb{N}$.\label{item:thm:completed-1}
		\item For any point $y$ of the inverse limit $\lim_{n\in\bb{N}}Y_n$ of locally ringed spaces, its residue field $\kappa(y)$ is a pre-perfectoid field with respect to any valuation ring $V$ of height $1$ extension of $\ca{O}_{\overline{K}}$ with fraction field $V[1/p]=\kappa(y)$.\label{item:thm:completed-2}
	\end{enumerate}
	Then, for any integer $q>d$ and any integer $k\in\bb{N}$, we have
	\begin{align}\label{eq:thm:completed-1}
		\colim_{n\in\bb{N}}H^q(Y_{n,\et},\bb{Z}/p^k\bb{Z})=0.
	\end{align} 
\end{mythm}
\begin{proof}
	 By \ref{cor:polystable-model}, there exists a directed inverse system $(X^{\triv}_n\to X_n)_{n\in\bb{N}}$ of open immersions of coherent schemes poly-stable over $\overline{\eta}\to\overline{S}$ with each $X_n$ proper over $\overline{S}$ and with an isomorphism of inverse systems of $\overline{\eta}$-schemes $(Y_n)_{n\in\bb{N}}\iso (X^{\triv}_n)_{n\in\bb{N}}$, where $\overline{S}=\spec(\ca{O}_{\overline{K}})$. Notice that the condition (\ref{item:thm:completed-2}) implies that $(X^{\triv}_n\to X_n)_{n\in\bb{N}}$ satisfies the conditions in \ref{thm:vanishing} by \ref{rem:val-criterion-faltings-acyclic}, and the conclusion follows from \ref{cor:vanishing-1} immediately.
\end{proof}

\begin{myrem}\label{rem:thm:completed}
	We give some rough remarks on each assumption in \ref{thm:completed}.
	\begin{enumerate}
		\renewcommand{\labelenumi}{{\rm(\theenumi)}}
		\item If each $Y_n$ is proper, then we don't need the assumption (\ref{item:thm:completed-1}). Indeed, in this case, we take any compactification $(X_n)_{n\in\bb{N}}$ of $(Y_n)_{n\in\bb{N}}$ and consider its $(Y_n)$-modifications $(X'_n)_{n\in\bb{N}}$. All these modifications (or equivalently the Riemann-Zariski spaces) form an object of $\pro(\fal^{\mrm{open}}_{\overline{\eta}\to \overline{S}})$ whose Faltings cohomology is Riemann-Zariski faithful, since its stalks are pre-perfectoid by the assumption (\ref{item:thm:completed-2}). On the other hand, as $Y_n$ is proper, we still have Faltings' main $p$-adic comparison theorem for each (non-smooth) $X'_n$ by \cite[5.17]{he2024falmain}. Therefore, similar as \ref{thm:vanishing}, the colimit of \'etale cohomology of $Y_n$ is computed by the colimit of coherent cohomology of $X'_n$, and we deduce \eqref{eq:thm:completed-1} from Grothendieck's vanishing as in \ref{cor:vanishing-1}.\label{item:rem:thm:completed-1}
		\item If $(Y_n)_{n\in\bb{N}}$ admits a compactification $(\overline{Y_n})_{n\in\bb{N}}$ over $\overline{\eta}$ such that the inverse limit of its analytification as adic spaces is representable by a perfectoid space in the sense of \cite[2.20]{scholze2013perfsurv}, then one can check that the assumption (\ref{item:thm:completed-2}) is satisfied. Hence, any inverse system of Shimura varieties that are perfectoid in the infinite level are examples of $(Y_n)_{n\in\bb{N}}$ in \ref{thm:completed} (see \ref{para:intro-shimura}). We expect that any Shimura datum satisfies the assumption (\ref{item:thm:completed-2}).\label{item:rem:thm:completed-2}
	\end{enumerate}
\end{myrem}

\begin{mycor}\label{cor:completed}
	Let $K$ be a complete discrete valuation field extension of $\bb{Q}_p$ with perfect residue field, $\overline{K}$ an algebraic closure of $K$, $\overline{\eta}=\spec(\overline{K})$, $(Y_n\to \overline{Y_n})_{n\in\bb{N}}$ a directed inverse system of open immersions of coherent $\overline{\eta}$-schemes such that each $\overline{Y_n}$ is proper smooth and $\overline{Y_n}\setminus Y_n$ is the support of a normal crossings divisor on $\overline{Y_n}$ and that $d=\limsup_{n\to\infty}\{\dim(Y_n)\}<+\infty$. Assume that the following conditions hold:
	\begin{enumerate}
		\renewcommand{\labelenumi}{{\rm(\theenumi)}}
		\item Either the poly-stable modification conjecture {\rm\ref{conj:poly-stable}} is true, or $Y_n$ is equidimensional of dimension $1$ for any $n\in\bb{N}$.\label{item:cor:completed-1}
		\item For any point $y$ of the inverse limit $\lim_{n\in\bb{N}}Y_n$ of locally ringed spaces, its residue field $\kappa(y)$ is a pre-perfectoid field with respect to any valuation ring $V$ of height $1$ extension of $\ca{O}_{\overline{K}}$ with fraction field $V[1/p]=\kappa(y)$.\label{item:cor:completed-2}
		\item The $p$-power torsion of $\lim_{k\in\bb{N}}\colim_{n\in\bb{N}}H^d(Y_{n,\et},\bb{Z}/p^k\bb{Z})$ is bounded.\label{item:cor:completed-3}
	\end{enumerate}
	Then, there is a canonical homomorphism of $\widetilde{\overline{K}}$-modules,
	\begin{align}\label{eq:cor:completed-1}
		\colim_{n\in\bb{N}}H^d(\widetilde{\overline{Y_n}},\ca{O}_{\widetilde{\overline{Y_n}}})\longrightarrow \left(\lim_{k\in\bb{N}}\colim_{n\in\bb{N}}H^d(Y_{n,\et},\ca{O}_{\widetilde{\overline{K}}}/p^k\ca{O}_{\widetilde{\overline{K}}})\right)[1/p],
	\end{align}
	where $\widetilde{\overline{K}}$ is the maximal completion of $\overline{K}$ {\rm(\ref{rem:max-val-exist})}, $\widetilde{\overline{\eta}}=\spec(\widetilde{\overline{K}})$ and $\widetilde{\overline{Y_n}}=\widetilde{\overline{\eta}}\times_{\overline{\eta}}\overline{Y_n}$, fitting into the following commutative diagram
	\begin{align}\label{eq:cor:completed-2}
		\xymatrix{
			\colim_{n\in\bb{N}}H^d(\widetilde{\overline{Y_n}},\ca{O}_{\widetilde{\overline{Y_n}}})\ar[r]\ar@{->>}[d]& \left(\lim_{k\in\bb{N}}\colim_{n\in\bb{N}}H^d(Y_{n,\et},\ca{O}_{\widetilde{\overline{K}}}/p^k\ca{O}_{\widetilde{\overline{K}}})\right)[1/p]\\
			\bigoplus_J\widetilde{\overline{K}}\ar[r]&\widehat{\bigoplus}_J\widetilde{\overline{K}}\ar[u]_-{\wr}
		}
	\end{align}
	where $J$ is a countable set, the horizontal arrow on the bottom is induced from the $p$-adic completion map $\bigoplus_J\ca{O}_{\widetilde{\overline{K}}}\to\widehat{\bigoplus}_J\ca{O}_{\widetilde{\overline{K}}}$ by inverting $p$, the left vertical arrow is surjective and the right vertical arrow is an isomorphism. In particular, \eqref{eq:cor:completed-1} has dense image with respect to the topology induced by the $p$-adic topology of $\lim_{k\in\bb{N}}\colim_{n\in\bb{N}}H^d(Y_{n,\et},\ca{O}_{\widetilde{\overline{K}}}/p^k\ca{O}_{\widetilde{\overline{K}}})$.
\end{mycor}
\begin{proof}
	By a similar argument as in \ref{cor:polystable-model}, there exists a directed inverse system $(X^{\triv}_n\to X_n)_{n\in\bb{N}}$ of open immersions of coherent schemes poly-stable over $\overline{\eta}\to\overline{S}$ with each $X_n$ proper over $\overline{S}$ and with an isomorphism of inverse systems of open immersions of $\overline{\eta}$-schemes $(Y_n\to \overline{Y_n})_{n\in\bb{N}}\iso (X^{\triv}_n\to X_{n,\overline{\eta}})_{n\in\bb{N}}$, where $\overline{S}=\spec(\ca{O}_{\overline{K}})$. Indeed, in the construction of \eqref{eq:cor:polystable-model-1}, one can take $Z_{n+1,\overline{\eta}}=\overline{Y_{n+1}}$ and take $Z_{n+1}$ be the base change of a compactification of a Noetherian approximation of $\overline{Y_{n+1}}$. Moreover, if we denote by $\scr{C}$ the category of such inverse systems $(X^{\triv}_n\to X_n)_{n\in\bb{N}}$, then by the same arguments we see that $\scr{C}$ is cofiltered. 
	
	Notice that the condition (\ref{item:thm:completed-2}) (resp. (\ref{item:cor:completed-3})) implies that each object $(X^{\triv}_n\to X_n)_{n\in\bb{N}}$ in $\scr{C}$ satisfies the conditions in \ref{thm:vanishing} by \ref{rem:val-criterion-faltings-acyclic} (resp. the conditions in \ref{cor:vanishing-3}). Therefore, \eqref{eq:cor:vanishing-3-1} defines a homomorphism of $\widetilde{\overline{K}}$-modules 
	\begin{align}\label{eq:cor:completed-3}
		\colim_{n\in\bb{N}}H^d(\widetilde{\overline{Y_n}},\ca{O}_{\widetilde{\overline{Y_n}}})\longrightarrow \left(\lim_{k\in\bb{N}}\colim_{n\in\bb{N}}H^d(Y_{n,\et},\ca{O}_{\widetilde{\overline{K}}}/p^k\ca{O}_{\widetilde{\overline{K}}})\right)[1/p]
	\end{align}
	with the required property by \ref{rem:cor:vanishing-3}.(\ref{item:rem:cor:vanishing-3-1}). Unwinding this construction, we see that it is functorial in $\scr{C}$. Thus, we conclude that \eqref{eq:cor:completed-3} does not depend on the choice of $(X^{\triv}_n\to X_n)_{n\in\bb{N}}\in\ob(\scr{C})$ as $\scr{C}$ is cofiltered.
\end{proof}

\begin{myrem}\label{rem:cor:completed}
	We give some rough remarks on the extra assumptions in \ref{cor:completed}.
	\begin{enumerate}
		\renewcommand{\labelenumi}{{\rm(\theenumi)}}
		\item The module $\colim_{n\in\bb{N}}H^d(\widetilde{\overline{Y_n}},\ca{O}_{\widetilde{\overline{Y_n}}})$ does not depend on the smooth compactification $(\overline{Y_n})_{n\in\bb{N}}$ of $(Y_n)_{n\in\bb{N}}$. Indeed, each $\overline{Y_n}$ is smooth and thus of rational singularities (\cite[(2), page 144]{hironaka1964resolution}) so that smooth modification of $\overline{Y_n}$ does not change its cohomology. Then, the claim follows from the fact that such smooth compactifications form a cofiltered category by Hironaka's resolution of singularities (\cite[Main Theorems I and II]{hironaka1964resolution}).\label{item:rem:cor:completed-1}
		\item If $(Y_n)_{n\in\bb{N}}$ is a tower of finite \'etale Galois covers of $Y_0$ corresponding to a countable shrinking basis of open normal subgroups of a $p$-adic analytic group and if the analytification of $Y_0$ is homotopic to a finite simplicial complex, then the assumption (\ref{item:cor:completed-3}) is satisfied by \cite[1.1]{calegariemerton2012completed} and \cite[\Luoma{11}.4.4]{sga4-3} (see the proof of \cite[2.1.5]{emerton2006interpolation}). In particular, inverse systems of Shimura varieties lie in this situation, see \cite[2.2]{emerton2006interpolation}.\label{item:rem:cor:completed-2}
	\end{enumerate}
\end{myrem}

\bibliographystyle{myalpha}
\bibliography{bibli}


\end{document}